\documentclass[a4paper, reqno]{amsart} %a4paper

\usepackage[utf8]{inputenc}
\usepackage[T1]{fontenc}

\usepackage[english]{babel}

\usepackage[margin=1.2in]{geometry}
\usepackage{enumitem}
\usepackage[bookmarks = true, pagebackref]{hyperref}
\hypersetup{  %% Set up hyperref color
	colorlinks=true,
	linkcolor=blue,
	citecolor=green,
	urlcolor=magenta,
}
\usepackage[all]{xy}
\usepackage{xcolor}
\usepackage{amsmath, amsthm, amssymb, bbm, mathrsfs}
\usepackage[alphabetic]{amsrefs} 
\usepackage{unicode-math}

\usepackage{varwidth}
\usepackage{tikz-cd}
\newtheorem{thm}{Theorem}[section]
\newtheorem*{thm*}{Theorem}
\newtheorem{lem}[thm]{Lemma}
\newtheorem{prop}[thm]{Proposition}
\newtheorem{cor}[thm]{Corollary}

\theoremstyle{definition}
\newtheorem{rmk}[thm]{Remark}
\newtheorem{notation}[thm]{Notation}
\newtheorem{defn}[thm]{Definition}
\newtheorem{ex}[thm]{Example}

\newtheorem{assumption}[thm]{Assumption}
\newtheorem{construction}[thm]{Construction}

\numberwithin{equation}{section}

\newcommand{\overbar}[1]{\mkern 1.5mu\overline{\mkern-1.5mu#1\mkern-1.5mu}\mkern 1.5mu}

\newcommand{\ob}{\overbar}

\newcommand{\mr}{\mathrm}

\newcommand{\wt}{\widetilde}
\newcommand{\wh}{\widehat}

\newcommand{\Res}{\operatorname{Res}}
\newcommand{\Fil}{\operatorname{Fil}}

\newcommand{\Orb}{\operatorname{Orb}}

\newcommand{\Int}{\operatorname{Int}}

\newcommand{\red}{\mathrm{red}}
\newcommand{\rk}{\mathrm{rk}}

\newcommand{\Hom}{\operatorname{Hom}}

\newcommand{\End}{\operatorname{End}}

\newcommand{\Lie}{\operatorname{Lie}}

\newcommand{\Spec}{\operatorname{Spec}}

\newcommand{\Spf}{\operatorname{Spf}}

\newcommand{\colim}{\operatorname{colim}}

\newcommand{\GL}{\operatorname{GL}}
\newcommand{\SL}{\operatorname{SL}}

\newcommand{\id}{\mathrm{id}}

\newcommand{\Gal}{\operatorname{Gal}}

\newcommand{\dF}{\dot{\bfF}}

\newcommand{\tensor}{\otimes}

\newcommand{\iso}{\cong}

\newcommand{\mbA}{\mathbb{A}}

\newcommand{\mbC}{\mathbb{C}}

\newcommand{\mbE}{\mathbb{E}}
\newcommand{\mbF}{\mathbb{F}}

\newcommand{\mbH}{\mathbb{H}}

\newcommand{\mbL}{\mathbb{L}}
\newcommand{\mbM}{\mathbb{M}}
\newcommand{\mbN}{\mathbb{N}}

\newcommand{\mbP}{\mathbb{P}}
\newcommand{\mbQ}{\mathbb{Q}}
\newcommand{\mbR}{\mathbb{R}}

\newcommand{\mbV}{\mathbb{V}}
\newcommand{\mbW}{\mathbb{W}}
\newcommand{\mbX}{\mathbb{X}}

\newcommand{\mbZ}{\mathbb{Z}}

\newcommand{\mcA}{\mathcal{A}}
\newcommand{\mcB}{\mathcal{B}}

\newcommand{\mcD}{\mathcal{D}}
\newcommand{\mcE}{\mathcal{E}}
\newcommand{\mcF}{\mathcal{F}}
\newcommand{\mcG}{\mathcal{G}}
\newcommand{\mcH}{\mathcal{H}}
\newcommand{\mcI}{\mathcal{I}}

\newcommand{\mcL}{\mathcal{L}}
\newcommand{\mcM}{\mathcal{M}}
\newcommand{\mcN}{\mathcal{N}}
\newcommand{\mcO}{\mathcal{O}}
\newcommand{\mcP}{\mathcal{P}}
\newcommand{\mcQ}{\mathcal{Q}}

\newcommand{\mcS}{\mathcal{S}}
\newcommand{\mcT}{\mathcal{T}}
\newcommand{\mcU}{\mathcal{U}}

\newcommand{\mcW}{\mathcal{W}}
\newcommand{\mcX}{\mathcal{X}}
\newcommand{\mcY}{\mathcal{Y}}
\newcommand{\mcZ}{\mathcal{Z}}

\newcommand{\mfA}{\mathfrak{A}}

\newcommand{\mfa}{\mathfrak{a}}

\newcommand{\mfd}{\mathfrak{d}}

\newcommand{\mfm}{\mathfrak{m}}

\newcommand{\mfp}{\mathfrak{p}}

\newcommand{\bfF}{\mathbf{F}}

\newcommand{\bfM}{\mathbf{M}}

\newcommand{\bfR}{\mathbf{R}}
\newcommand{\ov}{\overline}
\newcommand{\lr}{\longrightarrow}

\newcommand{\simto}{\overset{\sim}{\to}}
\newcommand{\simlr}{\overset{\sim}{\lr}}

\newcommand{\longmapsfrom}{\longleftarrow\!\shortmid}

\newcommand{\LT}{\mr{LT}}

\newcommand{\msF}{\mathscr{F}}
\newcommand{\msL}{\mathscr{L}}

\makeatletter
\def\iddots{\mathinner{\mkern1mu\raise\p@
\vbox{\kern7\p@\hbox{.}}\mkern1mu
\raise4\p@\hbox{.}\mkern1mu\raise7\p@\hbox{.}\mkern1mu}}
\makeatother

\newenvironment{altenumerate}
	{\begin{list}
			{(\theenumi) }
			{\usecounter{enumi}
				\setlength{\labelwidth}{0pt}
				\setlength{\labelsep}{0pt}
				\setlength{\leftmargin}{0pt}
				\setlength{\itemsep}{\the\smallskipamount}
				\renewcommand{\theenumi}{\arabic{enumi}}
		}}
		{\end{list}}
\newenvironment{altenumerate3}
	{\begin{list}
			{(\theenumi)}
			{\usecounter{enumi}
				\setlength{\leftmargin}{2em}
				\setlength{\itemsep}{\the\smallskipamount}
				\renewcommand{\theenumi}{\arabic{enumi}}
		}}
		{\end{list}}

	\newenvironment{altitemize}
	{\begin{list}
			{$\bullet$}
			{\setlength{\labelwidth}{0pt}
				\setlength{\itemindent}{5pt}
				\setlength{\labelsep}{5pt}
				\setlength{\leftmargin}{0pt}
				\setlength{\itemsep}{\the\smallskipamount}
		}}
		{\end{list}}

\newcommand{\BT}{\mr{BT}}
\renewcommand{\LT}{\mcL\mcT}
\newcommand{\Hk}{\mathrm{Hk}}
\newcommand{\CM}{\mathrm{CM}}
\newcommand{\U}{\mathrm{U}}
\newcommand{\GU}{\mathrm{GU}}

\begin{document}

\title{Unitary Shimura varieties at ramified primes and arithmetic transfer}
\author{Yu Luo}
\address{Department of Mathematics, University of Wisconsin-Madison, 480 Lincoln Drive, Madison, WI 53706, USA.}
\email{yluo237@wisc.edu}
\author{Andreas Mihatsch}
\address{School of Mathematical Sciences, Zhejiang University, 866 Yuhangtang Rd, Hangzhou, 310058, P.~R.~China.}
\email{mihatsch@zju.edu.cn}
\author{Zhiyu Zhang}
\address{Department of Mathematics, Stanford University, Building 380, Stanford, CA 94305, USA.}
\email{zyuzhang@stanford.edu}
\date{\today}

\begin{abstract}
We consider unitary Shimura varieties at places where the totally real field ramifies over $\mbQ$. Our first result constructs comparison isomorphisms between absolute and relative local models in this context, which relies on a reformulation of the Eisenstein condition of Rapoport--Zink and Rapoport--Smithling--Zhang. Related to that, we also provide a moduli description for the integral models of RSZ unitary Shimura varieties in new cases. Our second result lifts the comparison of local models to categories of $p$-divisible groups and, as a corollary, to various kinds of Rapoport--Zink spaces. Our third result is a proof of the arithmetic transfer conjecture of the third author in full generality. Using our statements about Rapoport--Zink spaces, we extend the previous proof from the unramified case to that of all $p$-adic local fields ($p$ odd).
\end{abstract}

\maketitle

\tableofcontents

\section{Introduction}\label{s:intro}

In this article, we prove the arithmetic transfer conjecture of the third author \cite{ZZhang21} in full generality. To this end, we give a new definition of integral models of unitary Shimura varieties at places that ramify over $\mbQ$. Our main results concern the local structure of these models, and ultimately allow to remove the ramification restrictions that were present in \cite{ZZhang21} before. Let us first describe the broader context.

\medskip The arithmetic Gan--Gross--Prasad (AGGP) conjecture states a Gross--Zagier type relation between algebraic cycles on Shimura varities and $L$-functions. In \cite{Zhang12}, Wei Zhang proposed a relative trace formula approach to the AGGP conjecture. An important component of this is his arithmetic fundamental lemma (AFL) which states an identity between intersection numbers on unitary Rapoport--Zink spaces (integral models of local Shimura varieties) and derivatives of orbital integrals. It applies at all unramified places with good reduction and can be thought of as a local analogue of the AGGP conjecture. The AFL was first proved over $\mbQ_p$ with $p > n$ by W. Zhang in \cite{Zhang21}, then extended to general local fields jointly with the second author \cite{MZ}, and finally to small primes by the third author \cite{ZZhang21}. We refer to the introductions of these articles for more background.

In order to obtain general global results, it is also necessary to understand the contributions from places with bad reduction. One still hopes to define local intersection numbers at such places and to relate them to derivatives of orbital integrals. This is called \emph{arithmetic transfer} and constitutes an analog of the smooth transfer of test functions in the context of trace formula comparisons in automorphic representation theory. Arithmetic transfer was introduced by Rapoport--Smithling--Zhang in the series \cites{RSZ-Duke,RSZ-Annalen,RSZ-AGGP} and is also the topic of work of Li--Rapoport--Zhang \cite{CRZ-quasiAFL} and joint work of one of us \cite{LRZ}. In these articles, the authors formulate arithmetic transfer conjectures in a variety of situations and prove them in some cases.

One of us \cite{ZZhang21} took up and extended the conjecture about the unramified case from \cite[\S10]{RSZ-Annalen}. Concretely, let $F/F_0$ be the quadratic extension used to define the unitary Shimura variety in question. His extended conjecture applies at all places of $F_0$ that are inert in $F$ and where the level is maximal parahoric. He moreover proves his conjecture at all places of $F_0$ that are unramified over $\mbQ$, the reason for this restriction being that the integral models of unitary Shimura varieties at places that ramify over $\mbQ$ were not sufficiently well understood. The original goal of our paper was to remove this restriction.

\medskip Motivated by the above problem, our main results provide a systematic study of unitary Shimura varieties at places of $F_0$ that are ramified over $\mbQ$. The situation can be described as follows. Let $V$ be a hermitian $F$-vector space and let $v$ be a non-archimedean place of $F_0$. Consider the local unitary group $\mr{U}(V_v)$ over $F_{0,v}$, the conjugacy class of minuscule cocharacters $μ$ coming from the Shimura datum, and a local level subgroup $K_v\subset \mr{U}(V_v)$. Let us assume that this local situation is well-understood in the sense that one can describe the local model and the Rapoport--Zink spaces in question. For example, $v$ might be inert in $F$, $μ$ might come from a signature $(n-1,1)$ setting, and $K_v$ might be maximal parahoric as in the arithmetic transfer conjecture of \cite{ZZhang21}.

The issue is now that the theory for $\mr{U}(V_v)$ does not directly apply to the Shimura variety because this variety is (essentially) attached to the Weil restriction group $\mr{Res}_{F_0/\mbQ} \mr{U}(V)$. If $v$ is ramified over $\mbQ$, then the Weil restriction picks up that ramification, leading to complications in local model theory and when formulating local or global moduli problems.

In a sense, this phenomenon should be considered an artifact of the Shimura variety formalism, and one can hope to be able to relate the geometry for $\mr{Res}_{F_0/\mbQ} \mr{U}(V)$ at a prime $p$ to the geometry of all $\mr{U}(V_v)$, where $v\mid p$. This is a non-trivial problem in general. For local models, a systematic study was begun by Pappas--Rapoport \cite{PRII} in terms of the splitting model. For Rapoport--Zink spaces, the problem was first studied by Rapoport--Zink \cite{RZ2} for Drinfeld moduli spaces. In our paper here, we provide a complete solution for the case that occurs frequently in arithmetic intersection theory: the case when $V$ is definite at all but one archimedean place.

To be more precise, the Shimura varieties we consider in the body of the paper are the RSZ unitary Shimura varieties introduced by Rapoport--Smithling--Zhang \cite{RSZ}. They are of particular interest because they are the natural hosting spaces for the special cycles of Kudla--Rapoport \cite{KR-global}. In addition to the AFL mentioned above, several further striking results about these cycles have been obtained recently, such as the modularity of their generating series by Bruinier--Howard--Kudla--Rapoport--Yang \cite{BHKRY}; arithmetic Siegel--Weil formulas by Li--Zhang \cites{Li-Zhang}, Li--Liu \cite{Li-Liu-Pi}, He--Li--Shi--Yang \cite{HLSY}, H. Yao \cite{YaoHaodong}, and the first author \cite{Luo-KR}; an arithmetic inner product formula due to Li--Liu \cites{Li-Liu-Annals, Li-Liu-Pi}; or the twisted AFL of the last author \cite{zhang2025non}. Our results have applications in the contexts of these works as well for which we give some examples in \S\ref{sec:furtherapplications} below.

\subsection{Overview of results}
\label{ss:intro_summary}

Let $F / F_0 $ be a quadratic CM extension of a totally real number field, and let $V$ be a hermitian $F$-vector space that has signature $(n-1,1)$ at one infinite place and is definite at all others. Let $K\subset \mr{U}(V)(\mbA_f)$ be a neat open compact subgroup. Given these data and a CM type for $F$, there is an RSZ type unitary Shimura variety $M_K$ of level $K$ over the reflex field $E$. These objects will be introduced in detail in \S\ref{sec:RSZ}. Our main results are as follows.

\medskip \noindent (R1) We provide a moduli description for the integral model $\mcM_{K,p}/\Spec O_{E,p}$ of $M_K$ at a prime $p$ in all situations of parahoric level structure. We similarly provide moduli descriptions for the local model and the involved Rapoport--Zink spaces. These description rely on a reformulation of the Eisenstein condition from \cite{RZ2} and \cite{RSZ-AGGP}. Prior to this work, such moduli descriptions were only known under assumptions on the ramification behavior of $p$ in $F$, see \cite[\S5--6]{RSZ}.

\medskip \noindent (R2) We carry out the comparison in local geometry between $\mr{Res}_{F_0/\mbQ}\mr{U}(V)$ and $\mr{U}(V)$ alluded to above. That is, we simplify the moduli descriptions for local models and Rapoport--Zink spaces from (R1) to more accessible ones that only involve the local group-theoretic data. For example, these local moduli spaces do not depend on the CM type that goes into the global definition.

\medskip \noindent (R3) We prove the arithmetic transfer conjecture for maximal parahoric level of the third author in full generality. We also prove new related almost modularity results on arithmetic theta series of Kudla--Rapoport divisors. 

\medskip Having a moduli description as in (R1) is useful beyond just giving an explicit definition of $\mcM_{K,p}$. Namely, it also allows to define integral models of special cycles on $\mcM_{K,p}$ as spaces of homomorphisms/endomorphisms involving the universal abelian scheme. For example, this is how we define the global cycles that come up in (R3).

There are also abstract constructions of integral models of parahoric level abelian type Shimura varieties, see the work of Kisin--Pappas \cite{Kisin-Pappas2018} and Kisin--Pappas--Zhou \cite{Kisin-Pappas-Zhou}. Moreover, Madapusi \cite{madapusi} has recently proposed a definition of special cycles for smooth integral models in the Hodge case which relies on $p$-adic Hodge theory. The advantage of such definitions is that they are group-theoretic in nature, and they will probably play an important role in extending the study of special cycles to other contexts. However, the study of moduli problems remains the more direct approach in the unitary case, and moreover also applies in ramified situations such as \cite{Li-Liu-Pi,HLSY,HLSsplitting,LRZ}.

A related aspect is that intersection problems on integral models often require a partial resolution of singularities. An example is the class of splitting models which is studied in \cite{Kramer,ZZ-Splitting,HLSsplitting} and which has found applications in \cite{HLSY} and \cite{LRZ}. In fact, the formulation of the arithmetic transfer conjecture we prove in (R3) also relies on an explicit partial resolution of singularities of a product of two semi-stable local models (see \S\ref{sec:global_intersection_numbers}). There is currently no known group-theoretic construction of such ``non-standard'' local models.

\medskip Let us also explain the relation of our results about $p$-divisible groups in (R2) with the theory developed by Scholze--Weinstein \cite{SW}. A major result of this book are two comparison isomorphisms between Rapoport--Zink spaces \cite[\S25.3, \S25.4]{SW} which rely on a reinterpretation of Rapoport--Zink spaces as moduli spaces of local shtukas. This method is group-theoretic and should also apply to the Rapoport--Zink spaces of our paper. By looking at generic fibers, it should moreover extend to a comparison of special cycles whenever these cycles are flat over $\mbZ_p$. However, this flatness is often not given. For example, the Kudla--Rapoport divisors we consider in \S\ref{sec:algebraiccycles} and \S\ref{sec:AT-conj} are usually not flat.

Our approach is different and yields stronger results. Namely, our main result is an equivalence between two categories of $p$-divisible groups with additional data. This equivalence holds in complete generality\footnote{Except that we need to restrict to biformal groups when including compatibility with duality, see Theorem \ref{thm:intro_PEL}.} in the sense that the base might be non-reduced, non-noetherian, and that the family of $p$-divisible groups need not be isogeny-trivial like over a Rapoport--Zink space. The comparison isomorphisms of Rapoport--Zink spaces follow as immediate corollaries of the categorical statement. It will moreover be clear that they are compatible with the definitions of special cycles and of other notions like Kottwitz--Rapoport strata.

In the next three sections, we give more details about our results and techniques.

\subsection{Local models}
\label{ss:intro_loc_models}
We are interested in the geometry of the Shimura variety $M_K$ at a $p$-adic place $ν$ of the reflex field. The global starting data involve a quadratic field extension, a CM type, and a hermitian space $V$ with signature $(r,s)$ at a distinguished archimedean place, and signature $(n,0)$ or $(0,n)$ at the others. This leads to the following local situation.

Taking $p$-adic completions of the global fields, we consider a finite extension $F_0/\mbQ_p$ and an étale quadratic extension $F/F_0$. Let $x\mapsto \ov x$ denote the Galois conjugation. The signatures of $V$ and the CM type define an embedding $ϕ_0:F\to \ov{\mbQ_p}$ and a subset $A\subset \Hom(F, \ov{\mbQ_p})$ such that there is a disjoint union decomposition
\begin{equation}
\Hom(F, \ov{\mbQ_p}) = A \sqcup \{ϕ_0, \ov{ϕ_0}\}\sqcup \ov A.
\end{equation}
Let us set $B = A\sqcup \{ϕ_0, \ov{ϕ_0}\}$. The pair $(A, B)$ gives rise to a local reflex field $E$. Moreover, Howard \cite{H} defined certain ideals
\begin{equation}
J_B\subset J_A \subset O_F\tensor_{\mbZ_p} O_E.
\end{equation}
We use these ideals to impose the following condition on the Lie algebra of the abelian variety or the $p$-divisible groups in our moduli problems.
\begin{defn}[$(A,B)$-strictness, see Definition \ref{def:AB-strict}]\label{def:intro_AB_strict}
Let $R$ be an $O_E$-algebra. We say that an $(O_F\tensor_{\mbZ_p} R)$-submodule $\mcF\subseteq (O_F\tensor_{\mbZ_p} R)^n$ is \emph{$(A, B)$-strict} if it is an $R$-module direct summand and satisfies
\begin{equation}\label{eq:intro_AB_strict}
J_B(O_F\tensor_{\mbZ_p} R)^n \subseteq \mcF \subseteq J_A(O_F\tensor_{\mbZ_p} R)^n.
\end{equation}
\end{defn}
The name of this condition comes from the fact that it is a generalization of the \emph{strictness} condition of Drinfeld (see \S\ref{ss:intro_p-div}). When $F/\mbQ_p$ is unramified, then \eqref{eq:intro_AB_strict} is equivalent to a Kottwitz determinant condition, but it is strictly stronger in ramified situations. In fact, Definition \ref{def:intro_AB_strict} is precisely equivalent to the Eisenstein condition of Rapoport--Zink \cite{RZ2} and Rapoport--Smithling--Zhang \cite{RSZ-AGGP}, see Proposition \ref{prop:RZ-eisenstein}, while the formulation in terms of \eqref{eq:intro_AB_strict} directly relates to the splitting model from Pappas--Rapoport \cite{PRII}.

The original global data also give rise to an $n$-dimensional hermitian $F$-vector space. Assume that $L$ is a self-dual lattice chain in $V$ and that $0\leq d\leq n$. Rapoport--Zink \cite{RZ1} attach to these data a local model $\bfM^L$ for the unitary group $\mr{U}(V)$ viewed as algebraic group over $F_0$. (Of course, if $F = F_0\times F_0$, then one is really working with $\GL_{n,F_0}$ and the definitions simplify. But we will stick to the uniform terminology for a moment.)

As explained before, the global setting leads one to additionally consider the $\mbQ_p$-algebraic group $\mr{Res}_{F_0/\mbQ_p}\mr{U}(V)$. Together with Definition \ref{def:intro_AB_strict}, the formalism from \cite{RZ1} also defines a variant $\bfM_A^L$ of $\bfM^L$ that parametrizes certain $(A,B)$-strict filtrations. Leaving the precise definitions to \S\ref{s:LM}, our main result about the two local models reads as follows:

\begin{thm}[\protect{see Theorems \ref{thm:local_model_comparison} and \ref{thm:LM-comparison}}]\label{thm:intro_LM}
There is a natural isomorphism
\begin{equation}\label{eq:intro_main_LM}
\bfM^L_A \simlr O_E\tensor_{O_F} \bfM^L.
\end{equation}
\end{thm}

The construction of this isomorphism is straightforward and based on the observation that $J_A/J_B \simto O_F\tensor_{O_{F_0}} O_E$. This observation was already made in \cite{PRII}. In a sense, Theorem \ref{thm:intro_LM} explains the remarkable effectiveness of the Eisenstein condition.

In the case $F = F_0\times F_0$, Theorem \ref{thm:intro_LM} extends to inner forms of $\GL_{n,F}$, meaning unit groups of central simple algebras. This generalizes some aspects of \cite{RZ2}, see Corollary \ref{cor:loc_mod_CSA} and Remark \ref{rmk:rel_to_RZ}.  

Theorem \ref{thm:intro_LM} is also related to \cite[Corollary 2.3]{TT23} which states that the local model essentially only depends on the adjoint datum. Namely, fixing $n$, $d$ and $L$, the adjoint local model triple that implicitly underlies $\bfM^L_A$ is independent of $A$. Then \eqref{eq:intro_main_LM} makes this independence explicit.

\subsection{$p$-Divisible groups}\label{ss:intro_p-div}
Let us first consider the simpler $\GL_n$-setting by considering only a field extension $F/\mbQ_p$, an embedding $ϕ_0:F\to \ov{\mbQ_p}$, and a subset $A\subset \Hom(F, \ov{\mbQ_p})$ with $ϕ_0\notin A$. Set $B = A\sqcup \{ϕ_0\}$. The condition of $(A,B)$-strictness from \eqref{eq:intro_AB_strict} still makes sense as written.

\begin{defn}\label{def:intro_AB_strict_p_div}
Let $X$ be a $p$-divisible group over an $O_E$-algebra $R$ together with an $O_F$-action $ι:O_F\to\End(X)$. We call $(X,ι)$ an \emph{$(A,B)$-strict pair} if the Hodge filtration of $X$ with its $O_F$-action via $ι$ is $(A,B)$-strict in the sense of Definition \ref{def:intro_AB_strict}.
\end{defn}

Note the following special case: if $A = \emptyset$ and $B = \{ϕ_0\}$, then $ϕ_0:F\simto E$ and $(A, B)$-strict pairs are literally the same as $p$-divisible $O_F$-modules, meaning $p$-divisible groups with strict $O_F$-action in the sense of Drinfeld. Our main result shows that this is a general phenomenon:

\begin{thm}[EL case, \protect{see Theorem \ref{thm:AB_strict_p_div}}]\label{thm:intro_EL}
Assume that $A$ and $B$ are of the form $B = A\sqcup \{ϕ_0\}$. There is an $O_F$-linear equivalence of fibered categories over $\Spf(O_E)$-schemes
\begin{equation}\label{eq:equiv_intro_EL}
Φ_A:\left\{\text{\begin{varwidth}{\textwidth}\centering $(A, B)$-strict pairs\end{varwidth}}\right\}\simlr \left\{\text{\begin{varwidth}{\textwidth}\centering $p$-divisible $O_F$-modules\end{varwidth}}\right\}.
\end{equation}
\end{thm}

This is the equivalence of categories mentioned in \S\ref{ss:intro_summary}. Let us come back to the case of a quadratic extension $F/F_0$ and subsets $A\subset B\subset \Hom(F, \ov{\mbQ_p})$ of the form $B = A\sqcup \{ϕ_0,\ov{ϕ_0}\}$. The construction of $Φ_A$ in \eqref{eq:equiv_intro_EL} can be carried out in this setting as well, and we prove that it is compatible with dualization. For technical reasons, however, we need to restrict to \emph{biformal} $p$-divisible $O_{F_0}$-modules in the sense that the module and its $O_{F_0}$-dual are both formal.

\begin{thm}[PEL case, \protect{see Theorem \ref{thm:AB_strict_p_div_PEL}}]\label{thm:intro_PEL}
Assume that $A$ and $B$ are of the form $B = A\sqcup \{ϕ_0, \ov{ϕ_0}\}$. There is an $O_F$-linear equivalence of fibered categories over $\Spf(O_E)$-schemes, compatible with dualization,
\begin{equation}\label{eq:equiv_intro_PEL}
Φ_A:\left\{\text{\begin{varwidth}{\textwidth}\centering relatively biformal\\$(A, B)$-strict pairs\end{varwidth}}\right\}\simlr \left\{\text{\begin{varwidth}{\textwidth}\centering biformal $p$-divisible\\$O_{F_0}$-modules\end{varwidth}}\right\}.
\end{equation}
Here, duality on the source is Serre duality for $p$-divisible groups; on the target it is Faltings duality for $p$-divisible $O_{F_0}$-modules.
\end{thm}

Specific cases of Theorems \ref{thm:intro_EL} and \ref{thm:intro_PEL} were already obtained by Rapoport--Zink \cite{RZ2}, the second author \cite{Mih}, S. Cho \cite{Cho18}, Li--Liu \cite{Li-Liu-Pi}, Y. Shi \cite{Shi-ramified}, and Kudla--Rapoport--Zink \cite{KRZ}. Among these, the most relevant for us is the article \cite{KRZ} where both theorems are obtained when $n = 2$. Our proof is a uniform version of their arguments that applies in all cases. In particular, it again relies on the theory of relative displays as developed by Zink \cite{Zink}, Lau \cite{Lau}, and Ahsendorf--Cheng--Zink \cite{ACZ}. Moreover, we make use of the results on the Ahsendorf functor in \cite[\S3]{KRZ}. We refer to Theorems \ref{thm:AB_equiv_displays} and \ref{thm:comp-display} for our two main intermediate statements.  

\medskip One of our improvements over previous results is that we give a direct construction of a quasi-inverse for $Φ_A$ by a double-dualization argument. In this way, we circumvent deformation theory arguments and do not need to restrict to rings with nilpotent nilradical.

Another improvement is that we work with the actual Faltings dual of $p$-divisible $O$-modules. In previous works, the dual of a $p$-divisible $O$-module was defined by transport of structure from relative displays, see Remark \ref{rmk:the_real_dual}. We close this gap in our appendix where we show that the functor $\mr{BT}_O$ from binilpotent relative displays to biformal $p$-divisible $O$-modules is compatible with duality (Proposition \ref{prop:Faltings_display_compatible}). The construction of the corresponding natural transformation, a key input, is taken from Zink \cite{Zink}. Our argument for why it is an isomorphism is elementary and avoids the cohomological formalism of biextensions. It can thus also be viewed to simplify the exposition in \cite[\S4]{Zink}.

\medskip Coming back to the general discussion, it is clear that Theorems \ref{thm:intro_EL} and \ref{thm:intro_PEL} have immediate applications to Rapoport--Zink spaces. Loosely speaking, they imply comparison isomorphisms between $(A,B)$-strict and strict (P)EL type moduli spaces in a variety of parahoric level situations. In particular, we give a direct construction of the comparison isomorphism of Drinfeld moduli problems from \cite{RZ2} and \cite[\S25.4]{SW}, see Corollaries \ref{cor:CDA} and \ref{cor:Drinfeld_comparison}.

The relevant situation for our applications to arithmetic transfer is the ``unramified maximal parahoric level case''. That is, $F/F_0$ is an unramified quadratic field extension and the parahoric level is defined by a lattice $L\subset V$ with $πL^\vee\subseteq L \subseteq L^\vee$. Set $t = [L^\vee:L]$. In this case, Corollary \ref{cor:unitary-comp} provides an equivariant isomorphism of Rapoport--Zink (RZ) spaces
\begin{equation}\label{eq:intro_unit_comp}
\mcN^{[t]}_{A, (n-1,1)} \simlr O_{\breve E}\wh{\tensor}_{O_{\breve F}} \mcN^{[t]}_{(n-1,1)}.
\end{equation}
Here, the left hand side is the ``complicated'' moduli space of $(A,B)$-strict objects that occurs naturally as the uniformizing space of the global moduli problem. The right hand side is the simplified moduli space that occurs in the formulation of the arithmetic transfer theorem.

\subsection{Arithmetic transfer}
\label{ss:intro_AT}

Let us first state our arithmetic transfer theorem on a qualitative level. Fix a prime number $p > 2$ and an unramified quadratic extension $F/F_0$ of $p$-adic local fields. Let $π\in O_{F_0}$ denote a uniformizer.

Let $V$ be an $F/F_0$-hermitian space and let $L\subset V$ be a vertex lattice. Recall that this means that $πL^\vee \subseteq L\subseteq L^\vee$ where $L^\vee$ is our notation for the dual lattice. The integer $t := [L^\vee:L]$ is called the type of $L$ and determines $L$ up to isomorphism. In this setting, we can define two types of $p$-adic quantities (see \S\ref{sec:AT conj} for their precise definitions):
\begin{altenumerate3}
\item On the orbital integral side, we consider the symmetric space
\begin{equation*}
S_n := \{\gamma\in \Res_{F/F_0} \GL_n \mid \gamma\ov{\gamma}=\id\}.
\end{equation*}
Attached to $L$, there are two natural test functions $f_L$ and $f_{L^\vee}$ on the product space $$S_n(F_0) \times F_0^n \times \Hom(F_0^n, F_0).$$
The first type of quantities are the special values of the first derivative of their $\GL_n$-orbital integrals,
$$\partial\!\Orb((γ, u_1, u_2), f_L),\quad \partial\!\Orb((γ, u_1, u_2), f_{L^\vee})\ \ \in\ \mbZ\cdot \log(q).$$
\item Attached to $L$, there is also a unitary RZ space $\mcN^{[t]}_{(n-1,1)}$. It is defined for the unitary group $\U(V)$, the basic isogeny class, and the level group $\U(L)$. The lattices $L$ and $L^\vee$ correspond to the $\mcZ$ and $\mcY$-families of Kudla--Rapoport divisors. The second type of quantity are arithmetic intersection numbers of these divisors against (derived) complex multiplication cycles
$$\wt {\Int}{}^{\mcZ}(g,u),\quad \wt{\Int}{}^\mcY(g,u)\ \ \in\ \mbZ.$$
\end{altenumerate3}
There is a notion of matching for the parameters $(γ, u_1, u_2)$ and $(g,u)$ which comes from the relative trace formula comparison of Jacquet--Rallis \cite{JR}. More precisely, we use its semi-Lie variant which was introduced by Y. Liu \cite{liu2021fourier}. Our main result is the following theorem.

\begin{thm}[\protect{\cite[Conjecture 6.4]{ZZhang21}}, see Theorem \ref{thm: AT vertex level}]\label{thm:AT_intro}
For every pair of matching regular semi-simple elements $(γ, u_1, u_2)$ and $(g,u)$, there are the identities
\begin{equation}
\begin{array}{rcr}
\partial\!\Orb ((\gamma,u_1,u_2),  f_{L}) & = & - \wt {\Int}{}^{\mcZ}(g, u) \log q,\\[1mm]
\partial\!\Orb ((\gamma,u_1,u_2),  f_{L^\vee}) & = & - (-1)^{t} \wt {\Int}{}^{\mcY}(g, u) \log q.
\end{array}
\end{equation}
\end{thm}
In \cite{ZZhang21}, this theorem was proved whenever $F_0$ is unramified over $\mbQ_p$. In the present paper, we remove this assumption and complete the proof in general.

Our proof of Theorem \ref{thm:AT_intro} relies on the techniques developed in \cite{ZZhang21}, most notably the ``double-modularity'' technique and the ``simple modification'' technique. (These were used to extend the proof of the arithmetic fundamental lemma from \cites{Zhang21, MZ} to the bad reduction situation.) An overview of this strategy can be found in \cite{zhang2025non}. The improvements of the present paper come on top and are as follows:
\begin{altenumerate}
\item In general, the proof in \cite{ZZhang21} is by globalizing both types of local quantities to global ones. On the intersection number side, this used the parahoric level non-archimedean uniformization results of S. Cho \cite{Cho18} which are formulated for $F/\mbQ_p$ unramified. In light of Theorem \ref{thm:intro_PEL}, especially the isomorphism \eqref{eq:intro_unit_comp}, we can carry out this step for all $F$.

\item However, our global integral model is usually not regular at ramified primes. The reason is obvious from \eqref{eq:intro_main_LM}: While the space $\mcN^{[t]}_{(n-1,1)}$ and its local model $\bfM^L$ are regular with semi-stable reduction, the local model for the Shimura variety is a base change as in \eqref{eq:intro_main_LM}. Thus the regularity is lost whenever $E/F$ is ramified (Remark \ref{rmk:non-regular}). In fact, \eqref{eq:intro_main_LM} only describes the base change to the local reflex field. The completion of the global reflex field, which contains the local reflex field, might involve further ramification.

For this reason, we prove in Theorem \ref{prop:diag-lci} that the global cycles in question are local complete intersections. This allows to use the Gysin map or derived tensor products to define intersection numbers in the non-regular context.

Actually, if one is only interested in proving Theorem \ref{thm:AT_intro}, then it would be possible to make a careful choice of global data that allows to work with a regular integral model at $p$. Our additional motivation for still including a discussion of the non-regular case is that it also extends the applicability of Theorem \ref{thm:AT_intro} to new global settings.

\item Recall that \cites{MZ, ZZhang21} are based on the construction of an admissible pairing between divisors and $0$-cycles on Shimura varieties. Its construction requires the $0$-cycles to be of degree $0$ on all connected components of the Shimura variety. This was a further reason for restricting to special local fields in our (Z.~Z.) previous work, compare \cite[Corollary 12.9, Conjecture 12.10]{ZZhang21}. In our proof of Theorem \ref{thm: AT vertex level}, we realize that the Shimura variety for $\U(V)$ is, in fact, connected as a variety over the reflex field, simplifying the construction of $0$-cycles of degree $0$. 
\end{altenumerate}

We mention that similar remarks and techniques also apply in the work of the first author \cite{Luo-KR}. There, he uses a global strategy to give a new proof of the Kudla--Rapoport conjecture \cite{KudlaRapopot-local}, previously proved by Li--Zhang \cite{Li-Zhang}. This proof will also cover the unramified maximal parahoric level structure case for both $\mcY$ and $\mcZ$-cycles, generalizing work of S. Cho \cite{Cho-Math-Ann} and Cho--He and the third author \cite{CHZ-Kudla}.

\subsection{Further applications}\label{sec:furtherapplications}

We end this introduction by listing potential further applications we are aware of.

\medskip \noindent {\bf Construction of flat integral models.} In \S\ref{sec:Shimura}-\S\ref{sec:AT-conj}, we will restrict to integral models of RSZ Shimura varieties at places where $F/F_0$ is unramified because this is the case of interest for the arithmetic transfer theorem (Theorem \ref{thm: AT vertex level}). However, Theorem \ref{thm:intro_LM} also allows to construct explicit flat integral models in some new ramified situations. For example, it shows that the conditions on ramification over $\mbQ$ in (3), (4) and (5) of \cite[\S5.2]{RSZ} are not necessary: after imposing the Eisenstein condition, one can require the strengthened spin condition for the direct summand $J_A\Lie(\mcA)$, see Corollary \ref{cor:LM-comparison} (3). This also answers the question posed in \cite[Remark 5.5]{RSZ}. The same logic applies to the Krämer model \cite{Kramer} (see Remark \ref{rmk:kramer model}) or, more generally, to the splitting models in \cites{HLSsplitting,ZZ-Splitting}.

\medskip \noindent {\bf Drinfeld level structure.} In \cite[\S4.3]{RSZ-AGGP} and \cite[Remark 5.6]{RSZ}, Rapoport--Smithling--Zhang introduce Drinfeld level structure at split places in the definition of RSZ Shimura varieties. This requires the existence of a suitable $1$-dimensional $p$-divisible group over the special fiber of the integral model at hyperspecial level. In \cites{RSZ-AGGP, RSZ}, this $p$-divisible group is obtained after imposing a matching condition for $p$, $F$, and the generalized CM type in question. With this condition, the $p$-divisible group of the universal abelian scheme has a suitable such $1$-dimensional factor.

Using Theorem \ref{thm:intro_EL}, this definition can most likely be extended to all split places. Namely, the $p$-divisible group of the universal abelian variety at such places has a canonical $(A,B)$-strict factor. Applying Theorem \ref{thm:intro_EL} constructs the desired $1$-dimensional $p$-divisible group over the $p$-adic completion of the integral model.

\medskip \noindent {\bf Globalization of local results.} Our proof of the arithmetic transfer conjecture is of the kind ``proof of a local results by globalization''. Conversely, our results allow the application of local results in the global context in new cases. For example, the restriction on ramification over $p$ in (G2) and (G4) for the arithmetic Siegel--Weil formula in \cite{Li-Zhang} can be removed.

Similarly, the ramification assumption over $p$ in \cite[(G2)]{HLSY} can be dropped. Moreover, the condition on the CM type in \cite[(G1)]{HLSY} can be removed; it is an artifact of the implicit usage of \cite{Mih} during uniformization. Moreover, in light of Remark \ref{rmk:kramer model}, the ramification condition on $v$ over $p$ in the Krämer model situation \cite[(G3)]{HLSY} can probably also be removed.

Some similar applications should be possible in the context of the arithmetic inner product formula \cite{Li-Liu-Annals, Li-Liu-Pi, Disegni-Liu-Invent} and the recent work on the Beilinson--Bloch--Kato conjecture for Rankin--Selberg motives \cite{LTXZZ}. Especially in situations where the relative local model is smooth (hyperspecial case, exotic smooth case), Theorem \ref{thm:intro_LM} provides smooth absolute local models. In particular, the height pairings used in, for example, \cites{Li-Liu-Annals, Li-Liu-Pi, Disegni-Liu-Invent} are still given by the intersection pairing on the integral model.

\subsection{Structure of the article}

The first part of the article contains all results that apply at a general parahoric level place of an RSZ unitary Shimura variety. After recalling the Eisenstein condition in \S2, we compare $(A,B)$-strict with strict objects in the context of local models (\S3), displays (\S4) and Rapoport--Zink spaces (\S5). We then formulate the global moduli problem and explain its non-archimedean uniformization in terms of relative Rapoport--Zink spaces (\S6).

The second part of the article contains all material that is more specific to the arithmetic transfer theorem. In particular, we define Kottwitz--Rapoport strata (\S7) and global algebraic cycles (\S8) for the moduli spaces introduced in Part 1. The formulation and proof of the arithmetic transfer theorem is contained in \S9.

Finally, in the appendix, we recall the definition of the Faltings dual of a strict $O$-module and prove its compatibility with the Ahsendorf functor. 

\subsection{Acknowledgement}
We heartily thank Michael Rapoport for his interest and comments on an earlier version of our article.

This work was started at IHES in 2022. A.~M. and Z.~Z. would like to thank IHES for its hospitality. Y.~L. would like to thank the Morningside Center of Mathematics at the Chinese academy of sciences for its hospitality during Summer 2024 when part of this work was done. Y.~L. and Z.~Z. would like to thank all organizers of the Oberwolfach Seminar ``Reduction of Arithmetic Varieties'' and the hospitality of the Mathematisches Forschungsinstitut Oberwolfach where part of this work was done. 

\part{Unitary Shimura varieties}

\section{The Eisenstein condition}\label{s:Eis cond}
In this section, we study the local Eisenstein condition from \cites{RZ2, RSZ-AGGP}. Our approach is based on an equivalent formulation that is inspired by \cite{H} and that can also be understood as a special case of the splitting model definition in \cite{PRII}.

Let $F/K$ be a finite extension of $p$-adic local fields. We stick to the ``Eisenstein'' terminology from \cite{RZ2} even though we do not assume $F/K$ to be totally ramified. We allow all residue characteristics, including $p = 2$.

\subsection{Eisenstein ideals}
For a subset $A\subseteq \Hom_K(F, \ov{K})$ and an element $ζ\in O_F$, we define the monic polynomial 
\begin{equation}\label{eq:def_eisenstein_element}
e_{ζ,A}(T) := \prod_{ϕ\in A} (T-ϕ(ζ)) \in \ov K[T].
\end{equation}
Let $E\subset \ov{K}$ be a finite extension of $K$ such that $\mr{Gal}(\ov{K}/E)$ fixes $A$. The minimal such $E$ is called the \emph{reflex field of $A$} and denoted by $E_A$. Tautologically, it can be described by
\begin{equation}\label{eq:reflex_E_A}
E_A = \ov K^{\, \{σ\in \mr{Gal}(\ov K/K)\,\mid\, σ\circ A = A\}}.
\end{equation}
Then $e_{ζ,A}(T)$ lies in $O_E[T]$. The following lemma has occurred in various forms before, see for example \cite[Proposition 2.2.1]{KRZ} or \cite[Lemma 6.11]{RSZ}.

\begin{lem}\label{lem:indep_of_generator}
Let $ζ_1,ζ_2\in O_F$ be two elements with $O_K[ζ_1] = O_K[ζ_2]$. Then the two elements
$$e_{ζ_1,A}(ζ_1\tensor 1),\ e_{ζ_2,A}(ζ_2\tensor 1) \in O_F\tensor_{O_K} O_E$$
differ by a unit.
\end{lem}
\begin{proof}
For every polynomial $P(T)\in O_K[T]$, and every embedding $ϕ\in\Hom_K(F,\ov{K})$, the polynomial $T-ϕ(ζ_1)$ divides $P(T) - P(ϕ(ζ_1))$. Taking the product over $ϕ\in A$ we find
$$e_{ζ_1,A}(T) \mid e_{P(ζ_1), A}(P(T)).$$
Apply this with some polynomial $P$ such that $ζ_2 = P(ζ_1)$. It follows that
$$e_{ζ_1,A}(ζ_1\tensor 1) \mid e_{ζ_2,A}(ζ_2\tensor 1).$$
The converse divisibility follows in the same way by symmetry.

\begin{lem}\label{lem:unit_local_ring}
Let $A$ be a local ring with maximal ideal $\mfm$. Let $x,y\in A$ be two elements with $x\mid y$ and $y\mid x$. Then there exists a unit $u$ such that $y = ux$.
\end{lem}
\begin{proof}
Let $u,v\in A$ be two elements with $y = ux$ and $x = vy$. Then $x = uvx$. There is nothing to show if $uv\in A^\times$, so assume $uv \in \mfm$. Then we obtain $\mfm\cdot (x) = (x)$ which by Nakayama's Lemma implies $x = 0$. Since $x\mid y$, also $y = 0$. So in this case, too, $x$ and $y$ differ by a unit, and the proof is complete.
\end{proof}

The tensor product $O_F\tensor_{O_K} O_E$ is a finite $O_K$-algebra and hence a product of local rings because $O_K$ is henselian \cite[Tag 04GG]{stacks-project}. Applying Lemma \ref{lem:unit_local_ring} to $e_{ζ_1,A}(ζ_1\tensor 1)$ and $e_{ζ_2, A}(ζ_2 \tensor 1)$ in each of its factors proves Lemma \ref{lem:indep_of_generator}.
\end{proof}

\begin{defn}\label{def:Eisenstein_ideal}
Let $ζ\in O_F$ be a generator as $O_K$-algebra. We define the \emph{Eisenstein ideal} $J_A \subseteq O_F\tensor_{O_K} O_E$ as the ideal generated by $e_{ζ, A}(ζ\tensor 1)$. It is independent of the choice of $ζ$ by Lemma \ref{lem:indep_of_generator}. When $ζ$ has been fixed or when the precise choice does not matter, then we also sometimes write
\begin{equation}\label{eq:def_e_S}
e_A := e_{ζ,A}(ζ\tensor 1) \in O_F\tensor_{O_K} O_E.
\end{equation}
Given an $O_E$-algebra $R$, we define
$$J_{A,R} := J_A \cdot (O_F\tensor_{O_K} R).$$
\end{defn}

The following lemma is \cite[Lemma 2.1.2]{H}; part (1) and (2) were also already observed in \cite[\S2]{PRII}.

\begin{lem}\label{lem:eisenstein_ideal}
\begin{altenumerate}
\item The ideal $J_{A,R} \subseteq O_F\tensor_{O_K} R$ is an $R$-module direct summand which is free of rank $|A^c|$.
\item It equals the kernel of the multiplication homomorphism
\begin{equation*}
    e_{A^c}: O_F\tensor_{O_K}R\lr O_F\tensor_{O_K}R.
\end{equation*}
\item The Eisenstein ideal can equivalently be described as
\begin{equation}\label{eq:Eisenstein_nat_map}
J_A = \ker \big[(ϕ)_{ϕ\in A}:O_F\tensor_{O_K} O_E \lr \prod_{ϕ\in A} \ov K\big].
\end{equation}
\end{altenumerate}
\end{lem}
\begin{proof}
\begin{altenumerate}
\item Let $P_ζ(T) \in O_K[T]$ be the characteristic polynomial of $ζ \in O_F$ over $K$. The polynomial $e_{ζ,A}(T) \in O_E[T]$ is monic and divides $P_ζ(T)$. Thus the quotient
$$(O_F\tensor_{O_K}R)/J_{A,R} \iso (R[T]/P_ζ(T))/(e_{ζ,A}(T)) = R[T]/(e_{ζ,A}(T))$$
is finite and free of rank $|A|$ as $R$-module. Since $O_F\tensor_{O_K}R$ is free as $R$-module, the kernel $J_{A,R}$ is locally free as $R$-module of complementary rank $|A^c|$. We obtain that the natural map
\begin{equation}\label{eq:bc_J}
J_{A}\tensor_{O_E} R \lr J_{A, R}
\end{equation}
is a surjective map of finite projective $R$-modules of the same rank, and hence an isomorphism. The ideal $J_{A, O_E}$ is free over $O_E$ because $O_E$ is a discrete valuation ring. We obtain by base change from \eqref{eq:bc_J} that $J_{A, R}$ is free.

\item We first observe that
\begin{equation}\label{eq:zero}
e_A\cdot e_{A^c} = P_ζ(ζ\tensor 1) = P_ζ(ζ)\tensor 1 = 0.
\end{equation}
Next, we use that the image of $e_{A^c}$ (which is $J_{A^c,R}$ by definition) is finite free of rank $|A|$ as $R$-module by (1). It follows that $\ker(e_{A^c})$ is finite, locally free as $R$-module, and an $R$-module direct summand of $O_F\tensor_{O_K} R$. It contains $J_{A,R}$ by \eqref{eq:zero}, and hence is equal to it, because $J_{A, R}$ is a direct summand and free of the same rank as well.

\item We clearly have the inclusion $\subseteq$ in \eqref{eq:Eisenstein_nat_map}. To show the converse inclusion, it suffices by (1) to show that the image of the map $O_F\tensor_{O_K} O_E \to \prod_{ϕ\in A}\ov K$ has $O_E$-rank at least $|A|$. But this is clear because a lower bound for this $O_E$-rank is given by $\dim_{\ov K} \mr{Im}(F\tensor_K \ov K \to \prod_{ϕ\in A} \ov K) = |A|$.
\end{altenumerate}
\end{proof}

\begin{lem}\label{lem:Eisenstein_ideal_intermediate_field}
Let $K\subseteq F_0 \subseteq F$ be an intermediate field and let $A_0\subseteq \Hom_K(F_0, \ov K)$ be a subset. Let $A$ be the inverse image of $A_0$ under the projection $\Hom_K(F, \ov K)\twoheadrightarrow \Hom_K(F_0, \ov K)$. Then the reflex field $E$ of $A_0$ agrees with the reflex field of $A$, and for any $O_E$-algebra $R$ we have
\begin{equation}\label{eq:Eis_ideal_intermediate_field}
J_{A, R} = J_{A_0, R}\cdot (O_F\tensor_{O_K} R).
\end{equation}
\end{lem}
\begin{proof}
The claim about reflex fields is clear from \eqref{eq:reflex_E_A}. In order to prove \eqref{eq:Eis_ideal_intermediate_field}, we first note that the right hand side agrees with
$$J_{A_0, R} \tensor_{O_{F_0}\tensor_{O_K} R} (O_F\tensor_{O_K} R)$$
by flatness of $O_{F_0}\to O_F$. The formation of the Eisenstein ideal commutes with base change in $R$ (Lemma \ref{lem:eisenstein_ideal} (1)), so this reduces to proving \eqref{eq:Eis_ideal_intermediate_field} when $R = O_E$. The claim now follows from Lemma \ref{lem:eisenstein_ideal} (3).
\end{proof}

\subsection{$(A,B)$-Strictness}
Let $A\subseteq B\subseteq \Hom_K(F,\ov{K})$ be two subsets. Denote by $E=E_AE_B$ the join of their reflex fields. Then both $e_{ζ,A}(T),e_{ζ,B}(T) \in O_E[T]$.

\begin{defn}\label{def:AB-strict}
Let $R$ be an $O_E$-algebra and let $M$ be a locally free $O_F\otimes_{O_K} R$-module of $R$-rank $n[F:K]$. A short exact sequence of $O_F\tensor_{O_K} R$-modules
\begin{equation}\label{equ:hodge-ses}
    0 \lr \mcF\lr M \lr \mcL\lr 0
\end{equation}
is said to be \emph{$(A,B)$-strict} if $\mcL$ is locally free as $R$-module and if
\begin{equation}\label{eq:AB-strict}
    J_B\cdot M\subseteq \mcF\subseteq J_A\cdot M.
\end{equation}
If \eqref{equ:hodge-ses} is $(A,B)$-strict, then we say that $\mcL$ is of \emph{relative rank $d$} if
\begin{equation*}
    \rk_R ((J_A\cdot M)/\mcF) = d
    \qquad(\text{or equivalently }\rk_R(\mcL) = n|A|+d).
\end{equation*}
We say that a filtration $\mcF\subseteq M$, resp. a quotient $M\twoheadrightarrow \mcL$, is $(A,B)$-strict if it extends to an $(A,B)$-strict exact sequence as in \eqref{equ:hodge-ses}.
\end{defn}

\begin{ex}\label{ex:strict}
Consider the case when $B = \{ϕ_0\}$ is a singleton and $A = \emptyset$. We have $E_{\emptyset} = K$ and $E_{\{ϕ_0\}} = ϕ_0(F)$. Their composite is $E = ϕ_0(F)$. Being $(\emptyset, \{ϕ_0\})$-strict is then the same as being strict in the classical sense, meaning that $O_F$ acts on $\mcL$ in \eqref{equ:hodge-ses} by the composition $ϕ_0:O_F\to ϕ_0(O_F)\to R$.
\end{ex}

\begin{lem}\label{lem:AB-strict_dualizing}
Let $M$ be a locally free $O_F\otimes_{O_K}R$-module of $R$-rank $n[F:K]$. Consider its $R$-dual $\Hom_R(M,R)$ which is again a locally free $O_F\otimes_{O_K}R$-module.
Let $\mcF\subseteq M$ be an $O_F$-stable submodule that is locally an $R$-module direct summand. Let $\mcF^\perp \subseteq \Hom_R(M,R)$ be its orthogonal complement.
\begin{altenumerate}
\item $\mcF\subseteq M$ is $(A,B)$-strict if and only if $\mcF^\perp$ is $(B^c,A^c)$-strict.
\item $\mcF\subseteq M$ is of relative rank $d$ if and only if $\mcF^\perp$ is of relative rank $|B\setminus A|\cdot n-d$.
\end{altenumerate}

\end{lem}
\begin{proof}
The submodule $J_{A^c}\cdot \Hom_R(M, R) \subseteq \Hom_R(M, R)$ is an $R$-module direct summand by Lemma \ref{lem:eisenstein_ideal} (1). It is locally free of $R$-rank $n|A|$, which is the same as the rank of $(J_A\cdot M)^\perp$. Moreover, since $J_AJ_{A^c} = 0$, we have
\begin{equation}\label{eq:dual_J}
J_{A^c}\cdot \Hom_R(M, R) \subseteq (J_A\cdot M)^\perp,
\end{equation}
and hence these two submodules of $\Hom_R(M, R)$ are equal. The lemma now follows by taking duals and orthogonal complements in \eqref{eq:AB-strict}.
\end{proof}

\begin{rmk}
A similar lemma is found in \cite[Lemma 2.3.2]{KRZ}.
\end{rmk}

\begin{prop}\label{prop:e_S-inverse}
Let $S \subseteq A \subseteq B \subseteq \Hom_K(F, \ov K)$ be three subsets. Let $E = E_SE_AE_B$, let $R$ be an $O_E$-algebra, and let $M$ be a locally free $O_F\tensor_{O_K}R$-module of $R$-rank $n[F:K]$. Then multiplication by $e_S$ defines an isomorphism
\begin{equation}\label{eq:iso-e_S}
    e_S:(J_{A\setminus S}\cdot M)/(J_{B\setminus S}\cdot M) \simlr (J_A\cdot M)/(J_B\cdot M).
\end{equation}
In particular, there is a bijection
\begin{equation}\label{eq:bij-e_S}
\begin{aligned}
    \{\text{$(A,B)$-strict $\mcF\subseteq M$}\} & \simlr \{\text{$(A\setminus S, B\setminus S)$-strict $\mcF\subseteq M$}\}\\
    \mcF & \longmapsto e_S^{-1}(\mcF)\\
    e_S\cdot \mcF & \longmapsfrom \mcF.
\end{aligned}
\end{equation}
\end{prop}

\begin{proof}
First note that all three elements $e_S$, $e_{A\setminus S}$ and $e_{B\setminus S}$ lie in $O_F\tensor_{O_K} O_E$ by definition of $E$. Moreover,
$$e_A = e_S\cdot e_{A\setminus S}\quad\text{and}\quad e_B = e_S\cdot e_{B\setminus S}.$$
In particular, multiplication by $e_S$ defines surjective maps $J_{A\setminus S}\cdot M \twoheadrightarrow J_A\cdot M$ and $J_{B\setminus S}\cdot M \twoheadrightarrow J_B\cdot M$. From this we obtain a map on quotients as claimed,
$$(J_{A\setminus S}\cdot M)/(J_{B\setminus S}\cdot M) \lr (J_{A}\cdot M)/(J_{B}\cdot M).$$
By Lemma \ref{lem:eisenstein_ideal} (2), its kernel is
\begin{equation}\label{eq:kernel_complicated}
[(J_{B\setminus S}\cdot M) + (J_{A\setminus S}\cdot M) \cap (J_{S^c} \cdot M)] / (J_{B\setminus S}\cdot M).
\end{equation}
Since $(A\setminus S) \subseteq (B\setminus S) \subseteq S^c$, we have
$$J_{S^c} \subseteq J_{B\setminus S}\subseteq J_{A\setminus S}$$
so the kernel in \eqref{eq:kernel_complicated} equals
$$(J_{B\setminus S}\cdot M)/(J_{B\setminus S}\cdot M) = 0.$$
This proves that \eqref{eq:iso-e_S} is an isomorphism. The bijection in \eqref{eq:bij-e_S} is a direct consequence.
\end{proof}

Definition \ref{eq:AB-strict} is phrased in terms of the submodules $J_B\cdot M$ and $J_A\cdot M$, and hence directly refers to the module $M$. We end this section by proving that $(A,B)$-strictness is equivalent to the Eisenstein condition of Rapoport--Zink which is purely in terms of $\mcL$.

\begin{prop}\label{prop:RZ-eisenstein}
Let $A\subseteq B\subseteq \Hom_K(F, \ov K)$ be two subsets and let $E = E_A E_B$. Let $R$ be an $O_E$-algebra and let $M$ be a locally free $O_F\tensor_{O_K} R$-module of $R$-rank $n[F:K]$. Let
\begin{equation}\label{eq:std_ex_seq}
0 \lr \mcF\lr M\lr \mcL\lr 0
\end{equation}
be an exact sequence of $(O_F\tensor_{O_K}R)$-modules such that $\mcL$ is locally free as $R$-module of rank $n|A|+d$. The following two conditions are equivalent:
\begin{altenumerate}
    \item The filtration \eqref{eq:std_ex_seq} is $(A, B)$-strict, i.e.
    $$J_{B} M \subseteq \mcF \subseteq J_{A} M.$$
    \item The quotient $\mcL$ satisfies the Eisenstein condition from \cite{RZ2},
    $$e_B \mcL = 0 \quad\quad \text{and} \quad\quad \bigwedge^{d+1}_R(e_A \mid \mcL) = 0.$$
\end{altenumerate}
\end{prop}
\begin{proof}
The conditions $J_{B} M \subseteq \mcF$ and $e_B \mcL = 0$ are obviously equivalent. The non-trivial part is to show that
$$\mcF\subseteq J_{A}M \quad \Longleftrightarrow \quad \bigwedge^{d+1}_R(e_A \mid \mcL) = 0.$$
First we note that $e_BM$ and $e_AM$ are $R$-module direct summands of $M$ by Lemma \ref{lem:eisenstein_ideal}. Next, we observe that $e_A \mcL = e_A M / (e_A M \cap \mcF)$. Thus if the condition $\mcF\subseteq e_AM$ holds, then $e_A \mcL$ is locally free of rank $d$ as $R$-module and hence $\bigwedge^{d+1}(e_A \mid \mcL) = 0$. For the converse implication, we first recall a lemma from \cite{RZ2}.
\begin{lem}[\protect{\cite[Lemma 4.9]{RZ2}}]
Let $R$ be a local ring with residue field $k$. Let $\mcL$ be a finitely generated $R$-module of rank $m$. Let $f:\mcL\to \mcL$ be an endomorphism. Let $s = \dim_k(\mcL/f(\mcL))\tensor_R k$ and let $d = m - s$. Assume that $\bigwedge^{d+1} f = 0.$ Then $\mcL/f(\mcL)$ is a free $R$-module of rank $s$.
\end{lem}

We apply this lemma to the endomorphism $e_A:\mcL \to \mcL$. Assuming the Eisenstein condition $\bigwedge^{d+1} (e_A\mid \mcL) = 0$, and using that $\mr{rk}_R(\mcL) = n|A|+d$, a direct computation shows that $\dim_k(\mcL/e_A(\mcL)) = n|A|$ for every closed point $R\to k$. The lemma now implies that $\mcL/e_A(\mcL)$ is locally free of rank $n|A|$ as $R$-module. We deduce that $e_A\mcL$ is locally free of rank $d$ as $R$-module.

Using again that $e_A\mcL = e_A M/ (e_AM \cap \mcF)$ and that $e_AM \subset M$ is a direct summand by Lemma \ref{lem:eisenstein_ideal}, we thus obtain that both
$$\mcF,\quad\text{and}\quad\ e_A M \cap \mcF\ \subseteq\ M$$
are $R$-module direct summands with cokernel of the same rank. It follows that $\mcF = e_A M \cap \mcF$, and the proof of Proposition \ref{prop:RZ-eisenstein} is complete.
\end{proof}

\subsection{Relation with hermitian forms}\label{sec:duality}

\newcommand{\btr}{\boldsymbol{tr}}

With a view towards later applications in a unitary setting, we now study the bijection \eqref{eq:bij-e_S} in the context of hermitian forms. Let $τ \in \mr{Gal}(F/K)$ be an involution with fixed field $F_0$. We also write $\ov a := τ(a)$ for elements $a\in F$. For an $O_K$-algebra $R$, we again use $τ$ and $x\mapsto \ov x$ to denote the involution $τ\tensor 1$ on $O_F\tensor_{O_K} R$.

Given $ϕ\in \Hom_K(F, \ov K)$ or $S\subseteq \Hom_K(F, \ov K)$, we define $\ov{ϕ} = ϕ\circ τ$ resp. $\ov S = \{\ov{ϕ} \mid ϕ\in S\}$. Let $ϕ_0\in \Hom_K(F, \ov K)$ and $A\subset \Hom_K(F, \ov K)$ be such that
$$\Hom_K(F, \ov K) = A \sqcup \{ϕ_0, \ov{ϕ_0}\}\sqcup \ov A.$$
Set $B = A\cup \{ϕ_0, \ov{ϕ_0}\}$ and let $E = E_AE_B$ be the reflex field of this situation. Note that it contains $E_{\{ϕ_0, \ov{ϕ_0}\}} = ϕ_0(F_0)$.

\begin{lem}\label{lem:ideal_conjugation}
The Eisenstein ideals $J_A, J_B\subseteq O_F\tensor_{O_K} O_E$ satisfy $\ov{J_A} = J_{\ov A}$ and $\ov{J_B} = J_{\ov B}$.
\end{lem}
\begin{proof}
Let $ζ\in O_F$ be an $O_K$-algebra generator. By definitions, see \eqref{eq:def_e_S}, we have
\begin{equation}\label{eq:e_A_conjugate}
\begin{aligned}
    τ(e_{ζ, A}(ζ\tensor 1)) & = e_{ζ, A}(\ov{ζ} \tensor 1)\\
    & = \prod_{ϕ\in A} (\ov{ζ}\tensor 1 - 1\tensor ϕ(ζ))\\
    & = \prod_{ϕ\in A} (\ov{ζ}\tensor 1 - 1\tensor \ov{ϕ}(\ov{ζ}))\\
    & = e_{\ov{ζ}, \ov A}(\ov{ζ}\tensor 1).
\end{aligned}
\end{equation}
Looking at the generated ideals, this means $\ov{J_A} = J_{\ov{A}}$. The argument for $J_B$ is the same.
\end{proof}

Let $R$ be an $O_E$-algebra and let $M_1$, $M_2$ be two locally free $O_F\tensor_{O_K} R$-modules. Assume that
$$ψ:M_1\times M_2\lr R$$
is an $R$-bilinear form that is $O_F$-conjugate linear. By this we mean that for all $a\in O_F$, $r\in R$, and $x_i\in M_i$
\begin{equation}\label{eq:hermitian-definition}
ψ((a\tensor r) x_1, x_2) = ψ(x_1, (\bar a \tensor r)x_2).
\end{equation}
Then $J_{A}\cdot M_i \perp J_{B}\cdot M_{3-i}$ because $A\cup \ov B = \Hom_K(F, \ov K)$ together with Lemma \ref{lem:ideal_conjugation} implies $J_A \ov{J_B} = 0$. In particular, we obtain an induced $R$-bilinear pairing
$$(J_A M_1 / J_B M_1) \times (J_A M_2 / J_B M_2) \lr R$$
which is still $O_F$-conjugate linear in the sense of \eqref{eq:hermitian-definition}. Next, recall the formalism of lifting pairings along the trace. Let $\vartheta \in \mcD^{-1}_{F_0/K}$ be a generator of the inverse different of $F_0/K$, and consider the map
$$\begin{aligned}
    \btr:O_{F_0}\tensor_{O_K}R & \lr R\\
    a\tensor r & \longmapsto \mr{tr}_{F_0/K}(\vartheta a)\tensor r.
\end{aligned}$$
Then $\btr$ has the property that the map
\begin{equation*}
   \btr:\Hom_{O_{F_0}\otimes_{O_K}R}(O_{F_0}\otimes_{O_K}R, O_{F_0}\otimes_{O_K}R)\lr \Hom_{R}(O_{F_0}\otimes_{O_K}R, R),\quad \left[f\longmapsto \btr \circ f\right]
\end{equation*}
is an isomorphism. The next lemma extends this to all modules.
\begin{lem}\label{lem:keep-isotropic}
For every $O_{F_0}\otimes_{O_K}R$-module $M$, the map
\begin{equation*}
    \btr:\Hom_{O_{F_0}\otimes_{O_K}R}(M,O_{F_0}\otimes_{O_K}R)\lr \Hom_R(M,R),\quad f\longmapsto \btr \circ f
\end{equation*}
is an isomorphism of $O_{F_0}\tensor_{O_K} R$-modules.
\end{lem}
\begin{proof}
Write $S := O_{F_0}\otimes_{O_K}R$ during the proof. Choose a presentation
\begin{equation*}
    S^{\oplus J}\lr S^{\oplus I}\lr M\lr 0.
\end{equation*}
We apply the functors and natural transformation $\Hom_S(-,S)\overset{\btr}{\to} \Hom_R(-,R)$ to get the following commutative diagram:
\begin{equation*}
\begin{tikzcd}
0\arrow[r]&\Hom_{S}(M,S)\arrow[d]\arrow[r]& \prod_I \Hom_{S}(S,S)\arrow[r]\arrow[d,"\cong"]&\prod_J \Hom_{S}(S,S)\arrow[d,"\cong"]\\
0\arrow[r]&\Hom_{R}(M,R)\arrow[r]&\prod_I \Hom_{R}(S,R)\arrow[r]&\prod_J \Hom_{R}(S,R).
\end{tikzcd}
\end{equation*}
Here, the identity $\Hom_R(S^{\oplus I}, R) = \prod_I \Hom_R(S, R)$ comes from the finiteness of $S$ over $R$; analogously for $J$ in place of $I$. The claimed isomorphism now follows from the five lemma.
\end{proof}
Lemma \ref{lem:keep-isotropic} implies that there exists a unique $O_{F_0}\tensor_{O_K}R$-bilinear form
$$ψ_{F_0}:M_1\times M_2\lr O_{F_0}\tensor_{O_K}R$$
with $ψ = \btr \circ ψ_{F_0}$. By functoriality of the construction, it is again $O_F$-conjugate linear in the sense of \eqref{eq:hermitian-definition}. Define
$$A_0 = \{ϕ\vert_{F_0} \mid ϕ\in A\} \subseteq \Hom_{K}(F_0, \ov K).$$
By Lemmas \ref{lem:Eisenstein_ideal_intermediate_field} and \ref{lem:ideal_conjugation}, the Eisenstein ideal $J_{A_0}\subset O_{F_0}\tensor_{O_K} O_E$ satisfies
\begin{equation}\label{eq:Eis_ideal_rel}
J_A\cdot J_{\ov A} = J_{A_0}\cdot (O_F\tensor_{O_K} O_E).
\end{equation}
We have three distinguished kinds of generators for this ideal: the elements $e_{A\cup \bar A}$, the products $e_A\cdot \ov{e_A}$ (both kinds defined for generators $ζ\in O_F$), and the elements $e_{A_0}$ defined for generators $ζ_0\in O_{F_0}$. Note that
\begin{equation}\label{eq:invariants_R}
(O_F\tensor_{O_K}R)^{τ = \mr{id}} = O_{F_0}\tensor_{O_K}R
\end{equation}
whenever $R$ is $O_K$-flat or $F/F_0$ at most tamely ramified. In particular, the products $e_A\cdot \ov{e_A}$ lie in $O_{F_0}\tensor_{O_K} O_E$. Moreover, by Lemma \ref{lem:unit_local_ring}, all mentioned types of generators differ by units from $(O_F\tensor_{O_K} O_E)^\times$.
\begin{lem}\label{lem:restrict_lifted_form}
The lifted pairing $ψ_{F_0}$ restricts to an $R$-bilinear and $O_F$-conjugate linear pairing
$$(J_A M_1 / J_B M_1) \times (J_A M_2 / J_B M_2) \lr J_{A_0, R} \subseteq O_{F_0}\tensor_{O_K} R.$$
\end{lem}
\begin{proof}
As noted before, we have $ψ(J_AM_1, J_BM_2) = ψ(J_BM_1, J_AM_2) = 0$. By the uniqueness of the lifted pairing from Lemma \ref{lem:keep-isotropic}, this implies $ψ_{F_0}(J_AM_1, J_BM_2) = ψ_{F_0}(J_BM_1, J_AM_2) = 0$. Moreover, we have seen in Lemma \ref{lem:ideal_conjugation} and \eqref{eq:Eis_ideal_rel} that $J_A\ov{J_A} = J_{A_0}\cdot (O_F\tensor_{O_K} O_E)$. So we obtain
$$\begin{aligned}
ψ_{F_0}(J_A M_1, J_A M_2) & = ψ_{F_0}(J_A \ov{J_A} M_1, M_2)\\
& = ψ_{F_0}(J_{A_0} M_1, M_2)\\
& \subseteq J_{A_0,R}
\end{aligned}$$
as was to be shown.
\end{proof}

\begin{prop}[Properties of $ψ_{F_0}$]\label{prop:duality-comp}
Let $R$ be an $O_E$-algebra and let $M_1$, $M_2$, $ψ$ and $ψ_{F_0}$ be as above.
\begin{altenumerate}
\item The pairing $ψ$ is perfect if and only if $ψ_{F_0}$ is perfect.
\item Assume $ψ$ and $ψ_{F_0}$ are perfect and let $S \in \{A, B\}$. Then for each $i\in \{1,2\}$,
\begin{equation}\label{eq:simple_orthognality}
(J_S\cdot M_i)^\perp = J_{\bar S^c}\cdot M_{3-i}.
\end{equation}
In particular, the restrictions of $ψ$ and $ψ_{F_0}$ define perfect pairings of $R$-modules that fit into the commutative diagram
\begin{equation}\label{eq:restricted_pairings_diagram}
\xymatrix{
& J_{A_0, R} \ar[dd]^{\btr}_{\iso} & & R \ar[ll]_-{e_A\ov{e_A}}^-{\iso}\\
(J_{A} M_1)/(J_{B}M_1)\times (J_{A} M_2)/(J_{B}M_2) \ar[ru]^-{ψ_{F_0}} \ar[rd]_-{ψ} & & &\\
& R. & &}
\end{equation}
Here, the upper right isomorphism comes from Lemma \ref{lem:eisenstein_ideal} (2) and the natural identifications
$$(O_{F_0} \tensor_{O_K} R)/J_{\{ϕ_0\vert F_0\}, R} \simlr R.$$
\item Assume that $\mcF_i\subseteq M_i$ are two $(A, B)$-strict filtrations in the sense of Definition \ref{def:AB-strict}. Then
$$\begin{aligned}
\mcF_1 \perp_{ψ} \mcF_2 & \ \Longleftrightarrow\ \mcF_1 \perp_{ψ_{F_0}} \mcF_2\\[1mm]
& \ \Longleftrightarrow\ (\mcF_1/J_{B}M_1) \perp (\mcF_2/J_{B}M_2)
\end{aligned}$$
where the third orthogonality relation is with respect to either of the restricted pairings in \eqref{eq:restricted_pairings_diagram}.
\end{altenumerate}
\end{prop}
\begin{proof}
Statement (1) is a direct consequence of the fact that $\btr$ in Lemma \ref{lem:keep-isotropic} is an isomorphism. The fact that $J_{S}M_i \perp J_{\bar S^c}M_{3-i}$ was already noted during the proof of Lemma \ref{lem:restrict_lifted_form} before. A simple variation of the arguments around \eqref{eq:dual_J} allows to conclude that the two submodules are mutual orthogonal complements. The first equivalence in statement (3) comes from Lemma \ref{lem:keep-isotropic} applied to $\mcF_1$ and $\mcF_2$: the unique lifting along $\btr$ of the $0$-pairing is the $0$-pairing. The second equivalence is an immediate consequence using $J_{B}M_i \perp J_{A}M_{3-i}$ as in statement (2).
\end{proof}

\section{Local models}\label{s:LM}

In this section, we define and compare relative and absolute local models for general linear and unitary groups. 

\subsection{The linear case}\label{sec:linear-LM}
Let $F/K$ be a finite extension of $p$-adic local fields and let $V$ be an $F$-vector space of dimension $n$. Let $π\in F$ denote a uniformizer.
\begin{defn}\label{defn:lattice-chain}
A \emph{lattice chain in $V$} is a set $L$ of $O_F$-lattices $\Lambda\subset V$ which form a chain, that is,
\begin{enumerate}
\item For any two lattices $\Lambda_1$ and $\Lambda_2\in L$, either $\Lambda_1\subseteq \Lambda_2$ or $\Lambda_2\subseteq \Lambda_1$;
\item For any $\Lambda\in L$, we have $\pi\Lambda\in L$.
\end{enumerate}
\end{defn}
Fix a lattice chain $L$ in $V$ and an integer $0 \leq d \leq n$. For the following definitions, we adhere to the terminology from \cite{PRI} and call local models for the general linear group \emph{standard local models}.

\begin{defn}\label{def:rel_loc_mod_GL}
The \emph{relative standard model} $\bfM_n^L(d)$ is the projective scheme over $\Spec O_F$ whose $R$-points for an $O_F$-algebra $R$ are the set of tuples $(\msF_\Lambda)_{\Lambda\in L}$ that satisfy:
\begin{altenumerate}

\item $\msF_\Lambda \subseteq Λ\tensor_{O_F} R$ is an $R$-module direct summand of rank $n-d$;
\item For each inclusion $i:\Lambda_1\subseteq \Lambda_2$ in $L$, the induced morphism 
\begin{equation*}
i\tensor \mr{id}_R:\Lambda_1\otimes_{O_F}R\lr \Lambda_2\otimes_{O_F}R
\end{equation*}
satisfies $(i\tensor \mr{id}_R)(\msF_{\Lambda_1}) \subseteq \msF_{\Lambda_2}$;
\item For every $Λ\in L$, the multiplication map $π:\Lambda\simto πΛ$ induces an isomorphism 
\begin{equation*}
    (π\tensor \mr{id}_R)\vert_{\msF_Λ}:\msF_{\Lambda}\simlr\msF_{\pi\Lambda}.
\end{equation*}
\end{altenumerate}
\end{defn}

\begin{thm}\label{thm:Goertz}
The relative standard model $\bfM_n^L(d)$ is flat normal and Cohen--Macaulay. When $d=1$ or $d=n-1$, the relative standard model is semi-stable. When $d = 0$ or $d = n$, it is isomorphic to $\Spec O_F$.
\end{thm}
\begin{proof}
The flatness is \cite[Corollary 4.20]{Gortz}.
The normality and Cohen-Macaulayness is \cite[Theorem 2.1]{TT23}.
When $d=1$, a standard computation (see \cite[Proposition 4.13]{Gortz} for the Iwahori level case) shows that the strict complete local ring in a closed point of $\bfM_n^L(d)$ is isomorphic to
\begin{equation*}
    O_{\breve F}[\![x_1,\cdots,x_s,y_{s+1},\cdots,y_n]\!]/(x_1\cdots x_s-\pi)
\end{equation*}
for some $1\leq s\leq n$. Therefore, it is semi-stable. The case $d=n-1$ follows by taking the dual. When $d = 0$ resp. $d = n$, then Definition \ref{def:rel_loc_mod_GL} specializes to $(\msF_Λ)_Λ = (Λ\tensor_{O_F}R)_Λ$ resp. $(\msF_Λ)_Λ = (0)_Λ$.
\end{proof}

Next, we define the absolute version. 

\begin{defn}\label{def:abs_loc_model_GL}
Let $ϕ_0\in \Hom_K(F, \ov K)$ be an element and let $A\subset \Hom_K(F, \ov K)$ be a subset with $ϕ_0\notin A$. Set $B = A\cup \{ϕ_0\}$. Denote by $E = E_A ϕ_0(F)$ the reflex field.\footnote{See Remark \ref{rmk:reflex}.} The \emph{absolute standard model} $\bfM^L_{n,A}(d)$ is the projective scheme over $\Spec O_E$ whose $R$-points for an $O_E$-algebra $R$ are the set of tuples $(\mcF_\Lambda)_{\Lambda\in L}$ that satisfy:
\begin{altenumerate}

\item $\mcF_\Lambda \subseteq Λ\tensor_{O_K}R$ is an $R$-module direct summand of $n|A^c|-d$. It is required to be stable under the $O_F$-action;

\item For each inclusion $i:\Lambda_1\subseteq \Lambda_2$ of lattices from $L$, the induced morphism 
\begin{equation*}
    (i\tensor \mr{id}_R):\Lambda_1\otimes_{O_K}R\lr \Lambda_2\otimes_{O_K}R
\end{equation*}
satisfies $(i\tensor \mr{id}_R)(\mcF_{\Lambda_1}) \subseteq \mcF_{\Lambda_2}$;
\item For every $Λ\in L$, the multiplication map $π:Λ\simto πΛ$ induces an isomorphism
\begin{equation*}
(π\tensor \mr{id}_R):\mcF_{\Lambda}\simlr\mcF_{\pi\Lambda};
\end{equation*}
\item (Eisenstein condition) Finally, we demand that each $\mcF_Λ$ is $(A, B)$-strict in the sense of Definition \ref{def:AB-strict}. That is, there are the inclusions
\begin{equation*}
    J_B\cdot (\Lambda\otimes_{O_K}R) \subseteq \mcF_\Lambda \subseteq J_A \cdot (\Lambda\otimes_{O_K}R).
\end{equation*}
\end{altenumerate}
\end{defn}
\begin{rmk}[Banal case \protect{\cite[Appendix B]{RSZ-AGGP}}]\label{rmk:reflex}
The reflex field is $E_A$ if $d = 0$, and it is $E_B$ if $d = n$. Indeed, one can define $\bfM^L_{n,A}(0)$ already over $O_{E_A}$ by the moduli description
\begin{equation}\label{eq:trivial_model_A}
\bfM^L_{n,A}(0)(R) = \{(J_A\cdot (Λ\tensor_{O_K}R))_{Λ\in L}\}.
\end{equation}
This set is a singleton, so one then gets $\bfM^L_{n,A}(0) = \Spec(O_{E_A})$. Note that \eqref{eq:trivial_model_A} does not require the choice/existence of an element $ϕ_0\in \Hom_K(F, \ov K)\setminus A$. Similarly, also given $ϕ_0$, one can define $\bfM^L_{n,A}(n)$ over $O_{E_B}$ by
$$\bfM^L_{n,A}(n)(R) = \{(J_B\cdot (Λ\tensor_{O_K}R))_{Λ\in L}\}.$$
This results in $\bfM^L_{n,A}(n) = \Spec(O_{E_B})$. For simplicity, we will however work with $\bfM^L_{n,A}(d)$ over $O_E$ in all cases.
\end{rmk}
Recall that there are isomorphisms
\begin{equation}\label{eq:standard-model-lattice-iso}
\begin{aligned}
\xymatrix{
J_A (\Lambda\otimes_{O_K} R)\,/\,J_B (\Lambda\otimes_{O_K}R)\ar[r]_-\sim^-{e_A^{-1}}& \Lambda\otimes_{O_K}R\,/\,J_{\{\phi_0\}}(\Lambda\otimes_{O_K}R)\ar@{=}[r]& \Lambda\otimes_{O_F}R,
}
\end{aligned}
\end{equation}
where the isomorphism $e_A^{-1}$ is defined in \eqref{eq:iso-e_S}, and where the second equality follows from Example \ref{ex:strict}. 

\begin{thm}\label{thm:local_model_comparison}
View $F$ as a subfield of $E$ via $ϕ_0$. There exists an isomorphism of $O_E$-schemes from the absolute standard model to the relative standard model:
\begin{equation}\label{eq:iso_loc_mod_GL}
\begin{aligned}
\xymatrix{
\bfM^L_{n,A}(d)\ar[r]^-{\sim}& O_E\tensor_{O_F} \bfM_{n}^L(d),&
(\mcF_\Lambda)_{Λ\in L}\ar@{|->}[r]&(\msF_\Lambda)_{Λ\in L}:=(e_A^{-1}(\mcF_\Lambda))_{Λ\in L}
}
\end{aligned}
\end{equation}
In particular, the absolute standard model $\bfM^{L}_{n,A}(d)$ is flat, normal and Cohen-Macaulay.
\end{thm}
\begin{proof}
Let $(\mcF_\Lambda)_{\Lambda\in L}$ be an $R$-point of the absolute standard model $\bfM_{n,A}^L(d)$. Due to the $(A,B)$-strictness in Definition \ref{def:abs_loc_model_GL}, we may use \eqref{eq:standard-model-lattice-iso} to define a tuple $(\msF_Λ)_{Λ\in L}$ of $O_F$-stable $R$-module direct summands of $(Λ\tensor_{O_F}R)$, also see \eqref{eq:bij-e_S}. 

Since each $\mcF_Λ$ is locally free of $R$-rank $n|A^c| -d$, it is clear that each $\msF_Λ$ is locally free of $R$-rank $n-d$. It is also clear that this definition translates between the functoriality conditions (2) and (3) of Definitions \ref{def:rel_loc_mod_GL} and \ref{def:abs_loc_model_GL}. Hence \eqref{eq:iso_loc_mod_GL} is defined.

The inverse functor is given by $(\msF_Λ) \mapsto (e_A\cdot \msF_Λ)$. The flatness property follows from Theorem \ref{thm:Goertz}. The normality and Cohen--Macaulayness follow from Theorem \ref{thm:Goertz} and \cite[Lemma 5.7]{TT23}.
\end{proof}

\begin{rmk}
The absolute standard model we defined here is a special case of the splitting model defined by Pappas--Rapoport in \cite{PRII}, except that we worked over $E$ instead of the Galois closure $F^{\mr{Gal}}$. More precisely, in the terminology of \cite{PRII}, the model $\bfM_{n,A}^L(d)$ arises by enumerating $\Hom_K(F, \ov K)$ as $A\sqcup \{ϕ_0\}\sqcup \text{(complement)}$ and choosing the tuple $(r_1, \ldots, r_{[F:K]})$ as $(n, \ldots, n, d, 0, \ldots, 0)$, where the middle entry is for $\{ϕ_0\}$.

A basic observation in \cite{PRII} is that the splitting model can be understood as an iterated fibration of $[F:K]$ many relative local models, see \cite[p. 3]{PRII}. Theorem \ref{thm:local_model_comparison} is essentially the special case when all but one of these relative local models is étale.
\end{rmk}

We explain how to apply Theorem \ref{thm:local_model_comparison} in the context of \cite{RZ2}. Let $m\geq 1$ and let $D/F$ be a central division algebra of degree $c = (\dim_F(D))^{1/2}$. Let $O_D\subseteq D$ be the ring of integers and denote by $\mr{charred}_{D/F}:O_D\to O_F[T]$ the reduced characteristic polynomial.

\begin{defn}[Special $O_D$-actions]\label{def:special}
Let $R$ be an $O_F$-algebra and let $M$ be an $O_D\tensor_{O_F}R$-module which is locally on $\Spec R$ free of rank $m$ as $O_D\tensor_{O_F}R$-module. A filtration $\msF\subseteq M$ or, equivalently, a quotient map $M\twoheadrightarrow \msL = M/\msF$, is called \emph{special} if $\msF$ is $O_D$-stable, if $\msL$ is locally free of rank $c$ as $R$-module, and if for every $a\in O_D$,
\begin{equation}\label{eq:special_LM}
\mr{char}(a\mid \msL; T) = \mr{charred}_{D/F}(a; T).
\end{equation}
\end{defn}
Let $A$, $ϕ_0$ and $B = A\sqcup \{ϕ_0\}$ be as before, and again set $E = E_A ϕ_0(F)$.
\begin{defn}[$(A, ϕ_0)$-special $O_D$-actions]\label{def:AB_special}
Let $R$ be an $O_E$-algebra and let $M$ be an $O_D\tensor_{O_K}R$-module which is locally on $\Spec R$ free of rank $m$ as $O_D\tensor_{O_K}R$-module. A filtration $\mcF\subseteq M$ or, equivalently, a quotient map $M\twoheadrightarrow \mcL = M/\mcF$, is called $(A, ϕ_0)$-\emph{special} if $\mcF$ is $O_D$-stable, if $\mcL$ is $(A,B)$-strict and locally free of rank $mc^2\cdot |A| + c$ as $R$-module, and if for every $a\in O_D$,
\begin{equation}\label{eq:A_phi_0_special}
\mr{char}(a\mid J_A\cdot \mcL; T) = \mr{charred}_{D/F}(a; T).
\end{equation}
Here, note that $J_A\cdot \mcL$ is locally free of rank $c$ over $R$ by the required $(A, B)$-strictness. 
\end{defn}

Let $Λ$ be an $O_D$-module that is free of rank $m$. Denote by $\bfM^D_m$ the projective scheme over $O_F$ whose $R$-points is the set of special filtrations on $Λ\tensor_{O_F} R$. Denote by $\bfM^D_{m,A}$ the projective $O_E$-scheme whose $R$-points is the set of $(A, ϕ_0)$-special filtrations on $Λ\tensor_{O_K} R$.

\begin{cor}[to Theorem \ref{thm:local_model_comparison}]\label{cor:loc_mod_CSA}
View $F$ as a subfield of $E$ via $ϕ_0$. There exists an isomorphism of $O_E$-schemes
\begin{equation}\label{eq:iso_comparison_loc_mod_CSA}
\bfM^D_{m,A} \simlr O_E\tensor_{O_F} \bfM^D_m.
\end{equation}
In particular, $\bfM^D_{m,A}$ is flat, normal and Cohen--Macaulay.
\end{cor}
\begin{proof}
Consider Theorem \ref{thm:local_model_comparison} for $V = F\cdot Λ$ with lattice chain $\{π^kΛ\mid k\in \mbZ\}$, and with quotient of rank $mc^2\cdot |A| + c$ for the absolute standard local model (resp. of rank $d = c$ for the relative standard local model). Source and target in \eqref{eq:iso_comparison_loc_mod_CSA} identify with the closed subfunctors of $(A, ϕ_0)$-special $O_D$-stable (LHS) and special $O_D$-stable (RHS) filtrations of \eqref{eq:iso_loc_mod_GL}. Clearly, \eqref{eq:iso-e_S} restricts to an isomorphism of these closed subfunctors, and hence to an isomorphism as in \eqref{eq:iso_comparison_loc_mod_CSA}.

It is well-known that $\bfM^D_m$ is regular with semi-stable reduction (Drinfeld). For example, let $Q/F$ be an unramified field extension of degree $c$. Then $Q$ splits $D$, and $O_Q\tensor_{O_F} \bfM^D_m$ is isomorphic to the standard local model $\bfM^{L}_{mc}(1)$ over $O_Q$ for a suitable choice of lattice chain $L$. The regularity and semi-stable reduction property then follow from Theorem \ref{thm:Goertz}.

The base change $O_E\tensor_{O_F} \bfM^D_m$ in \eqref{eq:iso_comparison_loc_mod_CSA} is finite and flat over $\bfM^D_m$, hence at least Cohen--Macaulay. Its generic fiber is smooth while its special fiber is reduced, so $O_E\tensor_{O_F} \bfM^D_m$ is normal by \cite[Proposition 9.2]{Pappas-Zhu}. This finishes the proof.
\end{proof}
\begin{rmk}\label{rmk:rel_to_RZ}
Consider the case where $m = 1$ and where the Hasse invariant of $D$ is $1/c$. Then $\bfM^D_{m,A}$ agrees with the scheme $\mbM_r$ from \cite[\S5]{RZ2}. (Here, the data $A$ and $r$ that go into the definitions are in direct correspondence.) Corollary \ref{cor:loc_mod_CSA} gives a concrete description of $\bfM^D_{m,A}$ that allows to obtain some of the statements in \cite[\S5]{RZ2} directly. For example, over the special fiber, \eqref{eq:iso_comparison_loc_mod_CSA} reduces to the isomorphism \cite[(5.4)]{RZ2}. Note that for general $m$ and $D$, Corollary \ref{cor:loc_mod_CSA} also gives a concrete description of the singularities of $\bfM^D_{m,A}$: its strict complete local rings in closed points are of the form
$$O_{\breve E}[\![X_1, \ldots, X_{c\cdot m}]\!]/(X_1\cdots X_r - π_E^e),$$
where $π_E\in E$ is a uniformizer and where $e$ is the ramification index of $E/F$. In particular, $\bfM^D_{m,A}$ is not regular if $e \geq 2$.
\end{rmk}

\subsection{The unitary case}
\label{ss:unitary_local_model}
Recall the notation from \S \ref{sec:duality}: let $K\subseteq F_0 \subset F$ be field extensions with $[F:F_0] = 2$, and let $τ\in \Gal(F/F_0)$ be the non-trivial element. We also write $\ob a = τ(a)$ for elements $a\in F$.

Let $V$ be an $n$-dimensional $F$-vector space and let $\psi: V\times V\to K$ be a perfect $K$-bilinear pairing that is $F$-conjugate linear in the sense of \eqref{eq:hermitian-definition}. By Lemma \ref{lem:keep-isotropic}, there exists a unique perfect $F_0$-bilinear pairing $\psi_{F_0}:V\times V\to F_0$ with $\psi=\btr\circ\psi_{F_0}$. It is again $F$-conjugate linear. We denote the dual lattice of an $O_F$-lattice $\Lambda\subset V$ by
\begin{equation*}
    \Lambda^\vee:=\{x\in V\mid \psi(x,v)\in O_K\quad\text{for all }v\in \Lambda\}.
\end{equation*}
By Lemma \ref{lem:keep-isotropic}, one can equivalently take the dual with respect to $ψ_{F_0}$, that is,
\begin{equation*}
Λ^\vee = \{x\in V\mid \psi_{F_0}(x,v)\in O_{F_0}\quad\text{for all }v\in \Lambda\}.
\end{equation*}

\begin{defn}\label{def:self-dual-chain}
A \emph{self-dual lattice chain} is a set $L$ of $O_F$-lattices $\Lambda\subset V$ such that:
\begin{altenumerate}
\item For any two lattices $\Lambda_1$ and  $\Lambda_2\in L$, we have either $\Lambda_1\subseteq \Lambda_2$ or $\Lambda_2\subseteq \Lambda_1$;
\item For any $\Lambda\in L$, we have $\pi\Lambda\in L$;
\item For any $\Lambda\in L$, we have $\Lambda^\vee\in L$.
\end{altenumerate}
\end{defn}

Let $L$ be a self-dual lattice chain in $V$ and let $0\leq r \leq n$ be an integer. Set $s = n - r$. We now recall the definition and properties of the relative local model for $L$ and signature $(r,s)$. Let $E_0$ denote the reflex field, defined as
\begin{equation}\label{eq:reflex_field_unitary_relative}
E_0 = \begin{cases}
F &\text{ if $r\neq s$};\\
F_0 & \text{ if $r=s$}.
\end{cases}
\end{equation}
\newcommand{\naive}{\mathrm{naive}}
\begin{defn}\label{def:unitary_loc_mod_relative}
(1) The \emph{naive relative local model} $\bfM_n^{L,\naive}(r,s)$ 
is the projective scheme over $\Spec O_{E_0}$ whose $R$-points for an $O_{E_0}$-algebra are the set of tuples $(\msF_\Lambda)_{\Lambda\in L}$ that satisfy the following conditions:
\begin{enumerate}[label=(\roman*)\,]

\item $\msF_\Lambda \subseteq Λ\tensor_{O_{F_0}}R$ is an $O_F$-stable $R$-module direct summand of rank $n$.

\item For each inclusion $i:\Lambda_1\subseteq \Lambda_2$ of lattices from $L$, the induced morphism 
\begin{equation*}
(i\tensor \mr{id}_R):\Lambda_1\otimes_{O_{F_0}}R\lr \Lambda_2\otimes_{O_{F_0}}R
\end{equation*}
satisfies $(i\tensor \mr{id}_R)(\msF_{\Lambda_1}) \subseteq \msF_{\Lambda_2}$.
\item For each $Λ \in L$, the multiplication map $π:Λ\simto πΛ$ induces an isomorphism
\begin{equation*}
(π\tensor \mr{id}_R):\msF_{\Lambda}\simlr\msF_{\pi\Lambda}.
\end{equation*}
\item For each lattice $Λ\in L$, the two submodules $\msF_\Lambda$ and $\msF_{\Lambda^{\vee}}$ are orthogonal to each other with respect to $(\psi_{F_0}\vert_{Λ\times Λ^\vee})\tensor \mr{id}_R$.
\item (Kottwitz condition) For every $Λ\in L$, let $\msL_Λ := (Λ\tensor_{O_F} R)/\msF_Λ$ denote the quotient. Then we demand that for each $a\in O_F$, we have the following identity of polynomials:
\begin{equation*}
    \mr{char}_R(a\otimes 1\mid \msL_{\Lambda})=(T-a)^r(T-\ov{a})^s.
\end{equation*}
Here, the right hand side defines a polynomial in $O_{E_0}[T]$ which is viewed over $R$ by the structure map $O_{E_0}\to R$.
\end{enumerate}

\smallskip \noindent (2) The \emph{relative local model} $\bfM_n^{L}(r,s)$ is defined as the flat closure of the generic fiber of $\bfM^{L,\naive}_n(r,s)$ in itself:
\begin{equation*}
\begin{aligned}
\xymatrix{
\bfM_n^{L,\naive}(r,s)\otimes_{O_{E_0}}E_0\ \ar@{^(->}[r]\ar@{^(->}[d]& \bfM_n^{L,\naive}(r,s)\\
\bfM_n^{L}(r,s).\ar@{^(->}[ur]&
}
\end{aligned}
\end{equation*}
Then $\bfM_n^L(r,s)$ is flat over $O_{E_0}$ by definition, see \cite[Proposition 14.14]{GW}.
\end{defn}

The relative unitary local model has been studied extensively. We summarize the most important results here:
\begin{thm}\label{thm:rel_LM_flat}
\begin{altenumerate}
\item When $F/F_0$ is unramified, for any $(r,s)$, the naive local model $\bfM_{n}^{L,\naive}(r,s)$ is flat.

\item When $F/F_0$ is ramified and $p\neq 2$, the naive local model $\bfM_n^{L,\naive}(r,s)$ is not flat in genreal.

\item When $F/F_0$ is unramified, the local model $\bfM^L_n(r,s)$ is normal and Cohen--Macaulay. When $(r,s)=(n-1,1)$ or $(1,n-1)$, it is regular with semi-stable reduction.

\item When $F/F_0$ is ramified and $p\neq 2$, the local model $\bfM^L_n(r,s)$ is normal and Cohen--Macaulay.
\end{altenumerate}
\end{thm}
\begin{proof}
When $F/F_0$ is unramified, a standard result shows that there is an isomorphism between local models $O_F\tensor_{O_{E_0}}\bfM_n^L(r,s)\cong \bfM_n^{L_0}(r)$, where $L_0$ is the set of sub-lattices $\Lambda_0\subset \Lambda\otimes_{O_{F_0}}O_F$, $\Lambda\in L$ such that $O_F$ acts strictly. See for example \S \ref{sec:KR-strata} when $(r,s)=(1,n-1)$ and \cite[Prop. 2.14]{HPR} for general. Now part (1) follows from \cite[Cor. 4.20]{Gortz}.

For part (2), the failure of the flatness follows from \cite[Prop. 3.8]{Pa}, see also \cite[P.27]{PRS13}.
The normality and Cohen--Macaulay property in parts (3) and (4) follow from the general result of \cite[Thm. 2.1]{TT23}.

Finally, when $F/F_0$ is unramified and $(r,s)=(n-1,1)$ or $(1,n-1)$, the semi-stability follows the comparison with $\bfM_n^{L_0}(1)$ resp. $\bfM_n^{L_0}(n-1)$ and \cite[Prop. 4.13]{Gortz} as in (1).
\end{proof}

\begin{rmk}\label{rmk:Luo_spin_condition}
Suppose $F/F_0$ is ramified and $p\neq 2$.
When $(r,s)=(n-1,1)$ or $(1,n-1)$, the relative local model $\bfM^L_{n}(r,s)$ can be represented as the closed subscheme of $\bfM_n^{L,\naive}(r,s)$ that satisfies the \emph{strengthened spin condition}, see \cite[\S 2.4]{Luo}. This was previously conjectured by \cite{Sm15}, inspired by the work of \cite{PRIII}, and known to hold for almost $\pi$-modular lattice chains by \cite[Theorem 1.4]{Sm15}.
When $(r,s)=(n-1,1)$ or $(1,n-1)$, for self-dual lattice chain, the relative local model $\mbM_n^L(r,s)$ can be represented by imposing the \emph{wedge condition} \cite{Pa}. For $\pi$-modular lattice chain, the relative local model can be represented by imposing the \emph{spin condition} \cite{PRIII,RSZ-Annalen}.
\end{rmk}

Recall the CM type data from \S \ref{sec:duality}: given $ϕ\in \Hom_K(F, \ov K)$ or $S\subseteq \Hom_K(F, \ov K)$ we define $\ov{ϕ} = ϕ\circ τ$ resp. $\ov S = \{\ov{ϕ} \mid ϕ\in S\}$. Let $ϕ_0\in \Hom_K(F, \ov K)$ and $A\subset \Hom_K(F, \ov K)$ be such that
$$\Hom_K(F, \ov K) = A \sqcup \{ϕ_0, \ov{ϕ_0}\}\sqcup \ov A$$
and set $B = A\sqcup \{ϕ_0, \ov{ϕ_0}\}$. Let $L$ (a lattice chain in $V$) and $(r,s)$ be as before. We define the reflex field
\begin{equation*}
E = \begin{cases}
E_AE_Bϕ_0(F) & \text{ if $r \neq s$};\\
E_AE_B & \text{ if $r = s$}.
\end{cases}
\end{equation*}

\begin{defn}\label{def:abs-LM}
(1) The \emph{naive absolute local model} $\bfM_{n,A}^{L,\naive}(r,s)$ is the projective scheme over $\Spec O_E$ whose $R$-point for an $O_E$-algebra $R$ are the set of tuples $(\mcF_\Lambda)_{\Lambda\in L}$ that satisfy the following conditions:
\begin{enumerate}[label=(\roman*)\,]
\item $\mcF_\Lambda \subset Λ\tensor_{O_K} R$ is a locally free $R$-submodule direct summand of rank $n[F_0:K]$. We require it to be stable under the $O_F$-action.

\item For each inclusion $i:\Lambda_1\subseteq \Lambda_2$ of lattices from $L$, the induced map
\begin{equation*}
i\tensor \mr{id}_R:\Lambda_1\otimes_{O_K}R\lr \Lambda_2\otimes_{O_K}R
\end{equation*}
satisfies $(i\tensor \mr{id}_R)(\mcF_{\Lambda_1}) \subseteq \mcF_{\Lambda_2}$.
\item For every $Λ\in L$, the multiplication map $π:\Lambda \simto πΛ$ induces an isomorphism
\begin{equation*}
    (π\tensor \mr{id}_R):\mcF_{\Lambda}\simlr\mcF_{\pi\Lambda}.
\end{equation*}
\item For every $Λ\in L$, the submodules $\mcF_\Lambda$ and $\mcF_{\Lambda^{\vee}}$ are orthogonal to each other with respect to $(\psi\vert_{Λ\times Λ^\vee})\tensor \mr{id}_R$.
\item (Eisenstein condition) For every $Λ\in L$, the submodule $\mcF_\Lambda$ is $(A,B)$-strict in the sense of Definition \ref{def:AB-strict}. That is, there are the inclusions
\begin{equation*}
    J_B\cdot (\Lambda\otimes_{O_K}R)\subseteq \mcF_\Lambda\subseteq J_A\cdot (\Lambda\otimes_{O_K}R).
\end{equation*}
\item (Kottwitz condition) By the Eisenstein condition, for every $Λ \in L$, the quotient $\mcL_Λ := (Λ\tensor_{O_K}R)/\mcF_Λ$ has the property that $e_A\cdot \mcL_Λ$ is locally free of rank $n$ over $R$. For any $a\in O_F$, we now require the following identity of polynomial functions:
\begin{equation*}
\mr{char}_R(a\tensor 1 \mid e_A\mcL_Λ;\,T) = (T-\phi_0(a))^r (T-\phi_0(\ov a))^s.
\end{equation*}
\end{enumerate}

\smallskip\noindent (2) The \emph{absolute local model} $\bfM_{n,A}^{L}(r,s)$ 
is defined as the flat closure of the generic fiber of the naive absolute local model:
\begin{equation*}
\begin{aligned}
\xymatrix{
\bfM_{n,A}^{L,\naive}(r,s)\otimes_{O_E}E\ \ar@{^(->}[r]\ar@{^(->}[d]&\bfM_{n,A}^{L,\naive}(r,s)\\
 \bfM_{n,A}^{L}(r,s)\ar@{^(->}[ur]&
}
\end{aligned}
\end{equation*}
Then $\bfM_{n,A}^L(r,s)$ is flat over $O_E$ by definition, see \cite[Proposition 14.14]{GW}.
\end{defn}

Our main result in this section is that the relative and the absolute local model become isomorphic after base change:
\begin{thm}\label{thm:LM-comparison}
\begin{altenumerate}
View $E_0$ as a subfield of $E$ via $ϕ_0$.
\item There exists an isomorphism between the naive absolute local model and the base change of the naive relative local model:
\begin{equation}\label{eq:iso_loc_mod_unit_naive}
\begin{aligned}
\xymatrix{
\bfM^{L,\naive}_{n,A}(r,s)\ar[r]^-{\sim} & O_E\tensor_{O_{E_0}} \bfM_{n}^{L,\naive}(r,s),&
(\mcF_\Lambda)_{Λ\in L}\ar@{|->}[r]&(\msF_\Lambda)_{Λ\in L}:=(e_A^{-1}(\mcF_\Lambda))_{Λ\in L}.
}
\end{aligned}
\end{equation}
\item It restricts to an isomorphism between absolute and relative local model:
\begin{equation}\label{eq:iso_loc_mod_unit}
\begin{aligned}
\xymatrix{
\bfM^{L}_{n,A}(r,s)\ar[r]^-{\sim} & O_E\tensor_{O_{E_0}} \bfM_{n}^{L}(r,s).
}
\end{aligned}
\end{equation}
\end{altenumerate}
\end{thm}
\begin{proof}
Consider first the naive local models. We have already explained in Proposition \ref{prop:e_S-inverse} and Theorem \ref{thm:local_model_comparison} how the formula in \eqref{eq:iso_loc_mod_unit_naive} provides a bijection between ``absolute'' tuples $(\mcF_Λ)_{Λ\in L}$ that satisfy (i), (ii), (iii) and (v) of Definition \ref{def:abs-LM} and ``relative'' tuples $(\msF_Λ)_{Λ\in L}$ that satisfy (i), (ii) and (iii) of Definition \ref{def:unitary_loc_mod_relative}. Moreover, it is clear that this construction identifies $(e_A\cdot \mcL_Λ)_{Λ\in L}$ with $(\msL_Λ)_{Λ\in L}$. So the Kottwitz condition holds for $(\mcF_Λ)_{Λ\in L}$ if and only if it holds for the corresponding tuple $(\msF_Λ)_{Λ\in L}$. It is now only left to check that $\mcF_Λ\perp \mcF_{Λ^\vee}$ if and only if $\msF_Λ\perp \msF_{Λ^\perp}$. To this end, we construct a slight extension of the diagram \eqref{eq:restricted_pairings_diagram}.

Assume that $Λ_1 \subseteq Λ_2^\vee$ for two lattices $Λ_1, Λ_2\in L$. We can view $ψ_{F_0}$ as $O_K$-bilinear pairing and obtain by scalar extension, for every $O_E$-algebra $R$, an $O_F$-conjugate linear pairing
\begin{equation}\label{eq:flexible_base_change}
Λ_1\tensor_{O_K} R\, \times\, Λ_2\tensor_{O_K} R \,\lr\, O_{F_0}\tensor_{O_K}R.
\end{equation}
If we base change this pairing along the quotient map $O_{F_0}\tensor_{O_K}R \twoheadrightarrow R$ using $ϕ_0$, then we recover the ``usual'' base chang pairing
$$Λ_1\tensor_{O_{F_0}} R\, \times\, Λ_2\tensor_{O_{F_0}} R \,\lr\, R.$$
The pairing from \eqref{eq:flexible_base_change} can also be restricted to submodules, however. This yields the upper square in the promised extension of \eqref{eq:restricted_pairings_diagram}:
\begin{equation}\label{eq:restricted_pairings_diagram_extended}
\xymatrix{
Λ_1 \tensor_{O_{F_0}} R \times Λ_2 \tensor_{O_{F_0}} R \ar[rr]^-{\psi_{F_0}}\ar[d]_-{\iso}^-{(e_A\ ,\,e_A\ )} && R \ar[d]_-{\iso}^-{e_A\ov{e_A}}\\
\frac{J_A(Λ_1\tensor_{O_K}R)}{J_B(Λ_1\otimes_{O_K}R)}\times \frac{J_A(Λ_2\tensor_{O_K}R)}{J_B(Λ_2\tensor_{O_K}R)}\ar[rr]^-{\psi_{F_0}} \ar[rrd]_{ψ} && J_{A_0, R} \ar[d]^{\btr}_{\iso}\\
&& R.}
\end{equation}
Consider now the situation $(Λ_1, Λ_2) = (Λ, Λ^\vee)$. Then a pair of filtrations $(\msF_Λ, \msF_{Λ^\vee})$ in the top left of \eqref{eq:restricted_pairings_diagram_extended} is mutually perpendicular if and only if $(e_A\msF_Λ, e_A\msF_{Λ^\vee})$ in the middle left has this property. Since always $J_B(Λ\tensor_{O_K} R) \perp J_B(Λ\tensor_{O_K} R)$ as in Lemma \ref{lem:restrict_lifted_form}, this is equivalent to $\mcF_Λ \perp \mcF_{Λ^\vee}$. The proof of (1) is then complete.

Statement (2) follows immediately because the flat closure construction that defines $\bfM^L_{n}(r,s) \subseteq \bfM^{L,\naive}_n(r,s)$ commutes with the base change along $O_{E_0} \to O_E$.
\end{proof}

\begin{cor}\label{cor:LM-comparison}
\begin{altenumerate}
\item When $F/F_0$ is unramified, for any $(r,s)$, the naive absolute local model $\bfM_{n}^{L,\naive}(r,s)$ is flat, normal, and Cohen--Macaulay.
\item When $F/F_0$ is ramified and $p\neq 2$, the absolute local model $\bfM_{n}^{L}(r,s)$ is flat, normal, and Cohen--Macaulay.
\item When $(r,s)=(n-1,1)$ or $(1,n-1)$, the absolute local model $\bfM_{n,A}^{L}$ can be characterized as the closed subscheme of $\bfM_{n,A}^{L,\naive}$ where the \emph{relative} strengthened spin condition holds. To be more precise, for any point $(\mcF_\Lambda)\in \bfM^{L,\naive}_{n,A}$, the relative filtration $(e_A^{-1}(\mcF_\Lambda))$ determines a point in $O_E\otimes_{O_{E_0}}\bfM_n^{L,\naive}(r,s)$, which is required to satisfy the strengthened spin condition from \cite[\S 2.4]{Luo}.
\end{altenumerate}
\end{cor}
\begin{proof}
Combine Theorems \ref{thm:rel_LM_flat} and \ref{thm:LM-comparison}, as well as Remark \ref{rmk:Luo_spin_condition}.
\end{proof}

\begin{rmk}\label{rmk:kramer model}
\begin{altenumerate}
\item  Corollary \ref{cor:LM-comparison} (1) can also be deduced from Theorem \ref{thm:local_model_comparison}.

\item Definition \ref{def:abs-LM} and Theorem \ref{thm:LM-comparison} can also be formulated for Krämer models. Concretely, the Krämer datum is added for the $O_F\tensor_{O_{F_0}} R$-module $e_A\mcL$ in Definition \ref{def:abs-LM}.
\end{altenumerate}
\end{rmk}

We have commented in the $\mathrm{GL}_n$-case, see Remark \ref{rmk:reflex}, that the definition of $\bfM^L_{n,A}(0)$ does not refer to the element $ϕ_0$. A similar ``banal'' variant comes up in the unitary case: Assume $A\subset \Hom_K(F, \ov K)$ satisfies
$$\Hom_K(F, \ov K) = A \sqcup \ov A,$$
and let $L$ be a self-dual lattice chain as before. Then we define the banal local model over $O_{E_A}$ by
\begin{equation}\label{eq:banal_LM_unitary}
\bfM^L_{n,A, \mr{banal}}(R) = \{(J_A\cdot (R\tensor_{O_K} Λ))_{Λ\in L}\}\quad\quad\text{(a singleton)}.
\end{equation}
Clearly, this definition just recovers $\Spec(O_{E_A})$. We have made this comment here mainly for later reference.

\section{$O$-Displays}
\label{s:displays}

\subsection{$O$-Displays}
\label{ss:displays}
We first recall some necessary background on relative displays. Our main references for this will be \cite{ACZ} and \cite[\S3 and \S4]{KRZ}. Let $K/\mbQ_p$ be a finite extension with ring of integers $O = O_K$. Let $q$ be the residue cardinality of $K$ and let $π\in O$ be a uniformizer. We consider $π$-adically complete $O$-algebras with the $π$-adic topology, and we simply call these \emph{$π$-adic $O$-algebras}. With this convention, at the level of topological spaces, $\Spf(R) = \Spec(R/π)$ for any $π$-adic $O$-algebra $R$. 

We denote by $W_O(R)$ the ring of $O$-Witt vectors of a $π$-adic $O$-algebra $R$. We write $I_O(R) = \ker(W_O(R) \twoheadrightarrow R)$ for the kernel of the projection to $R$ as well as $σ,V:W_O(R)\to W_O(R)$ for Frobenius and Verschiebung. Recall that $W_O(R)$ is an $O$-algebra, that $σ$ is an $O$-algebra endomorphism, and that $σ$ and $V$ satisfy $σ\circ V = π$. Also recall that $V$ is injective with image $I_O(R)$. In particular, we may define the $σ$-linear surjection
$$\dot{σ}:I_O(R) \lr W_O(R),\quad \dot{σ}(x) := V^{-1}(x).$$

\begin{defn}[$O$-displays, \protect{\cite[\S2.1]{ACZ}}]\label{def:O_display}
Let $R$ be a $π$-adic $O$-algebra. An \emph{$O$-display over $R$} is a quadruple $\mcP = (P, Q, \bfF, \dF)$ whose entries are of the following kind:
\begin{altitemize}
    \item $P$ is a finite projective $W_O(R)$-module;
    \item $Q\subseteq P$ is a submodule with $I_O(R)P\subseteq Q$ and such that $P/Q$ is a projective $R$-module;
    \item $\bfF:P\to P$ and $\dF:Q\to P$ are two $σ$-linear maps in the sense that
    $$\bfF(ξx) = σ(ξ) \bfF(x),\quad \dF(ξy) = σ(ξ)\dF(y)\quad \text{for $ξ\in W_O(R)$, $x\in P$ and $y\in Q$.}$$
\end{altitemize}
We require that these data satisfy the following two axioms:
\begin{altenumerate}
    \item $\dF(Q)$ generates $P$ as $W_O(R)$-module;
    \item For $x\in P$ and $ξ\in I_O(R)$, we have
    \begin{equation}\label{eq:display_axiom}
    \dF(ξx) = \dot{σ}(ξ)\bfF(x).
    \end{equation}
\end{altenumerate}
The rank of $P$ as $W_O(R)$-module is called the \emph{height of $\mcP$}. The quotient $P/Q$ is called the \emph{Lie algebra of $\mcP$} and denoted by $\Lie(\mcP)$. Its rank as $R$-module is called the \emph{dimension of $\mcP$}. Height and dimension of $\mcP$ are locally constant functions on $\Spf R$. The filtration $Q/I_O(R)\subseteq P/I_O(R)P$ is called the \emph{Hodge filtration} of $\mcP$. Note that there is an exact sequence of finite projective $R$-modules
\begin{equation}\label{eq:Hodge_O_display}
0 \lr Q/I_O(R)P \lr P/I_O(R)P \lr \Lie(\mcP)\lr 0.
\end{equation}
\end{defn}

\begin{ex}\label{ex:displays}
(1) The quadruple $\mcW_O(R) := (W_O(R), I_O(R), σ, \dot{σ})$ is an $O$-display over $R$. It is sometimes called the \emph{multiplicative $O$-display} because its role in the category of $O$-displays is analogous to that of $μ_{p^\infty}$ in the category of $p$-divisible groups.

\smallskip \noindent (2) If $R$ is a perfect field, or more generally a perfect $O$-algebra, then $O$-displays over $R$ are the same as $O$-Dieudonné modules over $R$, see \cite[Proposition 3.1.9]{KRZ}.

\smallskip \noindent (3) $O$-Displays $(P, Q, \bfF, \dF)$ such that $P/Q \iso R^{\oplus d}$ and $Q/I_O(R) \iso R^{\oplus h-d}$ are both free modules can be described explicitly in terms of invertible $(h\times h)$-matrices, see \cite[after Definition 1]{Zink}. Note that $P$ and $P/Q$ are always locally free over $\Spf R$, so in principle this provides an explicit description of all $O$-displays.
\end{ex}

The operator $\bfF$ is in fact uniquely determined by $\dF$: For every $x\in P$, by the second display axiom \eqref{eq:display_axiom},
\begin{equation}\label{eq:Frobenius_determined}
\bfF(x) = \dot{σ}(V(1))\bfF(x) = \dF(V(1)\cdot x).
\end{equation}
This is convenient because it allows to define homomorphisms and bilinear pairings of $O$-displays entirely in terms of $\dF$.

\begin{defn}\label{def:hom_display}
Let $\mcP_i = (P_i, Q_i, \bfF_i, \dF_i)$ be two $O$-displays. A homomorphism from $\mcP_1$ to $\mcP_2$ is a $W_O(R)$-linear map $f:P_1\to P_2$ such that $f(Q_1)\subseteq Q_2$ and such that $f\circ \dF_1 = \dF_2 \circ f$.
\end{defn}

\begin{defn}[\protect{\cite[\S3.2]{KRZ}}]\label{def:bihom_display}
Let $\mcP_i = (P_i, Q_i, \bfF_i, \dF_i)$ for $i = 1,2,3$, be three $O$-displays. A bilinear form of displays
$$(\ ,\ ):\mcP_1\times \mcP_2 \lr \mcP_3$$
is a $W_O(R)$-bilinear pairing $P_1\times P_2 \to P_3$ that satisfies $(Q_1, Q_2) \subseteq Q_3$ and
\begin{equation}\label{eq:bilin_pairing_display}
(\dF_1(x_1), \dF_2(x_2)) = \dF_3((x_1, x_2))\quad \text{for all $x_1\in Q_1$, $x_2\in Q_2$}.
\end{equation}
\end{defn}
Given an $O$-display $\mcP = (P, Q, \bfF, \dF)$, there is a unique $O$-display structure $\mcP^\vee := (P^\vee, Q^\vee, \bfF^\vee, \dF^\vee)$ on $P^\vee := \Hom_{W_O(R)}(P, W_O(R))$ such that the canonical pairing $P\times P^\vee \to W_O(R)$ defines a bilinear form of displays $\mcP\times \mcP^\vee \to \mcW_O(R)$. This form is universal in the sense that for every other $O$-display $\mcP'$ over $R$,
$$\Hom(\mcP', \mcP^\vee) \simlr \mr{BiHom}(\mcP\times \mcP', \mcW_O(R)).$$
The $O$-display $\mcP^\vee$ is called the dual display of $\mcP$. It can be described explicitly by the relations $Q^\vee/I_O(R)P^\vee = (Q/I_O(R)P)^\perp$ and
\begin{equation}\label{eq:dual_display_key}
(\dF^\vee(\ell))(\dF(x)) = \dot{σ}(\ell(x))\quad \text{for all $x\in Q,\ \ \ell\in Q^\vee$}.
\end{equation}
In particular, the Hodge filtration of $\mcP^\vee$ is the orthogonal complement of the Hodge filtration of $\mcP$. Dualization defines a self anti-equivalence of the category of $O$-displays. It is an involution in the sense that the natural map $P\simto (P^\vee)^\vee$ defines an isomorphism $\mcP\to (\mcP^\vee)^\vee$.

\begin{defn}[Base change of $O$-displays, \protect{\cite[\S3.1]{KRZ}}]
Let $R\to R'$ be a map of $π$-adic $O$-algebras and let $\mcP = (P, Q, \bfF, \dF)$ be an $O$-display over $R$. The \emph{base change $R'\tensor_R \mcP$ of $\mcP$ to $R'$} is defined as the $O$-display $(P', Q', \bfF', \dF')$ with
\begin{equation}
P' = W_O(R')\tensor_{W_O(R)}P,\quad Q' = \mr{Im}(W_O(R') \tensor_{W_O(R)} Q \lr P')
\end{equation}
and with $\bfF'$, $\dF'$ the unique $σ$-linear maps that are given on $W_O(R')$-module generators by
\begin{equation}
\bfF'(1\tensor x) = 1\tensor \bfF(x),\quad \dF'(1\tensor y) = 1\tensor \dF(y)\quad x\in P,\ y\in Q.
\end{equation}
Clearly, base change is functorial in $\mcP$. Moreover, its effect on the Hodge exact sequence \eqref{eq:Hodge_O_display} is given by
\begin{multline}
    R'\tensor_R [0 \lr Q/I_O(R)P \lr P/I_O(R)P \lr \Lie(\mcP) \lr 0] \\[1mm]
    \simlr [0 \lr Q'/I_O(R')P' \lr P'/I_O(R')P' \lr \Lie(R'\tensor_R \mcP) \lr 0].
\end{multline}
Finally, base change induces a natural map
$$\mr{BiHom}(\mcP_1 \times \mcP_2, \mcP_3) \lr \mr{BiHom}(R'\tensor_R \mcP_1 \times R'\tensor_R\mcP_2, R'\tensor_R\mcP_3)$$
and commutes with dualization up to the natural isomorphism given by
$$W_O(R')\tensor_{W_O(R)} (P^\vee)\simlr (W_O(R')\tensor_{W_O(R)} P)^\vee.$$
\end{defn}

\begin{ex}[Étale $O$-displays]\label{ex:etale_displays}
An $O$-display $(P, Q, \bfF, \dF)$ over $R$ is said to be \emph{étale} if $P = Q$. Then, by definition, $\dF$ is a \emph{$σ$-linear isomorphism}. That is, its linearization defines an isomorphism
$$\dF^\# := \mr{id}\tensor \dF:W_O(R) \tensor_{σ,W_O(R)} P \simlr P.$$
Conversely, let $(P, \dF)$ be a finite projective $W_O(R)$-module together with a $σ$-linear isomorphism $\dF:P\to P$. Then $(P, P, \bfF, \dF)$ is an étale $O$-display, where $\bfF$ is defined by \eqref{eq:Frobenius_determined}. Pairs $(P, \dF)$ of this kind are called \emph{étale Frobenius modules} in \cite{KRZ}, and the above constructs an equivalence with étale $O$-displays. In a slight abuse of terminology, we continue to call such pairs $(P, \dF)$ étale $O$-displays. By \cite[Proposition 4.5.2]{KRZ}, there is an equivalence of fibered categories over $π$-adic $O$-algebras
\begin{equation}\label{eq:equiv_etale}
\begin{aligned}
\left\{\text{\begin{varwidth}{\textwidth}\centering Étale $O$-displays \\ over $R$\end{varwidth}}\right\} &\ \overset{\sim}{\longleftrightarrow}\ \left\{\text{\begin{varwidth}{\textwidth} \centering Étale local systems of finite\\projective $O$-modules over $R$\end{varwidth}}\right\}\\[1mm]
(P, \dF) &\ \,\longmapsto\ P^{\dF = \mr{id}}\\
(W_O \tensor_O \mbP, σ\tensor \mr{id})(R) & \ \,\longmapsfrom\ \mbP.
\end{aligned}
\end{equation}
Here, several objects have to be understood as sheaves for the fpqc (or proétale) topology. For example, if $(P, \dF)$ is an étale $O$-display over $R$, then $P^{\dF = \mr{id}}$ is the functor on $π$-adic $R$-algebras given by
$$P^{\dF = \mr{id}}(A) := \big(W_O(A)\tensor_{W_O(R)} P\big)^{(\mr{id}\tensor \dF) = \mr{id}}.$$
For the quasi-inverse functor, if $\mbP$ is a local system over $R$, then $W_O\tensor_O \mbP$ denotes the sheafification 
$$\big[A\longmapsto W_O(A)\tensor_{\underline{O}(A)} \mbP(A)\big]^{\mr{Sh}}$$
and \eqref{eq:equiv_etale} denotes the evaluation at $R$.

Assume that $\mcP_i = (P_i, \dF_i)$ for $i = 1,2,3$ are three étale $O$-displays and that $λ:\mcP_1\times \mcP_2 \to \mcP_3$ is a bilinear form of $O$-displays. Then we obtain from \eqref{eq:bilin_pairing_display} that $λ$ restricts to a bilinear form of local systems
$$\mcP_1^{\dF_1 = \mr{id}} \times \mcP_2^{\dF_2 = \mr{id}} \lr \mcP_3^{\dF_3 = \mr{id}}.$$
It is clear from the explicit description in \eqref{eq:equiv_etale} that this defines a bijection
\begin{equation}\label{eq:iso_duality_etale}
\mr{BiHom}(\mcP_1\times \mcP_2, \mcP_3) \simlr \mr{BiHom}(\mcP_1^{\dF_1 = \mr{id}} \times \mcP_2^{\dF_2 = \mr{id}}, \mcP_3^{\dF_3 = \mr{id}}).
\end{equation}
Moreover, we have a notion of internal homomorphisms: there is a unique étale $O$-display structure $\mcH om(\mcP_1, \mcP_3) = (\Hom(P_1, P_3), α)$ such that the natural pairing $P_1\times \Hom(P_1, P_3) \to P_3$ defines a bilinear pairing of $O$-displays
$$\mcP_1\times \mcH om(\mcP_1, \mcP_3) \lr \mcP_3.$$
Then $\mcH om(\mcP_1, \mcP_3)$ represents the functor of bilinear pairings with $\mcP_1$ to $\mcP_3$ in the sense that the following natural map is an isomorphism
\begin{equation}\label{eq:univ_pairing_etale}
\Hom(\mcP_2, \mcH om(\mcP_1, \mcP_3)) \simlr \mr{BiHom}(\mcP_1\times \mcP_2, \mcP_3).
\end{equation}
\end{ex}

\subsection{The modification functor}
\label{ss:modification_functor}
Let $F/K$ be a finite field extension and let $A\subseteq B\subseteq \Hom_K(F,\ob K)$ be two subsets. Let $E = E_A E_B \subseteq \ov K$ be the join of their reflex fields.

Assume that $R$ is a $π$-adic $O$-algebra and that $\mcP = (P, Q, \bfF, \dF)$ is an $O$-display over $R$. In the following, we will consider such $O$-displays with an additional $O$-linear $O_F$-actions $ι:O_F\to \End(\mcP)$. When we speak of \emph{pairs} $(\mcP, ι)$, then we always mean such a datum. By \cite[Lemma 3.1.15]{KRZ}, for every pair $(\mcP, ι)$, the module $P$ is locally free as $O_F\tensor_O W_O(R)$-module.

\begin{defn}\label{def:ABdis}
Let $R$ be a $π$-adic $O_E$-algebra and let $\mcP = (P, Q, \bfF, \dF)$ be an $O$-display over $R$. An $O_F$-action $ι:O_F\to \End(\mcP)$ is said to be \emph{$(A,B)$-strict} if the induced $O_F$-action on the Lie algebra quotient $P/I_O(R)P \twoheadrightarrow \Lie(\mcP)$ is $(A, B)$-strict in the sense of Definition \ref{def:AB-strict}.

For brevity, we simply speak of \emph{$(A,B)$-strict pairs} $(\mcP, ι)$. Let $(A,B)\text{-Disp}$ denote the fibered category of $(A, B)$-strict pairs on the category of $π$-adic $O_E$-algebras. The morphisms in this category are the $O_F$-linear maps of $O$-displays.
\end{defn}

\begin{lem}\label{lem:dual_display_strict}
Let $(\mcP, ι)$ be an $(A, B)$-strict pair over $R$. Let $\mcP^\vee$ be the dual $O$-display with dual action $ι^\vee(a) := ι(a)^\vee$. Then $(\mcP^\vee, ι^\vee)$ is $(B^c, A^c)$-strict. In other words, duality defines an anti-equivalence
\begin{equation}\label{equ:LT-dual}
(A,B)\text{-Disp} \simlr (B^c,A^c)\text{-Disp}.
\end{equation}
\end{lem}
\begin{proof}
Write $\mcP = (P, Q, \bfF, \dF)$. The Lie algebra of $\mcP^\vee$ by definition equals $\Hom(Q/I_O(R)P, R)$. It is $(B^c, A^c)$-strict by Lemma \ref{lem:AB-strict_dualizing}.
\end{proof}

Assume that $S\subseteq A\subseteq \Hom_K(F, \ob K)$ is a further subset. Our aim will be to compare $(A,B)$-strict and $(A\setminus S, B\setminus S)$-strict pairs which requires a field extension. So we redefine $E$ as $E := E_SE_AE_B$. Note that $E$ also contains the join $E_{A\setminus S}E_{B\setminus S}$. The categories $(A, B)$-Disp and $(A\setminus S, B\setminus S)$-Disp are from now on understood as fibered categories over $π$-adic $O_E$-algebras for this new $E$. We remark that in applications, one is usually interested in the situation $S = A$ in which no enlargement of $E$ is necessary.

\begin{thm}\label{thm:AB_equiv_displays}
There is an equivalence of fibered categories on $π$-adic $O_E$-algebras
$$Φ_S:(A,B)\text{-Disp} \simlr (A\setminus S, B\setminus S)\text{-Disp}.$$
\end{thm}

\begin{defn}\label{def:modification_functor}
We call the functor $Φ_S$ that will be constructed during the proof of Theorem \ref{thm:AB_equiv_displays} the \emph{modification functor}. 
\end{defn}

Similarly to the proofs in \cite{KRZ, Mih}, the idea is to choose a suitable lifting $\wt{e_S}$ of one of the elements $e_S$ from \eqref{eq:def_e_S} along the surjection
$$O_F\tensor_O W_O(O_E) \twoheadrightarrow O_F\tensor_O O_E,$$
and then to ``divide'' the display structure of an $(A,B)$-strict pair $(\mcP,ι)$ by $\wt{e_S}$. To this end, given an element $ζ\in O_F$, define the polynomial $\wt {e_{ζ,S}}(T) \in W_O(O_{E_S})[T]$ by the formula
\begin{equation}
\wt{e_{ζ,S}}(T) = \prod_{ϕ\in S} (T - 1\tensor [ϕ(ζ)]).
\end{equation}
A priori, $\wt{e_{ζ,S}}(T)$ is defined in $W_O(O_{F^{\mr{Gal}}})[T]$, but all its coefficients are fixed under $\mr{Gal}(\ob K/E_S)$ and hence lie in $W_O(O_{E_S})$. Note that $\wt{e_{ζ,S}}(T)$ lifts the polynomial $e_{ζ,S}(T)$ from \eqref{eq:def_eisenstein_element}.

We introduce some notation. Let $k_F$ be the residue field of $F$, and let $\ov k$ be the residue field of $\ov K$. There is a decomposition
\begin{equation}\label{eq:decomp_psi}
O_F\tensor_O W_O(\ov k) \simlr \prod_{ψ:k_F\to k} O_F\tensor_{W_O(k_F), ψ} W_O(\ov k).
\end{equation}
Here, the embedding $W_O(k_F)\to O_F$ is the one that induces the identity on residue fields. Each factor in \eqref{eq:decomp_psi} is a DVR (the completion of a maximal unramified extension of $O_F$, more precisely). Denote by $S_ψ\subseteq S$ the subset of homomorphisms that specialize to $ψ:k_F\to \ov k$. 
\begin{lem}\label{lem:valuation_lift_W}
Assume that $ζ\in O_F$ is an $O$-algebra generator of the form $ζ = μ + π_F$, where $μ\in O_F$ is a root of unity and $π_F\in O_F$ a uniformizer. Let $\ov e_ψ$ denote the image of $\wt{e_{ζ,S}}(ζ\tensor 1)$ in the $ψ$-factor of \eqref{eq:decomp_psi}. Then the normalized valuation of $\ov e_ψ$ is $|S_ψ|$.
\end{lem}
\begin{proof}
Let $ϕ\in S$ be any element. The $ψ$-component $ξ_{ϕ,ψ}$ of the image of $ζ\tensor 1 - 1\tensor [ϕ(ζ)]$ in \eqref{eq:decomp_psi} is
$$\begin{aligned}
(μ + π_F)\tensor 1 - 1\tensor [ϕ(μ)] = π_F\tensor 1 + 1\tensor ([ψ(μ)] - [ϕ(μ)]).
\end{aligned}$$
The Teichmüller lifting is multiplicative, so $[ψ(μ)]$ and $[ϕ(μ)]$ are the unique roots of unity that agree with the images of $ψ(μ)$ resp. $ϕ(μ)$ in $\ov k$. It follows that $[ψ(μ)] = [ϕ(μ)]$ if $ϕ\in S_ψ$, or $[ψ(μ)] - [ϕ(μ)] \in W_O(\ov k)^\times$ if $ϕ\notin S_ψ$. So we find that $ξ_{ϕ,ψ}$ has normalized valuation
$$\begin{cases}
    1 & \text{if $ϕ\in S_ψ$}\\
    0 & \text{if $ϕ\notin S_ψ$.}
\end{cases}$$
Taking the product over all $ϕ\in S$, we see that the normalized valuation of $\ov e_ψ$ is $|S_ψ|$ as claimed.
\end{proof}

\begin{proof}[Proof of Theorem \ref{thm:AB_equiv_displays}.] Let $S\subseteq A \subseteq B$ and $E = E_SE_AE_B$ be as in the theorem. Fix an $O$-algebra generator $ζ\in O_F$ of the form $μ + π_F$ as in Lemma \ref{lem:valuation_lift_W}. Put
\begin{equation}\label{eq:def_e_S_tilde}
\wt{e_S} := \wt{e_{ζ, S}}(ζ\tensor 1)\in O_F\tensor_OW_O(O_E).
\end{equation}
Let $(P,Q,F,\dot F, ι)$ be an $(A, B)$-strict pair over a $π$-adic $O_E$-algebra $R$. Define
\begin{equation}\label{equ:const-Q_0}
Q_0 := \ker \left[\wt{e_S}: P \lr P/Q \right].
\end{equation}
Recall from Proposition \ref{prop:e_S-inverse} that if $M$ is a locally free $O_F\tensor_O R$-module, then multiplication by $e_S$ defines an isomorphism
$$e_S:(e_{A\setminus S}M)/(e_{B\setminus S}M) \simlr (e_AM)/(e_BM).$$
Applying this with $M = P/I_O(R)P$ shows that $P/I_O(R)P \twoheadrightarrow P/Q_0$ is $(A\setminus S, B\setminus S)$-strict. We define
\begin{equation}\label{equ:const-F_0}
\dF_{0}:Q_0\longrightarrow P,\ x\longmapsto \dF(\wt{e}_S x).
\end{equation}
Then $\dF_0:Q_0\to P$ is a $σ$-linear epimorphism because this property can be checked in geometric points of $\Spf(R)$ where it reduces to a simple Dieudonné module calculation. This calculation is just like the one in \cite[from (4.3.14) to (4.3.18)]{KRZ}; it is here that Lemma \ref{lem:valuation_lift_W} is used.

In the same way, define $\bfF_0(x) = \bfF(\wt{e_S} x).$ Then $\Phi_S(\mcP) := (P, Q_0, \bfF_0, \dF_0)$ is a display with $(A\setminus S, B\setminus S)$-strict $O_F$-action. We claim that $Φ_S$ is an equivalence. Since it is tricky to invert the construction in \eqref{equ:const-F_0} directly, we prove this with a dualizing argument.
\begin{prop}\label{prop:commutativ_square}
The following square commutes up to the natural isomorphism $\mcP\simto (\mcP^\vee)^\vee$
\begin{equation*}
\xymatrixcolsep{2cm}
\xymatrix{
(A,B)\mathrm{-Disp} \ar[r]^-{Φ_S} \ar[d]_{\cong}^{\mr{dualize}} & (A\setminus S, B\setminus S)\mathrm{-Disp} \ar[d]^{\cong}_{\mr{dualize}}\\
(B^c,A^c)\mathrm{-Disp} & (B^c\cup S,A^c\cup S)\mathrm{-Disp}. \ar[l]_-{Φ_{S}}
}
\end{equation*}
\end{prop}
\begin{proof}
We consider an object $(P, Q, \bfF, \dF)$ in the source category and its images. We omit the $O_F$-action in the notation:
\begin{equation}
\xymatrixcolsep{1.5cm}
\xymatrix{
(P, Q, \bfF, \dF) \ar@{|->}[r]^-{Φ_S} & (P, Q_0, \bfF_0, \dF_0) \ar@{|->}[d]\\
(P^\vee, Q^\vee_{00}, \bfF^\vee_{00}, \dF^\vee_{00}) & (P^\vee, Q_0^\vee, \bfF_0^\vee, \dF_0^\vee). \ar@{|->}[l]_-{Φ_S}
}
\end{equation}
We claim that the natural pairing $P\times P^\vee\to W_O(R)$ defines an isomorphism
$$(P, Q, \bfF, \dF) \simlr (P^\vee, Q^\vee_{00}, \bfF^\vee_{00}, \dF^\vee_{00})^\vee.$$
Set $I = I_O(R)$ in the following. First we check
$$\begin{array}{rcl}
Q^\vee_{00}/IP^\vee &\ \overset{\text{Def. of $Φ_S$}}{=}\ & e_S^{-1} (Q_0^\vee/IP)\\
&\ \overset{\text{Def. of $\vee$}}{=}\ & e_S^{-1} (Q_0/IP)^\perp\\
& = & (e_SQ_0/IP)^\perp\\
& \overset{\text{Def. of $Φ_S$}}{=} & (Q/IP)^\perp\\
& \overset{\text{Def. of $\vee$}}{=} & Q^\vee/IP^\vee
\end{array}$$
and thus $Q^\vee_{00} = Q^\vee$. We claim that $\dF^\vee_{00} = \dF^\vee$. These are both homomorphisms $Q^\vee_{00} = Q^\vee \to P^\vee$. By $W_O(R)$-$σ$-linearity, it suffices to check the identity
\begin{equation}\label{eq:dual_to_check}
(\dF^\vee_{00}(\ell))(w) = (\dF^\vee(\ell))(w)
\end{equation}
for all $\ell\in Q^\vee$ and all $w$ within a set of $W_O(R)$-module generators of $P$. Such a set is given by $\{\dF(\wt {e_S} x) \mid x\in Q_0\}$. We have
$$\begin{array}{rcl}
\big(\dF_{00}^\vee (\ell)\big)(\dF(\wt {e_S} x)) & = & \big(\dF_0^\vee(\wt {e_S}\ell)\big)(\dF_0(x))\\[1mm]
& \overset{\eqref{eq:dual_display_key}}{=} & V^{-1}\big((\wt {e_S}\ell)(x)\big)\\[1mm]
& = & V^{-1}\big(\ell(\wt {e_S} x)\big)\\[1mm]
& \overset{\eqref{eq:dual_display_key}}{=} & \big(\dF^\vee(\ell)\big)(\dF(\wt {e_S} x)).
\end{array}$$
The proof of the proposition is complete.
\end{proof}
Proposition \ref{prop:commutativ_square} shows that the functor $Φ_S$ has the quasi-inverse $(\text{dual})\circ Φ_S\circ (\text{dual})$ and is hence an equivalence. This concludes the proof of Theorem \ref{thm:AB_equiv_displays}.
\end{proof}

\subsection{Lubin--Tate displays}
\label{ss:LT}
Let $F_0/K$ be a finite extension and let $R$ be a $π$-adic $O_{F_0}$-algebra. (We will later work with a quadratic extension $F/F_0$ which explains this choice of notation.) Then we have a notion of \emph{strict pair} $(\mcP, ι_0)$ over $R$, that is, an $O$-display $\mcP$ over $R$ together with a strict $O_{F_0}$-action $ι_0$ in the sense of Example \ref{ex:strict}. An equivalent condition is as follows: Write $\mcP = (P, Q, \bfF, \dF)$ and recall that for any $O_{F_0}$-action on $\mcP$, the module $P$ is projective over $O_{F_0}\tensor_OW_O(R)$. Define $J_O(R)$ as the kernel of the composition
$$J_O(R) = \ker\big[O_{F_0}\tensor_O W_O(R) \lr O_{F_0}\tensor_O R \lr R\big],$$
where the second map is given by multiplication. Then an $O_{F_0}$-action on $\mcP$ is strict if and only if
$$J_O(R)\cdot P \subseteq Q \subseteq P.$$
\begin{defn}\label{def:LT_display_general}
Let $R$ be a $π$-adic $O_{F_0}$-algebra. A \emph{Lubin--Tate $O$-display} (with respect to $F_0$) is a strict pair $(\mcL, ι_0)$ over $R$ where $\mcL$ has height $[F_0:K]$ and dimension $1$.
\end{defn}
If $\mcL = (L, Q, \bfF, \dot \bfF, ι_0)$ is a Lubin--Tate display, then $L$ is a line bundle over $O_{F_0}\tensor_O W_O(R)$. It follows that $Q$ is uniquely determined as $Q = J_O(R)\cdot L$.

Given strict pairs $(\mcP_i = (P_i, Q_i, \bfF_i, \dF_i), ι_{0,i})$ for $i = 1,2,3$, we write $\Hom_{O_{F_0}}(\mcP_1, \mcP_2)$ for the module of $O_{F_0}$-linear homomorphisms of $O$-displays, and $\mr{BiHom}_{O_{F_0}}(\mcP_1\times \mcP_2, \mcP_3)$ for the module of $O_{F_0}$-bilinear pairings of $O$-displays.
\begin{prop}\label{prop:dual_LT}
Let $R$ be a $π$-adic $O_{F_0}$-algebra and let $\mcL = (L, J_O(R)L, α, \dot{α})$ be a Lubin--Tate display over $R$. For every strict pair $(\mcP, ι_0)$ over $R$, there exists a unique strict pair structure $(\mcP^\dagger, ι^\dagger_0)$ on
$$P^\dagger := \Hom_{O_{F_0}\tensor_OW_O(R)}(P, L)$$
such that the natural pairing $P\times P^\dagger \to L$ defines an bilinear pairing $\mcP\times \mcP^\dagger \to \mcL$ that is universal: for every strict pair $(\mcP', ι'_0)$, the induced map
$$\Hom_{O_{F_0}}(\mcP',\mcP^\dagger) \simlr \mr{BiHom}_{O_{F_0}}(\mcP \times \mcP', \mcL)$$
is an isomorphism.
\end{prop}
\begin{proof}
There is a perfect bilinear pairing of finite projective $R$-modules
$$(P/J_O(R)P) \times (P^\dagger/J_O(R)P^\dagger) \lr L/J_O(R)L.$$
Define $Q^\dagger \subseteq P^\dagger$ by the two conditions
$$J_O(R)P^\dagger \subseteq Q^\dagger\subseteq P^\dagger,\quad Q^\dagger/J_O(R)P^\dagger = (Q/J_O(R)P)^\perp.$$
Then the natural pairing $P\times P^\dagger \to L$ tautologically satisfies $(Q, Q^\dagger)\subseteq J_O(R)L$. We need to define a $σ$-linear surjection $\dF^\dagger:Q^\dagger \to P^\dagger$ such that for every $\ell\in Q^\dagger$ and $x\in Q$,
\begin{equation}\label{eq:def_dagger}
(\dF^\dagger(\ell))(\dF(x)) = \dot{α}(\ell(x)).
\end{equation}
It is enough to characterize $\dF^\dagger(\ell)$ on a $W_O(R)$-generating set of $P$ and the image $\dF(Q)\subseteq P$ is such a set. So \eqref{eq:def_dagger} already pins down a unique map of sets $\dF^\dagger:Q^\dagger \to P^\dagger$. It is immediately checked that this map is $σ$-linear. Moreover, it is checked by reduction to geometric points of $\Spf(R)$ that $\dF^\dagger(Q^\dagger)$ generates $P^\dagger$ as $W_O(R)$-module; this is a simple calculation with Dieudonné modules which we omit, see Example \ref{ex:displays} (2).

Recall that if there exists a Frobenius map $\bfF^\dagger: P^\dagger \to P^\dagger$ such that $(P^\dagger, Q^\dagger, \bfF^\dagger, \dF^\dagger)$ is a display, then it is uniquely determined by $\bfF^\dagger(\ell) = \dF^\dagger(V(1)\cdot \ell)$, see \eqref{eq:Frobenius_determined}. Also, clearly, this formula defines a $σ$-linear map $\bfF^\dagger:P^\dagger \to P^\dagger$. It is left to check that with this definition the display axiom \eqref{eq:display_axiom} is satisfied, that is, $\dF^\dagger(ξ\cdot \ell) = \dot{σ}(ξ) \bfF(\ell)$ for all $ξ\in I_O(R)$ and $\ell\in P^\dagger$. We again check this by pairing against the generating set $\dF(Q)\subseteq P$:
\begin{equation}\label{eq:verify_dual}
\begin{array}{rcl}
(\dF^\dagger(ξ\cdot \ell))(\dF(x)) & = & \dot{α}(ξ\ell(x))\\[1mm]
& \overset{\eqref{eq:display_axiom} \text{ for }\mcL}{=} &\dot{σ}(ξ) α(\ell(x))\\[1mm]
& \overset{\eqref{eq:Frobenius_determined} \text{ for }\mcL}{=} & \dot{σ}(ξ) \dot{α}(V(1)\cdot \ell(x))\\[2mm]
& = & (\dot{σ}(ξ)\cdot \dF^\dagger(V(1)\cdot \ell))(\dF(x))\\[1mm]
& \overset{\text{Def. of }\bfF^\dagger}{=} & (\dot{σ}(ξ)\cdot \bfF^\dagger(\ell))(\dF(x)).
\end{array}
\end{equation}
This proves the existence of a strict pair structure $(\mcP^\dagger, ι^\dagger_0)$ on $P^\dagger$ such that the map $P\times P^\dagger \to L$ induces an $O_{F_0}$-bilinear pairing $\mcP\times \mcP^\dagger \to \mcL$. The universality (and hence uniqueness) of this construction is seen directly from definitions: first, note that for every projective $O_{F_0}\tensor_O W_O(R)$-module $P'$ there is an isomorphism
$$\Hom_{O_{F_0}\tensor_O W_O(R)}(P', P^\dagger) \simlr \mr{BiHom}_{O_{F_0}\tensor_O W_O(R)}(P \times P', L).$$
This already implies that for every strict pair $(\mcP' = (P', Q', \bfF', \dF'), ι_0')$ the map $\Hom_{O_{F_0}}(\mcP', \mcP^\dagger) \to \mr{BiHom}_{O_{F_0}}(\mcP\times \mcP', \mcL)$ is injective. For surjectivity, assume an $O_{F_0}$-bilinear pairing $(\ ,\ ):\mcP\times \mcP'\to \mcL$ is given. By definition, see Definition \ref{def:bihom_display}, this in particular means $(Q, Q')\subseteq J_O(R)L$, so the induced map $f:P'\to P^\dagger$ satisfies $f(Q')\subseteq Q^\dagger$. The compatibility $f\circ \dF' = \dF^\dagger \circ f$ follows from the definition of $\dF^\dagger$ in \eqref{eq:def_dagger}. 
\end{proof}

In order to define PEL type moduli spaces of $O$-displays or $p$-divisible $O$-modules, one usually needs to fix a Lubin--Tate object to pin down a notion of duality. Our last statement in this section implies that this choice often does not matter:

\begin{lem}\label{lem:LT_object_essentially_unique}
Any two Lubin--Tate displays are pro-(finite étale) locally isomorphic.
\end{lem}
\begin{proof}
Let $\phi_0:F_0\to \ov K$ be an embedding and identify $F_0$ with $\phi_0(F_0)$. As remarked after Definition \ref{def:LT_display_general}, any Lubin--Tate display $(L, Q, \bfF, \dF, \iota_0)$ satisfies $Q = J_O(R)L$. In the terminology of \S\ref{ss:modification_functor} (with $F_0$ in place of $F$), this means that Lubin--Tate displays are the same as $(\{\phi_0\}, \{\phi_0\})$-strict pairs $(\mcP, \iota_0)$ of height $[F_0:K]$. Via Theorem \ref{thm:AB_equiv_displays}, these are equivalent to $(\emptyset,\emptyset)$-strict pairs of height $[F_0:K]$.

A pair $(\mcP = (P, Q, \bfF, \dF), \iota_0)$ is $(\emptyset,\emptyset)$-strict if and only if $Q = P$. That is, those pairs are nothing but \'etale $O$-displays with $O_{F_0}$-action as in Example \ref{ex:etale_displays}. By the equivalence from \eqref{eq:equiv_etale}, all such pairs of height $[F_0:K]$ are pro-(finite \'etale) locally isomorphic.
\end{proof}

\subsection{Duality I: The strict case}
\label{ss:duality_and_modification}
We begin with a simple observation about strictness and duality. Let $F_0/K$ be a finite field extension and let $F/F_0$ be a quadratic field extension. We denote the Galois conjugation by $τ$ or by $a\mapsto \ov a$. Let $S \subseteq A\subseteq B \subseteq \Hom_K(F, \ov K)$ be any subsets and let $E \subseteq \ov K$ contain the reflex fields of $S$, $A$, $B$ as well as their $F/F_0$-conjugates.

\newcommand{\hv}{\text{herm-$\vee$}}
\newcommand{\hdagger}{\text{herm-$\dagger$}}

\begin{defn}[Hermitian dual]\label{def:herm_dual}
Let $(\mcP, ι)$ be a pair, that is, an $O$-display with $O_F$-action, over a $π$-adic $O_E$-algebra $R$. We define the \emph{hermitian $\vee$-dual} of $(\mcP, ι)$ as
$$(\mcP^\vee, \ob{ι}^\vee(a) := ι^\vee(\ob a)),\quad a\in O_F.$$
\end{defn}
It is clear from Lemma \ref{lem:dual_display_strict} that if $(\mcP, ι)$ is $(A, B)$-strict, then $(\mcP^\vee, \ob{ι}^\vee)$ is $(\ob B^c, \ob A^c)$-strict. By Proposition \ref{prop:commutativ_square} and Theorem \ref{thm:AB_equiv_displays}, we now have a commutative square (up to the natural isomorphism $\mcP\simto (\mcP^\vee)^\vee$) of equivalences 
\begin{equation*}
\xymatrixcolsep{2cm}
\xymatrix{
(A,B)\mathrm{-Disp} \ar[r]^-{Φ_S} \ar[d]_{\cong}^{\text{herm.~dual}} & (A\setminus S, B\setminus S)\mathrm{-Disp} \ar[d]^{\cong}_{\text{herm.~dual}}\\
(\ov B^c,\ov A^c)\mathrm{-Disp} & (\ov B^c\cup \ov S,\ov A^c\cup \ov S)\mathrm{-Disp}. \ar[l]_-{Φ_{\ov S}}
}
\end{equation*}

Our main case of interest is when $S = A$ and when there is a pair $\{ϕ_0, \ov{ϕ_0}\} \subseteq \Hom_K(F, \ov K)$ such that
\begin{equation}\label{eq:CM_type_unitary}
\Hom_K(F, \ov K) = A \sqcup \ov A \sqcup \{ϕ_0, \ov{ϕ_0}\},\quad B = A \sqcup \{ϕ_0, \ov{ϕ_0}\}.
\end{equation}
Then $(\ob B^c, \ob A^c) = (A, B)$ and hence $(\mcP, ι)\mapsto (\mcP^\vee, \ob{ι}^\vee)$ is a self anti-equivalence of the category $(A,B)$-Disp. Our aim in this section is to show that the modification functor $Φ_A$ is compatible with this notion of duality in the sense that $Φ_A(\mcP^\hv) \simto Φ_A(\mcP)^\hdagger$, where the hermitian $\dagger$-dual is with respect to a suitable Lubin--Tate display.

We begin by defining this Lubin--Tate display. As during the construction of $Φ_S$, let $ζ\in O_F$ denote an $O$-algebra generator of the form $ϕ = μ + π_F$, see Lemma \ref{lem:valuation_lift_W}. Let again $E = E_A E_B \subseteq K$. We will always view $F_0$ as a subfield of $E$ via $\{ϕ_0, \ov{ϕ_0}\}$. Let $\wt{e_A} \in O_F\tensor_OW_O(O_E)$ be as in \eqref{eq:def_e_S_tilde}. Let $\wt{f_A}$ be its $O_F/O_{F_0}$-conjugate. Equivalently, in the notation of \eqref{eq:def_e_S_tilde},
$$\wt{f_A} = \wt{e_{\ob{ζ}, \ob A}}$$
where $\ob{ζ}$ and $\ob A$ denote the Galois conjugates of $ζ$ and $A$. Let
\begin{equation}\label{def:theta_tilde}
\wt{θ} := \wt{e_A} \wt{f_A}
\end{equation}
denote the product and let $θ = e_{ζ, A} e_{\ov{ζ}, \ov A}$ be the image in $O_F\tensor_O O_E$. By Lemma \ref{lem:indep_of_generator}, the two elements $e_{\ov{ζ}, \ov A}$ and $e_{ζ, \ov A}$ differ by a unit. Hence $θ$ and $e_{A\cup \ov A}$ differ by a unit.\footnote{Beware that this does not imply that $\wt{θ}$ and $\wt{e_{A\cup \ov A}}$ differ by a unit.} The advantage of $θ$ and $\wt{θ}$ over $e_{A\cup \ov A}$ and $\wt{e_{A\cup \ov A}}$ is that they are already defined over $F_0$ in the following sense.
\begin{lem}\label{lem:theta_in_F_0}
The elements $θ$ and $\wt{θ}$ satisfy
\begin{equation}\label{eq:thetas}
θ \in O_{F_0}\tensor_O O_E\quad\text{and}\quad \wt{θ}\in O_{F_0}\tensor_O W_O(O_E).
\end{equation}
\end{lem}
\begin{proof}
The subring $O_{F_0}\subset O_F$ is the set of conjugation invariant elements. Both $O_E$ and $W_O(O_E)$ are $π$-torsion free, so by scalar extension $O_{F_0}\tensor_O O_E$ and $O_{F_0}\tensor_O W_O(O_E)$ are the rings of conjugation invariants in $O_F\tensor_O O_E$ and $O_F\tensor_OW_O(O_E)$. By construction, $θ$ and $\wt{θ}$ are conjugation invariant and we obtain \eqref{eq:thetas}.
\end{proof}

\begin{defn}\label{def:LT_display}
Let $R$ be a $π$-adic $O_E$-algebra. We define from $\wt{θ}$ the Lubin--Tate display
\begin{equation}\label{eq:LT_display}
\mcL(R) := \mcL_{O_{F_0}/O, \wt{θ}}(R) := (O_{F_0} \tensor_O W_O(R), J_O(R), α, \dot{α})
\end{equation}
as follows. The second Frobenius $\dot{α}$ is defined as
\begin{equation}\label{eq:def_sigma_dot}
\begin{aligned}
\dot{α}:J_O(R) & \lr O_{F_0}\tensor_O W(R)\\
x & \longmapsto (\mr{id}\tensor V^{-1})(\wt{θ} x).
\end{aligned}
\end{equation}
The Frobenius $α$ is determined by $α(x) = \dot{α}(V(1)\cdot x)$, see \eqref{eq:Frobenius_determined}.
\end{defn}
We explain why \eqref{eq:def_sigma_dot} is indeed defined. Namely, for every choice of $O$-algebra generator $ζ_0\in O_{F_0}$, we can described $J_O(R)$ by
\begin{equation}\label{eq:J_O_R}
J_O(R) = (ζ_0\tensor 1 - 1 \tensor [ζ_0]) + O_{F_0}\tensor_O I_O(R).
\end{equation}
By Lemma \ref{lem:Eisenstein_ideal_intermediate_field}, the ideal $(ζ_0\tensor 1 - 1\tensor ζ_0)\cdot (O_F\tensor_O O_E)$ agrees with the Eisenstein ideal $J_{\{ϕ_0, \ov{ϕ_0}\}}$. Since $A\cup \ov A \cup \{ϕ_0, \ov{ϕ_0}\} = \Hom_K(F, \ov K)$, we obtain $θ\cdot (ζ_0 \tensor 1 - 1\tensor ζ_0) = 0$ in $O_{F_0}\tensor_O O_E$. In terms of \eqref{eq:J_O_R}, this means that $\wt{θ}J_O(R) \subseteq O_{F_0}\tensor_O I_O(R)$ so \eqref{eq:def_sigma_dot} makes sense. It is shown by specialization to geometric points that $\dot{α}(J_O(R))$ generates $O_{F_0}\tensor_O W_O(R)$. Finally, we see for $ξ\in I_O(R)$ and $x\in O_{F_0}\tensor_O W_O(R)$ that
\begin{equation}\label{eq:def_LT_display}
\begin{aligned}
\dot{α}(ξx) & = (\mr{id}\tensor V^{-1})(\wt{θ}\cdot ξ x)\\[1mm]
& = V^{-1}(ξ)\cdot (\mr{id}\tensor σ)(\wt{θ}\cdot x)\\[1mm]
& = V^{-1}(ξ)\cdot (\mr{id}\tensor V^{-1})(V(1)\cdot \wt{θ} \cdot x)\\[1mm]
& = V^{-1}(ξ)\cdot \dot{α}(V(1)\cdot x)\\[1mm]
& = V^{-1}(ξ)\cdot α(x)
\end{aligned}
\end{equation}
as required by the display axioms. This shows that Definition \ref{def:LT_display} indeed constructs a display.

\begin{defn}\label{def:herm_dagger_dual}
Let $(\mcP, ι)$ be a pair over a $π$-adic $O_{F_0}$-algebra such that the restriction $ι\vert_{O_{F_0}}$ is strict. (In other words, $(\mcP, ι\vert_{O_{F_0}})$ is a strict pair in the sense of \S\ref{ss:LT}.) We define the \emph{hermitian $\dagger$-dual} of $(\mcP, ι)$ as the $\dagger$-dual from Proposition \ref{prop:dual_LT} with respect to the Lubin--Tate display from Definition \ref{def:LT_display}, together with conjugated $O_F$-action:
$$(\mcP^\dagger, \ob{ι}^\dagger(a) := ι^\dagger(\ov a)),\quad a\in O_F.$$
\end{defn}

Note that an $O$-display with $O_F$-action $(\mcP, ι)$ is $(\emptyset, \{ϕ_0, \ov {ϕ_0}\})$-strict if and only if the restriction $ι\vert_{O_{F_0}}$ is a strict $O_{F_0}$-action. This follows from Lemma \ref{lem:Eisenstein_ideal_intermediate_field}. Thus the modification functor $Φ_A$ from Theorem \ref{thm:AB_equiv_displays} is an equivalence
\begin{equation}\label{eq:Phi_A_display}
Φ_A:(A,B)\text{-Disp} \simlr \left\{\text{\begin{varwidth}{\textwidth} \centering $O$-displays with $O_F$-action $(\mcP, ι)$\\
s.t. $ι\vert_{O_{F_0}}$ is strict\end{varwidth}}\right\}.
\end{equation}
Combining Theorem \ref{thm:AB_equiv_displays} and Proposition \ref{prop:dual_LT} we obtain a diagram of functors, all of which are equivalences:
\begin{equation}\label{eq:compatibility_duality}
\xymatrix{
(A,B)\text{-Disp} \ar[rr]^-{\text{herm.~$\vee$-dual}} \ar[d]_{Φ_A} & &(A,B)\text{-Disp} \ar[d]^{Φ_A}\\
\big\{\text{$(\mcP, ι)$ s.t. $ι\vert_{O_{F_0}}$ strict}\big\} \ar[rr]^{\text{herm.~$\dagger$-dual}} & & \big\{\text{$(\mcP, ι)$ s.t. $ι\vert_{O_{F_0}}$ strict}\big\}.
}
\end{equation}

Fix a generator $\vartheta$ of the inverse different of $F_0/K$ and define the trace map by
\begin{equation}\label{eq:def_trace}
\btr : O_{F_0} \lr O,\quad a  \longmapsto \mr{tr}_{F_0/K}(\vartheta a).
\end{equation}
For every $O$-algebra $W$, we use $\btr$ to denote the $W$-linear extension
$$\btr \tensor \mr{id}_W:O_{F_0}\tensor_OW\lr W.$$
Given a finite projective $O_{F_0}\tensor_OW$-module $P$, define $P^\dagger := \Hom_{O_{F_0}\tensor_OW}(P, O_{F_0}\tensor_OW)$. By Lemma \ref{lem:keep-isotropic}, for every such $P$, composition with $\btr$ defines an isomorphism of $O_{F_0}\tensor_O W$-modules
\begin{equation}\label{eq:trace_iso}
\btr: P^\dagger\simlr P^\vee,\quad \ell \longmapsto \btr\circ \ell.
\end{equation}

\begin{thm}\label{thm:comp-display}
For every $(A,B)$-strict pair $(\mcP, ι)$, the trace isomorphism $\btr:P^\dagger \simlr P^\vee$ is an $O_F$-linear isomorphism of displays,
\begin{equation}\label{eq:nat}
Φ_A(\mcP, ι)^\hdagger \simlr Φ_A(\mcP^\hv).
\end{equation}
In particular, the diagram \eqref{eq:compatibility_duality} commutes up to the natural isomorphism defined by \eqref{eq:nat}.
\end{thm}
\begin{proof}
Fix a $π$-adic $O_E$-algebra $R$ and set $W = W_O(R)$, $I = I_O(R)$ as well as $J = J_O(R)$. All $O$-displays in the following will be over $R$. It will be clear that all constructions and identifications below are natural in $R$ and functorial.

Let $(P, Q, \bfF, \dF, ι)$ be an $O$-display with $(A,B)$-strict $O_F$-action over $R$. Tracing its images along the equivalences from \eqref{eq:compatibility_duality}, we obtain the objects of the following diagram.
\begin{equation}\label{eq:display_chase}
\xymatrix{
(P, Q, \bfF, \dF, ι) \ar@{|->}[rr]^-{\text{herm.~$\vee$-dual}} \ar@{|->}[d]_{Φ_A} & & (P^\vee, Q^\vee, \bfF^\vee, \dF^\vee, \ob{ι}^\vee) \ar@{|->}[d]^{Φ_A}\\
(P, Q_0, \bfF_0, \dF_0, ι) \ar@{|->}[rd]^{\ \ \text{herm.~$\dagger$-dual}} & & (P^\vee, (Q^\vee)_0, (\bfF^\vee)_0, (\dF^\vee)_0, \ob{ι}^\vee)\\
& (P^\dagger, (Q_0)^\dagger, (\bfF_0)^\dagger, (\dF_0)^\dagger, \ob{ι}^\dagger) \ar@{-->}[ru]^{\btr}. &
}
\end{equation}
The claim is that the dotted arrow $\btr:P^\dagger \to P^\vee$ in the lower right is an isomorphism of $O$-displays. By definition, this means that it has the two properties
\begin{enumerate}
    \item[(a)] $\btr((Q_0)^\dagger) = (Q^\vee)_0$;
    \item[(b)] $\btr \circ (\dF_0)^\dagger = (\dF^\vee)_0 \circ \btr$ as maps $(Q_0)^\dagger \to P^\vee$.
\end{enumerate}
We first check (a). Recall the elements $\wt{e_A}$, $\wt{f_A}$ and $\wt{θ}$ from \eqref{def:theta_tilde}. By definition of $(Q^\vee)_0$, see \eqref{equ:const-Q_0}, and because we have taken the hermitian dual instead of the plain dual, having $\btr((Q_0)^\dagger)\subseteq (Q^\vee)_0$ as in (a) means that $\wt{f_A}\btr((Q_0)^\dagger) \subseteq Q^\vee$. By definition of $Q^\vee$, this holds if and only if for all $x\in Q$ and all $\ell\in (Q_0)^\dagger$,
\begin{equation}\label{eq:to_show_Qs}
(\wt{f_A}(\btr\circ \ell))(x) = \btr(\ell(\wt{f_A}x)) \in I.
\end{equation}
Since $Q = \wt{e_A}Q_0 + IP$, and since \eqref{eq:to_show_Qs} is obviously satisfied for $x\in IP$, we may assume $x = \wt{e_A}y$ for some $y\in Q_0$. Recall that $\ell\in P^\dagger$ lies in $(Q_0)^\dagger$ if and only if $\ell(Q_0)\subseteq J$. So we obtain
$$\begin{aligned}
\btr(\ell(\wt{f_A}\wt{e_A}y)) & = \btr(\wt{θ}\cdot \ell(y))\\
& \in \btr(\wt{θ}\cdot J) \subseteq \btr(O_{F_0}\tensor_O I) \subseteq I.
\end{aligned}$$
This proves the inclusion relation $\btr((Q_0)^\dagger)\subseteq (Q^\vee)_0$. The quotients $P^\dagger/(Q_0)^\dagger$ and $P^\vee / (Q^\vee)_0$ are both projective $R$-modules of the same rank. So we obtain $\btr((Q_0)^\dagger) = (Q^\vee)_0$ and completed the proof of (a).

We now show (b) which states that for every $x\in P$ and every $\ell\in (Q_0)^\dagger$,
\begin{equation}\label{eq:to_show_dot}
\btr(((\dF_0)^\dagger(\ell))(x)) = ((\dF^\vee)_0(\btr\circ \ell))(x).
\end{equation}
Just like when we checked \eqref{eq:dual_to_check} before, it suffices to verify \eqref{eq:to_show_dot} for all $x$ in a $W$-module generating set of $P$. Such a set is given by $\{\dF(\wt{e_A}y) \mid y \in Q_0\}$. The left hand side in \eqref{eq:to_show_dot} for a generator $x = \dF(\wt{e_A}y)$ equals
\begin{equation}\label{eq:ts_lhs}
\begin{array}{rcl}
\btr(((\dF_0)^\dagger(\ell))(\dF(\wt{e_A}y))) & = & \btr(((\dF_0)^\dagger(\ell))(\dF_0(y))\\[1mm]
& \overset{\eqref{eq:bilin_pairing_display}}{=} & \btr(\dot{α}(\ell(y)))\\[1mm]
& \overset{\eqref{eq:def_sigma_dot}}{=} & \btr((\mr{id}\tensor V^{-1})(\wt{θ}\cdot \ell(y)))\\[1mm]
& = & \btr((\mr{id}\tensor V^{-1})(\ell(\wt{θ}y))).
\end{array}
\end{equation}
We now unpack the right hand side of \eqref{eq:to_show_dot}. Recall for this that by definitions
$$(\dF^\vee)_0(α) = \dF^\vee(\wt{f_A}α).$$
So for a generator $x = \dF(\wt{e_A}y)$, we obtain
\begin{equation}\label{eq:ts_rhs}
\begin{array}{rcl}
((\dF^\vee)_0(\btr\circ\ell))(x) & = & (\dF^\vee(\wt{f_A}\cdot (\btr\circ \ell)))(\dF(\wt{e_A} y))\\[1mm]
& \overset{\eqref{eq:dual_display_key}}{=} & (\mr{id}\tensor V^{-1})((\wt{f_A} \cdot (\btr \circ \ell))(\wt{e_A}y))\\[1mm]
& = & (\mr{id}\tensor V^{-1})(\btr(\ell(\wt{θ}y))).
\end{array}
\end{equation}
Since $\btr$ is defined as the $W$-linear extension of a map $\btr:O_{F_0}\to O$, it commutes with $\mr{id}\tensor V^{-1}$. Hence \eqref{eq:ts_lhs} equals \eqref{eq:ts_rhs} and the proof of Theorem \ref{thm:comp-display} is complete.
\end{proof}

\begin{rmk}
Let $W$ be an $O_{F_0}$-algebra and let $P$ be a finite projective $O_{F_0}\tensor_O W$-module. There are natural isomorphisms
\begin{equation}\label{eq:dual_dual}
P\simlr (P^\vee)^\vee\quad\text{and}\quad P\simlr (P^\dagger)^\dagger
\end{equation}
given by evaluation. Recall that an $O_{F_0}$-linear homomorphism
$$λ:P\lr P^\vee\quad\text{or}\quad λ:P\lr P^\dagger$$
is called symmetric resp. skew-symmetric if it satisfies $λ^\vee = λ$ or $λ^\dagger = λ$ (resp. $λ^\vee = -λ$ or $λ^\dagger = -λ$) with respect to \eqref{eq:dual_dual}. The trace isomorphism $\btr$ is compatible with this notion in the sense that $λ:P\to P^\dagger$ is (skew-)symmetric if and only if $\btr\circ λ:P\to P^\vee$ is (skew-)symmetric.
\end{rmk}

\subsection{Duality II: The étale case}
\newcommand{\Hdagger}{\text{$\mcH$-$\dagger$}}

For applications to unitary Shimura varieties, we also need to consider Theorem \ref{thm:AB_equiv_displays} for $(A, A)$-strict pairs. In this case, too, we need a compatibility of $Φ_A$ with duality in the sense of Theorem \ref{thm:comp-display}. The basic observation is that this can be achieved with the same arguments as before. In fact, it would have been possible to give a unified treatment of the two cases in \S\ref{ss:LT} and \S\ref{ss:duality_and_modification}, but for better readability we decided to present the main case of interest first.

We begin by constructing a specific étale $O$-display with $O_{F_0}$-action that will serve as our dualizing object. Its definition will be analogous to that in \eqref{eq:LT_display}. Let $A\subseteq \Hom_K(F, \ov K)$ be such that
$$\Hom_K(F, \ov K) = A \sqcup \ov A$$
and denote by $E = E_A$ the reflex field. Fix an $O$-algebra generator $ζ\in O_F$ of the form $ζ = μ + π_F$ as before to define $\wt{e_A} \in O_F\tensor_O W_O(O_E)$ as in \eqref{eq:def_e_S_tilde}. Let $\wt{f_A}$ be its $F/F_0$-conjugate and define $\wt {θ} = \wt{e_A}\wt{f_A}$. Then $\wt{θ}$ lies in $O_{F_0}\tensor_O I_O(O_E)$ by the reasoning of Lemma \ref{lem:theta_in_F_0} and because $e_Af_A = 0$ in $O_{F_0}\tensor_O O_E$. For a $π$-adic $O_E$-algebra $R$, define the étale $O$-display
\begin{equation}\label{eq:def_H_et_LT_disp}
\mcH(R) := \mcH_{O_{F_0}/O, \wt{θ}}(R) := \big(O_{F_0}\tensor_O W_O(R),\ \dot{β}(x) = (1\tensor V^{-1})(\wt{θ}\cdot x)\big).
\end{equation}

\begin{defn}[Hermitian $\mcH$-$\dagger$-dual]\label{def:duality_et}
Let $R$ be a $π$-adic $O_E$-algebra. An \emph{étale pair} over $R$ is a pair $(\mcP = (P, \dF), ι)$ consisting of an étale $O$-display $\mcP$ over $R$ with an $O_F$-action $ι:O_F\to \End(\mcP)$. Here, we have followed the notation convention from Example \ref{ex:etale_displays}.

Given an étale pair $(\mcP = (P,\dF), ι_0)$ over $R$, consider again the module
$$P^\dagger := \Hom_{O_{F_0}\tensor_O W_O(R)}(P, O_{F_0}\tensor_O W_O(R)).$$
Then one shows as indicated in Example \ref{ex:etale_displays} or as in Proposition \ref{prop:dual_LT} that there is a (necessarily unique) étale $O$-display structure $\mcP^\Hdagger$ on $P^\dagger$ such that the natural pairing $P\times P^\dagger \to O_{F_0}\tensor_OW_O(R)$ defines a pairing of $O$-displays
\begin{equation}\label{eq:H_dagger_dual}
    \mcP\times \mcP^\Hdagger \lr \mcH(R).
\end{equation}
The \emph{hermitian $\mcH$-$\dagger$-dual} of $(\mcP, ι)$ is defined as $(\mcP^\Hdagger, \ob{ι}^\Hdagger)$.
\end{defn}

Recall that we defined a hermitian $\vee$-dual in Definition \ref{def:herm_dual}. Also recall that we defined the trace isomorphism $\btr:P^\dagger \simto P^\vee$ in \eqref{eq:trace_iso}.
\begin{thm}\label{thm:comp-display-banal}
For every $(A, A)$-strict pair $(\mcP, ι)$, the trace isomorphism $\btr:P^\dagger \simto P^\vee$ defines an $O_F$-linear isomorphism of displays,
$$Φ_A(\mcP, ι)^\text{herm-$\mcH$-$\dagger$} \simlr Φ_A(\mcP^\text{herm-$\vee$}).$$
\end{thm}
\begin{proof}
Let $(\mcP = (P, Q, \bfF, \dF), ι)$ be an $(A,A)$-strict pair. Apply the functors $Φ_A$, hermitian $\vee$-dual, and hermitian $\mcH$-$\dagger$-dual to obtain objects as in \eqref{eq:display_chase}. (Replace the hermition $\dagger$-dual with the hermitian $\mcH$-$\dagger$-dual for our current situation.) Note that after applying $Φ_A$, we are working with étale objects. That is, in the notation of \eqref{eq:display_chase}, we have $Q_0 = P$, $(Q^\vee)_0 = P^\vee$, and $(Q_0)^\Hdagger = P^\dagger$.

From here on, the remainder of the proof is literally the same as that for Theorem \ref{thm:comp-display}. Even simpler, (a) is immediate because the displays are étale. Concerning (b), the manipulations in \eqref{eq:ts_lhs} and \eqref{eq:ts_rhs} that prove \eqref{eq:to_show_dot} are purely formal and go through without change.
\end{proof}

\section{Rapoport--Zink spaces}
\label{s:RZ}

\subsection{$O$-displays and formal $p$-divisible $O$-modules}

Let $K$ be a $p$-adic local field with ring of integers $O = O_K$, uniformizer $\pi\in O$, and residue field $\mbF_q$. Let $R$ be a $π$-adic $O$-algebra. By \emph{$p$-divisible $O$-module over $R$}, we mean a $p$-divisible group $X/R$ together with an action $O\to \End(X)$ that is strict on $\Lie(X)$. A $p$-divisible $O$-module is called \emph{formal} if its underlying $p$-divisible group is formal.

Let $\mcP = (P, Q, \bfF, \dot \bfF)$ be an $O$-display over $R$. There is also a definition of \emph{nilpotent} $O$-displays. A simple characterization of nilpotency is that for every homomorphism to a perfect field $R\to k$, the Dieudonné module obtained by specialization is $V$-nilpotent, see \cite[before Definition 3.1.10]{KRZ}.

The main theorem of Ahsendorf--Cheng--Zink \cite{ACZ} establishes an equivalence from nilpotent $O$-displays to formal $p$-divisible $O$-modules. For $O = \mbZ_p$, this equivalence was previously proved by Zink \cite{Zink} and Lau \cite{Lau}. We now recall their result. Afterwards, we will combine it with our results in \S\ref{s:displays} to construct equivalences between categories of $p$-divisible groups with $O_F$-action.

\begin{thm}[\protect{\cite[Theorem 1.1]{ACZ}}]\label{thm:equiv_displays}
There is an equivalence of fibered categories over $π$-adic $O$-algebras
\begin{equation}\label{eq:equiv_display}
\BT:\left\{\text{nilpotent $O$-displays}\right\} \simlr \left\{\text{\begin{varwidth}{\textwidth} \centering formal $p$-divisible\\$O$-modules \end{varwidth}}\right\}.
\end{equation}
It is compatible with Hodge filtrations in the sense that for any $O$-display $\mcP = (P, Q, \bfF, \dF)$, there is a natural isomorphism
\begin{equation}\label{eq:compatib_Hodge_filt}
\big[P/I_O\cdot P\, \twoheadrightarrow\, P/Q\big] \simlr \big[\mcD_{\BT(\mcP)}(R) \,\twoheadrightarrow\, \Lie(\BT(\mcP))\big].
\end{equation}
Here, $\mcD_X$ denotes the relative covariant Grothendieck--Messing crystal of a $p$-divisible $O$-module $X$.
\end{thm}

Theorem \ref{thm:equiv_displays} will be all that is needed in the EL context (\S\ref{ss:EL_RZ}). For applications to PEL moduli problems (\S\ref{sec:uni-RZ}), we show in the appendix that $\BT$ is compatible with duality (Proposition \ref{prop:Faltings_display_compatible}).

\subsection{EL type moduli spaces}
\label{ss:EL_RZ}

Our notation is the same as in \S\ref{s:displays}: $F/K$ is a finite extension, $A \subseteq B \subseteq \Hom_K(F, \ov K)$ are two subsets, and $E = E_AE_B$ is the join of their reflex fields.

\begin{defn}
Let $S$ be a scheme over $\Spf O_E$. An $O_F$-action $ι:O_F\to \End(X)$ on a $p$-divisible $O$-module $X/S$ is called \emph{$(A, B)$-strict} if the induced $O_F$-action on $\Lie(X)$ is $(A,B)$-strict in the sense of the Eisenstein condition from Proposition \ref{prop:RZ-eisenstein} (2). Here, the integer $d$ with $\mr{rk}_{\mcO_S}(\Lie(X)) = n|A| + d$ is allowed to vary on $S$.
\end{defn}

\begin{lem}\label{lem:AB_strict_implies_formal}
Assume that $A \neq \emptyset$ and that $(X, ι)/S$ is a $p$-divisible $O$-module with $(A,B)$-strict $O_F$-action. Then $X$ is formal.

Similarly, if $A\neq \emptyset$ and if $(\mcP, ι)$ is an $O$-display with $(A, B)$-strict $O_F$-action, then $\mcP$ is nilpotent.
\end{lem}
\begin{proof}
Being formal means that $X$ has no étale part which can be checked point by point. This reduces to the case $S = \Spec k$ for an algebraically closed $O_E$-field $k$. We fix an embedding of the residue field of $\ov K$ into $k$. If $A\neq \emptyset$, then clearly $e_A\in O_F\tensor_{\mbF_q} k$ is not a unit: indeed, for $φ\in A$, the image of $ζ \tensor 1 - 1\tensor φ(ζ)$ along $φ\tensor 1:O_F\tensor k \to k$ is zero.

Consider now the Dieudonné module $(M, V, σ)$ of $X$. Note that $M$ is free over $O_F\tensor_{O} W_O(k)$. In particular, it decomposes into a direct sum $M = \bigoplus_{ψ}\,M_ψ$ indexed by the embeddings $ψ:k_F\to k$ as in \eqref{eq:decomp_psi}. Frobenius and Verschiebung are homogeneous of degree $\pm 1$ with respect to this decomposition. Thus $e_A$ not being a unit implies that $V$ is topologically nilpotent.

In the exact same way, an $O$-display is nilpotent if and only if each of its fibers in geometric points is, see \cite[before Definition 3.1.10]{KRZ}. By \cite[Proposition 3.1.9]{KRZ}, the claim reduces to Dieudonné modules as before.
\end{proof}

Recall that $(A,B)$-Disp is our notation for the fibered category of $(A,B)$-strict pairs $(\mcP, ι)$ over $\Spf(O_E)$-schemes, where $E = E_AE_B$. Similarly let $(A, B)$-Mod denote the fibered category of $p$-divisible $O$-modules with $(A, B)$-strict $O_F$-action $(X, ι)$.

\begin{thm}\label{thm:AB_strict_p_div}
(1) Assume $B = A$, in particular $E = E_A$. Then there is an equivalence of fibered categories over $\Spf(O_E)$-schemes
\begin{equation}\label{eq:equiv_AB_strict_banal}
\Phi_A:(A,A)\text{-}\mathrm{Mod} \simlr \left\{\text{\begin{varwidth}{\textwidth} \centering étale $p$-divisible\\$O_F$-modules \end{varwidth}}\right\}.
\end{equation}
(2) Assume that $B = A \sqcup \{ϕ_0\}$. In this case, $ϕ_0$ embeds $F$ into $E$. There is an equivalence of fibered categories over $\Spf(O_E)$-schemes
\begin{equation}\label{eq:equiv_AB_strict}
\Phi_A:(A,B)\text{-}\mathrm{Mod} \simlr \left\{\text{\begin{varwidth}{\textwidth} \centering strict $p$-divisible\\$O_F$-modules \end{varwidth}}\right\}
\end{equation}
such that the effect of $Φ_A$ on the Hodge filtration is given by \eqref{eq:bij-e_S}. In particular, for every $(A,B)$-strict pair $(\mcP, ι)$, there is a natural isomorphism
$$J_A\cdot \Lie(\mcP) \simlr \Lie(Φ_A(\mcP)).$$
\end{thm}

\begin{proof}
For the proof, we choose any quasi-inverse $Ψ$ to the equivalence $\BT$ from Theorem \ref{thm:equiv_displays} together with an isomorphism $γ:\mr{id} \simto \BT\circ Ψ$. In particular, via $γ$, $Ψ$ is again compatible with Hodge filtrations in a sense analogous to \eqref{eq:compatib_Hodge_filt}.

We first prove statement (1). If $A = \emptyset$, then there is nothing to show because $(\emptyset, \emptyset)$-Mod is the same as the category of étale $p$-divisible $O$-modules with $O_F$-action. (These are the same as étale $p$-divisible $O_F$-modules.) If $A \neq \emptyset$, then by Lemma \ref{lem:AB_strict_implies_formal} we can use the equivalence $Ψ$ and the functor $Φ_A$ from Theorem \ref{thm:AB_equiv_displays} to reduce to that case:
\begin{equation}\label{eq:chain_1}
(A, A)\text{-Mod}\underset{Ψ}{\simlr} (A, A)\text{-Disp}\underset{Φ_A}{\simlr} (\emptyset, \emptyset)\text{-Disp}\underset{\text{Ex. \ref{ex:etale_displays}}}{\simlr} (\emptyset, \emptyset)\text{-Mod}.\end{equation}

Next, we prove statement (2). Again, if $A = \emptyset$, then there is nothing to show because $(\emptyset, \{ϕ_0\})$-strictness is the same as the usual notion of strictness. So the two sides in \eqref{eq:equiv_AB_strict} are equal.

Having settled this special case, we assume from now on that $A\neq \emptyset$. Since $A\subsetneq B$ in statement (2), this also implies $(B\setminus A)^c \neq \emptyset$. So we can consider the following equivalences of categories:

\begin{equation}\label{eq:diag_proof_EL}
\xymatrix{
(A,B)\text{-Mod} \ar[d]^-{Ψ}_-{\text{Thm. \ref{thm:equiv_displays}}} \ar@{-->}[rr]^-{\cong}&&
(\emptyset,B\setminus A)\text{-Mod} &&
((B\setminus A)^c, \Hom_K(F, \ov K))\text{-Mod} \ar[ll]_-{\text{$O$-dual}}^-{\text{Def. \ref{def:dual}}}\\
(A,B)\text{-Disp} \ar[rr]^-{Φ_A}_-{\text{Thm. \ref{thm:AB_equiv_displays}}} &&
(\emptyset,B\setminus A)\text{-Disp} \ar[rr]^-{\text{display dual}}_-{\text{Lem. \ref{lem:dual_display_strict}}} &&
((B\setminus A)^c, \Hom_K(F, \ov K))\text{-Disp} \ar[u]_-{\BT}^-{\text{Thm. \ref{thm:equiv_displays}}}
}
\end{equation}

Here, the two applications of Theorem \ref{thm:equiv_displays} are justified by Lemma \ref{lem:AB_strict_implies_formal}. The lower ``display dual'' arrow denotes the dualization of $O$-displays. The upper ``$O$-dual'' arrow denotes the Faltings dual of $p$-divisible $O$-modules from Definition \ref{def:dual} (fix any Lubin--Tate object over $O_K$ for the definition). All named arrows in this diagram are equivalences; the dotted arrow in the upper left then provides the proof of Theorem \ref{thm:AB_strict_p_div}. The compatibility with Hodge filtrations in case (2) follows from similar compatibilities for all functors in \eqref{eq:diag_proof_EL}.
\end{proof}

Due to the categorical nature of Theorem \ref{thm:AB_strict_p_div}, we obtain without difficulties comparison isomorphisms for a broad class of EL type RZ spaces. Let $D/F$ be a central division algebra of degree $d = (\dim_F(D))^{1/2}$. Let $O_D\subseteq D$ denote its ring of integers and let $\mr{charred}_{D/F}:O_D\to O_F[T]$ denote the reduced characteristic polynomial. Let $m\geq 1$ be an integer.

\begin{defn}[Special $O_D$-modules]
Let $X$ be a $p$-divisible $O_F$-module over a $\Spf(O_F)$-scheme $S$. Assume that $O_F\text{-ht}(X) = md^2$ and that $\dim(X) = d$. An $O_D$-action $ι:O_D\to \End(X)$ is called \emph{special} if for all $a\in O_D$,
\begin{equation}\label{eq:special}
\mr{char}(ι(a)\mid \Lie(X); T) = \mr{charred}_{D/F}(a; T)
\end{equation}
where the identity is meant as identity of elements of $\mcO_S[T]$. The pair $(X, ι)$ is then called a \emph{special $O_D$-module}.

In fact, let $L/F$ be an unramified field extension of degree $d$ and fix an embedding $O_L\to O_D$. Then it suffices to demand \eqref{eq:special} only for $a\in O_L$. This weaker condition is equivalent to the following rank condition: Over $S' = O_L\tensor_{O_F} S$, the Lie algebra of $X$ decomposes into eigenspaces,
\begin{equation}\label{eq:special_rank}
\begin{aligned}
\Lie(X) & = \bigoplus_{ψ \in \mr{Gal}(L/F)} \Lie(X)_ψ,\\[1mm]
\Lie(X)_ψ & = \{x\in \Lie(X) \mid ι(a)(x) = (ψ(a)\tensor 1)\cdot x\text{ for }a\in O_L\}.
\end{aligned}
\end{equation}
The rank condition then is $\mr{rk}_{\mcO_{S'}}(\Lie(X)_ψ) = 1$ for all $ψ\in \Gal(L/F)$.
\end{defn}

\begin{defn}[$(A, ϕ_0)$-special $O_D$-modules]
Let $A\subseteq \Hom_K(F, \ov K)$ be a subset with $ϕ_0\notin A$ and let $B = A\sqcup \{ϕ_0\}$. Set $E = E_AE_B$. Let $X$ be a $p$-divisible $O$-module over an $O_E$-scheme $S$. Assume that $X$ is of height $md^2[F:K]$ and that $\dim(X) = md^2|A| + d$. An $O_D$-action $ι:O_D\to \End(X)$ is called \emph{$(A,\{ϕ_0\})$-special} if
\begin{altenumerate}
    \item The action $ι\vert_{O_F}$ is $(A,B)$-strict. In particular, $J_A\cdot \Lie(X)$ is a vector bundle of rank $d$ on $S$.
    \item The action of $O_D$ on $J_A\cdot \Lie(X)$ is special in the sense of \eqref{eq:special}.
\end{altenumerate}
The pair $(X, ι)$ is then called an \emph{$(A,ϕ_0)$-special $O_D$-module}.
\end{defn}

\begin{cor}[to Theorem \ref{thm:AB_strict_p_div}]\label{cor:CDA}
The functor $Φ_A$ from \eqref{eq:equiv_AB_strict} defines an equivalence of fibered categories over $\Spf (O_E)$-schemes
\begin{equation}\label{eq:equiv_AB_special}
\Phi_A:\left\{\text{$(A,ϕ_0)$-special $O_D$-modules}\right\} \simlr \left\{\text{special $O_D$-modules}\right\}.
\end{equation}
\end{cor}
\begin{proof}
Given an $(A,ϕ_0)$-special $O_D$-module $(X, ι)$, application of the equivalence $Φ_A$ from Theorem \ref{thm:AB_strict_p_div} produces a strict $O_F$-module with $O_D$-action $Φ(X, ι)$. Its effect on Hodge filtrations is given by \eqref{eq:bij-e_S} which translates the ``$(A, ϕ_0)$-special'' property into the ``special'' property.
\end{proof}
Let $\breve E/E$ be the completion of a maximal unramified extension of $E$. Let $(\mbX_A, ι)$ be an $(A,ϕ_0)$-special $O_D$-module over the residue field $\mbF$ of $\breve E$. Then $(\mbX_A,ι)$ defines an ``absolute'' RZ space $\mcM_A$ over $\Spf O_{\breve E}$. It is the formal scheme with moduli description
\begin{equation}\label{eq:def_RZ_AB_Drinfeld}
\mcM_A(S) = \left\{(X_A, ι, ρ) \left\vert \text{\begin{varwidth}{\textwidth} \centering $(X_A, ι)/S$ an $(A, ϕ_0)$-special $O_D$-module\\[1mm] $ρ:\ov S \times_{\Spec \mbF} (\mbX_A, ι) \dashrightarrow \ov S \times_S (X_A, ι)$ a quasi-isogeny\end{varwidth}}\right.\right\}.
\end{equation}
Here, $\ov S$ denotes the special fiber $\mbF\tensor_{O_{\breve E}} S$. The right hand side of \eqref{eq:def_RZ_AB_Drinfeld} is meant as the set of isomorphism classes. The quasi-isogeny $ρ$ is required to be $O_D$-linear.

Now we apply the functor $Φ_A$ to obtain a special $O_D$-module $(\mbX, ι) := Φ_A(\mbX_A, ι)$ over $\mbF$. Since $Φ_A$ is an equivalence of categories, it also defines an isomorphism $\mcQ\mcI sog(\mbX_A, ι)\simto \mcQ\mcI sog(\mbX, ι)$ between the groups of $D$-linear self quasi-isogenies of $\mbX_A$ and $\mbX$.

Let $\breve F/F$ be the completion of a maximal unramified extension. Extend $ϕ_0:F\to E$ to an embedding $\breve F\to \breve E$. In particular, $\mbF$ is also viewed as the residue field of $\breve F$. Thus we can define the ``relative'' RZ space $\mcM$ over $\Spf O_{\breve F}$ by
\begin{equation}\label{eq:def_RZ_Drinfeld}
\mcM(S) = \left\{(X, ι, ρ) \left\vert \text{\begin{varwidth}{\textwidth} \centering $(X, ι)/S$ a special $O_D$-module\\[1mm] $ρ:\ov S \times_{\Spec \mbF} (\mbX, ι) \dashrightarrow \ov S \times_S (X, ι)$ a quasi-isogeny\end{varwidth}}\right.\right\}.
\end{equation}
\begin{cor}[to Corollary \ref{cor:CDA}]\label{cor:Drinfeld_comparison}
There is an $\mcQ\mcI sog(\mbX_A, ι)$-equivariant isomorphism of formal schemes over $\Spf(O_{\breve E})$,
\begin{equation}\label{eq:RZ_Drinfeld_iso}
Φ_A:\mcM_A \simlr O_{\breve E}\wh{\tensor}_{O_{\breve F}} \mcM.
\end{equation}
\end{cor}

Historically, the isomorphism \eqref{eq:RZ_Drinfeld_iso} was first conjectured by Rapoport--Zink \cite{RZ2} for Drinfeld's half space (meaning $D$ of Hasse invariant $1/d$ and $m = 1$). Their main result, which relies on display theory as well, proved \eqref{eq:RZ_Drinfeld_iso} when $F/K$ is unramified. Moreover, for general $F$, they constructed $Φ_A$ on the maximal reduced subscheme and on the generic fiber separately. Scholze--Weinstein \cite[Theorem 25.4.1]{SW} later gave a full proof using the theory of integral models of moduli spaces of local shtukas. Corollary \ref{cor:Drinfeld_comparison} extends these results to all $D$ and $m$. Moreover, our Corollary \ref{cor:CDA} goes beyond RZ spaces in that it also applies to families of $p$-divisible $O$-modules whose fiberwise isogeny classes vary.

\subsection{PEL type moduli spaces}\label{sec:uni-RZ}
Our next objective is to formulate and prove comparison isomorphisms for RZ spaces in the unitary setting. Let $X$ be a $p$-divisible $O$-module. We denote by $X^\vee$ its Faltings dual with respect to the Lubin--Tate object $\BT((W_O(O), I_O(O), σ, V^{-1}))$ defined by the multiplicative $O$-display (Example \ref{ex:displays}). Our intermediate aim is to construct a variant of $Φ_A$ as in Theorem \ref{thm:AB_strict_p_div} that is compatible with duality.

We call a $p$-divisible $O$-module $X$ \emph{biformal} if both $X$ and $X^\vee$ are formal. We call an $O$-display $\mcP$ \emph{binilpotent} if both $\mcP$ and $\mcP^\vee$ are nilpotent. For every binilpotent $O$-display $\mcP$, we have a natural isomorphism $α_\mcP:\BT(\mcP^\vee)\simto \BT(\mcP)^\vee$, see Proposition \ref{prop:Faltings_display_compatible}. Let
\begin{equation}
Ψ: \left\{\text{\begin{varwidth}{\textwidth} \centering biformal $p$-divisible\\$O$-modules \end{varwidth}}\right\} \simlr \left\{\text{binilpotent $O$-displays}\right\}
\end{equation}
be any choice of quasi-inverse to $\BT$ and let $γ:\mr{id} \simto \BT\circ Ψ$ be any choice of isomorphism. Via pullback of \eqref{eq:compatib_Hodge_filt} and \eqref{eq:resulting_alpha} along $γ$, $Ψ$ is compatible with dualization and passage to the Hodge filtration in the sense that there are natural isomorphisms $Ψ(X^\vee) \simto Ψ(X)^\vee$ and
$$[\mcD_X(R)\twoheadrightarrow \Lie(X)]\simlr [P_{Ψ(X)}/I_O\cdot P_{Ψ(X)} \twoheadrightarrow P_{Ψ(X)}/Q_{Ψ(X)}].$$
Here, we briefly used the notation $Ψ(X) = (P_{Ψ(X)}, Q_{Ψ(X)}, \bfF_{Ψ(X)}, \dF_{Ψ(X)})$.

\begin{rmk}\label{rmk:assumption_psi}
When $K = \mbQ_p$, an explicit functor $Ψ$ was constructed by Lau \cite{Lau_smoothness}.
\end{rmk}

For clarity, we now consider the two cases of Theorem \ref{thm:AB_strict_p_div} separately.

\medskip \noindent \emph{(1) The case of $(A, A)$-$\mathrm{Mod}$.} Let $F_0/K$ be an extension of $p$-adic local fields and let $F/F_0$ be a quadratic field extension with Galois conjugation $a\mapsto \bar a$. Let $A\subseteq \Hom_K(F, \ov K)$ be a subset such that
$$\Hom_K(F, \ov K) = A \sqcup \ov A$$
and let $E = E_A$. This is a special case of the setting for Theorem \ref{thm:AB_strict_p_div} (1) and we consider the equivalence $Φ_A$ defined there. Recall that, by definition, it is given by the chain of equivalences \eqref{eq:chain_1}:
\begin{equation}\label{eq:chain_2}
(A, A)\text{-Mod}\underset{Ψ}{\simlr} (A, A)\text{-Disp}\underset{\text{Thm. \ref{thm:AB_equiv_displays}}}{\simlr} (\emptyset, \emptyset)\text{-Disp}\underset{\text{Ex. \ref{ex:etale_displays}}}{\simlr} (\emptyset, \emptyset)\text{-Mod}.\end{equation}
These four categories are endowed with the following notions of duality:
\begin{itemize}
    \item On $(A, A)$-Mod, we consider hermitian duality $(X_A, ι)\mapsto (X_A^\vee, \ov{ι}^\vee)$, where $X_A\mapsto X_A^\vee$ denotes Faltings $O$-duality with respect to the Lubin--Tate object $\BT((W_O(O), I_O(O), σ, V^{-1}))$ as fixed at the beginning of the section.

    \item On $(A, A)$-Disp, we consider hermitian $O$-display duality as in Definition \ref{def:herm_dual}.

    \item Let $\mcH = \mcH(O_E)$ be the étale dualizing object from \eqref{eq:def_H_et_LT_disp}. On $(\emptyset, \emptyset)$-Disp, we consider $\mcH$-$\dagger$-duality as in \eqref{eq:H_dagger_dual}.

    \item Let $\mbH = \mcH^{\dF = \mr{id}}$ be the étale local system defined from $\mcH$ under the equivalence from \eqref{eq:equiv_etale}. It is an étale local system of $O_F$-modules of rank $1$. On étale $p$-divisible $O_F$-modules, we consider hermitian duality with respect to $\mbH$.
\end{itemize}

\begin{prop}\label{prop:duality_compatibility_etale}
The equivalence $Φ_A$ from Theorem \ref{thm:AB_strict_p_div} (1) is compatible with duality in the sense that there are natural isomorphisms $Φ_A((X_A, ι)^\vee) \simlr Φ_A(X_A, ι)^\vee$ that are compatible with the natural isomorphisms $(X_A, ι) \simto ((X_A, ι)^\vee)^\vee$ and $Φ_A(X_A, ι)\simto (Φ_A(X_A, ι)^\vee)^\vee$
\end{prop}
\begin{proof}
By definition, $Φ_A$ is the composition of the three equivalences in \eqref{eq:chain_2}. Each of these has an analogous compatibility propperty: (1) For $Ψ$, this is by the choice of $γ:\mr{id}\simto \BT\circ Ψ$ and the compatibility results for $\BT$ in Lemma \ref{lem:symmetry} and Proposition \ref{prop:Faltings_display_compatible}. (2) For $Φ_A$, this is the compatibility from Theorem \ref{thm:comp-display-banal}. (3) For the passage to étale local systems, this follows from the isomorphism in \eqref{eq:iso_duality_etale}.
\end{proof}

\medskip \noindent \emph{(2) The case of $(A, B)$-$\mathrm{Mod}$.} Let $F_0/K$ and $F/F_0$ be as before; further assume that
$$\Hom_K(F, \ov K) = A \sqcup \ov A \sqcup \{ϕ_0, \ov{ϕ_0}\}$$
and set $B = A\sqcup \{ϕ_0, \ov{ϕ_0}\}$. Let $E = E_AE_B$ be the reflex field. In principle, we now again consider the diagram \eqref{eq:diag_proof_EL}. However, it is not immediately clear how to formulate a compatibility with duality for its right hand side square. For this reason, we now restrict to \emph{relatively} biformal objects:
\begin{defn}\label{def:binilpotent}
We call a pair $(\mcP, ι) \in (A,B)$-Disp \emph{relatively binilpotent} if $Φ_A(\mcP, ι)$ is binilpotent. We call a pair $(X, ι) \in (A, B)$-Mod \emph{relatively biformal} if $Ψ(X, ι)$ is relatively binilpotent.
\end{defn}
Our aim is to construct an equivalence
$$Φ_A:\left\{\text{\begin{varwidth}{\textwidth}\centering Relatively biformal\\$(X, ι)\in (A,B)$-Mod\end{varwidth}}\right\} \underset{?}{\simlr} \left\{\text{\begin{varwidth}{\textwidth}\centering Biformal strict $O_{F_0}$-modules\\
with $O_F$-action\end{varwidth}}\right\}$$
that is compatible with hermitian duality. This will require the Ahsendorf functor from \cite[Theorem 3.3.2]{KRZ} (originally due to Ahsendorf \cite{ACZ}) as intermediate step. Let $R$ be an $O_{F_0}$-algebra in which $p$ is nilpotent. The Ahsendorf functor is an equivalence of
\begin{equation}\label{eq:Ahsendorf_functor}
\mfA:\left\{\text{\begin{varwidth}{\textwidth}\centering Formal $O$-displays with\\strict $O_{F_0}$-action over $R$\end{varwidth}}\right\} \simlr \left\{\text{\begin{varwidth}{\textwidth}\centering Formal $O_{F_0}$-displays\\over $R$\end{varwidth}}\right\}.
\end{equation}
Let $\mcL$ be a Lubin--Tate display for $O_{F_0}/O$ over $R$ and let $\mcP\mapsto \mcP^{\mcL\text{-}\dagger}$ denote duality with respect to $\mcL$ as in Proposition \ref{prop:dual_LT}. Still by \cite[Theorem 3.3.2]{KRZ}, the Ahsendorf functor is compatible with duality in the sense that it comes with a natural isomorphism
\begin{equation}\label{eq:Ahsendorf_compatibility}
\mfA(\mcP^{\mcL\text{-}\dagger}) \simlr \mfA(\mcP)^{\mfA(\mcL)\text{-}\dagger}.
\end{equation}
This identity makes sense whenever $\mcP^{\mcL\text{-}\dagger}$ is again formal. We again call such $\mcP$ \emph{relatively binilpotent}. Note that this notion does not depend on $\mcL$ by Lemma \ref{lem:LT_object_essentially_unique}. Finally, we note that the Ahsendorf functor and the formation of the isomorphism \eqref{eq:Ahsendorf_compatibility} commute with base change in $R$.

After these preparations, we come to our main arguments. We work from now on over the ring $O_E$, where $E = E_AE_B$. Note that $\{ϕ_0, \ov{ϕ_0}\}:F_0\hookrightarrow E$, so the notion of strict $O_{F_0}$-actions makes sense for objects over $O_E$-schemes. Define $Φ_A$ as the composition of the following functors:
\begin{equation}\label{eq:Phi_A_square}
\xymatrix{
\left\{\text{\begin{varwidth}{\textwidth}\centering Relatively biformal\\$(X, ι)\in (A,B)$-Mod\end{varwidth}}\right\} \ar@{-->}[rr]^-{\sim}_-{\text{Def. of }Φ_A} \ar[d]^-{Ψ}&
&
\left\{\text{\begin{varwidth}{\textwidth}\centering Biformal strict $O_{F_0}$-modules\\with $O_F$-action\end{varwidth}}\right\}\\
\left\{\text{\begin{varwidth}{\textwidth}\centering Relatively binilpotent\\$(\mcP, ι)\in (A,B)$-Disp\end{varwidth}}\right\} \ar[r]^-{Φ_A}_-{\text{Thm. \ref{thm:AB_equiv_displays}}} &
\left\{\text{\begin{varwidth}{\textwidth}\centering Relatively binilpotent\\$(\mcP, ι) \in (\emptyset, \{ϕ_0, \ov{ϕ_0}\})$-Disp\end{varwidth}}\right\} \ar[r]^-{\mfA}_-{\eqref{eq:Ahsendorf_functor}} &
\left\{\text{\begin{varwidth}{\textwidth}\centering Biformal $O_{F_0}$-displays\\with $O_F$-action\end{varwidth}}\right\} \ar[u]^-{\BT}_-{\text{Thm. \ref{thm:equiv_displays}}}
}
\end{equation}
The five categories of the diagram are endowed with the following notions of duality:
\begin{itemize}
    \item On $(A, B)$-Mod, we consider hermitian duality $(X, ι)\mapsto (X^\vee, \ov{ι}^\vee)$, where $X\mapsto X^\vee$ denotes Faltings $O$-duality with respect to the Lubin--Tate object $\BT((W_O(O), I_O(O), σ, V^{-1}))$ as fixed at the beginning of the section.

    \item On $(A, B)$-Disp, we consider hermitian $O$-display duality as in Definition \ref{def:herm_dual}.

    \item Let $\mcL = \mcL(O_E)$ be the Lubin--Tate display from Definition \ref{def:LT_display}. On $(\emptyset, \{ϕ_0, \ov{ϕ_0}\})$-Disp, we consider $\mcL$-$\dagger$-duality as in Definition \ref{def:herm_dagger_dual}.

    \item Let $\mbL = \mfA(\mcL)$ be the $O_{F_0}$-display obtained from $\mcL$ by the Ahsendorf functor. On $O_{F_0}$-displays, we consider duality with respect to $\mbL$.
    
    \item Finally, on strict formal $O_{F_0}$-modules, we consider Faltings duality with respect to $\BT(\mbL)$.
\end{itemize}

\begin{thm}\label{thm:AB_strict_p_div_PEL}
The functor $Φ_A$ defined in \eqref{eq:Phi_A_square} is compatible with duality in the sense that there are natural isomorphisms $Φ_A((X_A, ι)^\vee)\simto Φ_A(X_A, ι)^\vee$ that are compatible with the natural isomorphisms $(X_A, ι)\simto ((X_A, ι)^\vee)^\vee$ and $Φ_A(X_A, ι) \simto (Φ_A(X_A, ι)^\vee)^\vee$.
\end{thm}
\begin{proof}
Each of the four functors in \eqref{eq:Phi_A_square} has an analogous compatibility property: (1) For $Ψ$, this compatibility comes from Lemma \ref{lem:symmetry}, Proposition \ref{prop:Faltings_display_compatible}, and its definition through $γ$. (2) For $Φ_A$, this compatibility is Theorem \ref{thm:comp-display}. (3) For $\mfA$, the compatibility is part of \cite[Theorem 3.3.2]{KRZ}. (4) For the equivalence $\BT$, the compatibility again follows from the appendix: For every biformal $O_{F_0}$-display with $O_F$-action $\mbP$, Construction \ref{construction:alpha} \emph{mutatis mutandis} defines a natural homomorphism $β_\mbP:\BT(\mbP^\text{$\mbL$-$\dagger$})\to \BT(\mbP)^\text{$\BT(\mbL)$-$\dagger$}$. Over an unramified extension of $O_{F_0}$, the $O_{F_0}$-display $\mbL$ becomes isomorphic to the multiplicative $O_{F_0}$-display $(W_{O_{F_0}}(O_{F_0}), I_{O_{F_0}}(O_{F_0}), σ, V^{-1})$. So $β_\mbP$ is equal to $α_\mbP$ étale locally (up to isomorphism), which shows that $β_\mbP$ is an isomorphism (Proposition \ref{prop:Faltings_display_compatible}). The claim about compatibility with double-duality for $β_\mbP$ follows by the same formal computation as for Lemma \ref{lem:symmetry}.
\end{proof}

We now deduce comparison isomorphisms of ``unitary vertex level RZ spaces for signature $(r,s)$ in the unramified case''. Let $n\geq 1$ be an integer and assume that $F/F_0$ is unramified. Let $A$ and $ϕ_0$ be as before; that is,
\begin{equation}\label{eq:setting_unitary_RZ}
\Hom_K(F, \ov K) = A\sqcup \ob A \sqcup \{ϕ_0, \ov{ϕ_0}\},\quad B = A\sqcup \{ϕ_0, \ov{ϕ_0}\}.
\end{equation}
Set
$$E = \begin{cases}
    ϕ_0(F_0)\cdot E_A & \text{if $r = s$}\\
    ϕ_0(F)\cdot E_A & \text{if $r \neq s$}.
\end{cases}$$
Then $E$ is the reflex field for the following moduli problem.
\begin{defn}\label{def:AB_strict_PEL} Let $R$ be an $O_E$-algebra.

\medskip \noindent (1) Let $(X_A, ι)$ be an $(A, B)$-strict pair of dimension $n\cdot [F_0:K]$ above $R$. In particular, $J_A\cdot \Lie(X_A)$ is finite projective of rank $n$. We say that $(X_A, ι)$ is \emph{of signature $(r,s)$} if for all $a\in O_F$,
$$\mr{char}(ι(a)\mid J_A\cdot \Lie(X_A); T) = (T-a)^r (T-\ov a)^s.$$
Here the right hand side lies in $O_E[T]$ and is viewed as an element $\mcO_S[T]$ by the map $O_E\to \mcO_S$.

\medskip \noindent (2) Let $(X_A, ι)$ be $(A, B)$-strict and of signature $(r,s)$ as in (1). When we consider polarizations $λ:X_A\to X_A^\vee$, we always implicitly assume that they are \emph{compatible with the $O_F$-action} in the sense that $λ\circ ι(a) = ι^\vee(\ob a)\circ λ$ for all $a\in O_F$. We say that $λ$ is of type $t$ with $0\leq t \leq n$ if
$$\ker(λ) \subseteq X_A[π_F],\quad \deg(λ) = q_F^t.$$
\end{defn}

Let $\mbF$ be the residue field of a maximal unramified extension $\breve E/E$ and let $(\mbX_A, ι, λ)$ over $\mbF$ be an $(A, B)$-strict pair of signature $(r,s)$ with polarization of type $t$. The absolute RZ space of $(\mbX_A, ι, λ)$ is the formal scheme over $\Spf O_{\breve E}$ with moduli description
\begin{equation}\label{eq:def_RZ_AB_PEL}
\mcN^{[t]}_{A,(r,s)}(S) = \left\{(X_A, ι, λ, ρ) \left\vert \text{\begin{varwidth}{\textwidth} \centering $(X_A, ι)/S$ an $(A, B)$-strict pair of signature $(r,s)$\\[1mm]
$λ:X_A\to X_A^\vee$ a polarization of type $t$\\[1mm]
$ρ:\ov S \times_{\Spec \mbF} (\mbX_A, ι, λ) \dashrightarrow \ov S \times_S (X_A, ι, λ)$ a quasi-isogeny\end{varwidth}}\right.\right\}.
\end{equation}

Fix an isomorphism $δ$ between $\mbL$ and $\BT((W_{O_{F_0}}(O_{F_0}), I_{O_{F_0}}(O_{F_0}), σ, V^{-1}))$ over $O_{\breve E}$. Then, if $λ_\mbL:\mbP\to \mbP^\text{$\mbL$-$\vee$}$ is an $\mbL$-polarization, the composition $λ = δ\circ λ_\mbL$ is a polarization $\mbP\to \mbP^\vee$ with respect to the ``standard'' choice of Lubin--Tate object.

Assume from now on that $\mbX_A$ is relatively biformal. Using $δ$, we construct a $p$-divisible $O_{F_0}$-module $(\mbX, ι, λ) := Φ_A(\mbX_A, ι, λ)$ of $O_{F_0}$-height $2n$, dimension $n$, together with an $O_F$-action $ι$ of signature $(r,s)$ and a polarization of type $t$. We define the relative RZ space over $O_{\breve F}$ by
\begin{equation}\label{eq:def_RZ_PEL}
\mcN^{[t]}_{(r,s)}(S) = \left\{(X, ι, λ, ρ) \left\vert \text{\begin{varwidth}{\textwidth} \centering $(X, ι)/S$ a $p$-divisible $O_{F_0}$-module with $O_F$-action of signature $(r,s)$\\[1mm]
$λ:X\to X^\vee$ a polarization of type $t$\\[1mm]
$ρ:\ov S \times_{\Spec \mbF} (\mbX, ι, λ) \dashrightarrow \ov S \times_S (X, ι, λ)$ a quasi-isogeny\end{varwidth}}\right.\right\}.
\end{equation}
The group $\mcQ \mcI sog(\mbX_A, ι, λ)$ acts on $\mcN^{[t]}_{A,(r,s)}$ by composition in $ρ$. Moreover, $Φ_A$ also provides an isomorphism $Φ_A:\mcQ \mcI sog(\mbX_A, ι, λ) \to \mcQ \mcI sog(\mbX, ι, λ)$.
\begin{cor}[to Theorem \ref{thm:AB_strict_p_div_PEL}]\label{cor:unitary-comp}
There is a $\mcQ \mcI sog(\mbX_A, ι, λ)$-equivariant isomorphism
$$Φ_A:\mcN^{[t]}_{A,(r,s)} \simlr O_{\breve E} \widehat{\tensor}_{O_{\breve F}} \mcN^{[t]}_{(r,s)}.$$
\end{cor}
\begin{proof}
This follows directly from Theorem \ref{thm:AB_strict_p_div_PEL}.
\end{proof}

\section{Unitary Shimura varieties}\label{sec:Shimura}

We now turn to the global theory. We define integral models for RSZ unitary Shimura varieties following \cite{RSZ-AGGP,RSZ} and \cite[\S C]{liu2021fourier}. The novelty of our definition is that we drop all the ramification assumptions on $F_0/\mbQ$ from \S5 of \cite{RSZ}. Even though this definition also works at ramified places of $F/F_0$, we only consider the unramified places for simplicity. (Only these will come up in later our application to arithmetic transfer.) We also restrict to vertex level structure, but our definition directly extends to lattice chain moduli problems as in Corollary \ref{cor:LM-comparison}.

\subsection{RSZ Shimura varieties}\label{sec:RSZ}
Let $F / F_0 $ be a quadratic CM extension of a totally real number field with non-trivial Galois involution $x \mapsto \ov{x}$. Let $\ov{\mbQ}$ be the algebraic closure of $\mbQ$ in $\mbC$, and let $Φ \subset \Hom(F, \ov{\mbQ})$ be a CM type. Let $V$ be a hermitian $F$-vector space of dimension $n\geq 2$. We assume that there exists an embedding $ϕ_0\in Φ$ such that for all $ϕ\in Φ$,
\begin{equation}\label{eq:sign_V}
\mr{sign}(V_ϕ) = \begin{cases}
(n-1,1) & \text{if $ϕ = ϕ_0$}\\
\text{$(n,0)$ or $(0,n)$} & \text{if $ϕ\neq ϕ_0$}.\end{cases}
\end{equation}
Starting from \S\ref{sec:global_intersection_numbers}, we will assume that only the positive definite case occurs for $ϕ\neq ϕ_0$, but we include the more general case before that.

Let $G=\mathrm{U}(V)$ be the unitary group of $V$, viewed as an algebraic group over $F_0$. As in \cite[\S 3]{RSZ}, we consider the following reductive groups over $\mathbb Q$. The notation $c(g)$ denotes the similitude factor of $g$:
$$\begin{aligned}
G^{\mathbb Q} & := \{g \in \Res_{F_0/\mathbb Q} \GU(V) | \, c(g) \in  \mathbb G_m \},\\
Z^{\mathbb Q} & :=\{ z \in  \Res_{F/\mathbb Q} \mathbb G_m | \, \mathrm{Nm}_{F/F_0}(z) \in \mathbb G_m \},\\
\widetilde{G} & :=Z^{\mathbb Q} \times_{\mr{Nm}_{F/F_0}, \mathbb G_m, c} G^{\mathbb Q}.
\end{aligned}$$
Note that $(z,g)\mapsto (z, z^{-1}g)$ defines an isomorphism $\widetilde{G}\simto Z^{\mathbb Q} \times \Res_{F_0/\mathbb Q} G$. Consider the RSZ Shimura datum from \cite[\S3.2]{RSZ},
\[
( \widetilde{G}, \{ h_{\widetilde{G}} \} ) = ( Z^{\mbQ}, h_{\Phi}) \times ( \Res_{F_0/\mathbb Q}G, \{ h_{\Res_{F_0/\mathbb Q}G} \}). 
\]
Let $E$ be the reflex field of $\{h_{\wt G}\}$. Then, for any neat compact open subgroup
$$\widetilde{K}=K_{Z^{\mbQ}} \times K \subset Z^{\mathbb Q}(\mathbb A_f) \times \Res_{F_0/\mathbb Q} G(\mathbb A_f) = \widetilde{G}(\mathbb A_f),$$
we have the \emph{RSZ unitary Shimura variety} attached to $( \widetilde{G}, \{ h_{\widetilde{G}} \} )$ with level $\wt{K}$
\[
M := \mathrm{Sh}_{\wt{K}}( \widetilde{G}, \{ h_{\widetilde{G}} \} ) \lr \Spec E.
\]
Define a function $r:\Hom(F, \mbQ)\to \mbZ_{\geq 0}$ through the formula
\begin{equation}\label{eq:def_r}
\mr{sign}(V_ϕ) = (r_ϕ, r_{\ov{ϕ}}),\quad ϕ\in Φ.
\end{equation}
By \cite[(3.4)]{RSZ}, the above reflex field $E$ can be described by
\begin{align*}
\Gal(\ov{\mbQ}/E)={}&\{\sigma\in \Gal(\ov{\mbQ}/\mbQ)\mid  σ\circ Φ = Φ \text{ and }r = r\circ σ\}.
\end{align*}
Note that we always have $ϕ_0(F)\subseteq E$. If $F$ is Galois over $\mbQ$, or if $F$ contains an imaginary quadratic subfield such that $Φ$ is induced from this subfield, we have $\phi_0: F \cong E$. In general, $E$ can be strictly larger than $\phi_0(F)$. 

\subsection{$(A,B)$-strictness for abelian schemes}
In complete analogy with \eqref{eq:reflex_E_A}, every subset $S\subseteq \Hom(F, \ov{\mbQ})$ defines a reflex field $E_S \subset \ov{\mbQ}$ by the formula
\begin{align*}
\Gal(\ov{\mbQ}/E_S)={}&\{\sigma\in\Gal(\ov{\mbQ}/\mbQ)\mid \sigma\circ S = S\}.
\end{align*}
We define the global Eisenstein ideal of $S$ (following Howard \cite{H})
\begin{equation*}
    J_S = \ker\left(\prod_{ϕ\in S} ϕ\tensor \mr{id} : O_F\otimes_{\mbZ}O_{E_S}\lr \prod_{\phi\in S}\ov{\mbQ}\right).
\end{equation*}
An equivalent way is to define $J_S$ as the ideal generated by all elements $e_{ζ,S}(ζ\tensor 1) \in O_F \tensor_\mbZ O_{E_S}$, for $ζ\in O_F$, where $e_{ζ,S}(T)$ is as in \eqref{eq:def_eisenstein_element}. The relation of global and local Eisenstein ideal is easy to describe. First, we recall the relation of local and global reflex fields, see \cite[(4.11)]{RSZ-AGGP}:

Let $ν\mid p$ be a place of $E_S$. Choose an embedding $α:\ov{\mbQ}\to \ov{\mbQ_p}$ that induces $ν$. Then there is a decomposition
\begin{equation}\label{eq:decomp_local_global}
\Hom(F, \ov{\mbQ}) \underset{α_*}{\simlr} \Hom(F, \ov{\mbQ_p}) = \coprod_{v\mid p} \Hom(F_v, \ov{\mbQ_p})
\end{equation}
by the places $v$ of $F$ above $p$. Let $S = \sqcup_{v\mid p} S_v$ be the induced decomposition of $S$. The Galois group of $\mbQ_p$ acts memberwise on the right hand side of \eqref{eq:decomp_local_global}. So we obtain
\begin{equation}\label{eq:reflex_localized}
\begin{aligned}
\Gal(\ov{\mbQ_p}/E_{S,ν}) = & \{σ\in \Gal(\ov{\mbQ_p}/\mbQ_p) \mid σ\circ S = S\}\\[1mm]
= & \{σ\in \Gal(\ov{\mbQ_p}/\mbQ_p) \mid σ\circ S_v = S_v \text{ for all $v$}\}.
\end{aligned}
\end{equation}
It follows that $E_{S,ν} = \langle E_{S_v}\rangle_{v\mid p}$, where $E_{S_v}\subset \ov{\mbQ_p}$ is the reflex field of $S_v$. Moreover, note that
$$(O_F\tensor_\mbZ O_{E_S})\tensor_{O_{E_S}} O_{E_S, ν} = \prod_{v\mid p} O_{F,v}\tensor_{\mbZ_p} O_{E_S, ν}.$$
Compatibly with this decomposition, one shows with Lemma \ref{lem:indep_of_generator} that
\begin{equation}\label{eq:localize_J}
    J_S \tensor_{O_{E_S}} O_{E_S, ν} = \prod_{v\mid p} J_{S_v}\tensor_{O_{E_{S_v}}}O_{E_S, ν}.
\end{equation}
We will make use of the natural embeddings $E_{S_v} \to E_{S,ν}$ during the uniformization of $\mcM$ later.

We come back to the data $Φ$ and $V$. Recall that $ϕ_0:F\hookrightarrow \ov{\mbQ}$ is the distinguished element of $Φ$, characterized by $\mr{sign}(V_{\phi_0})=(n-1,1)$. Set
\begin{equation}
\begin{aligned}
A = & \{ϕ\in Φ \mid \mr{sign}(V_ϕ) = (n,0)\}\\
&\ \cup \{ϕ \in \ov{Φ} \mid \mr{sign}(V_ϕ) = (0, n)\}.
\end{aligned}
\end{equation}
Since we assumed $n\geq 2$ which implies $(n-1,1)\notin \{(n,0), (0,n)\}$, there is a disjoint union decomposition
$$\Hom_{\mbQ}(F,\ov{\mbQ})=A\sqcup \ov{A}\sqcup \{\phi_0,\ov{\phi_0}\},$$
and we define $B:=A\sqcup \{\phi_0,\ov{\phi_0}\}$. Then it is clear from definitions that 
$$E= ϕ_0(F)E_AE_Φ.$$

\begin{notation}\label{notation:A}
Let $\mcA$ be an abelian scheme over a scheme $S$. We denote by $\mcD(\mcA)$ the locally free $\mcO_S$-module dual to the de Rham cohomology of $\mcA$. It is a locally free $\mcO_S$-module of rank $2\dim(\mcA)$. The Hodge filtration of $\mcA$ takes the form
\begin{equation*}
    0\lr \Fil(\mcA)\lr \mcD(\mcA)\lr \Lie(\mcA)\lr 0.
\end{equation*}

For rings of integers and homomorphism spaces of abelian varieties, we use the shorthand notations
\begin{equation}\label{eq:shorthands}
\begin{cases}\mbZ^\mfd = \mbZ[\mfd^{-1}],\quad O_F^\mfd = O_F[\mfd^{-1}],\ldots\\[2mm]
\Hom^\mfd(\mcA_1, \mcA_2) = \Hom(\mcA_1,\mcA_2)[\mfd^{-1}],\quad \End^\mfd(\mcA) = \End(\mcA)[\mfd^{-1}],\ldots.
\end{cases}
\end{equation}
Recall the following basic fact: Assume that $S$ is a $\mbZ^\mfd$-scheme, that $\mcA/S$ is an abelian scheme, and that $ι$ is an action $ι:O_F^\mfd \to \End^\mfd(\mcA)$. Then $[F_0:\mbQ]\mid \dim(\mcA)$ and $\mcD(\mcA)$ is a locally free $O_F\tensor_{\mbZ}\mcO_S$-module of rank $\dim(\mcA)/[F_0:\mbQ]$.
\end{notation}

\begin{defn}[Eisenstein and Signature conditions]\label{def:signature_abelian_variety} (1) Let $S$ be an $O_E^\mfd$-scheme and let $\mcA/S$ be an abelian scheme of dimension $n[F_0:\mbQ]$ together with an action $ι:O_F^\mfd\to \End^\mfd(A)$. We say that $ι$ is \emph{$(A, B)$-strict}, or that $ι$ satisfies the \emph{Eisenstein condition for $(A,B)$}, if
\begin{equation}\label{eq:AB_strict_abelian_variety}
J_B\cdot \mcD(\mcA) \subseteq \Fil(\mcA) \subseteq J_A\cdot \mcD(\mcA).
\end{equation}
Note that if $ι$ is $(A,B)$-strict, then by Lemma \ref{lem:AB-strict_dualizing} also the conjugate-dual action $\ov{ι}^\vee$ of $O_F$ on $\mcA^\vee$ is $(A,B)$-strict.

\medskip \noindent (2) Let $S$ as well as $(\mcA, ι)$ be as before and assume that $ι$ is $(A,B)$-strict. Then, by Lemma \ref{lem:eisenstein_ideal} (1) in conjunction with \eqref{eq:localize_J}, the product $J_A\cdot \Lie(\mcA)$ is a locally free $\mcO_S$-module of rank $n$. We say that $(\mcA, ι)$ satisfies the \emph{signature $(n-1,1)$ condition} if for every $a\in O_F$,
\begin{equation}\label{eq:signature_abelian_variety}
\mr{char}(ι(a)\mid J_A\cdot \Lie(\mcA); T) = (T - ϕ_0(a))^{n-1} (T-ϕ_0(\ov a)).
\end{equation}
\end{defn}

\begin{ex}[Relation of local and global signature conditions]
Let $ν\nmid \mfd$ be a place of $E$, and let $S$ be an $O_{E,ν}$-scheme. Let $\mcA/S$ be an abelian scheme of dimension $n[F_0:\mbQ]$ and let $ι:O_F^\mfd\to \End^\mfd(\mcA)$ be an action. Then
$$O_F^\mfd\tensor_\mbZ \mcO_S \simlr \prod_{v\mid p} O_{F,v}\tensor_{\mbZ_p} \mcO_S,$$
where the decomposition is over all places $v\mid p$ of $F$, and compatibly
$$\begin{aligned}
\Fil(\mcA) & = \oplus_{v\mid p} \Fil(\mcA)_v,\\[1mm]
\mcD(\mcA) & = \oplus_{v\mid p} \mcD(\mcA)_v,\\[1mm]
\Lie(\mcA) & = \oplus_{v\mid p} \Lie(\mcA)_v.
\end{aligned}$$
Each $\mcD(\mcA)_v$ is locally free of rank $n$ as $O_{F,v}\tensor_{\mbZ_p}\mcO_S$-module. Choose an embedding $α:\ov{\mbQ}\to \ov{\mbQ_p}$ that induces $ν$. This defines local sets $\{A_v\}_{v\mid p}$ and $\{B_v\}_{v\mid p}$ through \eqref{eq:decomp_local_global}. Moreover, each local reflex field $E_{A_v}E_{B_v}$ is naturally a subfield of $E_ν$. We obtain from \eqref{eq:localize_J} that
\begin{equation}\label{eq:translate_local_global_signature}
\text{$\Fil(\mcA)$ is $(A,B)$-strict}\quad \Longleftrightarrow\quad \text{each $\Fil(\mcA)_v$ is $(A_v, B_v)$-strict}.
\end{equation}
Let $v_0\mid p$ be the unique place such that $ϕ_0\in B_{v_0}$. First, consider a place $v\notin \{v_0, \ov{v_0}\}$. Then $A_v = B_v$, and the condition for $v$ on the right hand side of \eqref{eq:translate_local_global_signature} can be rephrased by
\begin{equation}\label{eq:banal_factor_abelian_variety}
\Fil(\mcA)_v = J_{A_v}\cdot \mcD(\mcA)_v.
\end{equation}
This is the so-called banal Eisenstein condition from \cite[Appendix B]{RSZ-AGGP} which we recalled in Remark \ref{rmk:reflex}.

Assume from now on that \eqref{eq:banal_factor_abelian_variety} holds for all $v\notin \{v_0, \ov{v_0}\}$, and focus on the place $v_0$. We first consider the case $\ov{v_0} \neq v_0$. Then $(\mcA, ι)$ satisfies the signature $(n-1,1)$ condition from Definition \ref{def:signature_abelian_variety} if and only if
\begin{equation}
\begin{cases}
\text{$\Fil(\mcA)_{v_0}$ is $(A_{v_0}, A_{v_0}\cup \{ϕ_0\})$-strict of relative dimension $n-1$, and}\\
\text{$\Fil(\mcA)_{\ov{v_0}}$ is $(A_{\ov{v_0}}, A_{\ov{v_0}}\cup\{\ov{ϕ_0}\})$-strict of relative dimension $1$.}
\end{cases}
\end{equation}
The local conditions here are the ones from Definition \ref{def:AB-strict}. Finally, we consider the case $\ov{v_0} = v_0$. Then $(\mcA, ι)$ satisfies the signature $(n-1,1)$ condition if and only if
\begin{equation}
\text{$\Fil(\mcA)_{v_0} \subset \mcD(\mcA)_{v_0}$ is of signature $(n-1,1)$ in the sense of Definition \ref{def:abs-LM} (v) and (vi)}.
\end{equation}
\end{ex}

\subsection{Integral models of toric Shimura varieties}
Let $\mfd \in \mbZ_{>1}$ be a product of primes such that the neat level $K_{Z^\mbQ}$ factors as
\begin{equation*}
K_{Z^\mbQ}=K_{Z^\mbQ,\mfd}\times K_{Z^\mbQ}^{\mfd}
\end{equation*}
where $K_{Z^\mbQ}^{\mfd}$ is the maximal compact subgroup in $Z^\mbQ(\mbA^\mfd_f)$. As an auxiliary datum, we also fix a hermitian $F_\mfd/F_{0,\mfd}$-module $\mbW_\mfd$ of rank $1$. Following \cite[\S 3.2]{RSZ-AGGP} and \cite[Definition 3.6, \S 4.1]{RSZ}, we define
$$
\mcM'_0 \lr \Spec O_{E_\Phi}^\mfd
$$
as the functor sending a scheme $S/O_{E_Φ}^\mfd$ to the set of isomorphism classes of tuples $(\mcA_0, \iota_0, \lambda_0, \ov \eta_0)$ where 
\begin{altenumerate}
 		\item $\mcA_0$ is an abelian scheme over $S$ of dimension $[F_0 : \mathbb Q]$.
		\item $\iota_0: O_{F}^\mfd\to \End^\mfd(\mcA_0)$ is an $O_F^\mfd$-action that is \emph{$(Φ, Φ)$-strict}, meaning that $\mr{Fil}(\mcA_0) = J_Φ\cdot \mcD(\mcA_0)$.\footnote{This condition is taken from \cite{H}.}
		\item $\lambda_0 \in \Hom^\mfd(\mcA_0, \mcA_0^\vee)$ is a \emph{$\mfd$-quasi polarization}, meaning it is a quasi-isogeny such that locally on $S$ some $\mfd^k$-multiple is a polarization.
        \item We require that $ι_0$ and $λ_0$ are compatible in the sense that for all $a \in O_F^\mfd$, we have
        $$ \lambda^{-1}_0 \circ \iota_0(a)^{\vee} \circ \lambda_0 = \iota_0(\overline{a}).$$
        \item We require that the degree of $λ_0$ is a product of powers of primes dividing $\mfd$. Put differently, we demand that $\ker(\mfd^k \lambda_0)[\mfd^{-1}] = 0$, where $k\geq 0$ is chosen (locally on $S$) such that $\mfd^k λ_0\in \Hom(\mcA, \mcA^\vee)$.
		\item $ \ov{\eta}_0$ is a $K_{Z^{\mathbb Q}_\mfd}$-level structure in the sense of \cite[C.3]{liu2021fourier} using $\mbW_0$.
\end{altenumerate}
Here, an isomorphism of such tuples
\begin{equation}\label{eq:iso_A_0}
γ_0:(\mcA_0, ι_0, λ_0, η_0) \simlr (\mcA_0', ι_0', λ_0', η_0')
\end{equation}
is an $O_F^\mfd$-linear element $γ_0\in \Hom^\mfd(\mcA_0, \mcA_0')$ such that $γ_0^*(λ_0') \in \mbZ[\mfd^{-1}]^\times\cdot λ_0$ and $\ov{η}_0 = \ov{η}_0'\circ γ_0$.

We emphasize that the datum $\mbW_0$ is only used to pin down a moduli description for the integral model of the toric Shimura variety. Suppose that $F/F_0$ is ramified at some place, or that there exists a non-split place of $F_0$ that divides $\mfd$. Then there exists a choice for $\mbW_0$ such that $\mcM_0'$ is non-empty, see \cite[Remark 3.7]{RSZ}. Only in the rather specific case that $F/F_0$ is everywhere unramified and that all places above $\mfd$ are split, it can happen that $\mcM_0'$ is empty for choice of $\mbW_0$. In this case, one could slightly modify the moduli problem as explained in \cite[after Definition 3.6]{RSZ}. For simplicity, we will stick to the above moduli problem and henceforth assume that $\mcM_0'$ is non-empty.

As explained in \cite[\S 3.4]{RSZ}, $\mcM_0'$ is representable by a finite etale $O^\mfd_{E_\Phi}$-scheme. Its generic fiber $(\mcM_0')_\mbQ$ is a disjoint union of copies of the Shimura variety $\mathrm{Sh}_{K_{Z^{\mbQ}}}(Z^{\mbQ},h_{\Phi})$. We work with a fixed copy of this union, and we denote by $\mcM_0\subseteq \mcM_0'$ its closure. It is an open and closed subscheme of $\mcM_0'$ that provides the finite, étale integral model of $\mathrm{Sh}_{K_{Z^{\mbQ}}}(Z^{\mbQ},h_{\Phi})$.

\subsection{Integral models of RSZ Shimura varieties}\label{sec:integral of RSZ Shimura}
Now we define the punctured integral model of $\mr{Sh}_{\wt K}(\wt G, \{h_{\wt G}\})$. Let $L$ be an $O_F$-lattice in $V$ and let $\mfd \in \mbZ_{>1}$ be chosen such that:
\begin{altitemize}
    \item If $v$ is a finite place of $F_0$ that is ramified in $F$, then $v\mid \mfd$;
    \item For any finite place $v\nmid \mfd$ of $F$ that is inert over $F_0$, the localization $L_v$ in $V_v=V \otimes_F F_v$ is a vertex lattice, i.e. $L_v \subseteq L_v^\vee \subseteq \pi_v^{-1}L_v$, where $\pi_v\in F_v$ is a uniformizer;
    \item For any pair $\{v, \ov v\}\nmid \mfd$ of places of $F$ split over $F_0$, the localization of $L$ is self-dual in the sense $L_v^\vee = L_{\ov v}$.
    \item The level $K$ factors as 
    \begin{equation*}
        K=K_\mfd\times K^\mfd
    \end{equation*}
    where $K_\mfd\subset G(F_{0,\mfd})$ is any open compact subgroup and where $K^\mfd=\mathrm{Stab}(\wh{L}^\mfd)$ is the stabilizer of the completion $\wh{L}$ in $G(\mbA_{0,f}^\mfd)$;
    \item The level $K_{Z^\mbQ}$ factors similarly as
    \begin{equation*}
        K_{Z^\mbQ}=K_{Z^\mbQ,\mfd}\times K_{Z^\mbQ}^{\mfd}
    \end{equation*}
    where $K_{Z^\mbQ,\mfd}\subset Z^\mbQ(\mbQ_\mfd)$ is any open compact subgroup and where $K_{Z^\mbQ}^{\mfd}$  is the maximal compact subgroup in $Z^\mbQ(\mbA^\mfd_f)$.
\end{altitemize}

\begin{defn}[compare \protect{\cite[Definition 4.1]{RSZ}}]\label{def:RSZ-integral}
The punctured integral model of the RSZ Shimura variety
$$
\mathcal M \lr \Spec O_E^{\mfd}
$$ 
is the functor that takes an $O_E^\mfd$-scheme $S$ to the set of isomorphism classes of tuples
$$(\mcA_0, ι_0, \lambda_0,\ov{\eta}_0; \mcA, \iota, \lambda,\ov{\eta})$$
that are of the following form:
\begin{altenumerate}
\item  $(\mcA_0, \iota_0, \lambda_0, \ov{\eta}_0) \in \mcM_0(S)$.
\item $\mcA$ is an abelian scheme over $S$ of dimension $n[F_0:\mathbb Q]$.
\item $\iota: O^{\mfd}_{F} \to \End^{\mfd}(\mcA)$ is an $O_F^\mfd$-action that is of signature $(n-1,1)$ in the sense of Definition \ref{def:signature_abelian_variety}.
\item $\lambda\in\Hom^\mfd(\mcA,\mcA^\vee)$ is a $\mfd$-quasi polarization.
\item We require $ι$ and $λ$ to be compatible in the sense that for $a \in O_{F}^\mfd$ we have $ \lambda^{-1} \circ \iota(a)^{\vee} \circ \lambda= \iota(\overline{a})$. We also require that for every inert place $v \nmid \mfd$ of $F/F_0$, corresponding to a prime ideal $\mfp_v \subset O_F^\mfd$, that $\ker(λ)[\mfp_v^\infty]$ is annihilated by $\mfp_v$ and of order $|L^\vee_v/L_v|$.

\item $\ov{\eta}$ is a $K_{G,\mfd}$-orbit of isometries of hermitian $F_\mfd = F\tensor_{\mbQ} \prod_{p\mid\mfd}\mbQ_p$-sheaves over $S$: 
\begin{equation*}
    \eta:V_{\mfd}(\mcA_0,\mcA)\overset{\sim}{\lr}V(F_{0,\mfd}).
\end{equation*}
Here, the notation is as follows. First,
\begin{equation*}
V_{\mfd}(\mcA_0,\mcA):=\prod_{p\mid\mfd}V_p(\mcA_0,\mcA)\quad\text{and}\quad
V_p(\mcA_0,\mcA)=\Hom_{F\otimes_{\mbQ}\mbQ_p}(V_p(\mcA_0),V_p(\mcA)),
\end{equation*}
where $V_p(\mcA_0)$ and $V_p(\mcA)$ denote the rational Tate modules at $p$; these are $\mathbb Q_p$-sheaves on $S$. Moreover, $V_\mfd(\mcA_0, \mcA)$ is made into a hermitian $F_\mfd$-vector space by
$$(x,y) := λ_0^{-1}\circ x^\vee\circ λ \circ y \in \End(V_\mfd(\mcA_0)) = F_\mfd.$$
Second, we consider
\begin{equation*}
V(F_{0,\mfd}):=F_{0,\mfd}\tensor_{F_0} V = \prod_{p\mid\mfd}\mbQ_p\otimes_{\mbQ}V
\end{equation*}
as locally constant $\prod_{p\mid\mfd}\mbQ_p$-sheaf on $S$.
\end{altenumerate}
An isomorphism between two such tuples
\begin{equation}\label{eq:iso_A_0_and_A}
(γ_0, γ):(\mcA_0, ι_0, \lambda_0,\ov{\eta}_0; \mcA, \iota, \lambda,\ov{\eta}) \simlr (\mcA_0', ι_0', \lambda_0',\ov{\eta}_0'; \mcA', \iota', \lambda',\ov{\eta}')
\end{equation}
is given by an isomorphism $γ_0:(\mcA_0, ι_0, \lambda_0,\ov{\eta}_0) \simto (\mcA_0', ι_0', \lambda_0', \ov{\eta}'')$ as in \eqref{eq:iso_A_0}, and an $O_F^\mfd$-linear $\mfd$-quasi isogeny $γ\in \Hom^\mfd(\mcA, \mcA')$ such that $γ^*(λ') = λ$ and $(γ_0, γ)^*(\ov{η}') = \ov{η}$.
\end{defn}

\begin{prop}\label{prop:representable}
$\mcM$ is representable by a flat, normal, and quasi-projective $O_E^\mfd$-scheme. The generic fiber $\mcM_E$ is isomorphic to the Shimura variety $\mr{Sh}_{\wt K}(\wt G, \{h_{\wt G}\})$. If $F_0\neq \mbQ$, then $\mcM$ is projective. 
\end{prop}
\begin{proof}
The representability, the quasi-projectivity over $O_E^\mfd$, and the fact that $\mcM_E$ is isomorphic to the Shimura variety of $(\wt G, \{h_{\wt G}\})$ for level $\wt K$ follow in the same way as \cite[Theorem 3.5]{RSZ}. Moreover, if $F_0\neq \mbQ$, then $\wt G$ is anisotropic modulo center because $V$ is definite at some archimedean place. It is well known that the Shimura variety and the PEL moduli problem for $G^\mbQ$ are then projective. In our case, $\wt G \subseteq Z^\mbQ \times G^\mbQ$ compatibly with the PEL description and hence $\mcM$ is projective as well.

The flatness and normality of $\mcM$ follow from the flatness and normality of the local models. We recall this argument from \cite[Theorem 5.2]{RSZ}. Let $ν\nmid \mfd$ be a place of $E$, and let $O_{E, ν}$ denote the completion at $ν$. Let $\mcM_{O_{E,ν}}$ be the base change along $O_E\to O_{E,ν}$. Our aim is to prove flatness and normality of $\mcM_{O_{E,ν}}$.

Denote by $v_0 \mid w_0 \mid p$ the places of $F\mid F_0\mid \mbQ$ below $ν$. Choose an embedding $α:E\to \ov{\mbQ_p}$ that induces $ν$; it gives rise to decompositions
$$A = \underset{v\mid p}{\sqcup}\ A_v,\quad B = \underset{v\mid p}{\sqcup}\ B_v$$
as in \eqref{eq:decomp_local_global}. Note that
$$v \notin \{v_0, \ov{v_0}\}\quad \Longrightarrow\quad B_v = A_v$$
while
$$B_{v_0} = \begin{cases}
    A_{v_0} \sqcup \{ϕ_0\} & \text{split case $v_0\neq \ov{v_0}$}\\[1mm]
    A_{v_0} \sqcup \{ϕ_0, \ov{ϕ_0}\} & \text{inert case $v_0 = \ov{v_0}$}.
\end{cases}$$
For each place $w\mid p$ of $F_0$, we have a local model as follows:
\begin{altitemize}
\item Assume $w$ splits in $F$. Pick one of the two places of $F$ above $w$, call it $v$; by a duality argument the choice does not matter. Let $\langle L_v\rangle = \{π_v^\mbZ L_v\}$ be the lattice chain generated by $L_v$.
\begin{itemize}
    \item[-] If $v \notin \{v_0, \ov{v_0}\}$, then $B_v = A_v$ and there is the banal local model $\bfM^{\langle L_v\rangle}_{n, A_v, \mr{banal}}$ over $O_{E_{A_v}}$ from Remark \ref{rmk:reflex}.
    \item[-] If $v \mid w_0$, then we may (by duality) assume without loss of generality that $v = v_0$. Then $B_v = A_v \sqcup \{ϕ_0\}$ and there is the local model $\bfM^{\langle L_v\rangle}_{n, A_v}(n-1)$ over $O_{ϕ_0(F_v)E_{A_v}}$ from Definition \ref{def:abs_loc_model_GL}.
\end{itemize}
\item Assume $w$ is inert in $F$, denote by $v$ its extension. The lattice $L_v$ generates a self-dual lattice chain in the sense of Definition \ref{def:self-dual-chain}. It is given by $\langle L_v\rangle = \{\pi_v^\mbZ L_v, \pi_v^\mbZ L_v^\vee\}$.
\begin{itemize}
    \item[-] If $v\neq v_0$, then there is the banal local model $\bfM^{\langle L_v\rangle}_{n, A_v, \mr{banal}}$ over $O_{E_{A_v}}$ from \eqref{eq:banal_LM_unitary}.
    \item[-] If $v = v_0$, then there is the local model $\bfM^{\langle L_v\rangle}_{n, A_v}(n-1, 1)$ over $O_{E_{A_v}E_{B_v}}$ from Definition \ref{def:abs-LM}. It coincides with the naive local model by Corollary \ref{cor:LM-comparison} (1).
\end{itemize}
\end{altitemize}
By \eqref{eq:reflex_localized}, all occurring reflex fields here are subfields of $E_ν$, and we consider the base changes to $O_{E,ν}$ of the above local models. All except the local model for the place $w_0$ are banal, so their base change is just $\Spec(O_{E,ν})$. By general theory, there is a local model diagram involving $\mcM_{O_{E,ν}}$ and the product of all local models. As just explained, only the factor for $w_0$ contributes non-trivially to the product, so we obtain a local model diagram as in \cite[(5.5)]{RSZ}:
\begin{equation}\label{equ:RSZ-LM}
\begin{aligned}
\xymatrix{
& \wt{\mcM}\ar[ld]_{\mr{pr}}\ar[rd]^{\wt{\varphi}} & \\
\mcM_{O_{E,ν}} && \mcM^{\mfa}_{0,O_{E,ν}}\times_{O_{E,ν}} \bfM_{v_0}
}
\end{aligned}
\end{equation}
with
\begin{equation}
    \bfM_{v_0} = O_{E, ν}\tensor \begin{cases} \bfM_{n,A_{v_0}}^{\langle L_{v_0}\rangle }(n-1) & \text{$w_0$ split}\\[2mm]
    \bfM_{n,A_{v_0}}^{\langle L_{v_0}\rangle }(n-1,1) & \text{$w_0$ inert}.
    \end{cases}
\end{equation}
The scheme $\wt \mcM$ here is defined as the functor that sends an $O_{E,ν}$-scheme $S$ to the set of isomorphism classes of points $(\mcA_0,\iota_{0},\lambda_0,\eta_0;\mcA,\iota,\lambda,\eta) \in \mcM(S)$ together with a trivialization of its first de Rham homology in the following sense:
\begin{itemize}
    \item If $w_0$ is split, then this trivialization is an $O_F$-linear isomorphism
    $$L_{v_0} \tensor_{O_{F, v_0}} \mcO_S \simlr \mcD(\mcA)_{v_0}.$$
    \item If $w_0$ is inert, then this trivialization is an $O_F$-linear isometry from the polarized lattice chain
    $$L_{v_0}\tensor_{O_{F,v_0}} \mcO_S \lr L_{v_0}^\vee\tensor_{O_{F,v_0}} \mcO_S \lr (π_{v_0}^{-1}L_{v_0})\tensor_{O_{F, v_0}} \mcO_S$$
    to the polarized chain
    $$\mcD(\mcA)_{v_0} \overset{\lambda}{\lr} \mcD(\mcA^\vee)_{v_0}\overset{\lambda^\vee}{\lr}\mcD(\mcA)_{v_0}.$$
\end{itemize}
The morphism $\mr{pr}$ in \eqref{equ:RSZ-LM} is the forgetful map while $\wt{\varphi}$ maps a point to its Hodge filtration. They are both smooth and of the same relative dimension, see \cite[Theorem 5.2]{RSZ} for a detailed discussion. It follows that the local rings of $\mcM_{O_{E, ν}}$ have the same flatness and normality properties as the local rings of $\bfM_{v_0}$. By Theorem \ref{thm:local_model_comparison} for the split case, or Corollary \ref{cor:LM-comparison} (1) for the inert case, the local model $\bfM_{v_0}$ is flat and normal, so the proof is complete.
\end{proof}

\begin{rmk}
The above moduli description can be extended to ramified places $v$ of $F$ where the localization $L_v$ is a vertex lattice. In this case, one needs to add the sign condition from \cite[Definition 4.1]{RSZ} to the moduli problem. One also needs to impose stronger signature conditions on the Lie algebras of $(\mcA_0, ι_0, \lambda_0,\ov{\eta}_0; \mcA, \iota, \lambda,\ov{\eta})$; we refer to \cite[\S 1.4]{Luo} for a summary on the second point.
\end{rmk}

\subsection{Basic uniformization}\label{sec:uniformofSV}

We now formulate the basic uniformization result for $\mcM$. This statement and its proof are relatively standard and taken from \cite{KR-global,liu2021fourier}; our own addition is that we can replace the occurring absolute RZ space $\mcN_A^{[t]}$ by the relative RZ space $\mcN^{[t]}$, thanks to Corollary \ref{cor:unitary-comp}. Unlike the basic uniformization in \cite{PR-Shtuka-I}, which is given by a group-theoretical construction, our approach is given explicit through a moduli description. This explicit description makes it clear that the construction also allows to uniformize special divisors and CM cycles in \S \ref{sec:algebraiccycles}.

\begin{assumption}\label{assumption:pos_def}
We make the following assumption throughout the rest of the article. We assume that $V$ is positive definite at all places except the one of signature $(n-1,1)$. There are two reasons for this. First, it conforms with the setting in \cite{liu2021fourier} which is our reference for the uniformization statement below. Second, it is needed to define KR divisors in \S\ref{sec:algebraiccycles} below and, in particular, underlies the whole global intersection theory setting.
\end{assumption}

Let $\nu \nmid \mfd$ be a $p$-adic place of $E$ such that the restriction $v = ν\vert_{F_0}$ is inert in $F$. Fix an embedding $E\hookrightarrow \ov{\mbQ}_p$ to obtain decompositions $A = \sqcup_{w\mid p} A_w$ and $B = \sqcup_{w\mid p} B_w$ as in \eqref{eq:decomp_local_global}. Here, the decomposition is over all places $w\mid p$ of $F$. Then $E_ν$ contains the local reflex field $E_{A_v}ϕ_0(F_v)$, and hence
$$(K = \mbQ_p, F_v/F_{0,v}, E_ν)$$
are local fields as in the setting for the definition of unitary RZ spaces in \eqref{eq:setting_unitary_RZ}. (The only small difference is that $E_ν$ contains the local reflex field from there but might be larger.) Let $t$ be the type of the vertex lattice $L_v$. Let again $\mbF$ denote the residue field of $\breve E_ν$. Fix a basic geometric point
$$(\mcA_0, ι_0, λ_0, \ov{η}_0; \mcA, ι, λ, \ov{η}) \in \mcM_v(\mbF).$$
The $p$-divisible group $\mcA[p^\infty]$ decomposes under the action of $(O_F\tensor_{\mbZ} \mbZ_p) = \prod_{w\mid p} O_{F,w}$. Let\footnote{We write ``$A$'' in the index for readability. In order to completely conform with the notation from \S\ref{sec:cycles on absolute RZ}, one should write $A_v$.} $(\mbX_A, ι_{\mbX_A}, λ_{\mbX_A})$ be the $v$-component of $(\mcA, ι, λ)$; set $(\mbX, ι, λ) = Φ_A(\mbX_A, ι_{\mbX_A}, λ_{\mbX_A})$. We use these framing objects to define the RZ spaces from \eqref{eq:def_RZ_AB_PEL} and \eqref{eq:def_RZ_PEL}:
\begin{equation}\label{eq:RZ_again_A}
\mcN_A^{[t]} \lr \Spf(O_{\breve E_ν}),\qquad \mcN^{[t]} \lr \Spf(O_{\breve F_v}).
\end{equation}
The comparison isomorphism in this situation (Corollary \ref{cor:unitary-comp}) reads
$$
\Phi_A: \mcN^{[t]}_{A}\simlr O_{\breve{E}_ν} \wh{\otimes}_{O_{\breve{F}_v}} \mcN^{[t]}.
$$
Denote by $\mathcal M^{\wedge}$ the formal completion of $O_{\breve E_ν}\tensor_{O_E}\mathcal M$ along the basic locus of the special fiber. Let $V^{(v)}$ be the nearby hermitian $F$-vector space of $V$ at $v$; that is, $V^{(v)}$ is positive definite at all archimedean places and isomorphic to $V$ after completion at all finite places $\neq v$. Consider the reductive groups
$$
G^{(v)}:=\Res_{F_0/\mathbb Q} \mathrm{U}(V^{(v)}), \quad \widetilde{G}^{(v)}:=Z^{\mathbb Q} \times \Res_{F_0/\mathbb Q} \mathrm{U}(V^{(v)}). 
$$

\begin{thm}[Basic uniformization] \label{thm: basic uniformization RSZ}
There is a natural isomorphism of formal schemes over $O_{\breve E_\nu}$:
\begin{equation}\label{eq: basic basic unif}
\Theta: \, \, \wt{G}^{(v)}(\mathbb{Q}) \backslash ( \mcN^{\prime} \times \wt{G} (\mathbb{A}_{f}^{p}) / \wt{K}^{p} )\overset{\sim}{\lr} \mathcal M^\wedge, 
\end{equation}  
where 
\[ \mcN^{\prime} \simeq (Z^{\mathbb Q} \left(\mathbb{Q}_{p}\right) / K_{Z^{\mathbb Q}, p} ) \times \bigl(O_{\breve{E}_\nu}\wh{\otimes}_{O_{\breve{F}_v}}\mcN^{[t]}\bigr) \times \prod_{w|p, w \not = v} \U(V)(F_{0, w}) / K_{w}. \]
Moreover, there is a natural projection map
\begin{equation}\label{projection map from RSZ to unitary}
\mathcal M^\wedge
\lr Z^{\mathbb Q}(\mathbb Q)  \backslash (Z^{\mathbb Q}(\mathbb A_{0, f}) / K_{Z^{\mathbb Q}} ) 
\end{equation}  
whose fibers $\mathcal{M}^{\wedge}_{0}$ are described by
\begin{equation}\label{eq:Theta_0}
\Theta_0: \mathrm{U}(V^{(v)})(F_0) \backslash [ \bigl(O_{\breve{E}_\nu}\wh{\otimes}_{O_{\breve{F}_{v}}}\mcN^{[t]}\bigr) \times \mathrm{U}(V)(\mathbb{A}_{0, f}^{v}) / K^{v} ]\overset{\sim}{\lr}\mathcal{M}^{\wedge}_{0}.
\end{equation}
\end{thm}
\begin{proof}
This uniformization map was constructed in \cite[Proposition C. 26]{liu2021fourier}, where after all efforts one gets an isomorphism
$$
\Theta_A:\wt{G}^{(v)}(\mbQ)\backslash (\mcN'_A\times \wt{G}(\mbA_f^p)/\wt{K}^p)\simlr \mcM^\wedge,
$$
where 
\[ \mcN_A^{\prime} \simeq (Z^{\mathbb Q} \left(\mathbb{Q}_{p}\right) / K_{Z^{\mathbb Q}, p} ) \times \mcN_{A}^{[t]} \times \prod_{w|p, w \not = v} \U(V)(F_{0, w}) / K_{w}. \]
Here, $\mcN_{A}^{[t]}$ is as in \eqref{eq:RZ_again_A}. Applying the comparison isomorphism $\Phi_A$ finishes the proof.
\end{proof}

\part{Arithmetic transfer}

\section{Kottwitz--Rapoport strata}\label{ss:KR_strata}
In this section, we study the (closed) Kottwitz--Rapoport strata of the special fiber of unitary Shimura varieties and Rapoport--Zink spaces. These strata relate to closed orbits in the associated local model, which was originally studied by Kottwitz and Rapoport \cite{Kottwitz-Rapoport-AdmPerm}, which justify the terminology adopted here.
The study of these strata will play an important role in the construction of the (derived) CM cycles (\S \ref{sec:algebraiccycles}), the almost intersection theory on Shimura variety (\S \ref{sec:arithPic}), and in the proof of the almost modularity of generating series of special divisors (\S \ref{sec:AT proof}).

\subsection{Kottwitz--Rapoport strata}\label{sec:KR-strata}
We use the notations from \S \ref{sec:Shimura}.
Let $S$ be an $O_E^\mfd$-scheme, and let $\mcA/S$ be an abelian scheme together with an action $ι:O_F^\mfd\to \End^\mfd(\mcA)$ of signature $(n-1,1)$. The Eisenstein condition ensures that $J_A\cdot\Lie(\mcA)$ is a locally free $\mcO_S$-module of rank $n$, on which the ring $O^\mfd_F$ acts $O^\mfd_{F_0}$-linearly. Recall that $\mfd$ was chosen such that $O_F^\mfd/O_{F_0}^\mfd$ is unramified. So the $O_F^\mfd$-action induces a splitting 
\begin{equation}\label{equ:Lie-split}
    J_A\cdot \Lie(\mcA) = (J_A\cdot \Lie(\mcA))_0 \ \oplus\ (J_A\cdot \Lie(\mcA))_1
\end{equation}
into $O_F^\mfd$-eigenspaces. That is, $O^\mfd_F$ acts on $(J_A\cdot \Lie(\mcA))_0$ via the structure morphism and on $(J_A\cdot\Lie(\mcA))_1$ via its conjugate. Moreover, by the signature condition, the two summands are of rank $n-1$ and $1$, respectively.
\begin{defn}\label{def:lb_mod_forms}
The \emph{line bundle of modular forms} of $\mcA$ is defined as
$$\omega_{\mcA} := (J_A\cdot\Lie(\mcA))_1.$$
This definition is functorial for $O_F^\mfd$-linear maps of abelian varieties. So given a point $(\mcA_0, ι_0, λ_0, \ov{η}_0;$ $\mcA, ι, λ, \ov{η}) \in \mcM(S)$ we may in particular consider the map on line bundles induced by $λ$:
$$ω_λ:ω_{\mcA} \lr ω_{\mcA^\vee}.$$
Let $\mcM^\mcZ\subseteq \mcM$ be the closed subscheme defined by the vanishing of $ω_λ$. Let $\mcM^\mcZ_\red$ be its underlying reduced subscheme. We call $\mcM^\mcZ$ the \emph{(Kottwitz--Rapoport) $\mcZ$-divisor} and $\mcM^\mcZ_\red$ the \emph{(Kottwitz--Rapoport) $\mcZ$-stratum}. 
\end{defn}

\begin{rmk}\label{rmk:correctlinebundle}
The above terminology is taken from \cite[\S 2]{BHKRY} and \cite[\S 4]{Howard-Linearinv}, but our definition is slightly different from theirs. Definition \ref{def:line bundle RZ} is sufficient for the study of Kottwitz--Rapoport strata, but to study global deformations of Kudla--Rapoport special divisors, one should twist it in the following way. Recall that Assumption \ref{assumption:pos_def} is in place and hence $Φ = A\cup \{ϕ_0\}$. So for any $S$-point $(\mcA_0,\mcA)$ in $\mcM$, the following two are line bundles:
\begin{equation*}
\Omega_{\mcA}:=\mcH\!om_{O_F\tensor \mcO_S}(\Fil(\mcA_0),(J_A\cdot\Lie(\mcA))_1),\quad
\Omega_{\mcA^\vee}:=\mcH\!om_{O_F\tensor \mcO_S}(\Fil(\mcA_0),(J_A\cdot\Lie(\mcA^\vee))_1).
\end{equation*}
These directly generalize the line bundles from \cite[\S 2]{BHKRY} and \cite[\S 4]{Howard-Linearinv}.
\end{rmk}

Note that $ω_λ$ is an isomorphism above all places where $L$ is self-dual, so $\mcM^\mcZ$ is supported above the inert primes $v$ of $F$ such that $L_v\neq L^\vee_v$. Let $ν\mid v$ be a finite place of $E$ above such a place $v$. By condition (5) in Definition \ref{def:RSZ-integral}, the map $ω_λ$ divides $π_v$. It follows that $\mcM^\mcZ$ (resp. $\mcM^\mcZ_\nu$) is a Cartier divisor. We denote by $\mcM^{\mcZ}_\nu$ the open subscheme of $\mcM^\mcZ$ supported above $\nu$.

\begin{defn}\label{def:KR_strata}
The \emph{(Kottwitz--Rapoport) $\mcY$-divisor $\mcM^\mcY_ν\subseteq \mcM_{\mcO_{E, ν}}$} above $ν$ is defined as the vanishing locus $V(π_vω_λ^{-1})$. Its maximal reduced subscheme $\mcM^\mcY_{ν,\red}$ is called the \emph{(Kottwitz--Rapoport) $\mcY$-stratum} above $ν$.

We define $\mcM^{\mcZ\cap \mcY}_ν$ as the intersection $\mcM^\mcZ\cap \mcM^\mcY_ν$. The reduced locus $\mcM^{\mcZ\cap \mcY}_{ν,\red}$ is called the \emph{linking stratum} above $ν$.
\end{defn}
\begin{prop}\label{prop:global-KR-Cartier}
(1) The strata $\mcM^\mcZ_{ν,\red}$ and $\mcM^\mcY_{ν,\red}$ are smooth over $\Spec \mbF_ν$.

\medskip\noindent (2) The structure map $\mcM_{O_{E, ν}} \to \Spec O_{E, ν}$ is smooth outside of $\mcM^{\mcZ\cap \mcY}_{ν,\red}$.
\end{prop}
\begin{proof}
This follows by comparison with the local model. Let us first consider the relative local model $\bfM^{\langle L_v\rangle}_n(n-1,1)$ over $O_{F,v}$. (That is, if $n = 2$ in which case the reflex field is $F_{0,v}$, then we extend scalars to $F_v$.) We have a splitting of the form $F_v\tensor_{F_{0,v}} V = V_0\oplus V_1$, where $F_v$ acts on $V_0$ by the identity, and on $V_1$ by conjugate identity. Since $F_v/F_{0,v}$ is unramified in our current setting, for any lattice $\Lambda \subset V_v$, we get an analogous splitting
$$
\Lambda \otimes_{O_{F_0,v}}O_{F,v}=\Lambda_0\oplus \Lambda_1
$$
with $\Lambda_i\subset V_i$. Given an $S$-valued point $(\msF_Λ)_Λ \in \bfM^{\langle L_v\rangle}_n(n-1, 1)(S)$, each filtration $\msF_Λ$ similarly splits into
\begin{equation*}
    \msF_{Λ,i} = \msF_{Λ,i}\oplus \msF_{Λ,1}
\end{equation*}
with $\msF_{Λ,i}\subseteq \mcO_S\tensor_{O_{F,v}}\Lambda_i$. By the signature condition, each quotient $\msL_{Λ,1} = (\mcO_S\tensor_{O_{F,v}} Λ)/\msF_{Λ,1}$ is a line bundle on $S$. This induces a map
\begin{equation*}
    \bfM_n^{\langle L_v\rangle }(n-1,1)\lr \bfM^{\langle L_{v}\rangle_1}_n(1),\quad \msF_i\longmapsto \msF_{i,1}.
\end{equation*}
It is clear from definitions that this is an isomorphism; or see \cite[Proposition 2.14]{HPR} for a general statement.

In our setting, $\langle L_v\rangle$ is generated by the two lattices $L_v\subsetneq L_v^\vee$. So $\langle L_v\rangle_1$ is generated by the two lattices $(L_v)_1$ and $(L^\vee_v)_1$. The respective vanishing loci of the two maps
\begin{equation}\label{eq:line_bundle_maps}
\msL_{(L_v)_1} \lr \msL_{(L^\vee_v)_1} \lr \msL_{π_v^{-1}(L_v)_1}
\end{equation}
define Cartier divisors $\bfM^\mcZ$ and $\bfM^\mcY$ on $\bfM_n^{\langle L_v\rangle_1}(1)$ whose sum equals the special fiber $V(\pi_v)$. The strict complete local rings of $\bfM_n^{\langle L_v\rangle_1}(1)$ in closed points are isomorphic to
\begin{equation}\label{eq:complete_rings_local_model}
\begin{cases}
\breve O_{F,v}[\![x_1,\ldots,x_{n-1}]\!] & \text{(smooth points)}\\[1mm]
\breve O_{F,v}[\![x_0,\ldots,x_{n-1}]\!]/(x_0x_1 - \pi_v) & \text{(singular points).}
\end{cases}
\end{equation}
The smooth case occurs precisely when one of the two maps in \eqref{eq:line_bundle_maps} is an isomorphism near the considered point. In this case, near that point, one out of $\{\bfM^\mcZ, \bfM^\mcY\}$ agrees with the special fiber while the other divisor is empty. Near a singular point and in terms of the coordinates from \eqref{eq:complete_rings_local_model}, one of the divisors equals $V(x_0)$ and the other $V(x_1)$. In summary, we see that $\bfM^\mcZ$ and $\bfM^\mcY$ are smooth over $\Spec(\mbF_v)$ and that $\bfM^{\langle L_v\rangle_1}_n(1) \to \Spec O_{F,v}$ is smooth outside of $\bfM^\mcZ\cap \bfM^\mcZ$.

The situation of the absolute local model (base changed to $O_{E_{A_v}ϕ_0(F_v)}$ if $n = 2$) is completely analogous: For each $Λ\in \langle L_v\rangle$ we have a splitting for the relative filtration quotients
\begin{equation*}
J_A\mcL_{Λ} \simlr \msL_Λ=\msL_{Λ,0}\oplus\msL_{Λ,1}.
\end{equation*}
Together with the comparison isomorphisms from Theorems \ref{thm:local_model_comparison} and \ref{thm:LM-comparison}, this induces a commutative diagram
\begin{equation*}
\begin{aligned}
\xymatrix{
\bfM_{n,A_v}^{\langle L_v\rangle}(n-1,1)\ar[rr]^-{\sim}\ar[d]_{\cong} && O_{E_{A_v}ϕ_0(F_v)}\tensor_{O_{F,v}} \bfM_{n}^{\langle L_{v}\rangle_1}(1)\ar[d]^{\cong}\\
\bfM_{n,A_v}^{\langle L_v\rangle} (n-1,1)\ar[rr]^-{\sim} && O_{E_{A_v}ϕ_0(F_v)}\tensor_{O_{F,v}}\bfM_{n}^{\langle L_v\rangle_1}(1).
}
\end{aligned}
\end{equation*}
Consider the vanishing loci $\bfM_A^\mcZ$ and $\bfM_A^\mcY$ of the two maps
\begin{equation}\label{eq:line_bundle_maps_absolute}
J_A \mcL_{(L_v)_1} \lr J_A \mcL_{(L_v^\vee)_1} \lr \mcL_{(π_v^{-1}L_v)_1}.
\end{equation}
These are the base changes of $\bfM^\mcZ$ and $\bfM^\mcY$ to $O_{E_{A_v}ϕ_0(F_v)}$. In particular, the reduced loci $\bfM^\mcZ_{A,\red}$ and $\bfM^\mcY_{A, \red}$ are smooth over $\mbF_ν$, and $\bfM_{n,A}^{\langle L_v\rangle_1}(1)$ is smooth over $\Spec(O_{E_{A_v}ϕ_0(F_v)})$ outside of $\bfM^\mcZ_{A, \red}\cap \bfM^\mcY_{A, \red}$. If we extend scalars further to $O_{E,ν}$, then we still have that
\begin{equation}\label{eq:reduced_base_change}
(O_{E,ν}\tensor_{O_{F,v}} \bfM^\mcZ)_\red = \mbF_ν \tensor_{\mbF_v} \bfM^\mcZ_\red
\end{equation}
is smooth over $\mbF_ν$; analogously for $(O_{E,ν}\tensor_{O_{F,v}} \bfM^\mcY)_\red$.

The various definitions of $\mcZ$-divisors and $\mcY$-divisors are compatible in terms of the local model diagram \eqref{equ:RSZ-LM} in the sense that $\mcM^\mcZ_ν$ and $\mcM^\mcY_ν$ are (the descent along ``$\mr{pr}$'' of) the pullbacks of $\bfM^\mcZ_A$ and $\bfM^\mcY_A$. In this way, the claimed properties for $\mcM^\mcZ$ and $\mcM^\mcY$ follow from those we just discussed for local models.
\end{proof}
\begin{rmk}\label{rmk:non-regular}
Concretely, the above arguments show that the situation for the strict complete local rings in closed points of $\mcM$ above $ν$ is as follows: These rings are isomorphic to
\begin{equation}\label{eq:complete_local_ring_M}
\begin{cases}
\breve O_{E, ν}[\![x_1,\ldots,x_{n-1}]\!] & \text{(smooth points)}\\[1mm]
\breve O_{E, ν}[\![x_0,\ldots,x_{n-1}]\!]/(x_0x_1 - π_ν^{e(ν\mid v)}) & \text{(singular points).}
\end{cases}
\end{equation}
(Here, $e(ν\mid v)$ denotes the ramification index of $ν$ above $v$.) Near smooth points, one out of $\{\mcM^\mcZ, \mcM^\mcY\}$ will be empty, the other equal to $V(π_ν^{e(ν\mid v)})$. Near singular points and in terms of the above coordinates, one of them equals $V(x_0)$ and the other $V(x_1)$. In particular, we also see that the reduced loci $\mcM^\mcZ_{ν,\mr{red}}$ and $\mcM^\mcY_{ν, \mr{red}}$ that occurred in Proposition \ref{prop:global-KR-Cartier} agree with the fibers of $\mcM^\mcZ_ν$ and $\mcM^\mcY_ν$ over $\Spec \mbF_ν$.
\end{rmk}

\subsection{Local Kottwitz--Rapoport strata}\label{sec:local KR strata}
In this subsection, we use our local notation from \S\ref{sec:uni-RZ}: Let $K$ be a $p$-adic local field, let $F_0/K$ be a finite extension, and let $F/F_0$ be an unramified quadratic field extension with Galois conjugation $a\mapsto \ov{a}$. Denote by $O_K\subseteq O_{F_0} \subset O_F$ their rings of integers. Let $π_F \in O_F$ denote a uniformizer. Fix a subset $A\subset \Hom_K(F,\ov{K})$ and an element $ϕ_0\in \Hom_K(F, \ov K)$ such that 
$$
\Hom_K(F,\ov{K})=A\sqcup \ov{A}\sqcup \{\phi_0,\ov{\phi_0}\}.
$$
Set $B=A\sqcup \{\phi_0,\ov{\phi_0}\}$. Let $E$ be a finite extension of the reflex field $E_Aϕ_0(F)$ and let $\mbF$ be the residue field of the completion $\breve E$ of a maximal unramified extension of $E$. Let $(\mbX_A,\iota_{\mbX_A},\lambda_{\mbX_A})$ be an $(A,B)$-strict pair of signature $(n-1,1)$ over $\mbF$ with polarization of type $t$; see Definition \ref{def:AB_strict_PEL}. We are mainly interested in applications to the basic locus of the Shimura variety $\mcM$, so we also assume that $(\mbX_A, ι_{\mbX_A}, λ_{\mbX_A})$ is basic. A simple way to phrase this is to say that $\mbX_A$ is isoclinic as $p$-divisible group. Using $(\mbX_A, ι_{\mbX_A}, λ_{\mbX_A})$ as framing object, we consider the absolute RZ space defined by \eqref{eq:def_RZ_AB_PEL}:
$$
\mcN_{A}^{[t]}:=\mcN_{A,(n-1,1)}^{[t]} \lr \Spf O_{\breve{E}}.
$$
We now introduce \emph{Kottwitz--Rapoport strata} on $\mcN_{A}^{[t]}$ as a local analogue of the definitions in \S \ref{sec:KR-strata}. These are called the \emph{formal balloon and ground strata} in \cite[\S7]{ZZhang21}.

Let $S$ be an $\Spf O_{\breve{E}}$-scheme and let $(X_A,\iota,\lambda,\rho)$ be a $S$-point of $\mcN_{A}^{[t]}$. 
The Eisenstein condition ensures that $J_A\cdot \Lie(X)$ is a locally free $\mcO_S$-module of rank $n$, on which the ring $O_F$ acts $O_{F_0}$-linearly. Since $F/F_0$ is unramified, the $O_F$-action induces a splitting 
$$
J_A\cdot \Lie(X_A)=(J_A\cdot \Lie(X_A))_0\ \oplus \ (J_A\cdot \Lie(X_A))_1
$$
into $O_F$-eigenspaces where $O_F$ acts on $(J_A\cdot \Lie(X_A))_0$ via the structure morphism and on $(J_A\cdot \Lie(X_A))_1$ via its conjugate. Moreover, by the signature condition, the two summands are of rank $n-1$ and $1$, respectively.

\begin{defn}\label{def:line bundle RZ}
The \emph{line bundle of modular forms} of $X_A$ is defined as
$$
\omega_{X_A} := (J_A\cdot \Lie(X_A))_1.
$$
\end{defn}
Recall that $λ$ being of type $t$ entails that $π_Fλ^{-1}:X_A^\vee\to X_A$ is a homomorphism. The definition of $ω_{X_A}$ is functorial, so we may consider the two maps
$$
\omega_\lambda:\omega_{X_A}\lr \omega_{X_A^\vee}\quad\text{resp.}\quad\omega_{π_F\lambda^{-1}}:\omega_{X_A^\vee}\lr \omega_{X_A}.
$$

\begin{defn}\label{def:KR strata RZ}
\begin{altenumerate}
\item The \emph{Kottwitz--Rapoport $\mcZ$-divisor} $\mcN_{A}^{[t],\mcZ}\subset \mcN_{A}^{[t]}$ is defined as the vanishing locus of $\omega_{\lambda}$. Its special fiber $\mcN_{A,\mbF}^{[t], \mcZ}$ is called the \emph{Kottwitz--Rapoport $\mcZ$-stratum}.\footnote{It could equivalently be defined as the maximal closed formal subscheme that is formally reduced. However, in order to prevent confusion with the maximal reduced closed \emph{subscheme}, we avoid this terminology.}

\item Similarly, the \emph{Kottwitz--Rapoport $\mcY$-divisor} $\mcN_{A}^{[t],\mcY}\subset \mcN_{A}^{[t]}$ is defined as the vanishing locus of $\omega_{π_F\lambda^{-1}}$. Equivalently, it can be defined as the vanishing locus $V(\pi_F \omega_{\lambda}^{-1})$ as in Definition \ref{def:KR_strata}. Its special fiber $\mcN^{[t], \mcY}_{A,\mbF}$ is called the \emph{Kottwitz--Rapoport $\mcY$-stratum}.

\item We define $\mcN_{A}^{[t],\mcZ\cap \mcY}$ as the intersection $\mcN_{A}^{[t],\mcZ}\cap\mcN_{A}^{[t],\mcY}$. Its special fiber $\mcN_{A,\mbF}^{[t],\mcZ\cap \mcY}$ is called the \emph{linking stratum}.
\end{altenumerate}
\end{defn}
Since $(π_F\lambda^{-1})\circ\lambda=[\pi_F]$, we obtain stratifications of the special fiber (a decomposition as sum of effective Cartier divisors) and of the maximal reduced closed subscheme (a set-theoretic union):
\begin{equation}\label{eq:stratif_RZ}
\mcN^{[t]}_{A,\mbF} = \mcN^{[t],\mcZ}_{A,\mbF} + \mcN^{[t],\mcY}_{A,\mbF},\qquad \mcN_{A,\red}^{[t]}=\mcN_{A,\red}^{[t],\mcZ}\cup \mcN_{A,\red}^{[t],\mcY}.
\end{equation}
We now apply the equivalence $Φ_A$ from Theorem \ref{thm:AB_strict_p_div_PEL} to define
$$(\mbX, ι_\mbX, λ_\mbX) := Φ_A(\mbX_A, ι_{\mbX_A}, λ_{\mbX_A}).$$
Recall that $\mbX_A$ is isoclinic which also implies that $\mbX$ is isoclinic. More precisely, since the $O_K$-height of $\mbX_A$ is twice its dimension by definition, $\mbX_A$ is isoclinic of $O_K$-slope $1/2$. In the same way, the $O_{F_0}$-height of $\mbX$ is twice its dimension, so it is isoclinic of $O_{F_0}$-slope $1/2$. In particular, Theorem \ref{thm:AB_strict_p_div} applies and $(\mbX, ι_\mbX, λ_\mbX)$ is a polarized triple of signature $(n-1,1)$ and of type $t$ for $F/F_0$. Let
$$\mcN^{[t]} := \mcN^{[t]}_{(n-1,1)} \lr \Spf(O_{\breve F})$$
be the attached relative RZ space from \eqref{eq:def_RZ_PEL}. As in Corollary \eqref{cor:unitary-comp} the equivalence $Φ_A$ provides an $\mcQ \mcI sog(\mbX_A, ι_{\mbX_A}, λ_{\mbX_A})$-equivariant isomorphism
\begin{equation}\label{eq:comp_iso_recap}
Φ_A:\mcN^{[t]}_{A,(n-1,1)} \simlr O_{\breve E} \widehat{\tensor}_{O_{\breve F}} \mcN^{[t]}_{(n-1,1)}.
\end{equation}
Observe that the definition of $\mcN^{[t]}$ is actually the special case $K = F_0$ of the definition of $\mcN^{[t]}_A$. In particular, Definition \ref{def:KR strata RZ} applies and defines Kottwitz--Rapoport divisors and strata on $\mcN^{[t]}$.

\begin{prop}\label{prop:comp-KR-divisors}
For every relatively biformal, $(A,B)$-strict triple $(X, ι, λ)$ of signature $(n-1,1)$, the modification functor $\Phi_A$ gives rise to a commutative diagram
$$\xymatrix{
ω_{X_A} \ar[rr]^-{ω_λ} \ar[d]^-{\cong} && ω_{X_A^\vee} \ar[d]^-{\cong} \ar[rr]^-{π_Fω_λ^{-1}} && ω_{X_A} \ar[d]^-{\cong}\\
ω_{Φ_A(X_A)} \ar[rr]^-{ω_{Φ_A(λ)}} && ω_{Φ_A(X_A)^\vee} \ar[rr]^-{π_Fω_{Φ_A(λ)^{-1}}} && ω_{Φ_A(X_A)}.
}$$
In particular, it defines isomorphisms
$$
\Phi_A:\mcN_A^{[t],\Box}\simlr O_{\breve{E}}\widehat{\tensor}_{O_{\breve F}} \mcN_{(n-1,1)}^{[t],\Box},\quad \Box\in\{\mcZ,\mcY,\mcZ\cap\mcY\}.
$$
\end{prop}
\begin{proof}
By the construction of the modification functor $\Phi_A$ in the proof of Theorem \ref{thm:AB_equiv_displays}, more precisely its effect on the Hodge filtration which is described in \eqref{eq:bij-e_S}, for any polarized $(A,B)$-strict triple $(X_A, ι, λ)$, we have isomorphisms of $O_F\tensor \mcO_S$-modules
$$
J_A\cdot\Lie(X_A)\cong \Lie(\Phi_A(X_A)),\quad J_A\cdot\Lie(X_A^\vee)\cong \Lie(\Phi_A(X_A^\vee)).
$$
These are functorial and respect maps induced by polarizations by Theorem \ref{thm:comp-display}. The assertions then follow from the way we defined the Kottwitz--Rapoport divisors.
\end{proof}

\begin{prop}\label{prop:KR-strata}
The Kottwitz--Rapoport divisors $\mcN_{A}^{[t],\mcZ}$ and $\mcN_{A}^{[t],\mcY}$ are Cartier divisors on $\mcN_A^{[t]}$. Their special fibers $\mcN_{A,\mbF}^{[t],\mcZ}$ and $\mcN_{A,\mbF}^{[t],\mcY}$ are formally smooth over $\mbF$. The structure map $\mcN_A^{[t]}\to \Spf(O_{\breve E})$ is formally smooth outside of the linking locus $\mcN_A^{\mcZ\cap\mcY}$.
\end{prop}
\begin{proof}
This can be proved by the same arguments as for Proposition \ref{prop:global-KR-Cartier}. Alternatively, the statements follow by using the comparison isomorphism \eqref{eq:comp_iso_recap}, its compatibility with the definition of Kottwitz--Rapoport divisors from Proposition \ref{prop:comp-KR-divisors}, and the known results about $\mcN^{[t]}$ from \cite[Proposition 5.9]{ZZhang21}.
\end{proof}

\begin{rmk}
The isomorphism in \eqref{eq:comp_iso_recap} also identifies the maximal reduced subschemes $\mcN_{A,\red}^{[t]}$ and $\mcN^{[t]}_\red$. In this way, all definitions and statements related to the Bruhat--Tits stratification on $\mcN^{[t]}$ carry over to $\mcN_A^{[t]}$.
\end{rmk}

\section{Algebraic cycles}\label{sec:algebraiccycles}
In this section, we study special divisors (in the sense of Kudla--Rapoport \cite{KR-global}) and (derived) CM cycles (in the sense of W. Zhang \cite{Zhang21}) on unitary Shimura varieties and Rapoport--Zink spaces. Since our integral models need not be regular in the presence of ramification, see \eqref{eq:complete_local_ring_M}, we provide an additional argument to define their intersection numbers. These numbers will later make up the geometric side of the arithmetic transfer theorem in \S \ref{sec:AT-conj}. 

\subsection{Cycles on the Shimura varietie}\label{sec:global_intersection_numbers}
We begin in the global setting and resume the notations from \S \ref{sec:Shimura}. Recall that $V$ is assumed positive definite at all archimedean places except $\{ϕ_0, \ov{ϕ_0}\}$ (Assumption \ref{assumption:pos_def}).

\begin{defn}\label{defn:hecke}
\begin{altenumerate}
\item Let $μ = \mu_{G,\mfd}\in K_{G,\mfd}\backslash G(F_{0,\mfd})\slash K_{G,\mfd}$ be a double coset. The \emph{Hecke correspondence} $\Hk_\mu$ defined by $\mu$ is the scheme representing the following functor. For an $O_{E}^{\mfd}$-scheme $S$, the set $\Hk_μ(S)$ is the set of isomorphism classes in
\begin{equation}\label{eq:def_Hecke}
\left\{(\mcA_0,\mcA,\ov{η}_\mcA,\mcB,\ov{η}_\mcB,φ) \left\vert
\begin{array}{c}\text{$(\mcA_0,\mcA,\ov{η}_\mcA),\,(\mcA_0,\mcB,\ov{η}_\mcB)\in \mcM(S)$ and}\\
φ\in \Hom_{O_F}^\mfd(\mcA,\mcB)\text{ s.t.}\\
\text{$φ^*(λ_\mcB) = λ_\mcA$ and $η_\mcB\circ φ\circ η_\mcA^{-1} \in µ$ for $η_\mcA\in \ov{η}_\mcA$, $η_\mcB\in \ov{η}_\mcB$}\end{array}\right\}\right..
\end{equation}
Note that for every tuple in this set, the degree $φ$ is $1$ because of the compatibility with $(\ov{η}_\mcA, \ov{η}_\mcB)$. If $μ = K_{G,\mfd}$, then we by definition recover the diagonal $Δ_\mcM$. Moreover, it is clear that $\Hk_μ$ is representable and that the two projections $\Hk_μ\to \mcM$ are finite and étale.

\item All fiber products in the following are taken over $O_E^\mfd$. Consider the blow up 
\[
\mathrm{Bl}: \wt{\mathcal M \times \mathcal M} \lr \mathcal M \times \mathcal M
\]
of $\mcM\times \mcM$ along the closed subscheme $\mcM^\mcZ\times \mcM^\mcZ$. If non-empty, this subscheme is purely of dimension $n-1$ by Proposition \ref{prop:global-KR-Cartier}. It is a Cartier divisor only over those inert places $v$ of $F$ with $L_v^\vee = π_v^{-1}L_v$; over these $v$, the blow-up is an isomorphism. Over those places $v$ with $L_v\subsetneq L_v^\vee \subsetneq π_v^{-1}L_v$, however, $\mcM^\mcZ\times \mcM^\mcZ$ is only a Weil divisor and the blow-up operation is non-trivial.

\item Consider a Hecke correspondence $\Hk_μ\to \mcM\times\mcM$ as in (1). The fiber product
$$\Hk_μ \times_{(\mcM\times \mcM)} (\mcM^\mcZ \times \mcM^\mcZ)$$
is a Cartier divisor on $\Hk_μ$. (This is clear if $\Hk_μ = Δ_{\mcM}$ for then this fiber product recovers $\mcM^\mcZ \subset \mcM$. The general case is proved by using that any $φ$ in \eqref{eq:def_Hecke} is a $\mfd$-quasi isogeny and, in particular, induces an isomorphism $ω_A\simto ω_B$.) By the universal property of blow ups, there hence exists a unique lifting in the diagram
\begin{equation}
\xymatrix{
&& \wt{\mcM\times \mcM} \ar[d]\\
\Hk_μ \ar[rr] \ar@{-->}[rru] && \mcM \times \mcM.
}
\end{equation}
We write
$$\wt{\Hk_μ} \lr \wt{\mcM\times\mcM}$$
for the resulting map. In particular, this defines the strict transform $\wt{\Delta_\mcM}$ of the diagonal.
\end{altenumerate}
\end{defn}

The following theorem will only be used for $\wt{\Delta_\mcM}$.

\begin{thm}\label{prop:diag-lci}
The schemes $\wt{\Hk_\mu} $ and $\wt{\Delta_{\mcM}}$ are relative local complete intersections over $\wt{\mathcal M \times \mathcal M}$.    
\end{thm} 
\begin{proof}
The schemes $\mcM$ and $\Hk_μ$ are smooth over $O_E^\mfd$ above all places $v_0$ of $F_0$ for which $L_{v_0}$ is self-dual or satisfie $L_{v_0}^\vee = π_{v_0}^{-1} L_{v_0}$. Also, the $\mcZ$-stratum is empty above such places. It follows that both $\Hk_μ$ and $\wt{\mcM\times\mcM}$ are smooth, hence regular above such places and the local complete intersection property is automatic. So we consider one of the remaining places. That is, $v$ is an inert place of $F$ with $L_v\subsetneq L_v^\vee \subsetneq π_v^{-1}L_v$.

\medskip \noindent \emph{Step 1: The relative local model.} Consider the relative local model $\bfM := \bfM_n^{\langle L_v\rangle}(n-1,1)$ which is regular. Denote by $\bfM^\mcZ$ its $\mcZ$-divisor as after \eqref{eq:line_bundle_maps}. By \cite[Theorem 5.11]{ZZhang21}, the blow up
\begin{equation}\label{eq:blow_up_relative_LM}
\wt{\bfM\times\bfM} := \mr{BL}_{\bfM^\mcZ\times \bfM^\mcZ}(\bfM \times_{\Spec O_{F,v}} \bfM)
\end{equation}
is regular. (Recall that the key here is to check that the blow up of
$$\Spf O_{F,v}[\![x_0,\ldots,x_n,y_0,\ldots,y_n]\!]/(x_0x_1-π_v, y_0y_1 - π_v)$$
along the ideal $(x_0, y_0)$ is regular.) Hence the strict transform of the diagonal
$$\wt{\Delta_{\bfM}}: \bfM \lr \wt{\bfM \times \bfM}$$
is a closed immersion of regular schemes, and in particular a local complete intersection.

\medskip \noindent \emph{Step 2: The absolute local model.} Let $ν$ be a place of $E$ above $v$. Consider the base change
$$\bfM_A := O_{E,ν}\tensor_{O_{E_{A,v} ϕ_0(F)}} \bfM_{n,A}^{\langle L_v\rangle} (n-1,1)$$
of the absolute local model to $O_{E, ν}$. It comes with a $\mcZ$-divisor $\bfM_A^\mcZ$ as after \eqref{eq:line_bundle_maps_absolute}. As explained there, the pair $(\bfM_A, \bfM_A^\mcZ)$ can be identified with the base change to $O_{E,ν}$ of $(\bfM, \bfM^\mcZ)$. Define
$$\wt{\bfM_A\times \bfM_A} := \mr{BL}_{\bfM_A^\mcZ\times \bfM_A^\mcZ}(\bfM_A\times_{O_{E,ν}} \bfM_A).$$
Since blowing up commutes with flat base change, we obtain that the strict transform
$$\wt{\Delta_{\bfM_A}}:\bfM_A \lr \wt{\bfM_A\times \bfM_A}$$
is a local complete intersection closed immersion.

\medskip \noindent \emph{Step 3: The moduli space $\mcM$.} Consider a geometric closed point $(\mcA_0, \mcA, \mcB, φ) \in \mathrm{Hk}_\mu$ of the special fiber of $\mcM$ over $ν$. Let $R = \wh{\mcO}_{\Hk_μ, (\mcA_0, \mcA, \mcB, φ)}$ be its strict complete local ring. We claim that the map
\begin{equation}\label{eq:strict_local_rings}
\Spf(R) \lr \Spf (\wh{\mcO}_{\mcM\times \mcM, ((\mcA_0, \mcA), (\mcA_0, \mcB))})
\end{equation}
can be identified with the strict completion of $\Delta_{\bfM_A}$ in a closed point.

Recall for this that the existence theorem of the local model diagram \cite[Theorem 3.11, Proposition 3.33]{RZ1} states that there exists an $O_{F,v}$-linear isomorphism between
$$L_v\tensor_{\mbZ_p} R \lr L_v^\vee\tensor_{\mbZ_p}R \lr (π_v^{-1} L_v)\tensor_{\mbZ_p}R$$
and
$$\mcD(\mcA) \lr \mcD(\mcA^\vee) \lr \mcD(\mcA)$$
that propagates into an isomorphism of self-dual chains. Moreover, $γ$ can be chosen such that the Hodge filtrations of $\mcA$ and $\mcA^\vee$ identify $R$ with the strict complete local ring $\bfR$ of a geometric closed point of $\bfM_A$. Since $φ$ is a $\mfd$-quasi-isogeny, the pair $(γ, φ\circ γ)$ then identifies the strict complete local ring of $\mcM\times_{O_{E, ν}}\mcM$ in $((\mcA_0, \mcA), (\mcA_0, \mcB))$ with $\bfR\wh{\tensor}_{\breve O_{E, ν}} \bfR$. This identification is such that \eqref{eq:strict_local_rings} becomes the diagonal map
$$\Spf(\bfR) \lr \Spf (\bfR\tensor_{\breve O_{E, ν}} \bfR).$$
Moreover, it is clear from definitions that these identifications are compatible with the formation of $\mcZ$-divisors and hence extend to blow ups. We conclude that the maps in strict complete local rings of $\wt{\Hk_μ} \to \wt{\mcM\times \mcM}$ can be identified with the maps arising on strict complete local rings of $\wt{\Delta_{\bfM_A}}$. So the relative complete intersection property for $\wt{\Hk_μ}$ follows from that of $\wt{\Delta_{\bfM_A}}$.
\end{proof}

 \begin{rmk}
For a finite type, flat, and separated scheme $X \to \Spec(W)$ over a discrete valuation ring $W$, the diagonal $\Delta: X \to X \times_{\Spec(W)} X$ is a perfect complex (equivalently, a local complete intersection) if and only if $X$ is smooth over $R$, see \cite[\href{https://stacks.math.columbia.edu/tag/0FDP}{Tag 0FDP}]{stacks-project}.
\end{rmk}

\newcommand{\Fix}{\mathrm{Fix}}

Let $f=f_{\mfd}\otimes f^{\mfd}$ be a $\mbQ$-valued Schwartz function in $\mcS(G(\mathbb A_f))^{K_G}$ such that $f^\mfd=\id$. Our next aim is to define certain derived CM cycle classes $\mathrm{CM}^{\mbL}(\alpha,f)$. These will be elements of $\mcZ_1'(\mcM)$, defined as follows:

\begin{defn}\label{def:Z_prime}
Let $\mcZ_1'(\mcM)$ denote the group of $1$-cycles on $\mcM$ with $\mbQ$-coefficients modulo fiberwise rational equivalence. That is, the equivalence relation is the one generated by cycles of the form $\mathrm{div}(h)$ where $h$ is a non-zero meromorphic function on an integral $2$-dimensional subscheme of $\mcM$ that lies in some fiber of $\mcM\to \Spec(O_E^\mfd)$.
\end{defn}

\begin{defn}\label{def:classical hecke} (1) The \emph{fixed point locus} of $\wt{\Hk_μ}$ is defined as the fiber product
\begin{equation}
    \Fix(μ) := \wt{\Delta_\mcM} \times_{\wt{\mcM\times \mcM}} \wt{\Hk_μ}.
\end{equation}
Recall from \cite[Prop. 11.10]{ZZhang21} that there is a disjoint union decomposition of $\Fix(μ)$ as
\begin{equation}\label{equ:hecke-decomp}
\Fix(μ) = \coprod_{\alpha\in O_F^\mfd[t]^{\circ}_{\deg n}}\CM(\alpha,\mu),
\end{equation}
where $O_F^{\mfd}[t]^{\circ}_{\deg n}$ is the set of all monic polynomials in $O_F^{\mfd}[t]$ of degree $n$ with $\alpha(0)t^n\ov{\alpha}(t^{-1})=\alpha(t)$. Concretely, $\CM(α, μ)$ is the open and closed subscheme of points $(\mcA_0, \mcA, \mcA, φ) \in \Fix(μ)$ such that the characteristic polynomial of $φ$ equals $α$.

Assume that $α$ is irreducible. Then $F[t]/(α(t))$ is a CM field of degree $n$ over $F$, and $\CM(α, μ)$ is a moduli space of abelian varieties with CM by that field. It follows that $\CM(α,μ)\to \Spec(O_E^\mfd)$ is proper. It is finite étale above all primes $p$ at which $O_F^\mfd[t]/(α(t))$ is a maximal order and unramified.

\newcommand{\Ch}{\mathrm{Ch}}

\medskip \noindent (2) In the following, all our Chow groups are taken with $\mbQ$-coefficients. By Theorem \ref{prop:diag-lci}, the strict transform of the diagonal $\wt{\Delta_\mcM}:\mcM\to \wt{\mcM\times \mcM}$ is a local complete intersection. Its codimension is $n-1$, so by \cite[\href{https://stacks.math.columbia.edu/tag/0FBI}{Tag 0FBI}]{stacks-project} it defines a Gysin homomorphism
$$\wt{\Delta_\mcM}{}^!:\Ch_\bullet(\wt{\Hk_μ}) \lr \Ch_{\bullet - n- 1}(\Fix(μ)).$$
Moreover, by \eqref{equ:hecke-decomp}, there is a decomposition
$$\Ch_{\bullet-n-1}(\Fix(μ)) = \bigoplus_{α\in O_F^\mfd[t]^\circ_{\text{deg $n$}}} \Ch_{\bullet -n-1}(\CM(α, μ)).$$
In particular, given some $α$, we may apply the Gysin map to the fundamental cycle $[\wt{\Hk_μ}]$ and consider the $α$-component to define a class
\begin{equation}\label{eq:Gysin_alpha}
(\wt{\Delta_\mcM}{}^!([\wt{\Hk_μ}]))_α \in \Ch_1(\CM(α, μ)).
\end{equation}
Assume that $α$ is irreducible. In particular, as already mentioned, all irreducible components of $\CM(α, μ)$ that map dominantly to $\Spec(O_E^\mfd)$ are $1$-dimensional. There is hence a pushforward map $\Ch_1(\CM(α, μ)) \lr \mcZ_1'(\mcM).$ We denote by
$$\CM^\mbL(α, μ) \in \mcZ_1'(\mcM)$$
the image of \eqref{eq:Gysin_alpha}.

\medskip \noindent (3) Let $f = f_\mfd \tensor f^\mfd$ be a $\mbQ$-valued Schwartz function in $S(G(\mbA_f))^{K_G}$ such that $f^\mfd=\id$. We define the \emph{modified derived CM cycle} $\CM^{\mbL}(\alpha,f)$ as the (finite) sum
\begin{equation*}
    \CM^{\mbL}(\alpha,f) := \sum_{\mu \in K_{G,\mfd}\backslash G(F_{0,\mfd})\slash K_{G,\mfd}}f^\mfd(\mu)\CM^{\mbL}(\alpha,\mu)\in \mcZ_1'(\mcM).
\end{equation*}
\end{defn}
\begin{rmk}\label{rmk:two def of cycles}
The construction we provided agrees with previous constructions using the derived tensor product. In the following, we refer to the book \cite{Fulton-Intersection} whose main results apply to our arithmetic setting as explained in its \S 20.1, see in particular \cite[p. 395]{Fulton-Intersection}.

Consider the following diagram:
\begin{equation*}
\begin{aligned}
\xymatrix{
\CM(\alpha,\mu)\ar@{^(->}[r]^-{i}&\Fix_\mu \ar@{}[rd]|{\square}\ar@{^(->}[r]\ar[d]^-{p}&\wt{\Hk_\mu}\ar[d]^-{f}\\
&\wt{\Delta_{\mcM}}\ar@{^(->}[r]&\wt{\mcM\times\mcM}.
}
\end{aligned}
\end{equation*}
In terms of $K$-groups, the definition of the derived CM cycle is
\begin{equation}\label{eq:def_in_terms_of_K_group}
i^*[\mcO_{\wt{\Hk_\mu}}\otimes^{\mbL}_{\mcO_{\wt{\mcM\times\mcM}}}\mcO_{\wt{\Delta_{\mcM}}}]\in G_0(\CM(\alpha,\mu)),
\end{equation}
where $G_0(X)$ is our notation for the $\mbQ$-coefficient $K$-group of isomorphism classes of coherent sheaves on a scheme $X$. (This is also denoted by $K'_0(X)_\mbQ$ in the literature.) Note that $i^*$ is nothing but the restriction to an open and closed subscheme of $\mr{Fix}_μ$.

Recall that there is a Riemann--Roch isomorphism, see \cite[\S18]{Fulton-Intersection},
$$τ_{\mr{Fix}_μ}:G_0(\mr{Fix}_μ) \simlr \mr{Ch}_*(\mr{Fix}_μ).$$
By the description on \cite[bottom of p. 366]{Fulton-Intersection}, we have
$$\tau_{\Fix_\mu}([\mcO_{\wt{\Hk_\mu}}\otimes^{\mbL}_{\mcO_{\wt{\mcM\times\mcM}}}\mcO_{\wt{\Delta_{\mcM}}}])=\tau_f([\mcO_{\wt{\Hk_\mu}}])\cap \tau_{\wt{\mcM\times\mcM}}(\mcO_{\wt{\Delta_\mcM}})\in \mathrm{Ch}_*(\Fix_\mu)$$
where
$$\tau_f: K(\wt{\Hk_\mu}\to \wt{\mcM\times\mcM})\to A(\wt{\Hk_\mu}\to \wt{\mcM\times\mcM})$$
is the map defined on \cite[top of p. 366]{Fulton-Intersection}.
Since $f$ is local complete intersection, by \cite[Example 18.3.17]{Fulton-Intersection}, the induced element $\tau_f([\mcO_{\wt{\Hk_\mu}}])=c_f$ is the higher Gysin map. So we find for the element in \eqref{eq:Gysin_alpha} that
$$(\wt{\Delta_\mcM}{}^!([\wt{\Hk_μ}]))_α = τ_{\mr{Fix}_μ}([\mcO_{\wt{\mr{Hk}_μ}}\tensor_{\mcO_{\wt{\mcM\times \mcM}}}^\mbL \mcO_{\wt{Δ_\mcM}}])_α.$$
\end{rmk}

Next, we recall the definition of Kudla--Rapoport special divisors. Let $F_{0,+}\subseteq F_0^\times$ denote the subset of all totally positive elements.

\begin{defn}\label{def:KR} (1) Let $\xi\in F_{0, +}$ and let $μ = \mu_\mfd\subset V(F_{0,\mfd})$ be an open compact subset that is $K_{G,\mfd}$-stable. The \emph{Kudla--Rapoport special divisors} $\mcZ(\xi,\mu)$ and $\mcY(\xi,μ)$ are defined as the functors taking an $O_{E}^{\mfd}$-test scheme $S$ to the set of isomorphism classes of tuples $(\mcA_0,\mcA,\ov{\eta},u)$ where:
\begin{itemize}
\item $(\mcA_0,\mcA,\ov{\eta})\in \mcM(S)$;
\item $u\in \Hom^0_F(\mcA_0,\mcA)$ with $(u,u)=\xi$;
\item For the definition of $\mcZ(ξ, μ)$ we require $u\in\Hom^{\mfd}(\mcA_0,\mcA)$; for the definition of $\mcY(ξ, μ)$ we instead demand $\lambda\circ u\in\Hom^{\mfd}(\mcA_0,\mcA^\vee)$;
\item $\ov{\eta} \circ u$ is a $K_{G, \mfd}$-orbit in $\mu$.
\end{itemize}
\medskip \noindent (2) Let $ϕ_\mfd \in S(V(F_{0,\mfd}))^{K_{G,\mfd}}$ be a $\mbQ$-coefficient Schwartz function. For $i\in\{1,2\}$, extend it to a Schwartz function $\phi_i=\phi_{\mfd}\otimes \phi_i^\mfd\in \mcS(V(\mbA_{0,f}))^{K_G}$ with 
\begin{equation}\label{equ:phi_i_mfd}
    \phi_1^\mfd = 1_{\wh L^\mfd},\quad\text{and}\quad\phi_2^\mfd = 1_{\wh L^{\vee, \mfd}}.
\end{equation}
Given $\xi\in F_{0,+}$, we associate to $ϕ_i$ its KR divisor as follow: Write $ϕ_\mfd = \sum_{j = 1}^r a_j \, 1_{μ_j}$ as a linear combination of indicator functions. Then set
\begin{equation}\label{def:ZY_weighted}
\begin{aligned}
    \mcZ(\xi,\phi_{1}) & := \sum_{j = 1}^r \,a_j\, \mcZ(\xi,\mu_j)\\[2mm]
    \mcZ(\xi,\phi_2) & := \sum_{j = 1}^r\, a_j\,\mcY(\xi,\mu_j).
\end{aligned}
\end{equation}
These are algebraic cycles on $\mcM$.
\end{defn}

\begin{prop}
The schemes $\mcZ(\xi,\mu)\to \mcM$ and $\mcY(\xi,\mu)\to \mcM$ are finite, unramified and relatively representable. 
They are \'etale locally Cartier divisors over $\mcM$.
\end{prop}
\begin{proof}
The given morphisms are finite, unramified and relatively representable by \cite[Lemma A.2.3]{RChen24-III}. It remains to show that they are \'etale locally Cartier divisors. We verify this for $\mcZ(\xi,\mu)$, the case of $\mcY(ξ,μ)$ being analogous. Let $z\in \mcZ(ξ,μ)$ be a geometric closed point and let $\wh \mcO_{\mcZ(ξ,μ),z}$ and $\wh \mcO_{\mcM,z}$ denote the strict complete local rings in $z$. Let $\mcI\subseteq \mcO_{\mcM,z}$ be the kernel of the surjection $\wh \mcO_{\mcM,z}\to \wh \mcO_{\mcZ(ξ, μ),z}$, and let $\mfm_z$ denote the maximal ideal of $\wh \mcO_{\mcM,z}$. We need to show that $\mcI$ is generated by a regular element. Note that the residue characteristic of $z$ is $p$ because $\mcM$ is projective and $z$ closed.

Let $(\mcA_0,\mcA,\ov{\eta})$ denote the universal family over $S := V(\mfm_z\mcI)\subseteq \Spec(\wh \mcO_{\mcM,z})$ and let $u:\mcA_0\vert_{V(\mcI)}\to \mcA\vert_{V(\mcI)}$ be the homomorphism defined over $\Spec(\wh \mcO_{\mcZ(ξ, μ),z})$. Since $\mfm_z\supseteq\mfm_z\mcI\supset\mfm_z^2$, the closed subscheme defined by the ideal sheaf $\mfm_z\mcI$ is a square-zero thickening of $V(\mcI)$. In particular, $u$ defines a homomorphism $\mcD(u):\mcD(\mcA_0)\to \mcD(\mcA)$ between Grothendieck--Messing crystals evaluated at $S$. By Grothendieck--Messing theory, $V(\mcI) \subseteq S$ is the locus of $S$ where the induced diagram
\begin{equation}\label{eq:lifting_diagram}
\xymatrix{
0\ar[r]&\Fil(\mcA_0)\ar[r]&\mcD(\mcA_0)\ar[d]^{\mcD(u)}\ar[r]&\Lie(\mcA_0)\ar[r]&0\\
0\ar[r]&\Fil(\mcA)\ar[r]&\mcD(\mcA)\ar[r]&\Lie(\mcA)\ar[r]&0
}
\end{equation}
gives rise to a map $\Fil(\mcA)_0\to\Fil(\mcA)$. Equivalently, it is the vanishing locus of the composite morphism $γ:\Fil(\mcA_0)\to \Lie(\mcA)$. Recall that, by definitions, $\mcA_0$ is $(\Phi,\Phi)$-strict which means that $\Fil(\mcA_0) = J_Φ\cdot \mcD(\mcA_0)$. Since all maps in \eqref{eq:lifting_diagram} are $O_F$-linear, this implies that $γ$ takes values in $J_Φ\cdot \Lie(\mcA)$. By Assumption \ref{assumption:pos_def}, $Φ = A\cup \{ϕ_0\}$, so $J_Φ\cdot \Lie(\mcA) = (J_A\cdot\Lie(\mcA))_{1}$ which is a line bundle by the signature $(n-1,1)$ condition \eqref{eq:signature_abelian_variety} and \eqref{equ:Lie-split}.

Moreover, $J_{\{\ov{ϕ_0}\}}\cdot (J_A\cdot \Lie(\mcA))_1 = 0$, so the map $γ$ factors over the quotient line bundle $\Fil(\mcA_0)/J_{\{\ov{ϕ_0}\}}\cdot \Fil(\mcA_0)$. Hence, we see that $\mcI$ is defined by the vanishing of the map
$$γ:\Fil(\mcA_0)/J_{\{\ov{ϕ_0}\}}\cdot \Fil(\mcA_0) \lr (J_A\cdot \Lie(\mcA))_1$$
which shows that $\mcI$ can be generated by a single element.

Since the generic fibers $\mcZ(\xi,\mu)_E\to \mcM_E$ and $\mcY(\xi,\mu)_E\to \mcM_E$ are Cartier divisors by complex uniformization, and since $\mcM$ is flat over $\mcO_E$, this generator cannot be a zero divisor.
\end{proof}

Define the volume factor $\tau(Z^{\mbQ}) := \# Z^{\mbQ}(\mbQ)\backslash Z^{\mbQ}(\mbA_f)/K_{Z^{\mbQ}}$ and fix an irreducible polynomial $α\in O_F^\mfd[t]^\circ_{\text{deg $n$}}$. Consider an element $\xi\in F_{0,+}$, a function $f=f_{\mfd}\otimes f^{\mfd}\in \mcS(G(\mathbb A_{0,f}))^{K_G}$ and a function $\phi_i=\phi_i^{\mfd}\otimes\phi_{i,\mfd}\in \mcS(V(F_{0,\mfd}))^{K_{G,\mfd}}$ as before. By the same argument as \cite[Theorem 8.5]{RSZ-AGGP}, the intersection between the supports of $\mcZ(ξ, ϕ_i)$ and $\CM^\mbL(α, f)$ is contained in the basic locus above finitely many inert places of $F$. Thus, for each $i \in \{1,2\}$, we can define the global intersection number
\begin{equation}\label{eq:def_global_int_number}
    \Int(\xi,\phi_i, f):=\frac{1}{\tau(Z^{\mbQ})[E:F]}(\mcZ(\xi,\phi_i),\CM^{\mbL}(\alpha,f))\ \ \in \sum_{p\nmid \mfd} \mbQ\,\log(p).
\end{equation}

Later, we will define the global arithmetic intersection number $\Int^?(\xi,f,\phi_i)$ for $?=\mathbf{K}$, $\mathbf{B}$ and $\mathbf{K}-\mathbf{B}$ for any $\xi\in F_0$.

\subsection{Cycles on the RZ spaces}\label{sec:cycles on absolute RZ}
We switch to the local notation from \S\ref{sec:local KR strata}. In what follows, we introduce CM cycles and Kudla--Rapoport special divisors on $\mcN_A^{[t]}$ and explain why these are the same as the isomorphic images of the ``usual'' cycles on $\mcN^{[t]}$ under the comparison isomorphism \eqref{eq:comp_iso_recap}.

We first construct and compare CM cycles on absolute and relative RZ space. Fiber products in the following will be over $O_{\breve F}$ (for $\mcN^{[t]}$) or over $O_{\breve E}$ (for $\mcN_A^{[t]}$). Following \cite[Definition 5.10]{ZZhang21}, we define
\[
\wt{\mcN^{[t]} \times \mcN^{[t]}} \lr \mcN^{[t]} \times \mcN^{[t]}
\]
as the blow up along the closed formal subscheme 
$$\mcN^{[t],\mcZ} \times \mcN^{[t],\mcZ} \,\subset\, \mcN^{[t]} \times \mcN^{[t]}.$$
This blow up is regular by \cite[Theorem 5.11]{ZZhang21}. Similarly, for the absolute RZ space $\mcN_{A}^{[t]}$, we define
\[
\wt{\mcN_{A}^{[t]} \times \mcN_{A}^{[t]}} \lr \mcN_{A}^{[t]} \times \mcN_{A}^{[t]}
\]
as the blow up along the closed formal subscheme 
$$\mcN_A^{[t],\mcZ}\times \mcN_A^{[t],\mcZ}\,\subset\, \mcN_{A}^{[t]} \times \mcN_{A}^{[t]}.$$
Consider the diagonal embedding 
$$\Delta_A:\mcN_A^{[t]}\lr \mcN_A^{[t]}\times \mcN_A^{[t]}.$$
Since $\mcN_A^{[t],\mcZ}$ is a Cartier divisor, see Proposition \ref{prop:KR-strata}, the diagonal $\Delta_A$ lifts uniquely to a morphism
$$\wt{\Delta_A}:\mcN_A^{[t]} \lr \wt{\mcN_A^{[t]} \times \mcN_A^{[t]}}.$$
More generally, let $g\in \U(\mbV_A)$ be any element. The action of $g$ on $\mcN_A^{[t]}$ is purely in terms of the framing $ρ$ and, in particular, preserves Kottwitz--Rapoport strata. So the same argument shows that there exists a unique lifting
$$\wt{Γ_{A,g}}:\mcN_A^{[t]} \lr \wt{\mcN_A^{[t]} \times \mcN_A^{[t]}}$$
of the graph map $(1,g)$ to the blow up. All these definitions in particular apply to $(\mbX, ι_\mbX, λ_\mbX)$, i.e. in the case $K = F_0$ and $A = \emptyset$), which defines analogous cycles
$$\wt{\Delta},\, \wt{Γ_g}: \mcN^{[t]} \lr \wt{\mcN^{[t]} \times \mcN^{[t]}}.$$
\begin{prop}\label{prop:compar_strict_transform_local}
For every $g\in \mathrm U(\mbV_A)$, the comparison isomorphism $Φ_A$ from \eqref{eq:comp_iso_recap} defines a commutative diagram
\begin{equation}\label{eq:square_blow_up_phi_A}
\xymatrix{
\mcN_A^{[t]} \ar[rr]^-{\wt{Γ_{A,g}}} \ar[d]_-{Φ_A}^-\cong && \wt{\mcN_A^{[t]}\times \mcN_A^{[t]}} \ar[d]^-{Φ_A}_-\cong\\
O_{\breve E}\wh{\tensor}_{O_{\breve F}} \mcN^{[t]} \ar[rr]^-{\wt{Γ_{Φ_A(g)}}} && O_{\breve E}\wh{\tensor}_{O_{\breve F}}\big(\wt{\mcN^{[t]}\times \mcN^{[t]}}\big).
}
\end{equation}
In particular, for every $g$, the strict transform $\wt{Γ_{A,g}}$ is a local complete intersection in $\wt{\mcN_A^{[t]}\times \mcN_A^{[t]}}.$
\end{prop}
\begin{proof}
By Proposition \ref{prop:comp-KR-divisors}, $Φ_A$ identifies the Kottwitz--Rapoport divisors on $\mcN_A^{[t]}$ and $O_{\breve E}\wh{\tensor}_{O_{\breve F}}\mcN^{[t]}$. Since blowing up commutes with flat base change, it follows that $Φ_A$ extends to an isomorphism
$$\wt{\mcN_A^{[t]} \times \mcN_A^{[t]}} \simlr O_{\breve E} \wh{\tensor}_{O_{\breve F}} \big(\wt{\mcN^{[t]}\times \mcN^{[t]}}\big).$$
The existence of \eqref{eq:square_blow_up_phi_A} is then clear. It is known by \cite[Theorem 5.11]{ZZhang21} that $\wt{Γ_g}$ is a local complete intersection. By base change, $\wt{Γ_{A,g}}$ is a local complete intersection as well.
\end{proof}

\begin{defn}
The local CM cycle defined by $g$ is the intersection
$$\CM_A(g) := \wt{Γ_g} \cap \wt{\Delta_g} \subseteq \mcN_{A}^{[t]}.$$
The derived CM cycle defined by $g$ is the complex
\[
\CM^\mbL_A(g) := \big(\mcO_{\wt{\Gamma_g}} \otimes^{\mathbb L} \mcO_{\wt{\Delta_A}}
\big) \in D(\mcN_{A}^{[t]})\]
where the tensor product is over the structure sheaf of $\wt{\mcN_{A}^{[t]} \times \mathcal N_{A}^{[t]}}$.
\end{defn}

\begin{prop}[Comparison of CM cycles] \label{prop-com-cycles}
For every $g\in \mathrm U(\mbV_A)$ and under the isomorphism $Φ_A$ from \eqref{eq:comp_iso_recap}, there is a natural isomorphism
$$\CM^\mbL_A(g) \simlr Φ_A^*(O_{\breve{E}} \otimes_{O_{\breve{F}}} \CM^\mbL(g))\ \in D(\mcN_A^{[t]}).$$
In particular, its class $[\CM^\mbL_A(g)]$ in $G_0(\CM_A(g))$ lies in the filtration subgroup $F_1G_0(\CM^\mbL_A(g))$ generated by coherent sheaves with support of dimension $\leq 1$.  
\end{prop}
\begin{proof}
The first statement follows from definitions and Proposition \ref{prop:compar_strict_transform_local}. The second statement holds because $[\CM^\mbL(g)]$ lies in $F_1G_0(\CM(g))$ because it arises by intersection of two $n$-dimensional subspaces of a regular, $(2n-1)$-dimensional space.
\end{proof}

Next, we define and compare the special divisors on absolute and relative RZ space.
Let $(\mbE_A,\iota_{\mbE_A},\lambda_{\mbE_A})$ be an $(A,B)$-strict, isoclinic $O_K$-module over $\Spec \mbF$ of height $2[F_0:K]$ and dimension $[F_0:K]$, of signature $(1,0)$, and such that $\lambda_{\mbE_A}$ is a principal polarization. Define
$$(\mbE, ι_\mbE, λ_\mbE) := Φ_A(\mbE_A,\iota_{\mbE_A},\lambda_{\mbE_A})$$
which is a principally polarized strict $O_{F_0}$-module of height $2$, dimension $1$, $O_{F_0}$-height $1/2$, and signature $(1,0)$. It is well-known that $(\mbE, ι_\mbE, λ_\mbE)$ has a unique deformation $(\mathcal E, \iota_{\mathcal E}, \lambda_{\mathcal E})$ to $\Spf O_{\breve{F}}$. In the same way, $(\mbE_A,\iota_{\mbE_A},\lambda_{\mbE_A})$ has a unique deformation $(\mcE_A,\iota_{\mcE_A},\lambda_{\mcE_A})$ to $\Spf O_{\breve{E}}$. (Use the equivalence $Φ_A$ or the fact that the banal local model in \eqref{eq:banal_LM_unitary} is étale.) We choose the deformation $\mcE$ as
$$(\mcE, ι_\mcE, λ_\mcE) = Φ_A(\mcE_A, ι_{\mcE_A}, λ_{\mcE_A}).$$
Define the two hermitian $F$-vector spaces
$$\mbV_A = \Hom_F^0(\mbE_A, \mbX_A),\qquad \mbV = \Hom_F^0(\mbE, \mbX).$$
The hermitian forms are given by Kudla--Rapoport construction:
$$(x,y) := \lambda_{\mbE_A}^{-1} \circ y^\vee \circ \lambda_{\mbX_A} \circ x,\quad \text{resp.}\quad λ_\mbE^{-1}\circ y^\vee \circ λ_\mbX \circ x.$$
These compositions are $F$-linear quasi-endomorphisms of $\mbE_A$ resp. $\mbE$ and hence can be viewed as elements of $F$. Elements of $\mbV_A$ and $\mbV$ are called \emph{special quasi-homomorphisms}. The equivalence $Φ_A$ from Theorem \ref{thm:AB_strict_p_div_PEL} provides an isometry
$$Φ_A: \mbV_A\simlr \mbV.$$
Moreover, it is known that
$\mcQ\mcI sog(\mbX_A, ι_{\mbX_A}, λ_{\mbX_A}) \simlr \mathrm U(\mbV_A)$ via composition.
\begin{defn}
Let $x\in \mbV_A$ be special quasi-homomorphism. 

\medskip \noindent (1) The \emph{Kudla--Rapoport $\mcZ$-divisor} $\mcZ_A(x)\subseteq \mcN_{A}^{[t]}$ is defined as the closed formal subscheme whose $S$-points are those tuples $(X_A,ι_A, λ_A, ρ_A)$ such that $\rho_A \circ x \in \Hom(\mathcal E_{A,S}, X_A)$. In other words, the composition of quasi-homomorphisms over $\ov{S}$ 
$$
\mcE_{A,\ov S} \simlr \mbE_{A,\ov{S}}\overset{x}{\lr} \mbX_{A,\ov{S}} \overset{\rho_A}{\lr} X_{A,\ov S}
$$
lifts to a homomorphism $\mcE_{A,S}\lr X_A$.

\medskip \noindent (2) In the same way, the \emph{Kudla--Rapoport $\mcY$-divisor} $\mcY_A(x)\subseteq \mcN_{A}^{[t]}$ is the closed formal subscheme whose $S$-points are those tuples $(X_A,ι_A, λ_A, ρ_A)$ such that
$$λ_A\circ \rho_A \circ x \in \Hom(\mathcal E_{A,S}, X_A^\vee).$$
\end{defn}

When applied to $(\mbX, ι, λ)$ which is the case $K=F_0$ and $A=\emptyset$, this recovers the standard definition of KR divisors $\mcZ(x), \mcY(x)\subseteq \mcN^{[t]}$.

\begin{prop}[Comparison of special cycles] \label{prop-com-KR-cycles}
For every $x\in \mbV_A$, the $\mathrm{U}(\mathbb V)$-equivariant isomorphism $\Phi_A$ in \eqref{eq:comp_iso_recap} restricts to isomorphisms
$$
\Phi_A: \mcZ_A(x) \simlr O_{\breve E}\,\wh{\tensor}_{O_{\breve F}}\, \mcZ(Φ_A(x)),\qquad \mcY_A(x) \simlr O_{\breve E}\,\wh{\tensor}_{O_{\breve F}}\, \mcY(Φ_A(x)).
$$
In particular, $\mcZ_A(x)$ and $\mcY_A(x)$ are Cartier divisors on $\mcN_{A}^{[t]}$.
\end{prop}
\begin{proof}
This follows from the moduli descriptions of both sides and the explicit construction of $\Phi_A$. The last statement follows from the fact that $\mcZ(x)$ and $\mcY(x)$ are Cartier divisors in $\mcN^{[t]}$; see \cite[Proposition 5.9]{ZZhang21}.
\end{proof}

A pair $(g,x)\in \mathrm{U}(\mbV) \times \mbV$ is called \emph{regular semi-simple} if $\{ g^ix\}_{i\geq 0}$ spans $\mbV$. In this situation, the intersections $\CM(g)\cap \mcZ(x)$ and $\CM(g)\cap \mcY(x)$ are projective schemes over $\Spf(O_{\breve F})$, \cite[Proposition 6.2]{ZZhang21}. Thus it makes sense to define the intersection numbers
\begin{equation}\label{eq:def_int_numbers}
\begin{aligned}
\wt {\Int}{}^{\mcZ}(g, x) & = χ(\CM^\mbL(g) \tensor^{\mbL} \mcO_{\mcZ(x)})\\[2mm]
\wt {\Int}{}^{\mcY}(g, x) & = χ(\CM^\mbL(g) \tensor^\mbL \mcO_{\mcY(x)}).
\end{aligned}
\end{equation}
The tensor product here is over $\mcO_{\mcN^{[t]}}$, and the Euler--Poincaré characteristic is defined by
$$χ(K_\bullet) = \sum_{i,j\in \mbZ} (-1)^{i+j} \ell en_{O_{\breve F}}(H^j(\mcN^{[t]}, H_i(K_\bullet))).$$
The same definition can be made on $\mcN_A^{[t]}$, but we do not need to introduce new notation for this because of the following result. Note that it does not matter whether one takes the length as $O_{\breve E}$-module or $O_{\breve F}$-module because $\breve E/\breve F$ is totally ramified.
\begin{cor}
Let $(g,x)\in \mathrm U(\mbV_A)\times \mbV_A$ be regular semi-simple. Then
$$χ(\CM_A^\mbL(g) \tensor^\mbL_{\mcO_{\mcN_A^{[t]}}} \mcO_{\mcZ_A(x)}) = [\breve E: \breve F]\,\wt{\Int}{}^\mcZ(g,x)$$
and similarly for the $\mcY$-divisor intersection number.
\end{cor}
\begin{proof}
This is clear from Proposition \ref{prop-com-KR-cycles}, Proposition \ref{prop:compar_strict_transform_local} and because taking the derived tensor product and taking cohomology commute with the finite flat base change $O_{\breve F}\to O_{\breve E}$.
\end{proof}

\subsection{Uniformization of algebraic cycles}
In this subsection, we prove the basic uniformization of CM cycles and special divisors. We resume notations in \S \ref{sec:uniformofSV}.

First, we state the basic uniformization for Hecke correspondences and CM cycles following \cite{ZZhang21}. Recall that $\Hk_μ$ was introduced in Definition \ref{defn:hecke}. Let $\mr{Hk}_{\mu}^{\wedge}$ be the locally noetherian formal scheme over $O_{\breve{E}_{\nu}}$ obtained by pullback of $\mr{Hk}_{\mu}$ 
along the map 
\[ 
\mathcal M^{\wedge} \times_{\mcM_0^\wedge} \mathcal M^{\wedge} \lr \mathcal M \times_{\mcM_0} \mathcal M.
\]
There is a natural projection 
\begin{equation}
\mr{Hk}_{\mu}^\wedge \lr Z^{\mathbb Q}(\mathbb Q)  \backslash Z^{\mathbb Q}(\mathbb A_{0, f}) / K_{Z^\mathbb Q} .
\end{equation}
Denote by $\mr{Hk}_{\mu,0}^{\wedge}$ any of its fibers. Consider the set-theoretic Hecke correspondence
\[
\mr{Hk}_{\mu,0}^{v}= \{ (g_1,g_2) \in G(\mathbb A_f^{v})/K^{v} \times G(\mathbb A_f^{v})/K^{v}\ |\ g_1^{-1}g_2 \in K^v\mu K^v \}.
\]
It is equipped with a diagonal action of $G^{(v)}(F_0)$ by left multiplication. Following the same reasoning as \cite[Proposition 7.16]{Zhang21}, and combining with our comparison result in the proof of Proposition \ref{prop-com-cycles}, we have a uniformization isomorphism:
\[ 
G^{(v)}(F_0) \backslash \big[ O_{\breve{E}_\nu}\wh{\otimes}_{O_{\breve{F}_v}}\mathcal N^{[t]} \,\times\, \mr{Hk}_{{\mu},0}^{v}\big] \simlr \mr{Hk}_{\mu,0}^{\wedge}\]
compatible with the uniformization map $\Theta_0$ from \eqref{eq:Theta_0}. 
For $(\delta', h') \in G^{(v)}(F_0) \times G(\mathbb A_{0,f}^v) / K^v$, consider the CM cycle
\begin{equation}
\CM (\delta', h')_{K^v} := [\CM(δ')_{O_{\breve E_ν}} \times 1_{h'K^v}] \lr [ O_{\breve E_ν}\tensor_{O_{\breve F_v}} \mathcal{N}^{[t]} \times G (\mathbb{A}_{0, f}^{v}) / K^v].
\end{equation}
The disjoint union
$$\coprod_{(\delta', h')} \CM (\delta', h')_{K^{v}} $$ over the $G^{(v)}(F_0)$-orbit of a given pair $(\delta, h)$ in $G^{(v)}(F_0) \times G(\mathbb A_{0,f}^v) / K^v$ descends to a finite unramified morphism $\CM(\delta, h) \to \mathcal{M}^{\wedge}_{O_{\breve{E}_{\nu}}, 0}$. Let $μ \in K_{\mfd}\backslash G(F_{0,\mfd})/K_{\mfd}$. As in the proof of \cite[Proposition 7.17]{Zhang21}, combined with Proposition \ref{prop-com-cycles}, we have
	\begin{prop}\label{prop:uni-CM}
		The restriction of the formal completion $\CM(\alpha, μ)^{\wedge}$ to any fiber of the projection $\Theta$ (\ref{eq: basic basic unif}) is the disjoint union
		\begin{equation}\label{basic naive Hk}
			\coprod_{(\delta, h) \in G^{(v)}(F_0) \backslash G^{(v)}(\alpha)(F_0)  \times G(\mathbb A_{0,f}^v)/K^{v} } 1_{μ \times K^{v,\mfd}} (h^{-1}\delta h) \cdot \CM (\delta, h)_{K^v}.
		\end{equation}
		Here $G^{(v)}(\alpha)(F_0) \subseteq G^{(v)}(F_0)$ consist of all elements with characteristic polynomial $\alpha$.\qed
	\end{prop}
    
Consider a Schwartz function $f = f_{\mfd} \otimes f^\mfd \in \mathcal S(G(\mathbb A_{0, f}))^K$ where $f^{\mfd}=1_{K^{\mfd}}$. 
Replacing $\mathrm{Fix}(\delta)$ by $ {\mathrm {Fix}}^\mathbb L  (\delta)$ in the definition, we get a virtual $1$-cycle $[{}^{\mathbb L}\CM (\delta, h) ]_{K^v}$.  As in the proof of \cite[Proposition 7.17]{Zhang21}, combined with Proposition \ref{prop-com-cycles}, we have
	\begin{cor}\label{cor:uni-DCM}
		The restriction of the basic uniformization ${}^{\mathbb L} \CM(\alpha, f)^{\wedge}$ to any fiber of the projection $\Theta$ (\ref{eq: basic basic unif}) is the sum
		\begin{equation}\label{basic derived Hk}
			\sum_{(\delta, h) \in  G^{(v)}(F_0) \backslash  G^{(v)}(\alpha)(F_0)  \times G(\mathbb A_{0,f}^v)/K^v} f^v(h^{-1}\delta h) \cdot [{}^{\mathbb L}\CM (\delta, h) ]_{K^v}.
		\end{equation}
        as an element in the group $\bigoplus_{μ\in K_{\mfd}\backslash G(F_{0,\mfd})\slash K_{\mfd}}F_1 G_0(\CM(α,μ)^\wedge)$.
        \qed
	\end{cor}

Now we state the basic uniformization of global Kudla--Rapoport special divisors at $ν$. Recall from Definition \ref{def:KR} that these are denoted by $\mcZ(\xi,\phi_i)$ where $ξ\in F_{0,+}$ is a totally positive element and where $\phi_i \in \mathcal S(V(\mathbb A_{0,f}))^{K_G}$ with $i=1$ or $2$ is a Schwartz function of the form $ϕ_i = ϕ_\mfd \tensor ϕ^\mfd_i$ with
$$
\phi^{\mfd}_1 = 1_{\wh L^{\mfd}}\qquad \text{and}\qquad \phi^{\mfd}_2 = 1_{\wh L^{\vee,\mfd}}.
$$
Let $\mathcal{Z}(\xi, \phi_i)^\wedge$ be the pullback of $\mathcal{Z}(\xi, \phi_i)$ along the formal completion $\mcM^\wedge\to \mcM$ from before. Moreover, let $\mathcal{Z}(\xi, \phi_i)_{0}^\wedge$ be any fiber of  $\mathcal{Z}(\xi, \phi_i)^\wedge$ under the projection map (\ref{projection map from RSZ to unitary}). We use analogous notation for the $\mcY$-divisors. For $\xi \in F_0$, we denote by $V^{(v)}_ξ$ the quadratic hypersurface $\{x \in V^{(v)} | (x,x)=\xi\}$.	
\begin{prop}\label{prop:uni-KR}
Under the uniformization isomorphism \eqref{eq:Theta_0},
\begin{equation}\label{equ:uniformforKRdivisor}
\begin{aligned}
\mcZ(\xi, \phi_1)_{0}^\wedge & = \sum_{ 
(x, g) \in G^{(v)}(F_0) \backslash  (V^{(v)}_{\xi}(F_0) \times G (\mathbb{A}_{0, f}^{v}) / K^{v})}  \phi_1^{v}(g^{-1}x) \, [\mcZ(x,g)]_{K^v}\\[2mm]
\mcY(\xi, \phi_2)_{0}^\wedge &= \sum_{(x, g) \in G^{(v)}(F_0) \backslash  (V^{(v)}_{\xi}(F_0) \times G (\mathbb{A}_{0, f}^{v}) / K^{v})  }  \phi_2^{v}(g^{-1}x)\, [\mcY(x,g)]_{K^v}.
\end{aligned}
\end{equation}
Here, $[\mcZ(x, g)]_{K^{v}}$ is defined as the descent of 
$$ \sum_{(x',g') \in G^{(v)}(F_0)\cdot (x, g)\ \subseteq\ V^{(v)}\times G(\mbA^{v}_{0,f})/K_G^{v}} \big(O_{\breve E_ν}\wh{\tensor}_{O_{\breve F_v}}\mcZ(x')\big) \times g'K^{v},$$
which is a $G^{(v)}(F_0)$-invariant Cartier divisor on the left hand side of \eqref{eq:Theta_0}, along $\Theta_0$ to $\mathcal{M}^{\wedge}_{0}$. The definition of $[\mcY(x,g)]_{K^v}$ is completely analogous.
\end{prop}
\begin{proof}
This follows from definitions and the construction of $\Theta_0$ and Proposition \ref{prop-com-KR-cycles}. We refer to \cite[Prop. 6.3]{KR-global} or \cite[\S 4.6]{RChen24-III} for more details.
\end{proof}

\section{Arithmetic transfer for unramified maximal parahoric level}\label{sec:AT-conj}
The arithmetic fundamental lemma, discovered by Wei Zhang \cite{Zhang12}, and subsequent arithmetic transfer conjectures \cite{RSZ-Duke, RSZ-Annalen, ZZhang21, CRZ-HeckeAFL, CRZ-quasiAFL} establish identities between derived orbital integrals of certain test functions and arithmetic intersection numbers of cycles on unitary Rapoport--Zink spaces with certain levels. In this section, we extend the arithmetic transfer identities with unramified maximal parahoric levels proven in \cite{ZZhang21} to general $p$-adic fields. Our approach uses and generalizes the global method developed in \cite{Zhang21, MZ, ZZhang21}.

\subsection{Arithmetic Picard group}\label{sec:arithPic}
\newcommand{\Pic}{\mathrm{Pic}}
Since the RSZ integral model $\mcM\to \Spec O_E^{\mfd}$ is not necessarily regular (see Remark \ref{rmk:non-regular}), the arithmetic Chow group must to be replaced by the arithmetic Picard group. In this subsection, we recall the arithmetic Picard group as introduced in \cite{GS90-Ann-I,GS92-Invent,BGS94-JAMES} and study its basic properties.

Let $\mcX$ be a flat, generically smooth, projective scheme over $O_E^{\mfd}$ (recall the notation in \eqref{eq:shorthands}), where $E/\mbQ$ is a number field and $\mfd\in\mbZ_{> 0}$.
Following \cite[\S 2.1.2]{BGS94-JAMES}, a \emph{hermitian vector bundle} on $\mcX$ is a pair $\wh{\mcE}=(\mcE,(\lVert-\rVert_\nu)_{\nu\in\Hom(E,\mbC)})$ where $\mcE$ is a locally free coherent $\mcO_X$-module and $\lVert-\rVert_\nu$ is a hermitian metric product on the holomorphic vector bundle $\mcE_{E_\nu}$ over $\mcX_{\nu}(\mbC)$ for each complex embedding $\nu\in \Hom(E,\mbC)$. An isomorphism $\varphi:\wh{\mcE}_1\to\wh{\mcE}_2$ of hermitian vector bundles is an algebraic isomorphism between $\mcE_1$ and $\mcE_2$ that induces an isometry from $\mcE_{1,\nu}$ to $\mcE_{2,\nu}$ for each $\nu\in \Hom(E,\mbC)$.

We denote the \emph{arithmetic Picard group} with $\mbQ$-coefficients by $\wh{\rm{Pic}}(\mcX)$. It is the $\mbQ$-vector space generated by the isomorphism classes of hermitian line bundles on $\mcX$ with group structure given by the tensor product. The identity element in $\wh{\rm{Pic}}(X)$ is $\mcO_X$ with the standard metric, and the inverse of the class of $\wh{\mcL}$ is given by the dual metrized line bundle.

To proceed, fix a K\"ahler metric $h_\nu$ on $X_\nu(\mbC)$ for each $\nu\in \Hom(E,\mbC)$ that is invariant under the complex conjugation. This metric induces a Laplace operator on differential forms on $X_\nu(\mbC)$. Let $A^{1,1}(X_{\nu}(\mbC),\mbR)$ be the $\mbR$-vector space of real differential forms of type $(1,1)$ on $X_\nu(\mbC)$.
A hermitian line bundle $\wh{\mcL}=(\mcL,(\lVert-\rVert_\nu)_{\nu\in \Hom(E,\mbC)})$ is called \emph{admissible} if its first Chern form $\wh{c}_1(\mcL_{\nu})\in A^{1,1}(X_\nu(\mbC),\mbR)$ is harmonic for each $\nu$, see \cite[\S 2.3]{GS90-Ann-I}. 
We denote by $\wh{\rm{Pic}}{}^{\rm{adm}}(\mcX)$ the subvector space spanned by the classes of admissible hermitian line bundles on $\mcX$. By \cite[\S 2.6]{GS90-Ann-I}, the admissible hermitian scalar product $\lVert-\rVert_\nu$ is unique up to a locally constant scalar on each factor\footnote{In \cite[\S 2.6]{GS90-Ann-I} and its references, this constant is fixed by the K\"ahler form, see \cite[\S 1 equation ($*$)]{Arakelov74}}.
Note that although \cite{GS90-Ann-I} requires $\mcX$ to be regular and $\mfd=0$, the existence and uniqueness up to local constant of the admissible metric is a statement purely on the complex fiber. Thus, these assumptions can be safely omitted.

\begin{defn}
\begin{altenumerate}
\item Let $\wh{\mathrm{Pic}}_{\mathrm{vert}}(\mcX)$ denote the subvector space of $\wh{\mr{Pic}}(\mcX)$ of metrized line bundles $\wh{\mcL}$ such that the restriction $\mcL\vert_{\mcX_η}$ to the generic fiber of $\mcX$ is trivial. That is, we have an exact sequence
$$0\lr \wh{\mathrm{Pic}}_{\mathrm{vert}}(\mcX)\lr 
\wh{\rm{Pic}}{}^{\rm{adm}}(\mcX)\lr \mathrm{Pic}(\mcX_{\eta}).$$

\item Let $\wh{\Pic}_{\infty}(\mcX)\subseteq  \wh{\rm{Pic}}_\mr{vert}(\mcX)$ be the subvector space of metrized line bundles $\wh{\mcL}$ such that $\mcL \cong \mcO_\mcX$.
\end{altenumerate}
\end{defn}

\begin{lem}\label{lem:picard pair deg zero}
Let $\mcL$ be a line bundle whose generic fiber $\mcL_η = \mcL\vert_{\mcX_η}$ is trivial. Then $\mcL$ is of the form $\mcO(D)$ for a rational Cartier divisor $D$ supported in the fibers over those places of $O_E^\mfd$ where $\mcX\to \Spec(O_E^\mfd)$ is not smooth. 
\end{lem}
\begin{proof}
Choose any isomorphism $s:\mcO_{\mcX_η}\simto \mcL_η$. Then $s$ extends to a meromorphic section of $\mcL$ by the flatness of $\mcX\to \Spec(O_E^\mfd)$ and we obtain $\mcL \simto \mcO_\mcX(D')$ for a vertical Cartier divisor $D'$ on $\mcX$. By the Stein factorization argument from \cite[Lemma 4.1]{MZ}, $D'$ is rationally equivalent to a a divisor support only over singular fibers.
\end{proof}

Recall the $\mbQ$-vector space $\mcZ'_1(\mcX)$ from Definition \ref{def:Z_prime}, i.e. the vector space of $1$-cycles up to fiber-wise rational equivalence. Write $\mbR_\mfd = \mbR/\sum_{p\mid \mfd} \mbQ \log(p)$. Recall from \cite[Proposition 2.3.1]{BGS94-JAMES} that we have a well-defined intersection pairing:
\begin{equation}\label{equ:arithmetic intersection}
    \wh{\Pic}(\mcX)\times \mcZ_1'(\mcX)\lr \mbR_{\mfd},
\end{equation}
see also the remarks on p. 941 of \cite{BGS94-JAMES}. 
To be more precise, the left hand side of \eqref{equ:arithmetic intersection} has a an intersection pairing map to $\wh{\rm{Ch}}{}^1(O_E^{\mfd})$. 
The pairing \eqref{equ:arithmetic intersection} is obtained by composing with the arithmetic degree $\wh{\rm{Ch}}{}^1(O_E^{\mfd})\to \mbR_\mfd$.

Let $\mcZ'_1(\mcX)^{\perp} \subseteq \mcZ'_{1}(\mcX)$ denote the orthogonal complement of $\wh{\Pic}_{\mr{vert}}(\mcX)$. 
Let $\Pic^{\rm{adm}}(\mcX_\eta)$ be the image of the map $\wh{\Pic}{}^{\rm{adm}}(\mcX)\to \Pic(\mcX_\eta)$. Then \eqref{equ:arithmetic intersection} restricts/descends to an arithmetic intersection pairing:
\begin{equation}\label{equ:almost intersection}
(\cdot,\cdot)^\text{adm}: \, \Pic(\mcX_\eta)^\text{adm}\times \mcZ'_{1}(\mcX)^{\perp}  \lr \mbR_\mfd.
\end{equation}

We resume our notation from \S \ref{sec:Shimura}. In particular, $\mcM$ is the punctured integral model of the RSZ Shimura variety defined over $\Spec O_E^{\mfd}$ (\S \ref{sec:integral of RSZ Shimura}).

Let $\nu$ be a finite place of $E$ such that the base change $\mcM_{O_{E_\nu}}$ is not regular. Recall from Definition \ref{def:lb_mod_forms} and Definition \ref{def:KR_strata} that the Kottwitz--Rapoport divisors $\mcM_\nu^{\mcZ}$ and $\mcM^{\mcY}_\nu$ are Cartier divisors in $\mcM$ supported at $\nu$, with their intersection $\mcM_\nu^{\mcZ}\cap \mcM_\nu^{\mcY}$ having codimension $2$ in $\mcM$. 
By Proposition \ref{prop:global-KR-Cartier}, the underlying reduced loci of the Kottwitz--Rapoport strata, denoted by $\mcM_{\nu,\red}^{\mcZ}$ and $\mcM_{\nu,\red}^{\mcY}$, are smooth over $\mbF_\nu$ and thus decompose into disjoint unions of irreducible components.
This induces a decomposition of Cartier divisors
$$
\mcM_\nu^{\mcZ}=\coprod_{i_\nu} \mcM^{\mcZ}_{i_\nu},\quad \mcM_\nu^{\mcY}=\coprod_{j_\nu} \mcM^{\mcY}_{j_\nu}.
$$
By construction, we have 
$$\mbF_\nu\otimes\mcM=\bigcup_i (\mbF_\nu\otimes\mcM^{\mcZ}_{i_\nu})\cup \bigcup_j (\mbF_\nu\otimes\mcM^{\mcY}_{j_\nu}),$$ 
which expresses $\mbF_\nu\otimes\mcM$ as a union of Weil divisors in $\mcM$.

Recall that for a normal scheme $\mcX$, the first Chern class map
$$c_1:\Pic(\mcX)\lr \mathrm{Ch}^1(\mcX)$$ 
is injective, see \cite[\S 7 Proposition 2.14]{Liu-AGbook}. A \emph{$\mbQ$-Cartier divisor} is a Weil divisor in $\mathrm{Ch}^1(\mcX)$ (recall that we always take $\mbQ$-values for $\Pic$ and $\mathrm{Ch}^1$) that lies in the image of $c_1(\Pic(\mcX))\subseteq \mathrm{Ch}^1(\mcX)$. 
\begin{prop}\label{prop:vert picard group}
Let $\nu$ be a finite place of $E$ such that $\mcM_{O_{E_\nu}}$ is not smooth.
The special fibers of the Kottwitz--Rapoport divisors $(\mbF_\nu\otimes\mcM^{\mcZ}_{i_\nu})$ and $(\mbF_\nu\otimes\mcM^{\mcY}_{j_\nu})$ are $\mbQ$-Cartier divisors of $\mcM$.
\end{prop}
\begin{proof}
Denote by $\mcU:= \mcM \setminus \mcM_\nu^{\mcZ\cap\mcY}$ the complement of the linking stratum above $ν$. By Proposition \ref{prop:global-KR-Cartier}, $\mcU \to \Spec(O_E^\mfd)$ is smooth above $ν$. The pull-back of line bundles from $\mcM$ to $\mcU$ induces the following commutative diagram:
\begin{equation*}
\begin{aligned}
\xymatrix{
0\ar[r]&\Pic_{\nu}(\mcU)\ar[r]&\Pic(\mcU)\ar[r]&\Pic(\mcU \setminus (\mbF_ν \tensor \mcU))\\
0\ar[r]&\Pic_{\nu}(\mcM)\ar[r]\ar[u]^{i_{\nu}}&\Pic(\mcM)\ar[r]\ar[u]^{i}&\Pic(\mcM \setminus (\mbF_ν\tensor \mcM))\ar[u]^{j}
}
\end{aligned}
\end{equation*}
By \cite[Lemma 2.17 in \S 9.2.3]{Liu-AGbook}, the middle morphism $i$ is injective, hence $i_{\nu}$ is also injective.
It suffices to show that $\Pic_{\nu}(\mcU)$ is generated by the restrictions $\mcO(-\mcM^{\mcZ}_{i_\nu})|_{\mcU}$ and $\mcO(-\mcM^{\mcY}_{i_\nu})|_{\mcU}$.

Since $\mcU$ is smooth above $ν$, we have $\Pic_{\nu}(\mcU)\cong \mathrm{Ch}^1_{\nu}(\mcU)$, where $\mathrm{Ch}^1_{\nu}(\mcU)$ is the subgroup of $\mathrm{Ch}^1(\mcU)=\mathrm{Ch}^1(\mcU)\otimes\mbQ$ generated by irreducible components of the special fibers at $\nu$, i.e., Weil divisors that are linear combinations of irreducible components of $(\mbF_\nu\otimes\mcM^{\mcZ}_{i_\nu})$ and $(\mbF_\nu\otimes\mcM^{\mcY}_{j_\nu})$. Since $\mcU$ is smooth above $ν$, the Cartier divisors $\mcM^{\mcZ}_{i_\nu}$ and $\mcM^{\mcY}_{j_\nu}$ are integral multiples of $(\mbF_\nu\otimes\mcM^{\mcZ}_{i_\nu})$ and $(\mbF_\nu\otimes\mcM^{\mcY}_{j_\nu})$, respectively. The assertion now follows.
\end{proof}

Let $Y\subset \mcM$ be a closed subscheme which is one-dimensional away from a finite set of places. Then we have the following commutative diagram (see Remark \ref{rmk:two def of cycles}):
\begin{equation}\label{equ:map to Z1M}
\begin{aligned}
\xymatrix@R=0.2pc{
\mathrm{Gr}_1G_0(Y)\ar[r]^-{\sim}\ar[dd]^{\sim}& \mathrm{Gr}_1G_0^Y(\mcM)\ar[rd]\ar[dd]^{\sim}&\\
&&\mcZ_1(\mcM)\\
\mathrm{Ch}_1(Y)\ar[r]^-{\sim}& \mathrm{Ch}_{1,Y}(\mcM)\ar[ur]&
}
\end{aligned}
\end{equation}
More generally, one can consider $Y\to \mcM$ to be a finite map whose image is one-dimensional away from a finite set of places. Then the diagram \eqref{equ:map to Z1M} is still commutative, but the horizontal maps are not necessarily isomorphism.

Let $C_\bullet$ be a bounded complex on $\mcM$ such that the support of $\mcH^i(C_\bullet)$ is proper over $\Spec(O_E^\mfd)$ and empty over the generic fiber. Recall that its \emph{arithmetic Euler--Poincaré characteristic} is defined by
\begin{equation*}
\chi(C_\bullet):=\sum_{i,j} (-1)^{i+j}\log \# H^j(\mcM, \mcH^i(C_\bullet)).
\end{equation*}
Let $s\in\mcL$ be a regular section of a hermitian line bundle $\wh{\mcL}$ on $\mcM$ and let $\mcF_\bullet$ be a perfect complex on $\mcM$, acyclic outside a closed subset $Y\subseteq \mcM$ as above, and such that its class $[\mcF_\bullet]\in G_0(Y)$ lies in the filtration subgroup $F_1G_0(Y)$ generated by coherent sheaves with dimension of support $\leq 1$. Suppose $\mr{div}(s)$ does not intersection $Y$ generically. Then the intersection pairing  \eqref{equ:almost intersection} of $\wh{\mcL}$ and the class in $\mcZ'_1(\mcM)$ defined by $\mcF_\bullet$ can be compute using the following formula:
\begin{equation}\label{equ:explicit intersection product}
\langle \wh{\mcL}, [\mcF_\bullet]\rangle\ =\ \chi(\mcF_\bullet\otimes_{\mcO_{\mcM}}(\mcL^{-1}\overset{s}{\lr}\mcO_{\mcM}))+\sum_{\nu\in\Hom(E,\mbC)}\sum_{i\in\mbZ} (-1)^i g_\nu(\mcH^i(\mcF_\bullet)_{E_\nu})\in \mbR,
\end{equation}
where $\mcH^i(\mcF_\bullet)_{E_\nu}$ is a coherent module over the complex fiber $\mcM_\nu(\mbC)$ with zero dimensional support, and we took lengths to view it as a linear combination of points in $\mcM_\nu(\mbC)$. This defines an intersection number valued in real number, which depends on the choice of the scetion $s\in\mcL$. (The image in $\mbR_\mfd$ does not depend on $s$.)

Let $\mcZ\subset \mcM$ be an effective Cartier divisor with associated line bundle $\mcO(-\mcZ)$. For each $\nu\in \Hom(E,\mbC)$, let $\mcG_\nu$ be a Green current associated to the analytic divisor $\mcZ_\nu\subset \mcM_\nu$ on the complex fiber. The pair $(\mcZ_\nu,\mcG_\nu)$ determines a hermitian line bundle $(\mcZ_\nu,\lVert-\rVert_\nu)$ on the complex fiber $\mcM_\nu(\mbC)$ (see \cite[\S 2]{GS90-Ann-I}). By a slight abuse of notation, we denote by $(\mcZ,(\mcG_\nu)_{\nu})$ the corresponding hermitian line bundle in the arithmetic Picard group (with a meromorphic section $s$ of the hermitian line bundle which defines this Cartier divisor).

For $?=\bf K, \bf B$ and $\xi \in F_{0}$,  consider the \emph{arithmetic Kudla--Rapoport divisor} (here $\nu|w$ are archimedean places of $E/F_0$): 
 \begin{equation}\label{equ:arithemtic KR divisor}
 	\widehat {\mathcal Z}^{?} (\xi, h_\infty, \phi)= (\mathcal Z (\xi, \phi), ( \mathcal{G}^{?}_{\nu}(\xi, h_w, \phi))_{ν | w | \infty} ) \in \widehat{\mathrm{Pic}}(\mathcal M).  
 \end{equation}
in the arithmetic Picard group as follows: 
\begin{itemize}
\item When $\xi\neq 0$, the first component $\mcZ(\xi,\phi)$ is defined in Definition \ref{def:KR}. And $\mcG_\nu^{\mathbf K}$ (resp. $\mcG_\nu^{\mathbf B}$) is the \emph{Kudla's Green current} (resp. \emph{automorphic Green current}). Set $\widehat {\mathcal Z}^? (\xi, \phi)=\widehat {\mathcal Z}^? (\xi, 1, \phi)$. We refer the reader to \cite[\S 3]{MZ} for their precise definition.
\item When $\xi=0$, we define arithmetic Kudla--Rapoport divisor using the Hodge bundle, see \cite[Definition 13.3, (13.5) and (13.6)]{ZZhang21}.
\end{itemize}

Let $\CM^{\mbL}(\alpha,\mu)\in F_1G_0(\CM(\alpha,\mu))$ be the modified derived CM cycle from Definition \ref{def:classical hecke} (2). It can also be defined in terms of the complex \eqref{eq:def_in_terms_of_K_group}, to which Equation \eqref{equ:explicit intersection product} applies for suitable $(\wh{\mcL}, s)$.

\subsection{Arithmetic transfer for maximal parahoric level}\label{sec:AT conj}
We next state our arithmetic transfer theorem, completely following \cite{ZZhang21}.
Let $F/F_{0}$ be an unramified quadratic extension of $p$-adic local fields where $p>2$. Let $q$ be the size of the residue field of $F_0$. Let $L \subseteq V$ be a vertex lattice of type $0 \leq t \leq n$ in a $n$-dimensional hermitian space $V$ over $F$. Let $\mathbb V$ be an $n$-dimensional hermitian space over $F$ that is not isomorphic to $V$. 

Choose an orthogonal basis $\{e_i\}_{i=1}^n$ of $L$. Choose a Galois involution $\overline{(-)}$ on $V$ with fixed subspace $V_0=\oplus_{i=1}^n F_0 e_i$. Let $L_0=\oplus_{i=1}^n O_{F_0}e_i$. Denote by $V_0^*$ the linear dual of $V_0$, and set $L_0^\vee=L^\vee \cap V_0$.

Consider the symmetric space over $F_0$:
\begin{equation}
		S(V_0)=\{ \gamma \in \GL(V) | \gamma \ov \gamma =\id \}.
\end{equation}
Then $\GL(V_0)$ acts on the space $S(V_0) \times V_0 \times V_0^*$ by
$$
	h.(\gamma, \, u_1, \, u_2)=(h^{-1}\gamma h, \, h^{-1}u_1, \, u_2h).
$$
Let $(S(V_0) \times V_0 \times V_0^*)(F_0)_{\mathrm{rs}}$ be the subset of regular semi-simple elements \cite[\S 2.2.]{Zhang21}.

Define $L_0^\vee:=L^\vee \cap V_0$. We normalize the Haar measure on $\GL(V_0)(F_0)$ such that the stabilizer $\GL(L_0, L_0^{\vee})$ of the lattice chain $L_0 \subseteq L_0^\vee$ is of volume $1$. 
	
For a Schwartz function $f' \in \mathcal{S}((S(V_0) \times V_0 \times V_0^{*})(F_0))$ and a regular semi-simple element $(\gamma,u_1,u_2) \in  (S(V_0) \times V_0 \times V_0^*)(F_0)_{\mathrm{rs}}$, define the one-variable orbital integral:
	\begin{equation}
		\Orb( (\gamma, u_1, u_2), f', s):= \omega(\gamma, u_1, u_2)  \int_{h \in \GL(V_0)} f'(h.(\gamma,u_1,u_2)) \eta(h) |h|^{s} d h, \quad s \in \mathbb C.
	\end{equation}
Here $ \eta(h):=\eta(\det h)$ and $|h|^s:=|\det h|^s$. And we define the \emph{transfer factor}
	\begin{equation}\label{transfer factor: semi-Lie}
		\omega(\gamma, u_1, u_2) :=\eta( \det (\gamma^{i} u_1 )_{i=0}^{n-1}) \in \{\pm 1\}.
	\end{equation} 
Define the derived orbital integral:
	\begin{equation}\label{derived orbital int}
		\partial\!\Orb((\gamma, u_1, u_2), f' \bigr):=\frac{d}{ds}\Big|_{s=0}  \Orb((\gamma, u_1, u_2), f', s).
	\end{equation}

	\begin{defn} \label{defn: semi-Lie match elements}
We say that $(g,u) \in (\mathrm{U}(\mathbb V) \times \mathbb V )(F_0)$ and $(\gamma,u_1,u_2) \in  (S(V_0) \times V_0 \times V_0^*)(F_0)$ \emph{match} if the following matching of \emph{invariants} holds:
	\[
	\det(T \id_{\mathbb V}+g)=\det(T \id_{V}+\gamma) \in F[T], \quad (g^iu,u)_{\mathbb V}=u_2 (\gamma^iu_1), \quad 0 \leq i \leq n-1.
	\]	
\end{defn}

Consider the following test functions on $(S(V_0) \times V_0 \times V_0^* )(F_0)$:
$$
f_{L}:=1_{S(L_0, L_0^{\vee})} \times 1_{L_0} \times 1_{(L^{\vee}_0)^*},\quad
f_{L^\vee} := 1_{S(L_0, L_0^{\vee})} \times 1_{L_0^\vee} \times 1_{(L_0)^*},
$$
where $S(L_0,L_0^\vee)$ is the stabilizer of the lattice chain $L\subseteq L^\vee$ inside $S(V_0)(F_0)$.
Recall from \eqref{eq:def_int_numbers} the definition of regular semi-simple pairs $(g,u)$ and the intersection numbers $\wt{\Int}{}^{\mcZ}(g,u)$ and $\wt{\Int}{}^{\mcY}(g,u)$.
\begin{thm}[\protect{see \cite[Conjecture 6.4]{ZZhang21}}] \label{thm: AT vertex level}
For any regular semi-simple pair $(g, u) \in (\mathrm{U}(\mathbb V) \times \mathbb V)(F_{0})$ matching $(\gamma, u_1, u_2) \in ( S(V_{0}) \times V_{0} \times V_{0}^* )(F_{0})$, there are the following equalities in $\mathbb Q \log q$:
\begin{equation}
\begin{array}{rcr}
\partial\!\Orb ((\gamma,u_1,u_2),  f_{L}) & = & - \wt {\Int}{}^{\mcZ}(g, u) \log q,\\[1mm]
\partial\!\Orb ((\gamma,u_1,u_2),  f_{L^\vee}) & = & - (-1)^{t} \wt {\Int}{}^{\mcY}(g, u) \log q.
\end{array} \mkern -80mu \tag{$\mathrm{AT}(F/F_0, L, g, u)$}
\end{equation}
\end{thm}    
This is proved in \cite[Theorem 1.1]{ZZhang21} when $F_0$ is unramified over $\mathbb Q_p$ for $p>2$. Our goal in the remainder of this section is to extend this proof to the general case. Note that the assumption on $F_0$ in \cite{ZZhang21} is only needed during the global part of the proof, and we focus on this aspect.

\subsection{Proof of Theorem \ref{thm: AT vertex level}}\label{sec:AT proof}
We switch to global notations. So $F_v/F_{0,v}$ is an unramified quadratic extension of $p$-adic local fields with $p>2$. And $L_v \subseteq V_v$ is a vertex lattice in an $n$-dimensional hermitian space $V_v$ over $F_{v}/F_{0,v}$. In this notation, we need to show the following statement.

\medskip

\begin{center}
\begin{minipage}{10cm}
\begin{center}
\emph{For every regular semi-simple pair $(g_v, u_v) \in (\mathrm{U}(\mathbb V_v) \times \mathbb V_v)(F_{0,v})$, the two identities $\mathrm{AT}(F_v/F_{0,v}, L_v, g_v, u_v)$ hold.}
\end{center}
\end{minipage}
\end{center}

\begin{proof}
We follow the proof of \cite[Theorem 1.1]{ZZhang21} in \cite[\S 15]{ZZhang21}. By the duality isomorphism of RZ spaces (see \cite[\S 5.1]{ZZhang21}), it suffices to prove the first identity of Theorem \ref{thm: AT vertex level}, and the second will follow.

\medskip

\noindent\textbf{Step 1.} We begin by globalizing our group-theoretical setup. We choose 
\begin{itemize}
    \item a CM extension of a totally real field $F/F_0$ and a finite place $v$ of $F$ such that $F_0 \not =\mathbb Q$ (so our RSZ Shimura varieties are projective), and the $v$-adic completion of $F/F_0$ is $F_v/F_{0,v}$. Moreover, we assume every other $p$-adic place $w \not = v$ of $F_0$ is split in $F$. We also assume that $F/F_0$ is ramified somewhere.
    \item a CM type $\Phi$ for $F/F_0$ and a distinguished embedding $\phi_0 \in \Phi$.
    \item an $F/F_0$-hermitian space $V$ of signature $(n-1,1)_{\phi_0}, (n,0)_{\phi \in \Phi - \phi_0}$ with a vertex lattice $L \subseteq V$ such that $L \otimes_{O_F} O_{F_v}=L_{v}$.
\end{itemize}  

For any place $w$ of $F_0$, recall that the nearby $F/F_0$-hermitian space  $V^{w}$ of $V$ at $w$ is defined as:
\begin{enumerate}
    \item If $w$ is finite, then $V^{w}$ is positive definite at all archimedean places of $F$. At the non-archimedean places it is isomorphic to $V$ except at $w$. In particular, can identify $V^v_v$ with the given local hermitian space $\mbV_v$.
    \item If $w$ is archimedean, then $V^w$ has signature $(n-1,1)$ at $w$ and is positive definite at all other places. Moreover, $V^w$ is isomorphic to $V$ at all non-archimedean places. In particular $V^{ϕ_0}$ is nothing but $V$ itself (up to isomorphism).
\end{enumerate}

Denote by $q_w$ the size of the residue field of $F_{0,w}$ for a finite place $w$ of $F_0$. Recall that $F_{0,+}$ denotes the cone of totally positive elements of $F_0$. By local constancy of orbital integrals and intersection numbers \cite[Theorem 1.2]{mihatsch2022local} which applies in our situation (see \cite[Remark 6.16]{ZZhang21}), we may choose a global regular semi-simple pair $(g_0,u_0) \in (\U(V^{v}) \times V^{v})(F_0)$ that is $v$-adically close enough to $(g_v,u_v)$ such that 
\begin{enumerate}
    \item Conjecture $\mathrm{AT}(F_{v}/F_{0,v}, g_v, u_v)$ and $\mathrm{AT}(F_{v}/F_{0,v}, g_0,u_0)$ are equivalent.
    \item $\xi_0=(u_0,u_0)$ is a totally positive element. 
\end{enumerate}
Let $\alpha \in F[t]$ be the characteristic polynomial of $g_0$.

We choose an integer $\mfd \in \mbZ_{\geq 1}$ such that 
\begin{itemize}
    \item $p$ does not divide $\mfd$;
    \item every $w \nmid \mfd$ is unramified in $F$, which is possible since there are only finitely many ramified place and all $p$-adic places are unramified;
    \item for every finite place $w\nmid \mfd$ unequal $v$, $L_w$ is self-dual; 
    \item for every finite inert place $w\nmid \mfd$ unequal $v$, the ring $O_{F_w}[t]/(\alpha(t))$ is regular, hence a product of discrete valuation rings.
\end{itemize}

\medskip
\noindent\textbf{Step 2.}
Next, we globalize the local intersection numbers from the AT conjecture. We work with the punctured integral models of the RSZ Shimura varieties
$$
\mathcal M \lr \Spec O_E^\mfd
$$
as in \S\ref{sec:integral of RSZ Shimura}. The level here is required to be of the form $K_{Z^{\mathbb Q}} \times K$ and to satisfy the assumptions at the beginning of \S\ref{sec:integral of RSZ Shimura}; in particular, $K^\mfd=\U(L^\mfd)$. Recall that the reflex field is given by $E=F^{\Phi}\varphi_0(F)$. We have also used the assumption that $F/F_0$ ramifies somewhere to define the moduli space $\mcM_0$.

Choose a $\mathbb Q$-valued test function $\Phi=f \otimes \phi \in \mathcal S((\U(V)\times V)(\mathbb A_{0,f}))^{K}$ such that for $w \nmid \mfd$, we have  $\Phi_{w}=1_{\U(L_w)}\otimes 1_{L_w}$. This defines a global intersection number as follows.

For the $f$-part, recall from Definition \ref{def:classical hecke} (3) the derived Hecke CM cycle
$$
{}^{\mathbb L} \mathrm{CM} (\alpha, f) \in \mcZ'_1(\mathcal M).
$$
For the $\phi$-part, for $?=\bf K, \bf B, \bf K- \bf B$ and $\xi \in F_{0}$,  recall from \eqref{equ:arithemtic KR divisor} the \emph{arithmetic Kudla--Rapoport divisor} $\widehat {\mathcal Z}^{?} (\xi, h_\infty, \phi)\in \wh{\Pic}(\mcM)$ in the arithmetic Picard group of $\mathcal M$.

Define the normalization volume factor $\tau(Z^\mathbb Q)=\#Z^{\mathbb Q}(\mathbb Q) \backslash Z^{\mathbb Q}(\mathbb A_f)/K_{Z^\mbQ}$. 
For $\xi \in F_0$, we define the \emph{global arithmetic intersection number}
	\begin{equation*}
		\Int^{?}(\alpha, \xi, \Phi) := \frac
		{1}{\tau(Z^{\mathbb Q}) [E:F]}
		(\widehat {\mathcal Z }^? (\xi, \phi), {}^{\mathbb L} \mathrm{CM}
        (\alpha, f)) \in \mathbb R_{\mfd}. 
	\end{equation*}
Here $(-,-)$ is the truncated arithmetic intersection pairing defined in \eqref{equ:almost intersection}. 

Set $\Int(\alpha, \xi, \Phi)=\Int^{\bf B}(\alpha, \xi, \Phi)$. When $\xi$ is totally positive, then the supports of $\mcZ(ξ, ϕ)$ and $\CM^\mbL(α,f)$ do not meet generically by \cite[Theorem 9.2]{Zhang21} which directly extends to our setting (also see \cite[Theorem 14.5]{ZZhang21}). Hence formula \eqref{equ:explicit intersection product} applies to the canonical section $\mcO_\mcM\to \mcO_\mcM(\mcZ(ξ,ϕ))$ and lifts $\Int(α, ξ, Φ)$ to a real value $\mbR$. It decomposed into a sum of local contributions over the archimedean and inert places $w \nmid \mfd$ of $F$; split places do not contribute by the cited theorems:
\begin{equation*}
\Int(\alpha, \xi, \Phi)=\sum_{w| \infty} \Int_w^{\bf B}(\alpha, \xi, \Phi)   +  \sum_{ w \nmid \mfd\ \text{finite inert}} \Int_w(\alpha, \xi, \Phi) \in \mathbb R.
\end{equation*}
We consider generating series of these intersection numbers to apply modularity results of geometric theta series: for $h_\infty  \in \SL_2(F_{0, \infty})$, define
\begin{equation*}
		\Int_{\alpha}(h_\infty, \Phi) = \Int(
        \alpha, 0, \Phi) + \sum_{\xi \in F_{0,+} } \Int (\alpha, \xi, \Phi) W_{\xi}^{(n)}(h_\infty)
\end{equation*}
where $W_{\xi}^{(n)}(h_\infty)$ is the weight $n$ Whittaker function (for $\xi$) on $\SL_2(F_{0,\infty}) \cong \prod_{w|\infty} \SL_2(\mathbb R)$.

Applying the uniformization Theorem \ref{thm: basic uniformization RSZ} together with Proposition \ref{prop:uni-KR} and Proposition \ref{prop:uni-CM} in the same way as in \cite[Theorems 9.4 and 10.1]{Zhang21}, also see \cite[Proposition 14.6]{ZZhang21}, we obtain the orbital integral decomposition of global intersection numbers into local ones:
\begin{prop}\label{prop:uni intersection number}
Let $\xi \in F_{0,+}$ be totally positive.
\begin{altenumerate}
    \item If $w \not =v$ is an inert place of $F/F_0$ not dividing $\mfd$, then
\[
\Int_{w}(\alpha, \xi,\Phi) =  
			2\log q_{w} \sum_{(g,u)\in [(\U(V^{(w)})(\alpha)\times V^{(w)}_\xi)(F_0)]} \Int_{w}(g,u ) \cdot \Orb\left((g,u), \Phi^{w} \right).
\]
Here, $\Int_{w}(g,u) \in \mathbb Q$ is the arithmetic intersection number that appears in the AFL for the quadratic field extension $F_{w}/F_{0,w}$ and the regular semi-simple pair $(g,u) \in (\U(V^{(w)}) \times V^{w})(F_{0,w})$. 
    \item If $w=v$, then 
 \[
\Int_{w}(\alpha, \xi,\Phi) =  
			2\log q_{w} \sum_{(g,u)\in [(\U(V^{(w)})(\alpha)\times V^{(w)}_\xi)(F_0)]} \Int_{w}(g,u ) \cdot \Orb\left((g,u), \Phi^{w} \right).
\]
Here, $\Int_{w}(g,u )=\wt {\Int}{}^{\mcZ}(g, u)$ is the arithmetic intersection number that appears in the main Theorem \ref{thm: AT vertex level}.  \qed
    
\end{altenumerate}
\end{prop}

For an archimedean place $w| \infty$ of $F_0$, by complex uniformization of RSZ Shimura varieties, we recall the following result.
\begin{prop}[\protect{\cite[Theorem 10.1, Corollary 10.3]{Zhang21}}] If $w|\infty$, then
\begin{equation}
	\Int_{w}^{\bf K}(\alpha, \xi,\Phi) =  
\sum_{(g,u)\in [(\U(V^{(w)})(\alpha)\times V^{(w)}_\xi)(F_0)]} \Int_{w}(g,u ) \cdot \Orb\left((g,u), \Phi \right).
\end{equation}
Here $\Int_{w}(g,u)={\mathcal G}^{\bf K}(u,1)(z_g)$ where $z_g$ is the unique fixed point of $g$ on the symmetric space $X_w$ for $G(F_w)$.\qed
\end{prop}

\medskip
\noindent\textbf{Step 3.} Next, we also globalize the orbital integral side of the AT conjecture, which is literally the same as in \cite{Zhang21,ZZhang21}. Denote by $\omega$ the Weil representation of $\SL_2(\mathbb A_{F_0}) \times \GL(V_0)(\mbA_{F_0})$ on $V'(\mathbb A_{F_0})$, see \cite[\S 11.1]{Zhang21}.

For every Gaussian decomposable function $\Phi' \in \mathcal S((S(V_0) \times V')(\mathbb A_{F_0}))$, and for $h \in \SL_2(\mathbb A_{F_0})$, consider the regularized integral ($s \in \mathbb C$): 
	\begin{equation}\label{defn of BJ}
		\mathbb J_{\alpha}(h,\Phi',s)=\int_{g \in [\GL(V_0)] } \left(\sum_{(\gamma, u')\in (S(V_0)(\alpha)\times V')(F_0)}\omega(h)\Phi'(g^{-1}. (\gamma, u'))\right)   |g|^s\eta(g)  dg.
	\end{equation}

Applying the geometric decomposition of the relative trace formula \cite[(12.32) and (12.38)]{Zhang21}, we find a local-global decomposition
	\begin{equation*}
\mathbb J_{\alpha}(h,\Phi',s)= \mathbb J_{\alpha}(h,\Phi',s)_{0}+ \sum_{(\gamma,u')\in [(S(V_0)(\alpha) \times V')(F_0)]_{\mathrm{rs}}} \Orb((\gamma,u'),\omega(h)\Phi',s),
	\end{equation*}
	where $\mathbb J_{\alpha}(h,\Phi',s)_{0}$ is the geometric term over nilpotent orbits (see \cite[\S 12.6]{Zhang21}).

For $\xi \in F_0^\times$ and $w$ a finite place of $F_0$, set 
\begin{equation*}
	\partial \mathbb J_{\alpha, w}(\xi, h, \Phi') := \sum_{(\gamma,u')\in [(S(V_0)(\alpha) \times V'_\xi)(F_0)]_{\mathrm{rs}} }  \partial \Orb((\gamma,u'), \omega(h)\Phi'_w) \cdot  \Orb((\gamma,u'), \omega(h)\Phi'^{w}).
\end{equation*}
This is the $\xi$-th Fourier coefficient of $\partial\mathbb J_{\alpha}(h,\Phi'):=\frac{\mathrm{d}}{\mathrm{d}s}\bigl|_{s=0}J_{\alpha}(h,\Phi',s)$ and is zero unless $\xi \in F_{0,+}$, see \cite[(14.5)]{Zhang21}.

\medskip
\noindent\textbf{Step 4.}
We now apply the modification construction from \cite[\S 12-13]{ZZhang21} to obtain modular forms. The goal is to obtain a \emph{modified} derived 1-cycle on $\mathcal M$
\begin{equation}\label{equ:modi cycle}
{}^{\mathbb L} \mathrm{CM} (\alpha, f)^{\mathrm{mod}} \in \mcZ_1'(\mcM)^\perp
\end{equation}
such that the difference ${}^{\mathbb L} \mathrm{CM} (\alpha, f)^{\mathrm{mod}}- {}^{\mathbb L} \mathrm{CM} (\alpha, f)$ is an explicit $1$-cycle on $\mathcal {M}$ whose intersection pairing against Kudla--Rapoport divisors is simple to understand. This difference will be specified in \eqref{equ:geometric modufication}. Being an element of the perp space in \eqref{equ:modi cycle} concretely means that
\begin{enumerate}
\item It is of degree $0$ on each connected component of each complex fiber $\mcM(\mathbb C)$ of $\mathcal M$.
    
\item Its intersection pairing against every vertical component of every geometric closed fiber of $\mathcal M$ is zero. 
\end{enumerate}
The point is that we may take the almost intersection pairing \eqref{equ:almost intersection} of ${}^\mbL\CM(α, f)^{\mr{mod}}$ against the geometric theta series $Z(h,\phi)$. This results in a a holomorphic modular form by the modularity of geometric theta series, see \cite[Theorem 8.1]{Zhang21}.

We call a polynomial $\alpha_0 \in O_F^\mfd[t]^\circ_{\deg n}$ of \emph{maximal order type} if $O_{F,w}[t]/(\alpha_0(t))$ is a product of discrete valuation rings for any finite place $w \nmid \mfd$ and moreover finite \'etale when $w=v$. Note that if $\alpha_0$ is of maximal order type, then $\mathrm{AT}(F_{v}/F_{0,v}, g_0,u_0)$ is known by \cite[\S 9]{ZZhang21}. In this situation, we already know that 
\begin{equation*}
2 \mathbb J_{\alpha_0} (\xi_0, \Phi') + \Int^{\mathbf K- \mathbf B} (\alpha_0, \xi_0, \Phi) + \Int(\alpha_0, \xi_0, \Phi)=0
\end{equation*}
by the local-global decomposition from Proposition \ref{prop:uni intersection number}, and the known case of arithmetic transfer for maximal orders from \cite[Theorem 9.14]{ZZhang21}.

The construction of the modified cycle \eqref{equ:modi cycle} is as follows. Over $\mathbb C$, we choose another derived CM cycle $\mathrm{CM}^{\mathbb L}(\alpha_0, f_0)$ where, after possibly enlarging $\mfd$, $α_0$ is of maximal order type, and where the test function $f_0 = f_{0,\mfd}\tensor 1_{K^\mfd}$ with $f_{0,\mfd} = 1_{\mu_{\mfd}}$ for some double coset $\mu_\mfd$. Choosing a suitable double coset, we arrange that the generic fiber of $\CM(α_0, μ_\mfd)$ is non-empty. Then the generic fiber of $\mathrm{CM}^{\mathbb L}(\alpha_0, f_0)$ has positive degree.

We compute the difference of the degrees of $\mathrm{CM}^{\mathbb L}(\alpha, f)$ and $\mathrm{CM}^{\mathbb L}(\alpha_0, f_0)$ on geometric connected components of the generic fiber of $\mcM$. As the Shimura datum for $\wt{G}$ is a product of Shimura data for $\Res_{F_0/\mathbb Q}(G)$ and $Z^{\mathbb Q}$, these CM cycles in the generic fiber all come by pullback along $\mcM_E \to M_G \tensor_F E$, where $M_G$ is the Shimura variety for $G$ with level $K$. Here, since our signature of $V$ is of strict Drinfeld type with respect to $Φ$, $M_G$ has reflex field $F$ and is moreover connected by \cite[Corollary 12.6]{ZZhang21}. Hence, the generic degrees of $\mathrm{CM}^{\mathbb L}(\alpha, f)$ and $\mathrm{CM}^{\mathbb L}(\alpha_0, f_0)$ are independent of the geometric connected component. So there is a constant $c$ such that the degree of $\mathrm{CM}^{\mathbb L}(\alpha, f)- c \mathrm{CM}^{\mathbb L}(\alpha_0, f_0)$ is zero on each geometric connected component.

Over $\mathbb F_p^{\mathrm{alg}}$, the modification follows the same logic as in \cite[\S 12.4]{ZZhang21} using the \emph{very special $1$-cycles} from \cite[Definition 12.14]{ZZhang21}. We add some details, because we use arithmetic Picard groups instead of arithmetic Chow groups.

Note that by Lemma \ref{lem:picard pair deg zero} and Proposition \ref{prop:vert picard group}, $\wh{\mathrm{Pic}}_{\mathrm{vert}}(\mathcal{M})$ is generated by irreducible components of closed fibers. These components are exactly the connected components of the Kottwitz--Rapoport $\mcZ$ and $\mcY$-strata because these strata are smooth.

First, consider the case when the type $t_v$ of $L_v$ is $1$. The proofs of \cite[Propositions 12.16 and 12.17]{ZZhang21} work in the same way, by relating the reduced Kottwitz--Rapoport $\mcZ$-stratum and $\mcY$-stratum in $\mbF_v\otimes \mcM$ to the Balloon-Ground stratification in the reduced locus of the relative RZ space $\mcN_{(n-1,1),\red}^{[1]}$.
Therefore, we may invoke \cite[Propositions 12.16 and 12.17]{ZZhang21} in our setting.
It is clear from the definition that the Balloon stratum lies in the reduced Kottwitz--Rapoport $\mcZ$-stratum $\mcN^{[1],\mcZ}_{(n-1,1),\red}$. In particular the reduced Balloon stratum is a disjoint union of dimension $n-1$ projective spaces by \cite{KR-BG}, see also \cite[\S 14.2]{CRZ-quasiAFL} for a precise statement.\footnote{The reduced Kottwitz--Rapoport $\mcZ$-stratum defined in Definition \ref{def:KR strata RZ} by the line bundle of modular forms is not the union of Balloon strata in $\mcN_{(n-1,1),\red}^{[t]}$ defined in \cite[\S 7]{ZZhang21}, see \cite[\S 2.4]{HLS}, where the construction and proof carry over directly to unramified extensions $F/F_0$.} 
Therefore, the construction of the very special $1$-cycles in \cite[Proposition 12.18]{ZZhang21} carries over to our setting when $t_v = 1$, via the uniformization in Theorem \ref{thm: basic uniformization RSZ}. The same argument applies to the case $t_v=n-1$.

Now consider the case when $1<t_v<n-1$. We need an analog of \cite[Proposition 12.24]{ZZhang21} which says that the reduced Kottwitz--Rapoport strata are irreducible in any connected component of $\mbF_v\otimes\mcM$. Note that the axiom (4c) of Rapoport–-He holds in our context by \cite[Corollary 1.5]{gleason2022connected}. 
Thus, the proof of \cite[Proposition 12.24]{ZZhang21} carries over to our setting and, combined with \cite[Proposition 12.25]{ZZhang21}, yields the desired very special $1$-cycles in the case $1<t_v<n-1$.

All in all, the above arguments showed that we there exists a modified CM cycle of the form  
\begin{equation}\label{equ:geometric modufication}
{}^{\mathbb L} \mathrm{CM} (\alpha, f)^{\mathrm{mod}}= \mathrm{CM}^{\mathbb L}(\alpha, f)- c\,\mathrm{CM}^{\mathbb L}(\alpha_0, f_0) - SC_{v} \in \mcZ_1'(\mcM)^\perp,
\end{equation}
where $SC_v$ is a sum of explicit special $1$-cycle in the special fibers of $\mathcal{M}$ above places $\nu | v$ of $E/F_0$ as in \cite[\S 15.3, (15.4)]{ZZhang21}.

\medskip

\noindent\textbf{Step 5.} Next, we combine the analytic and geometric generating series from Steps 3 and 4 and construct a holomorphic modular form $\mathrm{Diff}(h)^{\mathrm{mod}}$: Consider the modified analytic generating series
$$
\partial \mathbb J_{\rm hol}(h)^{\mathrm{mod}}:=2 \partial\mathbb J_{\alpha} (h, \Phi')^{\mathrm{mod}} + \Int^{\mathbf K- \mathbf B} _{\alpha}(h, \Phi)^{\mathrm{mod}},
$$ 
with
$$
\Int^{\mathbf K- \mathbf B} _{\alpha}(h, \Phi)^{\mathrm{mod}} = \sum_{w| \infty} \Int_w^{\bf K- \bf B}(\alpha, \xi, \Phi) - c\sum_{w| \infty} \Int_w^{\bf K- \bf B}(\alpha_0, \xi, \Phi) 
$$
and 
$$
2 \partial\mathbb J_{\alpha} (h, \Phi')^{\mathrm{mod}}= 2 \partial\mathbb J_{\alpha} (h, \Phi')- 2 c\, \partial\mathbb J_{\alpha_0} (h, \Phi') - \Theta_{SC_v}(h, \phi')
$$ where $\Theta_{SC_v}(h, \phi')$ is an explicit holomorphic theta series whose $\xi$-th Fourier coefficient ($\xi >0$) is equal to $(\mathcal{Z}(\xi, \phi), SC_v)$, see \cite[Theorem 13.9 and \S 15.3]{ZZhang21}.

\begin{prop}
The function $\partial \mathbb J_{\rm hol}(h)^{\mathrm{mod}}$ is a \emph{holomorphic} modular form on $ h \in \SL_2(\mathbb A_{F_0})$ with Fourier coefficients in $\mbR$.
\end{prop}
\begin{proof}
By \cite[Proposition 10.5]{MZ}, $\partial \mathbb J_{\rm hol}(h)$ is modular. By \cite[Proposition 14.5.]{Zhang21}, we know 
\[
2 \mathbb J_{\alpha} (h, \Phi') -  2 c\, \mathbb J_{\alpha_0} (h, \Phi') + \sum_{w| \infty} \Int_w^{\bf K- \bf B}(\alpha, \xi, \Phi) - c\sum_{w| \infty} \Int_w^{\bf K- \bf B}(\alpha_0, \xi, \Phi) 
\]
is holomorphic. So the result follows from holomorphic modularity of the theta series $\Theta_{SC_v}(h, \phi')$ \cite[Theorem 13.9]{ZZhang21}.
\end{proof}

For functions $\Phi$ in Step $2$ and $\Phi'$ in Step $3$ that are standard partial transfers \cite[Definition 14.7]{ZZhang21}, we consider (as \cite[Theorem 15.7]{ZZhang21}) the modified difference function on $h \in \SL_2(\mathbb A_{F_0})$
\begin{equation*}
\mathrm{Diff}(h)^{\mathrm{mod}}:=\partial \mathbb J_{\rm hol}(h)^{\mathrm{mod},\Phi'} + \Int (h, \Phi)^{\mathrm{mod}}\in\mcA_{\mathrm{hol}}(\SL_2(\mbA_{F_0}),K,n)_{\ov{\mbQ}}\otimes_{\ov{\mbQ}}\mbR_{\mfd,\ov{\mbQ}}
\end{equation*}
where we project the holomorphic modular form $\partial \mathbb J_{\rm hol}(h)^{\mathrm{mod}}$ to the quotient $\mbR_{\mfd}$ and where
$$
\Int(h,\Phi)^{\mathrm{mod}}:= (\mcZ^{\mathbf B}(h,\phi), {}^{\mathbb L} \mathrm{CM} (\alpha, f)^{\mathrm{mod}})^{\mathrm{adm}}
$$ 
comes from the admissible intersection pairing \eqref{equ:almost intersection}. This $\Int(h, Φ)^\mr{mod}$ is a holomorphic modular form (with Fourier coefficients in $\mbR_{\mfd,\ov{\mbQ}}$) by the modularity of the generic fiber generating series of Kudla--Rapoport divisors in \cite[Theorem 8.1]{Zhang21}.

\medskip
\noindent\textbf{Step 6.}
Finally, by subtracting the matching terms (see Step 4) we gave
$$
\mathrm{Diff}(h)^{\mathrm{mod}}=2 \mathbb J_{\alpha} (\xi_0, \Phi') + \Int^{\mathbf K- \mathbf B} (\alpha, \xi_0, \Phi) + \Int(\alpha, \xi_0, \Phi).
$$
By double induction \cite[Lemma 15.1]{ZZhang21}, the holomorphic modular form $\mathrm{Diff}(h)^{\mathrm{mod}}$ is a constant, hence has zero $\xi_0$-th Fourier coefficient. See \cite[Theorem 15.7]{ZZhang21} for precise argument. 
Namely, 
\begin{equation*}
2 \mathbb J_{\alpha} (\xi_0, \Phi') + \Int^{\mathbf K- \mathbf B} (\alpha, \xi_0, \Phi) + \Int(\alpha, \xi_0, \Phi)=0.
\end{equation*}
From this, the local arithmetic transfer identity $\mathrm{AT}(F_v/F_{0,v}, L_v, g_v, u_v)$ $(1)$ follows after enlarging $\mfd$ and choosing a suitable transferring pair $(Φ, Φ')$ as in \cite[Proposition 10.2]{MZ}.
\end{proof}

\appendix

\part{Appendix}

\section{Compatibility of $\BT$ with Faltings duality}

Let $K$ be a $p$-adic local field with ring of integers $O = O_K$, let $π\in O$ be a uniformizer, and let $q$ be the residue cardinality.

The aim of this appendix is to prove that the functor $\BT$ from Theorem \ref{thm:equiv_displays} intertwines duality for $O$-displays and Faltings duality for $p$-divisible $O$-modules (Proposition \ref{prop:Faltings_display_compatible}).

\subsection{Faltings duality of $p$-divisible $O$-modules}
\label{ss:Faltings}

 Fix a \emph{Lubin--Tate object} $\LT$ over $R$, that is, a $1$-dimensional $p$-divisible $O$-module whose underlying $p$-divisible group has height $[K:\mbQ_p]$. Faltings \cite{Faltings} defined the dual of a $p$-divisible $O$-module by considering its truncations and maps to $\LT$. This construction is completely analogous to that of the Serre dual of a $p$-divisible group. The purpose of this section is to recall Faltings' definitions for later reference. A more comprehensive presentation can be found in \cite{Hartl_Singh}.

\begin{defn}\label{def:small_deformations} Let $R$ be a ring.\\
\smallskip \noindent (1) We denote the augmentation map of an augmented $R$-algebra $A$ by $ε_A:A\twoheadrightarrow R$. We denote the augmentation ideal by $I_A := \ker(ε_A)$. Maps of augmented $R$-algebras are always understood to be compatible with the augmentation.

\smallskip \noindent (2) A \emph{small deformation} of an augmented $R$-algebra $A$ is a surjection $φ:B\twoheadrightarrow A$ of augmented $R$-algebras such that $I_B\cdot \ker(φ) = 0$. In particular, $φ$ is a square-zero thickening. Maps of small deformations of $A$ are always understood to be compatible with the surjection to $A$.

\smallskip \noindent (3) Assume that $A$ is a flat, augmented $R$-algebra that is a relative complete intersection. A small deformation $φ:B\twoheadrightarrow A$ of $A$ is called \emph{versal} if $B$ is a finite presentation $R$-algebra and if $φ$ is projective in the sense that for every small deformation $C\twoheadrightarrow A$ and every surjection $C\twoheadrightarrow B$ of deformations, there exists a section $B\to C$.  
\end{defn}

Whenever we speak of versal small deformations in the following, we implicitly assume that the assumptions on $A$ from (3) are satisfied. With this convention, versal small deformations always exist: Let $R[T_1,\ldots, T_d]/I \simto A$ be any presentation of $A$ as $R$-algebra. Let $I_0 \subset R[T_1,\ldots, T_d]$ be the augmentation ideal that comes from composition with $ε_A$. Then $R[T_1, \ldots, T_d]/(I_0\cdot I)$ is a versal small deformation of $A$.

\begin{lem}\label{lem:exist_versal_covering}
Let $B_i \twoheadrightarrow A$, for $i = 1,2$, be two versal small deformations of $A$. Then there exists a versal small deformation $C\twoheadrightarrow A$ together with surjections $C\twoheadrightarrow B_1$ and $C\twoheadrightarrow B_2$.
\end{lem}
\begin{proof}
Let $e \geq 0$ be such that each kernel $\ker(B_i \twoheadrightarrow A)$, $i = 1,2$, is generated by $\leq e$ elements as $A$-module. Also choose a surjection $ψ_0:R[T_1,\ldots,T_d] \twoheadrightarrow A$ as before. Extend $ψ_0$ to a map
$$ψ: S := R[T_1,\ldots,T_d, X_1,\ldots, X_e] \twoheadrightarrow A$$
by defining $ψ(X_k) = 0$ for $0\leq k \leq e$. Let $I = \ker(ψ)$ and define $C := S/(I\cdot I_S)$. Then there exist surjections of deformations $C\twoheadrightarrow B_i$ by lifting the images $ψ_0(T_j)\in A$ to $B_i$ and by mapping the $X_k$ to generators of $\ker(B_i\twoheadrightarrow A)$.
\end{proof}

\begin{defn}\label{def:cotangent}
The \emph{co-Lie complex} of a versal small deformation $B \twoheadrightarrow A$ with kernel $I$ is defined as the two term complex
\begin{equation}\label{eq:def_cotangent}
L^*_{B\twoheadrightarrow A} = [I/(I_B\cdot I) \lr I_B/(I_B)^2].
\end{equation}
\end{defn}

The co-Lie complex is independent of $B$ up to homotopy equivalence: By Lemma \ref{lem:exist_versal_covering}, it suffices to check this for surjections $φ:C\twoheadrightarrow B$ of versal small deformations. Choose a section $s:B\to C$ which exists by definition of versality. Then $C$ decomposes as $B \oplus \ker(φ)$ with multiplication
\begin{equation}\label{eq:decom_C_B}
(b_1, n_1) \cdot (b_2, n_2) = (b_1b_2, b_1n_2 + b_2n_1).
\end{equation}
Let $I = \ker(B\twoheadrightarrow A)$ and $J = \ker(C\twoheadrightarrow A)$. Then $J = I \oplus \ker(φ)$. Moreover, $I_C = I_B \oplus \ker(φ)$ and $I_B\cdot \ker(φ) = 0$, so we obtain
\begin{equation}\label{eq:cotangent_compare}
\begin{aligned}
L^*_{C\twoheadrightarrow A} & = \big[\,I/(I_B\cdot I) \oplus \ker(φ) \ \lr\ I_B/(I_B)^2 \oplus \ker(φ)\,\big]\\[1mm]
& = L^*_{B\twoheadrightarrow A} \oplus \big[\ker(φ)\, \overset{\mr{id}}{\lr}\, \ker(φ)\big].
\end{aligned}
\end{equation}
which is homotopy equivalent to $L^*_{B\twoheadrightarrow A}$.

We now specialize to the case when $A$ is the coordinate ring of an affine flat commutative group scheme $G$ of finite presentation over $R$; these are always relative complete intersections by \cite[\S2]{Faltings}. Let $G\hookrightarrow G^\flat = \Spec(A^\flat)$ be a versal small deformation of $G$. Denote by $\End(G^\flat, G)$ the set of $R$-scheme endomorphisms of $G^\flat$ that restrict to a group scheme endomorphism on $G\subseteq G^\flat$. Faltings explains that $\End(G^\flat, G)$ again forms a commutative group. Moreover, he explains that the restriction map $\End(G^\flat, G) \twoheadrightarrow \End(G)$ is surjective. Clearly, any element of $\End(G^\flat, G)$ defines an $R$-linear endomorphism of the co-Lie complex $L^*_{A^\flat \twoheadrightarrow A}$. 

\begin{defn}\label{def:O_action_Faltings}
Assume that $R$ is an $O$-algebra. Let $G\hookrightarrow G^\flat$ be a versal small deformation of an affine flat commutative $R$-group scheme of finite presentation as before. An $O$-action $O\to \End(G^\flat, G)$ is called \emph{strict} if the induced action of $O$ on the two terms of $L^*_{G\hookrightarrow G^\flat}$ is by scalar multiplication along the structure map $O\to R$.
\end{defn}

\begin{lem}[\protect{\cite[Remark \S2 (b)]{Faltings}}]\label{lem:O_action_universal}
Let $O\to \End(G^\flat, G)$ be a strict $O$-action as in Definition \ref{def:O_action_Faltings}. Then there is a unique way to define a strict $O$-action on every versal small deformation of $G$ such that every map between such deformations is $O$-linear.
\end{lem}
\begin{proof}
Let $φ:C\twoheadrightarrow B$ be a surjection of versal small deformations of $A = Γ(G, \mcO_G)$. Choose a section of $φ$ to decompose $C = B\oplus \ker(φ)$ as in \eqref{eq:decom_C_B}.

Assume we are given a strict $O$-action on $\Spec(C)$. Then, by definition of strictness and by \eqref{eq:cotangent_compare}, $O$ acts by scalar multiplication on $\ker(φ)$. In particular, the action preserves $\ker(φ)$ and hence descends to $B$. The descended action is strict because $L^*_{B\twoheadrightarrow A}$ is a direct summand of $L^*_{C\twoheadrightarrow A}$, see again \eqref{eq:cotangent_compare}.

Conversely, assume that we are given a strict $O$-action on $\Spec(B)$. The explicit formula \eqref{eq:decom_C_B} allows to extend it to a strict $O$-action on $\Spec(C)$ by letting $O$ act by scalar multiplication on $\ker(φ)$.

Assume from now on that $B$ and $C$ are endowed with compatible strict $O$-actions as above, and let $φ':C\to B$ be a second homomorphism of deformations of $A$. The difference $φ-φ'$ is described by a map $I_C/(I_C)^2 \to I/(I_B\cdot I)$, where $I = \ker(B\twoheadrightarrow A)$. By the strictness, the $O$-action on source and target of this map is by scalar multiplication. Hence $φ-φ'$ is $O$-linear. So since $φ$ is $O$-linear, also $φ'$ is $O$-linear.

Using Lemma \ref{lem:exist_versal_covering}, it follows that the constructions from the beginning of the proof provides the unique way to extend the strict $O$-action from $G^\flat$ to all versal small deformations.
\end{proof}

\begin{defn}[\protect{\cite[Definition 1]{Faltings}}]
A \emph{strict $O$-action} on an affine, flat, commutative $R$-group scheme $G$ of finite presentation is the datum of compatible strict $O$-actions on all versal small deformations of $G$ in the sense of Lemma \ref{lem:O_action_universal}.

A homomorphism $G\to H$ of group schemes with strict $O$-actions is called \emph{strict} if it lifts to an $O$-linear map $G^\flat \to H^\flat$ for any (equivalently, one) pair $(G^\flat, H^\flat)$ of versal small deformations of $G$ and $H$.
\end{defn}

Assume that $X_1 \to X_2 \to \ldots$ is a sequence of strict $O$-group schemes over an $O$-algebra $R$ in which $p$ is nilpotent. Assume that there is an integer $h$ such that each $X_n$ is finite and locally free of degree $q^{nh}$, and that the transition maps are strict closed immersions with $X_n = X_{n+1}[π^n]$. Then $X = \colim_{n\geq 1}\,X_n$ is a $p$-divisible $O$-module. The integer $h$ is called its \emph{height}. The height of the underlying $p$-divisible group is $[K:\mbQ_p]\cdot h$.

Conversely, let $X$ be a $p$-divisible $O$-module over $R$. We claim that each truncation $X_n := X[π^n]$ is naturally a strict $O$-group scheme with strict transition maps $X_n\to X_{n+1}$.

\begin{lem}\label{lem:truncation_p_div_strict}
(1) Assume $π^mR = 0$. Define $X_n^\flat$ as the maximal small deformation of $X_n$ in $X_{n+m}$. Then $X_n^\flat$ is a versal small deformation of $X_n$.

\medskip \noindent (2) The $O$-action on $X_{n+m}$ reduces to a strict action on $X_n^\flat$ and hence endows $X_n$ with the structure of a strict $O$-group scheme.
\end{lem}
\begin{proof}
(1) Let $X_n^\flat \hookrightarrow Y$ be any closed immersion into a small deformation $Y$ of $X_n$. It is well-known that $p$-divisible groups are formally smooth functors, so there exists an extension of $ι:X_n^\flat \hookrightarrow X$ to a map $Y\to X$. Since $Y$ is a square-zero thickening of $X_n$ and since multiplication by $π^m$ annihilates the Lie algebra of $X$ by assumption, the composition $[π^m]\circ ι$ factors through $X_n$. This means that $ι$ factors through a map $Y\to X_{n+m}$. Since $Y$ is a small deformation of $X_n$, its image is contained in $X_n^\flat$. Thus $ι$ provides a splitting $Y\to X_n^\flat$ as desired.

(2) The $O$-action on $X_{n+m}$ preserves $X_n$, and hence restricts to an action on $X_n^\flat$. We need to show that this action is strict. Introduce the following notation: Given an affine flat commutative group scheme $G$ over $R$, put
$$n_G := H^{-1}(L^*_{G\hookrightarrow G^\flat}),\quad ω_G := H^0(L^*_{G\hookrightarrow G^\flat}).$$
These do not depend on the choice of versal small deformation $G^\flat$. Moreover, $ω_G$ is nothing but the co-Lie algebra $e^*\Omega^1_{G/R}$.

First assume that $n \geq m$. The second term of $L^*_{X_n\hookrightarrow X_n^\flat}$ is the cotangent space $e^*\Omega^1_{X_n^\flat/R}$. Under our assumption on $n$, it coincides with the co-Lie algebra $ω_{X_n} = ω_X$ of $X$. The $O$-action on $ω_X$ is the natural one by the definition of $p$-divisible $O$-modules. Moreover, since $e^*\Omega^1_{X_n^\flat/R} = ω_{X_n}$, we see that
$$L^*_{X_n\hookrightarrow X_n^\flat} = [n_{X_n} \overset{0}{\lr} ω_{X_n}].$$
Consider the short exact sequence $0\to X_n \to X_{n+m} \overset{π^n}{\to} X_m\to 0$. By a statement of Messing, see \cite[Lemma 3.6]{Hartl_Singh}, it induces an exact sequence
\begin{equation}\label{eq:les_coLie}
0\lr n_{X_m} \lr n_{X_{n+m}} \lr n_{X_n} \lr ω_{X_m} \overset{0}{\lr} ω_{X_{n+m}} \simlr ω_{X_n}\lr 0.
\end{equation}
The terms of $L^*_{X_n\hookrightarrow X_n^\flat}$ are locally free as $R$-modules and of equal rank because $X_n$ is of relative dimension $0$. We obtain from \eqref{eq:les_coLie} that $n_{X_n}\simto ω_{X_m}$ and hence that the $O$-action on the first term of $L^*_{X_n\hookrightarrow X_n^\flat}$ is the natural one as well. This shows that the $O$-action on $X_n^\flat$ is strict.

For general $n$, we consider the multiplication map $[π^n]:X_{n+m}\to X_m$. It extends to $[π^n]:X_{n+2m}\to X_{2m}$ and hence restricts to an $O$-linear map $X_{n+m}^\flat\to X_m^\flat$. By \cite[Proposition 2]{Faltings}, there is a unique strict $O$-action on $X_n = \ker([π^n])$ such that the inclusion $X_n\to X_{n+m}$ is strict. Since $X_n^\flat$ includes into $X_{n+m}^\flat$, this $O$-action has to be the one we already defined.
\end{proof}

It is clear from construction that the transition maps $X_n\to X_{n+1}$ are strict. So we obtain that every $p$-divisible $O$-module $X$ is of the form $X = \colim_{n\geq 1} X_n$ for a sequence $X_1\to X_2\to \ldots$ of strict $O$-group schemes as before.

Let $\LT$ be a Lubin--Tate object over $R$, that is, a $p$-divisible $O$-module of dimension $1$ whose underlying $p$-divisible group has height $[K:\mbQ_p]$.

\begin{thm}[\protect{\cite[Theorem 8]{Faltings}}]\label{thm:Faltings}
Let $G$ be a finite, locally free strict $O$-group scheme over $R$. Then the functor $\underline{\Hom}_O^{\mr{strict}}(G, \LT)$ of strict group scheme homomorphisms to $\LT$ is representable by a finite, locally free strict $O$-group scheme $G^\vee$. It is of the same rank as $G$ and called the Cartier dual of $G$ (with respect to $\LT$).

Cartier duality defines a self anti-equivalence of the category of finite, locally free strict $O$-group schemes. It interchanges closed immersions and faithfully flat homomorphisms. For every $G$, the evaluation map $G\to (G^\vee)^\vee$ is an isomorphism.
\end{thm}
\begin{proof}
Faltings proves this for a specific choice of $\LT$. Since all such choices are fpqc locally (even proétale locally) isomorphic, see Lemma \ref{lem:LT_object_essentially_unique} and Theorem \ref{thm:equiv_displays}, his statement holds with any Lubin--Tate group.
\end{proof}

\begin{defn}[Faltings]\label{def:dual}
The dual of a $p$-divisible $O$-module $X$ (with respect to $\LT$) is defined as $X^\vee := \colim_{n\geq 1}\,\,X[π^n]^\vee$ where the transition maps are $[π]^\vee:X[π^n]^\vee\to X[π^{n+1}]^\vee$. It is again a $p$-divisible $O$-module.

Dualization defines a self anti-equivalence of the category of $p$-divisible $O$-modules. It is an involution in the sense that the natural isomorphisms $X[π^n]\simto (X[π^n]^\vee)^\vee$ define a natural isomorphism $X\simto (X^\vee)^\vee$.
\end{defn}

\begin{lem}\label{lem:Faltings_dual}
Let $R$ be an $O$-algebra in which $p$ is nilpotent, and let $X$ and $Y$ be two $p$-divisible $O$-modules over $R$. Let
$$f = (f_n)_{n\geq 1}: X = \colim_{n\geq 1} X[π^n] \lr \colim_{n\geq 1}\, \underline{\Hom}(Y[π^n], \LT[π^n])$$
be an $O$-linear homomorphism of group valued functors. Then each $f_n$ is strictly $O$-linear. That is, $f$ defines a homomorphism $f:X\to Y^\vee$.
\end{lem}
\begin{proof}
Fix $n$, and a schematic point $x\in X[π^n]$, our aim being to show that $f_n(x)$ is strictly $O$-linear. After base change we may assume $x$ to be $R$-valued. Let $m$ be such that $p^m R = 0$. Locally for the fppf topology, choose a point $x' \in X[π^{n+m}]$ such that $π^mx' = x$. Then $f_{n+m}(x') \in \underline{\Hom}(Y[π^{n+m}], \LT[π^{n+m}])$ is an $O$-linear map that extends $f_n(x)$. In the notation of Lemma \ref{lem:truncation_p_div_strict}, it restricts to an $O$-linear map $Y[π^n]^\flat \to \LT[π^n]^\flat$ which means that $f_n(x)$ is strict as claimed.
\end{proof}

\subsection{$\BT$ and duality}
\label{ss:BT_compatibility}
For $O = \mbZ_p$, it was proved by Zink \cite[\S4]{Zink} that $\BT$ is compatible with Serre duality. For general $O$, a sketch of this compatibility over complete Noetherian local rings with perfect residue field is given in \cite[\S5]{ACZ}. Our aim is to extend this statement to arbitrary bases, see Proposition \ref{prop:Faltings_display_compatible}. Our proof will be self-contained, taking as input only that the functor $\BT$ from Theorem \ref{thm:equiv_displays} is defined. It is a simplified (but essentially equivalent) version of the arguments in \cite[\S4]{Zink}.

\newcommand{\Nil}{\mcN il}

We begin by recalling the construction of $\BT$. To this end, note that the truncations of a formal $p$-divisible group over a ring $R$ are pointed, affine, finite type, infinitesimal thickenings of $\Spec R$. Let $\Nil_R$ be the opposite category of such thickenings, that is, the category of augmented finite type $R$-algebras with nilpotent augmentation ideal. By the Yoneda Lemma, in order to define or characterize a formal group over $R$, it suffices to describe its functor of points on $\Nil_R$. Observe that every object of $\Nil_R$ is of the form $R\oplus \mcN$ as $R$-module, where $\mcN$ denotes the augmentation ideal.

Let $\mcP = (P, Q, \bfF, \dF)$ be a nilpotent $O$-display over $R$. Following \cite{Zink} and \cite{ACZ}, define two functors $\wh P$ and $\wh Q$ on $\Nil_R$ as follows. First, for every $R\oplus \mcN \in \Nil_R$, denote by $\wh W_O(\mcN)\subseteq W_O(R\oplus \mcN)$ all those Witt vectors that only have finitely many non-zero entries all of which lie in $\mcN$. This is an ideal of $W_O(R\oplus \mcN)$. Then define
\begin{equation}\label{eq:def_P_hat}
\wh P(\mcN) := \wh W_O(\mcN)\tensor_{W_O(R)} P.
\end{equation}
Define $\wh Q(\mcN)\subseteq \wh P(\mcN)$ as the submodule generated by $V\wh P(\mcN)$ and the image of $\wh W_O(\mcN)\tensor_{W_O(R)} Q$. Define the map
\begin{equation}
\begin{aligned}
    \dF: \wh Q(\mcN) & \lr \wh P(\mcN)\\[2mm]
    V(ξ)\tensor x & \longmapsto {\mkern 24mu} ξ \tensor \bfF(x),\\
    ξ \tensor y & \longmapsto σ(ξ) \tensor \dF(y).
\end{aligned}
\end{equation}
Here, $ξ\in \wh W_O(\mcN)$, $x\in \wh P(\mcN)$ and $y\in \wh Q(\mcN)$. Then \cite{ACZ} defines $\BT(\mcP)$ by the formula
\begin{equation}\label{eq:def_BT}
    \BT(\mcP) := \mr{coker}\big[\dF - \mr{id}:\wh Q\lr \wh P\big].
\end{equation}
Note that $\dF - \mr{id}$ is injective because $\dF$ is $σ$-linear and because we only consider nilpotent $\mcN$. In particular, for each $R\oplus \mcN \in \Nil_R$, there is a short exact sequence
$$0 \lr \wh Q(\mcN) \overset{\dot \bfF - \mr{id}}{\lr} \wh P(\mcN) \lr \BT(\mcP)(\mcN)\lr 0.$$
Assume that the dual $O$-display $\mcP^\vee$ is nilpotent as well. We call such $\mcP$ \emph{binilpotent}. Since $\mcP^\vee$ is defined in terms of maps to $(W, VW, σ, V^{-1})$ it is natural to choose the Lubin--Tate object for dualization in Faltings' sense as
\begin{equation}\label{eq:def_LT}
\LT := \BT\left((W, VW, σ, V^{-1})\right)\ \overset{(\star)}{=}\ \wh W/ (V^{-1}-\mr{id})(V\wh W).
\end{equation}
Here, $(\star)$ is just the explicit description of $\LT$ from \eqref{eq:def_BT}. At this point, we have defined the two $p$-divisible $O$-module $\BT(\mcP^\vee)$ and $\BT(\mcP)^\vee$, and our next aim is to construct a homomorphism $α_\mcP:\BT(\mcP^\vee) \to \BT(\mcP)^\vee$. Our definition is by the standard translation from the biextension formalism used in \cite[\S4]{Zink} to that of bilinear pairings. (This formalism is due to Mumford \cite{Mumford}.) Most notably, the key formula \eqref{eq:def_alpha} below is taken from \cite[Equation (219)]{Zink}.
\begin{construction}\label{construction:alpha}
Let $\ell \in \BT(\mcP^\vee)[π^n]$ and $x\in \BT(\mcP)[π^n]$ be $π^n$-torsion points. Choose liftings $\wt \ell \in \wh{P^\vee}$ and $\wt x\in \wh P$. Then there exist unique elements $λ\in \wh{Q^\vee}$ and $y\in \wh Q$ such that
\begin{equation}\label{eq:y_lambda}
π^n \wt \ell = \dot \bfF^\vee(λ) - λ,\qquad π^n \wt x = \dot \bfF (y) - y.
\end{equation}
Define
\begin{equation}\label{eq:def_alpha}
(α_{\mcP, n}(\ell))(x)\ :=\ \text{class of }\ (\dot \bfF^\vee(λ))(\wt x) + \wt \ell(y)\ \in\ \LT.\medskip
\end{equation}
Here, the class in $\LT$ is meant in the sense of $(\star)$ in \eqref{eq:def_LT}. We claim that $α_{\mcP, n}$ is well-defined: Let $\wt x' = \wt x + (\dot F(z) - z)$ be another choice of lifting of $x$. Then $π^n\wt x' = \dot F(y') - y'$ for $y' = y + π^nz$. We obtain
\begin{equation}\label{eq:indep_choices}
\begin{aligned}
\big[(\dot \bfF^\vee(λ))(\wt x') + \wt \ell & (y') \big] - \big[(\dot \bfF^\vee(λ))(\wt x) + \wt \ell(y)\big]\\[1mm]
& \overset{\phantom{\eqref{eq:dual_display_key}}}{=} (\dot \bfF^\vee(λ))(\dot \bfF(z) - z) + \wt \ell(π^n z) \\[1mm]
& \overset{\phantom{\eqref{eq:dual_display_key}}}{=} (\dot \bfF^\vee(λ))(\dot \bfF(z) - z) + (\dot \bfF^\vee(λ) - λ)(z) \\[1mm]
& \overset{\phantom{\eqref{eq:dual_display_key}}}{=} (\dot \bfF^\vee(λ))(\dot \bfF(z)) - λ(z)\\[1mm]
& \overset{\eqref{eq:dual_display_key}}{=} (V^{-1} - \mr{id})(λ(z))\\[1mm]
& \overset{\phantom{\eqref{eq:dual_display_key}}}{=} 0\ \text{in}\ \LT.
\end{aligned}
\end{equation}
Independence of \eqref{eq:def_alpha} from the choice of $\wt \ell$ is shown by the same arguments. In this way, we have constructed an element
$$α_{\mcP, n} \in \Hom(\BT(\mcP)[π^n], \LT[π^n]).$$
By definition, the $α_{\mcP,n}$ are compatible in the following sense. For a pair $\ell\in \BT(\mcP)[π^n]$ and $x\in \BT(\mcP)[π^{n+m}]$,
$$(α_{\mcP,n+m}(\ell))(x) = (α_{\mcP, n}(\ell))(π^mx).$$
This means that the family of maps $(α_{\mcP, n})$ assembles to a homomorphism
$$α_\mcP:\BT(\mcP^\vee)\lr \colim_{n\geq 1}\,\Hom(\BT(\mcP)[π^n], \LT[π^n]).$$
By Lemma \ref{lem:Faltings_dual}, this map factors through the dual $p$-divisible $O$-module, that is, we have defined a homomorphism
\begin{equation}\label{eq:resulting_alpha}
α_\mcP:\BT(\mcP^\vee) \lr \BT(\mcP)^\vee
\end{equation}
as desired. It is clear from construction that the definition of $α_\mcP$ is compatible with base change in $R$.
\end{construction}

\begin{lem}\label{lem:symmetry}
The family $(α_\mcP)_{\mcP}$ defines a natural transformation $α:\BT\circ \vee \to \vee\circ \BT$. That is, for every homomorphism $f:\mcP_1\to \mcP_2$ of binilpotent $O$-displays, the following square commutes:
\begin{equation}\label{eq:square_1}
\xymatrix{
\BT(\mcP_2^\vee) \ar[rr]^-{\BT(f^\vee)} \ar[d]_{α_{\mcP_2}} & & \BT(\mcP_1^\vee) \ar[d]^{α_{\mcP_1}}\\
\BT(\mcP_2)^\vee \ar[rr]_-{\BT(f)^\vee} & & \BT(\mcP_1)^\vee.
}
\end{equation}
It is compatible with the evaluation isomorphisms $\mr{ev}_\mcP:\mcP\simto (\mcP^\vee)^\vee$ and $\mr{ev}_X:X\simto (X^\vee)^\vee$. That is, for every nilpotent $O$-display such that also $\mcP^\vee$ is nilpotent, the following square commutes:
\begin{equation}\label{eq:square_2}
\xymatrix{
\BT(\mcP) \ar[rr]^-{\BT(\mr{ev}_\mcP)} \ar[d]_{\mr{ev}_{\BT(\mcP)}} & & \BT((\mcP^\vee)^\vee) \ar[d]^{α_{\mcP^\vee}}\\
(\BT(\mcP)^\vee)^\vee \ar[rr]_-{α_{\mcP}^\vee} & & \BT(\mcP^\vee)^\vee.
}
\end{equation}
In particular, if $λ:\mcP\to \mcP^\vee$ is a (skew-)symmetric homomorphism of $O$-displays, then
$$
α_\mcP\circ \BT(λ):\BT(\mcP) \lr \BT(\mcP)^\vee
$$
is a (skew-)symmetric homomorphism of $p$-divisible $O$-modules.
\end{lem}
\begin{proof}
The commutativity of \eqref{eq:square_1} is immediate from the definition of $α$; we omit this detail and focus on \eqref{eq:square_2}. Let $x\in \BT(\mcP)[π^n]$ and $\ell\in \BT(\mcP^\vee)[π^n]$ be two schematic $π^n$-torsion points. We need to check the identity
\begin{equation}\label{eq:alpha_BT_check}
[(α_{\mcP^\vee} \circ \BT(\mr{ev}_\mcP))(x)](\ell) = [(α_\mcP^\vee\circ \mr{ev}_{\BT(\mcP)})(x)](\ell).
\end{equation}
Let $\wt x$ and $π^n\wt x = \dF(y) - y$ as well as $\wt \ell$ and $π^n\wt \ell = \dF^\vee(λ) - λ$ be as in \eqref{eq:y_lambda}. Directly by definitions, the left hand side of \eqref{eq:alpha_BT_check} is
$$((\dF^{\vee\vee})(\mr{ev}_\mcP(y))(\wt \ell) + \mr{ev}_\mcP(\wt x)(λ).$$
This is equal to $\wt \ell(\dF(y)) + λ(\wt x)$ because $\mr{ev}_\mcP$ is a display homomorphism. Again by definitions, the right hand side of \eqref{eq:alpha_BT_check} equals $(\dF^\vee(λ))(\wt x) + \wt \ell(y)$. The difference with the left hand side quantity vanishes as desired:
\begin{equation}\label{eq:symmetry_duality_definition}
\begin{aligned}
\big[(\dot \bfF^\vee(λ))(\wt x) + \wt \ell & (y)\big] - \big[λ(\wt x) + \wt \ell(\dot \bfF(y))\big]\\[1mm]
& = (\dot \bfF^\vee(λ) - λ)(\wt x) + \wt \ell(y - \dot \bfF(y))\\[1mm]
& = (π^n \wt \ell)(\wt x) + \wt \ell( - π^n \wt x)\\[1mm]
& = 0.
\end{aligned}
\end{equation}
This proves commutativity of \eqref{eq:square_2} and we come to the last claim: Assume that $λ:\mcP\to \mcP^\vee$ satisfies $λ^\vee\circ \mr{ev}_\mcP = \pm λ$. Then we obtain
\begin{equation}\begin{array}{rcl}
    (α_\mcP\circ \BT(λ))^\vee \circ \mr{ev}_{\BT(\mcP)} & = &\BT(λ)^\vee \circ α_\mcP^\vee \circ \mr{ev}_{\BT(\mcP)}\\[1mm]
    & \overset{\eqref{eq:square_2}}{=} & \BT(λ)^\vee \circ α_{\mcP^\vee} \circ \BT(\mr{ev}_\mcP)\\[1mm]
    & \overset{\eqref{eq:square_1}}{=} & α_\mcP \circ \BT(λ^\vee \circ \mr{ev}_\mcP)\\[1mm]
    & = & \pm (α_\mcP \circ \BT(λ)). 
\end{array}
\end{equation}
\end{proof}

\begin{prop}\label{prop:Faltings_display_compatible}
For each binilpotent $O$-display $\mcP$, the homomorphism $α_\mcP:\BT(\mcP^\vee)\to \BT(\mcP)^\vee$ is an isomorphism.
\end{prop}
\begin{proof}
\emph{Step 1. Reduction to specific displays over an algebraically closed field.} We know a priori that $\BT(\mcP^\vee)$ and $\BT(\mcP)^\vee$ are $p$-divisible $O$-modules. So it is enough to show that $α_\mcP$ is an isomorphism after base change to every geometric point $R \to k$. (Recall that $p$ is nilpotent in $R$, so $k$ is of characteristic $p$.) Since the construction of $α_\mcP$ commutes with base change, we may then simply assume that $R = k$ is an algebraically closed field.
\begin{rmk}
At this point, at least if $p \geq 3$, \cite[Theorem 5.9]{ACZ} directly implies Proposition \ref{prop:Faltings_display_compatible}. Since the proof of this theorem is only sketched and refers to \cite{Lau_dual}, we continue to give a self-contained argument.
\end{rmk}
Let $f:\mcP_1 \to \mcP_2$ be an isogeny of degree $d$ of binilpotent $O$-displays. Then both horizontal arrows in \eqref{eq:square_1} are isogenies of degree $d$, so $α_{\mcP_1}$ is an isomorphism if and only if $α_{\mcP_2}$ is an isomorphism. Since we are working over an algebraically closed field, by Dieudonné theory and additivity of all involved functors, we can hence assume that $\mcP$ is isoclinic of some slope $d/h$. By the same logic, it even suffices to prove the isomorphism property of $α_\mcP$ for a \emph{single} $O$-display $\mcP$ of slope $d/h$, for every possible $(d,h)$.

\emph{Step 2. Verification for a specific isoclinic $O$-display of dimension $d$ and height $h$.}
Let $(W, I) := (W_O(k), I_O(k))$ denote the ring of Witt vectors of $k$ and the Verschiebung ideal. Define a display $\mcP = (P, Q, \bfF, \dot \bfF)$ as follows. Let $P$ be the $W$-module $W^{\oplus h}$ with submodule $Q = I^{\oplus d} \oplus W^{\oplus (h-d)}$. Define the Frobenius $\dF:Q\to P$ by
\begin{equation}\label{eq:Frob_hd}
\dot \bfF := \begin{pmatrix}
0 & & & & & & σ\\
V^{-1} & \ddots & & & & &\\
& \ddots & \ddots & & & &\\
& & V^{-1} & \ddots & & &\\
& & & σ & \ddots & &\\
& & & & \ddots & \ddots &\\
& & & & & σ & 0
\end{pmatrix}.
\end{equation}
The datum $\bfF$ is determined by \eqref{eq:Frobenius_determined} and will not be used in the following. We claim that for this specific example, $α_\mcP$ is an isomorphism. To prove this, we first note that $\BT(\mcP^\vee)$ and $\BT(\mcP)^\vee$ are both formal groups of the same dimension. It is thus enough to show that $α_\mcP$ is injective on tangent spaces. Our proof is by making this explicit.

Consider the dual basis for $P^\vee = \Hom(P, W)$. It provides the $W$-module descriptions $P^\vee = W^{\oplus h}$ and $Q^\vee = W^{\oplus d}\oplus I^{\oplus (h-d)}$. In the dual basis, the dual Frobenius $\dot \bfF^\vee:Q^\vee \to P^\vee$ is given by
$$\dot \bfF^\vee := \begin{pmatrix}
0 & & & & & & V^{-1}\\
σ & \ddots & & & & &\\
& \ddots & \ddots & & & &\\
& & σ & \ddots & & &\\
& & & V^{-1} & \ddots & &\\
& & & & \ddots & \ddots &\\
& & & & & V^{-1} & 0
\end{pmatrix}.$$
Consider the dual numbers $k[ε]/(ε)^2$ with nilradical $\mcN = (ε)$. The tangent space $\BT(\mcP^\vee)(k[ε]/(ε)^2)$ is by definition of the functor $\BT$ the quotient in the exact sequence
$$0 \lr \wh {Q^\vee}(\mcN) \overset{\dot \bfF^\vee - \mr{id}}{\lr} \wh{P^\vee}(\mcN) \lr \BT(\mcP^\vee)(\mcN)\lr 0.$$
Since $\dot \bfF^\vee$ is $σ$-linear and $\mcN^2 = 0$, the first map is just the identity map $\wh {Q^\vee}(\mcN)\to \wh {P^\vee}(\mcN)$. Using the Teichmüller map
\begin{equation}\label{eq:Teichmuller}
\mcN\lr \wh W_O(\mcN),\quad n\longmapsto (n, 0, 0,\ldots),
\end{equation}
we obtain
$$0^{\oplus d} \oplus \mcN^{\oplus (h-d)} \simlr \wh {P^\vee}(\mcN)/\wh {Q^\vee}(\mcN) = \BT(\mcP^\vee)(k[ε]/(ε)^2).$$
Consider an arbitrary tangent vector
$$\ell = (0, \ldots, 0,n_{d+k}, n_{d+k+1},\ldots, n_h) \in 0^{\oplus d} \oplus \mcN^{\oplus (h-d)}.$$
Here, the indexing is chosen such that $1 \leq k \leq h-d$ and $n_{d+k} \neq 0$. We claim that $α_\mcP(\ell) \neq 0$. More precisely, we claim that there exist a $k[ε]/(ε)^2$-algebra $S$ and an $S$-valued point $x\in \BT(\mcP)[π]$ such that $\ell(x) \in \LT[π](R)$ is non-vanishing. One construction of such a point is as follows. 

Set $S = k[ε, η^{1/q^h}]/(ε^2, η^2)$ with nilradical $\mcU = (ε, η^{1/q^h})$. Consider the following element $y$ of $\wh Q(\mcU)$ where the notation means Witt vector coordinates with respect to $\wh P(\mcU) = \wh W_O(\mcU)^h$:
\begin{equation}\label{eq:point_y}
\begin{array}{rll}
    y = & (0, \ldots, 0, &\ \ \text{$d-1$ vanishing entries}\\[2mm]
    &(0, η^{1/q^{k-1}}, 0, 0, \ldots),&\ \ \text{$d$-th entry}\\[1mm]
    &(η^{1/q^{k-1}}, 0, 0,  \ldots),&\ \ \text{$(d+1)$-st entry}\\[1mm]
    &(η^{1/q^{k-2}}, 0, 0, \ldots),&\\
    &\ \ \ \vdots&\\
    &(η, 0, 0, \ldots),&\ \ \text{$(d+k)$-th entry}\\[2mm]
    &0, \ldots, 0)&\ \ \text{$h-d-k$ vanishing entries}.
\end{array}
\end{equation}
Then $\dot \bfF(y) - y \in V \wh P(\mcU)$. Note that since $π$ is zero in $S$, multiplication by $π$ on $W_O(S)$ is given by $π\cdot (w_0,w_1,\ldots) = (0, w_0^q, w_1^q, \ldots)$. Since $y$ actually has coordinates in the subring $k[η^{1/q^{h-1}}]/(η^2)$ of $S$, and since every element of this subring has a $q$-th root in $S$, there exists an element $\wt x\in \wh P(\mcU)$ such that $π \wt x = \dot \bfF (y) - y$. In particular, the image $x$ of $\wt x$ in $\BT(\mcP)(S)$ is $π$-torsion.

We claim that $(α_{\mcP,1}(\ell))(x) \neq 0$. In order to check this, we simply evaluate the definition in \eqref{eq:def_alpha}. Our elements $\wt x$ and $y$ are match the terminology there. We choose the lifting $\wt \ell \in P^\vee(\mcN)$ as the Teichmüller lifting of $\ell$ in $\wh W_O^{\oplus h}$. Since $\mcN^2 = 0$ and since multiplication by $π$ is $π\cdot (w_0, w_1, \ldots) = (0, w_0^q, w_1^q, \ldots)$, we have $π\cdot \wt \ell = 0$. In the notation of \eqref{eq:def_alpha}, this means $λ = 0$. So we find
$$(α_{\mcP,1}(\ell))(x) = \wt \ell(y) = (εη, 0, 0, \ldots)$$
which clearly does not lie in $(V^{-1} - \mr{id})(V\wh W_O)$ and hence defines a non-zero $S$-point of $\LT$. This finishes the proof of Proposition \ref{prop:duality-comp}.
\end{proof}

\begin{rmk}\label{rmk:the_real_dual}
Proposition \ref{prop:Faltings_display_compatible} also helps to clarify previous results, such as those in \cite{Mih} and \cite{KRZ}. Namely, since the relation of $O$-display duality with Faltings duality was not completely established at the time of their writing, these references \emph{defined} the dual $p$-divisible $O$-module in terms of the dual $O$-display, see \cite[Definition 11.9]{Mih} and \cite[after Corollary 3.4.13]{KRZ}. But this notion of duality is artificial because it depends on the choice of an inverse functor $Ψ$ to \eqref{eq:equiv_display} together with a choice of natural isomorphism $\BT\circ Ψ\simto \mr{id}$. Proposition \ref{prop:Faltings_display_compatible} makes this choice unnecessary.
\end{rmk}

\bibliographystyle{alpha}

\bibliography{reference}
\end{document}